\documentclass [leqno,11pt]{amsart}     
\usepackage{amssymb,amsmath,latexsym,amsfonts,amsbsy,amsthm,mathtools,graphicx,color,subeqnarray,cases,bm}      
\usepackage{float}   
\usepackage{hyperref}       
\usepackage{scalerel,stackengine}           
\usepackage{subcaption}                
\stackMath          
\newcommand\reallywidehat[1]{%          
\savestack{\tmpbox}{\stretchto{%   
  \scaleto{%
    \scalerel*[\widthof{\ensuremath{#1}}]{\kern-.6pt\bigwedge\kern-.6pt}%
    {\rule[-\textheight/2]{1ex}{\textheight}}%WIDTH-LIMITED BIG WEDGE
  }{\textheight}% 
}{0.5ex}}% 
\stackon[1pt]{#1}{\tmpbox}%  
}

\definecolor{myback}{RGB}{204,232,207} 
 
\usepackage{cancel}
\usepackage[dvipsnames]{xcolor}
\numberwithin{equation}{section}

\allowdisplaybreaks

\let\d=\delta

\let\f=\frac

\let\pa=\partial

\def\eqdef{\buildrel\hbox{\footnotesize def}\over =}

\newcommand{\beq}{\begin{equation}}
\newcommand{\eeq}{\end{equation}}
\newcommand{\ben}{\begin{eqnarray}}
\newcommand{\een}{\end{eqnarray}}
\newcommand{\beno}{\begin{eqnarray*}}
\newcommand{\eeno}{\end{eqnarray*}}

\newtheorem{theorem}{Theorem}[section]

\newtheorem{lemma}[theorem]{Lemma}

\newtheorem{remark}[theorem]{Remark}

\begin{document}
\title{Neutral curves and traveling waves in plane Poiseuille flow}
\author{Hui Li}
\address[H. Li]{School of Mathematics, Sichuan University, Chengdu 610064, P. R. China.}
\email{lihui92@scu.edu.cn}

\author{Yuxi Wang}
\address[Y. Wang]{School of Mathematics, Sichuan University, Chengdu 610064, P. R. China.}
\email{wangyuxi@scu.edu.cn}

\author{Zhifei Zhang}
\address[Z. Zhang]{School of  Mathematical Sciences, Peking University, Beijing 100871, P. R. China.}
\email{zfzhang@math.pku.edu.cn}

\begin{abstract}
We study the spectrum of the Orr--Sommerfeld operator associated with the incompressible plane Poiseuille flow in the high-Reynolds-number regime. In the Tollmien--Schlichting eigenvalue region, we prove the existence and uniqueness of the lower and upper branches of the neutral curve. For each sufficiently small fixed wavenumber $\alpha$, the viscosities corresponding to the lower and upper neutral branches satisfy $\nu\sim |\alpha|^7$ and $\nu\sim |\alpha|^{11}$, respectively. Equivalently, for each sufficiently small viscosity $\nu$, the corresponding lower and upper neutral wavenumbers satisfy $\alpha^2\sim \nu^{2/7}$ and $\alpha^2\sim \nu^{2/11}$, respectively. We also establish the simplicity of the neutral eigenvalues and verify the transversal crossing condition on both neutral branches. The proof is based on a boundary-adapted version of the Rayleigh--Airy iteration scheme, together with precise expansions and refined estimates for the correction terms and their parameter derivatives. These spectral results verify
the assumptions required in the classical Hopf bifurcation framework of Joseph--Sattinger \cite{JS1972} and Iooss \cite{Iooss1972}, and hence yield traveling-wave solutions bifurcating from the plane Poiseuille flow at the neutral points.
\end{abstract}
\maketitle 

\setcounter{tocdepth}{1}
{\small\tableofcontents}

\section{Introduction}
We consider the two-dimensional incompressible Navier--Stokes equations in the channel $(x,y)\in\mathbb R\times[-1,1]$ , driven by a constant body force,
\begin{equation}\label{eq-NS}
  \left\{
    \begin{array}{l}
      \pa_t U+U\cdot\nabla U-\nu\Delta U+\nabla P=2\nu \bm e_1,\\
      \nabla\cdot U=0,\\
      U(t,x,\pm1)=0,
    \end{array}
  \right.
\end{equation}
where $\bm e_1=(1,0)^\top$.

This system admits the steady solution
\begin{align*}
  U=\left(u_p(y),0\right)^\top,\quad P=0,
\end{align*}
where $u_p=1-y^2$ is the plane Poiseuille flow.

Let $\Omega=\partial_x U^{(2)}-\partial_y U^{(1)}$ be the vorticity. Taking the curl of \eqref{eq-NS} removes the pressure and gives
\begin{align*}
  \pa_t\Omega+U\cdot\nabla \Omega-\nu\Delta\Omega=0.
\end{align*} 

Writing $u=U-\left(u_p,0\right)^\top$, $\omega=\Omega-\omega_p=\Omega-2y$, we obtain the perturbation equation
\begin{align}\label{NS-pert-nonlinear}
  \pa_t\omega+u_p\pa_x \omega-u_p''\pa_x\phi-\nu\Delta\omega=-u\cdot\nabla\omega,
\end{align}
where $\phi$ is the stream function associated with $\omega$,
\begin{align*}
  \phi=\Delta^{-1}\omega, \quad u=\left(-\pa_y\phi,\pa_x\phi\right).
\end{align*}

For a nonzero wave number $\alpha$, we take the
normal-mode ansatz
\begin{align*}
  \phi(t,x,y)=\phi_\alpha(y)e^{\lambda t+i\alpha x},
\end{align*}
where  $\lambda=\lambda_r+i\lambda_i$.

The linearization of \eqref{NS-pert-nonlinear} around the plane Poiseuille flow then gives
\begin{align}\label{eq-Orr-Sommerfeld-ori}
  \left(u_p+\frac{\lambda}{i\alpha}\right)\left(\pa_y^2-\alpha^2\right)\phi_\alpha-u_p''\phi_\alpha-\frac{\nu}{i\alpha}\left(\pa_y^2-\alpha^2\right)^2\phi_\alpha=0,
\end{align}
with the no-slip boundary condition
\begin{align}\label{bc-ori}
  \phi_\alpha(\pm1)=0,\quad\phi'_\alpha(\pm1)=0.
\end{align}

For $\alpha>0$, set
\begin{align*}
  \varepsilon=\frac{\nu}{\alpha},\quad c=c_r+ic_i=\frac{-\lambda}{i\alpha}.
\end{align*}
Then \eqref{eq-Orr-Sommerfeld-ori} becomes the so-called Orr-Sommerfeld equation
\begin{align}\label{eq-Orr-Sommerfeld}
  i\varepsilon\left(\pa_y^2-\alpha^2\right)^2\phi_\alpha+\left(u_p-c\right)\left(\pa_y^2-\alpha^2\right)\phi_\alpha-u_p''\phi_\alpha=0.
\end{align}

For $\alpha<0$, by taking the complex conjugate of \eqref{eq-Orr-Sommerfeld-ori}, we have
\begin{align*}
  \left(u_p+\frac{\bar\lambda}{i|\alpha|}\right)\left(\pa_y^2-\alpha^2\right)\bar\phi_\alpha-u_p''\bar\phi_\alpha-\frac{\nu}{i|\alpha|}\left(\pa_y^2-\alpha^2\right)^2\bar\phi_\alpha=0.
\end{align*}
Then by taking
\begin{align*}
  \varepsilon=\frac{\nu}{|\alpha|},\quad c=\frac{-\bar\lambda}{i|\alpha|},
\end{align*}
we obtain
\begin{align}\label{eq-Orr-Sommerfeld-<0}
  i\varepsilon\left(\pa_y^2-\alpha^2\right)^2\bar\phi_\alpha+\left(u_p-c\right)\left(\pa_y^2-\alpha^2\right)\bar\phi_\alpha-u_p''\bar\phi_\alpha=0.
\end{align}

Thus, for both $\alpha>0$ and $\alpha<0$, the Orr--Sommerfeld equation has the same form. The case $\alpha<0$ is obtained from the case $\alpha>0$ by complex conjugation, and the corresponding eigenfunctions are conjugate to each other. 

The above Orr-Sommerfeld equation was introduced by Sommerfeld \cite{Sommerfeld1908} and Orr \cite{Orr1907} to study the stability of shear flows. It reduces the linear stability problem for shear flows to the spectral properties of the Orr--Sommerfeld operator
\begin{align}\label{eq-O-S-operator}
  Orr_{\alpha,\nu}(\omega_\alpha)=i\varepsilon\left(\pa_y^2-\alpha^2\right)\omega_\alpha+u_p\omega_\alpha-u_p''\Delta_\alpha^{-1}\omega_\alpha.
\end{align}
Here $\Delta_\alpha=(\pa_y^2-\alpha^2)$, and $\Delta_\alpha^{-1}\omega_\alpha=\phi_\alpha$ with boundary condition \eqref{bc-ori}.

Equivalently, the Orr--Sommerfeld equation can be written as
\begin{align*}
  Orr_{\alpha,\nu}(\omega_\alpha)=c\omega_\alpha.
\end{align*}
If $c$ is an eigenvalue of $Orr_{\alpha,\nu}$, then there exists a normal mode of the form 
\begin{align*}
  \phi(t,x,y)=\phi_\alpha(y)e^{i|\alpha|(x-ct)}.
\end{align*} 
For both $\alpha>0$ and $\alpha<0$, $c_i>0$ corresponds to linear growth,
$c_i<0$ corresponds to decay, and $c_i=0$ corresponds to a neutral mode.

In this paper, we study the spectral properties of the Orr--Sommerfeld operator for the plane Poiseuille flow, derive the neutral curves, and prove the bifurcation of traveling waves from the neutral curve.

\subsection{Background}

The study of hydrodynamic stability goes back to Reynolds' experiment in 1883, where the transition from laminar to turbulent motion was
observed as the Reynolds number $Re$ increases.  To understand such instability phenomena, Sommerfeld and Orr considered shear flows and derived the linearized equation, namely the equation later known as the Orr--Sommerfeld equation \cite{Sommerfeld1908,Orr1907}. Thus, the linear stability problem for a viscous shear flow can be formulated in terms of the spectral properties of the corresponding Orr--Sommerfeld operator.

Plane Poiseuille flow is one of the simplest shear flows, but it already exhibits a rich stability structure. The spectral instability of the Orr--Sommerfeld operator for plane Poiseuille flow at large Reynolds numbers was studied in the classical works of Heisenberg \cite{Heisenberg1924} and C. C. Lin \cite{LinCC1955}. Since the Poiseuille profile has no inflection point, Rayleigh's criterion implies inviscid spectral stability. Hence the Orr--Sommerfeld instability of plane Poiseuille flow is a genuinely viscous and boundary-induced phenomenon. Closely related viscous instability waves in boundary-layer flows were studied by Tollmien \cite{Tollmien1929} and Schlichting \cite{Schlichting1933}. The terminology Tollmien--Schlichting (T--S) waves originates from this boundary-layer theory, and it is also commonly used for the corresponding viscous unstable modes in channel flows. The set of parameters for which $c_i=0$ forms the neutral curve in the $(\alpha,Re)$-plane. In \cite{LinCC1955}, C. C. Lin derived the large-Reynolds-number asymptotic behavior of both the upper and lower branches of the neutral curve. Subsequent numerical studies computed this curve with high accuracy. In particular, Orszag determined the critical Reynolds number for plane Poiseuille flow to be approximately $Re_c = 5772.22$, above which unstable viscous modes of T--S type appear \cite{Orszag1971}. However, the laminar-to-turbulent transition observed in experiments may occur at Reynolds numbers well below this linear critical value. This indicates that the Orr--Sommerfeld linear stability theory alone does not fully describe the transition process, and motivates the study of nonlinear mechanisms near the neutral curve.

From the rigorous mathematical point of view, Grenier--Guo--Nguyen \cite{GGN16adv} proved the spectral instability of a class of symmetric shear flows in a two-dimensional channel at high Reynolds number. Their work introduced a Rayleigh--Airy iteration scheme and gave a rigorous
construction of unstable T--S modes. Related results for boundary-layer profiles were also obtained in \cite{GGN16duke}. Recently, Chen--Wu--Zhang proved the existence of neutral T--S modes for boundary-layer flows \cite{CWZ2026}; see also \cite{BG23-LWR} for related results. For unstable and neutral modes in compressible shear flows, we refer the reader to \cite{Kagei2015,YZ2023,YAZ2023,MWWZ2024,MWWZ2026}. 

After the development of the linear theory, attention naturally turned to the nonlinear behavior in the early stage of transition, see \cite{Stuart1960,Watson1960,DR1981}. In the classical work \cite{ReynoldsPotter1967}, Reynolds--Potter emphasized two fundamental questions in the nonlinear stability theory of parallel flows: whether a flow that is stable to infinitesimal disturbances may be unstable to disturbances of finite amplitude, and what finite-amplitude equilibrium flow may develop from an initial instability.  Since traveling-wave solutions form one of the simplest classes of finite-amplitude nonlinear solutions, they have become important objects in the study of such nonlinear instability phenomena. This point of view is also in line with the role of Hopf bifurcation emphasized by Ruelle--Takens \cite{RuelleTakens1971}: bifurcating time-periodic states may arise from a steady flow and may subsequently undergo further instabilities leading to more complicated flow dynamics.

In 1971--1972, Yudovich \cite{Iudovich1971}, Joseph--Sattinger \cite{JS1972}, and Iooss \cite{Iooss1972} developed Hopf-type bifurcation theory for the incompressible Navier--Stokes equations in two-dimensional channel domains. These results provide sufficient conditions for the existence of traveling-wave solutions bifurcating from shear flows under suitable spectral assumptions on the linearized Navier--Stokes operator. These assumptions are precisely the eigenvalue simplicity and the eigenvalue transversal crossing conditions (see Remark \ref{rmk-spect-assum}). As noted by Chen--Price in \cite{ChenPrice1999}, \emph{``these conditions are difficult to verify.''} The works of Joseph--Sattinger \cite{JS1972} and Iooss \cite{Iooss1972} also studied the stability of the traveling-wave solutions. For plane Poiseuille flow, Chen--Joseph analyzed the bifurcation type along the neutral curve by numerical methods  \cite{CJ1973}; see also the recent numerical work of Bian--Dai--Grenier \cite{BDG2026}. Related Hopf bifurcation results under similar spectral assumptions in other settings can be found in \cite{Galdi2016} for exterior domains and in \cite{BianGrenierIooss2025} for boundary-layer flows.

For compressible plane Poiseuille flow, Kagei--Nishida first studied the spectral properties of the linearized operator and established an instability criterion \cite{Kagei2015}. They later proved the existence of traveling waves bifurcating from compressible plane Poiseuille flow \cite{Kagei2019}. However, since the Mach number in the above works is required to satisfy a certain lower-bound condition, the resulting traveling waves belong to a genuinely compressible regime and are different from those arising in the incompressible setting. For mathematical results on Hopf bifurcation in other fluid models, we refer the reader to \cite{KloedenWells1983,INT1978,CI1985,CDI1987,ChenPrice1999}. We also
mention related traveling-wave results based on Crandall--Rabinowitz type
bifurcation methods; see, for instance, \cite{li2011resolution,LinZeng2011,castroLear2024} and the references therein.

In the present paper, we study the spectral properties of the Orr--Sommerfeld operator associated with the incompressible plane Poiseuille flow. We prove the existence and uniqueness of the lower and upper branches of the neutral curve in the large-Reynolds-number regime. Moreover, we verify the spectral conditions required in the Hopf bifurcation framework of Joseph--Sattinger and Iooss. As a consequence, we obtain the existence of traveling-wave solutions bifurcating from the plane Poiseuille flow at points on the neutral curve.

\subsection{Main results}
For $\varepsilon=\frac{\nu}{|\alpha|}$ sufficiently small, we study the Orr--Sommerfeld eigenvalue problem for the plan Poiseuille flow in the T--S eigen region $\mathbb H$, 
\begin{equation}\label{eq-region-H}
  \mathbb H=
  \left\{
  \left(c_r,c_i,\alpha^2,\nu\right)
  \;\middle|\;
  \begin{aligned}
    &-\frac{\varepsilon^{\frac12}}{2}\le c_i\le 2\max(M,100)\varepsilon^{\frac13},
    \quad
    \frac{\varepsilon^{\frac13}}{10^3}\le c_r\le 10^3 \varepsilon^{\frac15},\\
    &\frac{\varepsilon^{\frac13}}{10^3|\ln \varepsilon|}
    \le \alpha^2\le
    10^3|\ln \varepsilon| \varepsilon^{\frac15}
  \end{aligned}
  \right\}.
\end{equation}
where $M$ is a big enough constant fixed in Remark \ref{rmk:Airy-class-est}.
\begin{theorem}\label{thm1}
  For each $\alpha$ sufficiently small, there exists a unique eigenvalue curve $\left(c_r(\nu), c_i(\nu)\right)$ for the even modes in the T-S eigen region $\mathbb H$. Moreover, for every $\nu$ on this curve, the eigenvalue $c(\nu)$ is simple.

  There exist four positive $O(1)$ quantities 
  \begin{align*}
    J_{\nu,-}\left(\alpha\right)<\widetilde J_{\nu,-}\left(\alpha\right),\quad \widetilde J_{\nu,+}\left(\alpha\right)<J_{\nu,+}\left(\alpha\right),
  \end{align*}
  such that
  \begin{align*}
  -\frac{|\alpha|^3}{2}=c_i \left( \widetilde J_{\nu,-}(\alpha)|\alpha|^7 \right)< c_i\left(\nu\right)< 0=c_i \left(J_{\nu,-}(\alpha)|\alpha|^7 \right), \text{ for }J_{\nu,-}(\alpha)|\alpha|^7<\nu< \widetilde J_{\nu,-}(\alpha)|\alpha|^7,
\end{align*}
  \begin{align*}
  c_i\left(\nu\right)>0, \text{ for }J_{\nu,+}(\alpha)|\alpha|^{11}<\nu< J_{\nu,-}(\alpha)|\alpha|^7,
\end{align*}
  \begin{align*}
  -\frac{|\alpha|^5}{2}=c_i \left( \widetilde J_{\nu,+}(\alpha)|\alpha|^{11} \right)< c_i\left(\nu\right)< 0=c_i \left(J_{\nu,+}(\alpha)|\alpha|^{11} \right), \text{ for }\widetilde J_{\nu,+}(\alpha)|\alpha|^{11}<\nu< J_{\nu,+}(\alpha)|\alpha|^{11}.
\end{align*}
Moreover, 
\begin{align*}
  c_r(\nu)\sim\alpha^2,\ c_r \left(J_{\nu,-}(\alpha)|\alpha|^{7}\right)\sim  \nu^{\frac{2}{7}},\ c_r \left(J_{\nu,+}(\alpha)|\alpha|^{11}\right)\sim  \nu^{\frac{2}{11}},
\end{align*} 
and
\begin{align*}
  c_i(\nu)\sim \nu^{\frac{1}{2}}|\alpha|^{-\frac{3}{2}}, \text{ for }|\alpha|^{11}\ll \nu\ll |\alpha|^{7}.
\end{align*}
At the lower and upper neutral points, it holds that
\begin{align}
  \pa_{\nu}c_r(\nu)|_{\nu=J_{\nu,-}(\alpha)|\alpha|^7}\sim|\alpha|^{-5},\quad \pa_{\nu}c_i(\nu)|_{\nu=J_{\nu,-}(\alpha)|\alpha|^7}\sim-|\alpha|^{-5},\label{est-ci-nu-low}\\
  \pa_{\nu}c_r(\nu)|_{\nu=J_{\nu,+}(\alpha)|\alpha|^{11}}\sim|\alpha|^{-7},\quad \pa_{\nu}c_i(\nu)|_{\nu=J_{\nu,+}(\alpha)|\alpha|^{11}}\sim|\alpha|^{-7}.\label{est-ci-nu-up}
\end{align}
\end{theorem}

\begin{theorem}\label{thm2}
    For each $\nu$ sufficiently small, there exists a unique eigenvalue curve $\left(c_r(\alpha^2), c_i(\alpha^2)\right)$ for the even modes in the T-S eigen region $\mathbb H$. More precisely, for each point on this curve, the eigenvalue $c(\alpha^2)$ of the Orr--Sommerfeld of each wavenumber $\alpha=\pm\sqrt{\alpha^2}$ is simple, and the eigenspaces corresponding to $\alpha$ and $-\alpha$ are conjugate to each other.

   There exist four positive $O(1)$ quantities 
  \begin{align*}
    \widetilde J_{\alpha,-}\left(\nu\right)<J_{\alpha,-}\left(\nu\right),\quad J_{\alpha,+}\left(\nu\right)<\widetilde J_{\alpha,+}\left(\nu\right),
  \end{align*}
  such that
  \begin{align*}
  -\frac{\nu^{\frac{3}{7}}}{2}=c_i \left( \widetilde J_{\alpha,-}\left(\nu\right)\nu^{\frac{2}{7}} \right)<c_i\left(\alpha^2\right)< 0=c_i \left(J_{\alpha,-}\left(\nu\right)\nu^{\frac{2}{7}} \right), \text{ for }\widetilde J_{\alpha,-}\left(\nu\right)\nu^{\frac{2}{7}}<\alpha^2< J_{\alpha,-}\left(\nu\right)\nu^{\frac{2}{7}},
\end{align*}
  \begin{align*}
  c_i\left(\alpha^2\right)>0, \text{ for }J_{\alpha,-}\left(\nu\right)\nu^{\frac{2}{7}}<\alpha^2< J_{\alpha,+}\left(\nu\right)\nu^{\frac{2}{11}},
\end{align*}
  \begin{align*}
  -\frac{\nu^{\frac{5}{11}}}{2}=c_i \left( \widetilde J_{\alpha,+}\left(\nu\right)\nu^{\frac{2}{11}} \right)< c_i\left(\alpha^2\right)< 0=c_i \left(J_{\alpha,+}\left(\nu\right)\nu^{\frac{2}{11}} \right), \text{ for }J_{\alpha,+}\left(\nu\right)\nu^{\frac{2}{11}}<\alpha^2<\widetilde J_{\alpha,+}\left(\nu\right)\nu^{\frac{2}{11}}.
\end{align*}
Moreover, 
\begin{align*}
  c_r(\alpha^2)\sim\alpha^2,\ c_r \left(J_{\alpha,-}\left(\nu\right)\nu^{\frac{2}{7}}\right)\sim  \nu^{\frac{2}{7}},\ c_r \left(J_{\alpha,+}\left(\nu\right)\nu^{\frac{2}{11}}\right)\sim  \nu^{\frac{2}{11}},
\end{align*}
and
\begin{align*}
  c_i(\alpha^2)\sim \nu^{\frac{1}{2}}|\alpha|^{-\frac{3}{2}}, \text{ for }\nu^{\frac{2}{7}}\ll \alpha^2\ll\nu^{\frac{2}{11}},
\end{align*}
At the lower and upper neutral points, it holds that
\begin{align*}
  \left|\pa_{\alpha^2}c_r(\alpha^2)|_{\alpha^2=J_{\alpha,-}\left(\nu\right)\nu^{\frac{2}{7}}}\right|\lesssim 1,\quad \pa_{\alpha^2}c_i(\alpha^2)|_{\alpha^2=J_{\alpha,-}\left(\nu\right)\nu^{\frac{2}{7}}}\sim1,\\
  \pa_{\alpha^2}c_r(\alpha^2)|_{\alpha^2=J_{\alpha,+}\left(\nu\right)\nu^{\frac{2}{11}}}\sim1,\quad \pa_{\alpha^2}c_i(\alpha^2)|_{\alpha^2=J_{\alpha,+}\left(\nu\right)\nu^{\frac{2}{11}}}\sim-\nu^{\frac{2}{11}}.
\end{align*}
\end{theorem}
In fact, the above two theorems are consequences of the following more general result.
\begin{theorem}
  For $\frac{\nu}{|\alpha|}$ sufficiently small, sufficiently small, there exists a unique eigenvalue surface $\left(c_r(\alpha^2,\nu), c_i(\alpha^2,\nu)\right)$ for the even modes in the T-S eigen region $\mathbb H$.
\end{theorem}

\begin{remark}[Even and odd modes]\label{rmk-oddmode}
Since the plane Poiseuille profile is even, the eigenspaces decompose
into even and odd subspaces, and the eigenfunctions can be chosen  to be either even or odd. As pointed out in the classical literature
\cite{LinCC1955,Miles1960,Orszag1971}, the unstable and neutral T--S modes
occur in the even class. Therefore, in the proof of Theorems
\ref{thm1} and \ref{thm2}, as in \cite{LinCC1955}, we consider the
problem on the half-channel with boundary conditions
\begin{align}\label{bc-even}
  \phi_\alpha(-1)=0,\quad\phi'_\alpha(-1)=0,\quad\phi'_\alpha(0)=0,\quad\phi'''_\alpha(0)=0.
\end{align}
Using the techniques developed in this paper, one can also show that the odd
problem has no eigenvalues in the T--S eigen region \(\mathbb H\); see Remark \ref{rmk-odd-Wron}.
\end{remark}

\begin{remark}[Spectral conditions for Hopf bifurcation]\label{rmk-spect-assum}
\label{rmk-hopf-spectral-assumptions}
Fix \(0<\alpha^{[0]}\ll1\), and let \(\nu^{[0]}\) be one of the neutral
viscosity coefficients obtained in Theorem \ref{thm1}. Denote the corresponding
neutral eigenvalue by
\[
  c^{[0]}=c_r(\nu^{[0]})+ic_i(\nu^{[0]}) \text{ with } c_i(\nu^{[0]})=0.
\]
To apply the Hopf bifurcation theory, one needs the simplicity of the neutral eigenvalue, and the transversal crossing condition. In the \(2\pi/\alpha^{[0]}\)-periodic setting, the simplicity condition requires
that \(c^{[0]}\) be simple for the mode \(\alpha=\pm\alpha^{[0]}\), and that no higher wave number has the same neutral eigenvalue:
\[
  c^{[0]}\notin
  \sigma\!\left(\mathrm{Orr}_{k\alpha^{[0]},\nu^{[0]}}\right),
  \qquad |k|\ge2.
\]
In \cite{AH2026}, Almog--Helffer proved that, for symmetric shear flows in a channel, all possible eigenvalues of the Orr--Sommerfeld operator are stable in the regimes 
\begin{align*}
  \alpha^2\ll\varepsilon^{\frac{1}{3}} \text{ and }\alpha^2\gg\varepsilon^{\frac{1}{5}}.
\end{align*}
Therefore, for the plane Poiseuille flow, all neutral modes in the regime
considered here are captured by Theorems \ref{thm1} and \ref{thm2}. The
required simplicity condition then follows from the uniqueness statement in Theorem \ref{thm2}. Finally, the
estimates \eqref{est-ci-nu-low} and \eqref{est-ci-nu-up} imply
\[
  \partial_\nu c_i(\nu^{[0]})\neq0,
\]
which is the transversal crossing condition.
\end{remark}

By Theorems \ref{thm1}--\ref{thm2} and Remark
\ref{rmk-hopf-spectral-assumptions}, the spectral assumptions required in the Hopf bifurcation framework of Joseph--Sattinger \cite{JS1972} and Iooss
\cite{Iooss1972} are satisfied at the neutral points on the lower and upper
branches. Hence, taking \(\nu\) as the bifurcation parameter, we obtain
traveling-wave solutions bifurcating from the plane Poiseuille flow.

\begin{theorem}\label{thm-bif}
For $\alpha^{[0]}$ that sufficiently small, let $\nu^{[0]}$ be the upper or the lower neutral viscosity coefficient obtained in Theorem \ref{thm1}, and $c^{[0]}=c_r^{[0]}\sim(\alpha^{[0]})^2$ be the corresponding neutral eigenvalue. There exists $s_0>0$ sufficiently small and a branch of traveling-wave solutions $\left(\nu_s,\Phi_s(t,x,y)\right)$ of the nonlinear system \eqref{NS-pert-nonlinear} with $s\in[-s_0,s_0]$, such that 
\begin{align*}
  \nu_s=\nu^{[0]}+s^2\nu^{[2]},\\
  \Phi_s(t,x,y)= \phi_s(x-c_s t,y),\\
   \phi_s(x,y)=\phi_s(x+\frac{2\pi}{\alpha^{[0]}},y),\\
   \phi_s(x,y)=s\phi^{[1]}(x,y)+s^2 \phi^{[2]}(x,y),\\
  c_s=c_r^{[0]}+s^2c_r^{[2]}.
\end{align*}
where $\nu^{[2]}$, $c_r^{[2]}\in\mathbb R$, $\phi^{[1]}(x,y)=\Re \left(\phi_{\nu^{[0]},\alpha^{[0]}}(y)e^{i\alpha^{[0]}x}\right)$, $\phi_{\alpha^{[0]},\nu^{[0]}}(y)$ is the eigenfunction corresponding to $\alpha^{[0]}$ and $\nu^{[0]}$, and $\phi^{[2]}\in H^4\left(\mathbb T_{\frac{2\pi}{\alpha^{[0]}}}\times[-1,1]\right)$.
\end{theorem}

\begin{remark}
  In Theorem \ref{thm-bif}, if $c^{[0]}=c_r \left(\nu\right)|_{\nu=J_{\nu,-}(\alpha^{[0]})|\alpha^{[0]}|^7}$, then the bifurcation occurs on the lower branch. If $c^{[0]}=c_r \left(\nu\right)|_{\nu=J_{\nu,+}(\alpha^{[0]})|\alpha^{[0]}|^{11}}$, then the bifurcation occurs on the upper branch.
\end{remark}

\subsection{Notations}
Through this paper, we use the following notations.

We use $C$ to denote a positive constant that may be different from line to line, and use $f\lesssim g$ to denote
\begin{align*}
  f\le C g.
\end{align*}
We use $f\ll g$ to indicate $f\le \frac{1}{C}g$ for a sufficiently large constant $C$.

We use $f\sim g$ if there exists $C>0$ such that
\begin{align*}
  f\le C g, \text{ and }g\le C f.
\end{align*}
We use $f\approx g$ if 
\begin{align*}
  |f-g|\ll |f| \text{ or } |f-g|\ll |g|.
\end{align*}
For a function $f(y)$, we denote
\begin{align*}
  f'(y)=\pa_yf(y),\quad f''(y)=\pa_y^2f(y),\quad f{(k)}(y)=\pa_y^kf(y).
\end{align*}
 
\section{Main ideas and sketch of proof}
\subsection{Classical strategy for the spectral problem}
We briefly explain the main ideas and difficulties in the proof. The starting point is the classical strategy, going back to Heisenberg \cite{Heisenberg1924} and C. C. Lin \cite{LinCC1955}, for studying the Orr--Sommerfeld equation. One constructs four linearly independent
solutions $\left(\Phi_{1,\alpha},\Phi_{2,\alpha},\Phi_{3,\alpha},\Phi_{4,\alpha}\right)$ of the fourth-order Orr--Sommerfeld equation \eqref{eq-Orr-Sommerfeld} and writes a general solution as their linear combination
\begin{align*}
  \Phi(y)=B_1 \Phi_{1,\alpha}+B_2 \Phi_{2,\alpha}+B_3 \Phi_{3,\alpha}+B_4 \Phi_{4,\alpha}.
\end{align*}
For fixed parameters $(c,\alpha,\nu)$, the eigenvalue problem is reduced to whether there exists a nontrivial choice of the coefficients $B_j$ such that all boundary conditions in \eqref{bc-even} are satisfied. This gives a dispersion relation, written in the form of a Wronskian condition \eqref{eq-wron-cond}, and the zeros of this Wronskian give the Orr--Sommerfeld eigenvalues.

For small $\varepsilon$, the Orr--Sommerfeld equation
\eqref{eq-Orr-Sommerfeld} can be regarded as a perturbation of the inviscid
Rayleigh equation away from the critical layer. Here the critical layer refers to the region where $u_p(y)$ is close to the wave speed $c$, more precisely, $|u_p(y)-c|\lesssim\varepsilon^{\frac{1}{3}}$. Outside this region, the leading-order equation of \eqref{eq-Orr-Sommerfeld} is the Rayleigh equation
\begin{align}\label{eq-Ray-Heisen}
  \left(u_p-c\right)\left(\pa_y^2-\alpha^2\right)\phi_\alpha-u_p''\phi_\alpha=0.
\end{align}
Thus, two of the four solutions, $\Phi_{1,\alpha}$ and $\Phi_{2,\alpha}$, are constructed as perturbations of two linearly independent Rayleigh solutions $\phi_{1,\alpha}$ and $\phi_{2,\alpha}$ of \eqref{eq-Ray-Heisen}; these are the so-called slow modes. However, in the critical layer, where $\left|u_p(y)-c\right|\lesssim \varepsilon^{\frac{1}{3}}$, the Rayleigh equation becomes singular and the viscous term $i\varepsilon\left(\pa_y^2-\alpha^2\right)^2\phi_\alpha$ in the Orr--Sommerfeld equation cannot be neglected.

Inside the critical layer, the Orr--Sommerfeld equation is approximated by an Airy-type equation of the form
\begin{align}\label{eq-Airy-0}
  i\varepsilon\pa_y^4\phi(y)+\left(u_p-c-2i\varepsilon\alpha^2\right)\pa_y^2\phi(y)=f(y),\quad y\in[-1,0].
\end{align}
The other two independent solutions, $\Phi_{3,\alpha}$ and $\Phi_{4,\alpha}$, are therefore constructed from Airy-type solutions and are called fast modes. The mode $\Phi_{3,\alpha}$ decays exponentially from the wall to the center of the channel, namely from $y=-1$ to $y=0$, while $\Phi_{4,\alpha}$ grows exponentially in the same direction. As pointed out by Reid in the survey \cite{Reid1965}, viscous effects are expected to be small near the center of the channel, and the exponentially growing fast mode $\Phi_{4,\alpha}$ should not enter the leading-order matching. Since the boundary value of $\Phi_{3,\alpha}$ is exponentially small at $y=0$, the leading matching reduces to constructing an inviscid Rayleigh solution
\begin{align}\label{eq-Phi-Ray}
  \Phi_{Ray},\text{ with }\Phi_{Ray}'(0)=0,
\end{align}
and matching it with the decaying fast mode $\Phi_{3,\alpha}$  at the wall $y=-1$
\begin{align*}
  \left.\frac{\Phi_{Ray}'(y)}{\Phi_{Ray}(y)}\right|_{y=-1}=\left.\frac{\Phi_{3,\alpha}'(y)}{\Phi_{3,\alpha}(y)}\right|_{y=-1}.
\end{align*}
This is the mechanism underlying Lin's asymptotic derivation of the neutral
curve.

In C. C. Lin's work \cite{LinCC1955}, the boundary values of these linearly independent solutions were derived through formal asymptotic expansions. In \cite{GGN16adv}, Grenier--Guo--Nguyen introduced the Rayleigh--Airy iteration scheme and gave a rigorous construction of the four linearly independent solutions for $c_i>0$. They then proved the existence of unstable T--S modes for symmetric shear flows in a channel.

\subsection{Boundary-adapted iterations and the reduced dispersion relation}
We construct the slow modes by a Rayleigh--Airy iteration and the fast modes by two distinct Airy iterations. In both constructions, the solution operators are adapted to the relevant boundary conditions. The resulting exact zero entries in the dispersion matrix \eqref{eq-matrix-scale} play a important role in reducing the full dispersion relation to a perturbation of the matching condition between $\Phi_{1,\alpha}$ and $\Phi_{3,\alpha}$.

\subsubsection{A boundary-adapted Rayleigh--Airy iteration for the slow modes}
The slow modes $\Phi_{1,\alpha}$ and $\Phi_{2,\alpha}$ are constructed
within the framework of the Rayleigh--Airy iteration introduced in
\cite{GGN16adv}. The basic idea is to decompose the Orr--Sommerfeld
operator in two complementary ways:
\begin{align}\label{eq-OS-decomp-intro}
  Orr(\phi)=Ray_{\alpha}(\phi)+Diff(\phi)=Airy_4(\phi)+Reg(\phi),
\end{align}
where $Ray_{\alpha}$, $Diff$, $Airy_4$, and $Reg$ are defined in \eqref{eq-OS-decom-ray}-\eqref{eq-OS-decom-reg}. Starting from a homogeneous Rayleigh solution, one corrects the defect generated by $Diff$ through an inhomogeneous Airy equation. The resulting Airy correction generates a $Reg$ defect, which is then corrected through an inhomogeneous Rayleigh equation. Repeating this procedure yields a solution of the Orr--Sommerfeld equation.

A important step in the Rayleigh--Airy iteration is the construction of a suitable solution operator $AirySolver$ for $Airy_4$. This operator is constructed using the approximate Green function \eqref{eq-Green-Airy}. The Green function \eqref{eq-Green-Airy} consists of modified Airy functions, modified primitive Airy functions, and the associated nonlocal terms. In contrast to \cite{GGN16adv}, where the modified primitive Airy functions are defined from infinity, we define \eqref{eq-Airy-modi-1}--\eqref{eq-Airy-modi-2} from the boundary. This boundary-adapted construction ensures that
\begin{align}\label{eq-AirySolver-center-intro}
  \left.\pa_y AirySolver(f)\right|_{y=0}=0.
\end{align}
The Rayleigh solution operator $RaySolver_{\alpha}$ used in the iteration, as in \cite{GGN16adv}, satisfies 
\begin{align}\label{eq-RaySolver-center-intro}
  \left.\pa_y RaySolver_{\alpha}(f)\right|_{y=0}=0.
\end{align}
Consequently, every correction generated by the Rayleigh--Airy iteration has zero first derivative at the centerline.

Since the leading Rayleigh solution $\phi_{1,\alpha}$ constructed in Lemma \ref{lem-sol-Ray-hom} satisfies $\phi_{1,\alpha}'(0)=0$, this condition is preserved throughout the iteration, and hence $\Phi_{1,\alpha}'(0)=0$. Thus, $\Phi_{1,\alpha}$ is the viscous slow mode associated with the inviscid mode $\Phi_{Ray}$ used in \cite{LinCC1955,Reid1965}. For the same reason, we also have $\Phi_{2,\alpha}'(0)=\phi_{2,\alpha}'(0)=\frac{1}{1-\hat c}$.

A further difference from \cite{GGN16adv} is that our analysis includes the neutral case $c_i=0$, where the Rayleigh equation has a genuine critical-layer singularity. In \eqref{eq-OS-decom-ray}--\eqref{eq-OS-decom-diff}, we replace $c$ in the Rayleigh operator by the regularized parameter $\hat c=c+ic_0$ with $c_0>0$, and include the compensating term $ic_0(\pa_y^2-\alpha^2)\phi$ in $Diff(\phi)$. Such a regularization was previously used in Chen--Wu--Zhang \cite{CWZ2026}. 

In this paper, we choose $c_0=\varepsilon^{\frac{1}{2}}$. This choice balances two effects. On the one hand, a larger $c_0$ gives a stronger regularization of the critical-layer singularity in the Rayleigh solver, so that only a few principal correction terms in the Rayleigh--Airy iteration need to be treated sharply. On the other hand, $c_0$ should be small enough, since the regularization introduces the additional term $ic_0(\partial_y^2-\alpha^2)\phi$. The correction terms generated by this artificial term must remain lower-order contributions in the dispersion relation and be absorbed into the residual terms. The choice $c_0=\varepsilon^{\frac{1}{2}}$ satisfies both requirements: it sufficiently regularizes the Rayleigh solver while keeping the regularization error negligible in the leading-order matching.
\subsubsection{Boundary-adapted Airy iterations for the fast modes}
Unlike in \cite{GGN16adv}, where the fast modes are constructed through the Rayleigh--Airy iteration, we construct $\Phi_{3,\alpha}$ and $\Phi_{4,\alpha}$ directly by two distinct Airy iterations. Each iteration is based on a different boundary-adapted approximate Green function chosen to preserve the boundary behavior of the corresponding mode.

For $\Phi_{3,\alpha}$, we use the approximate Green function \eqref{eq-Green-Airy-}. In contrast to \eqref{eq-Green-Airy}, which is used in the construction of the slow modes, all nonlocal terms are assigned to the $x>y$ side. These two Green functions satisfy the same approximate Green-function equation \eqref{eq-Green-app}, but their different structures lead to solution operators with different boundary properties.

Starting from the unified exponentially decaying primitive Airy function $\phi_{A,1}$ in \eqref{eq-unifi-prim-Airy}, we apply the corresponding Airy iteration. The one-sided structure of \eqref{eq-Green-Airy-} preserves the exponential decay from the wall $y=-1$ toward the centerline $y=0$ throughout the iteration, see Lemma \ref{lem-AS--0}. We thereby obtain a fast mode $\Phi_{3,\alpha}$ satisfying
\begin{align}\label{eq-Phi3-center-intro}
  \Phi_{3,\alpha}'(0)=0,
  \qquad
  \left|\Phi_{3,\alpha}'''(0)\right|
  \lesssim
  e^{-\varepsilon^{-\frac13}}.
\end{align}

For $\Phi_{4,\alpha}$, we use the approximate Green function \eqref{eq-Green-Airy+}, in which the nonlocal terms are assigned to the $x<y$ side. The corresponding Airy iteration preserves the wall conditions and yields
\begin{align*}
  \Phi_{4,\alpha}(-1)=\Phi_{4,\alpha}'(-1)=0.
\end{align*}
\subsubsection{The reduced dispersion relation}
The boundary-adapted construction of the four independent solutions gives the dispersion matrix the special structure described in \eqref{eq-matrix-scale}. In particular, $\Phi_{4,\alpha}'''(0)\sim\kappa^{\frac92}$ is much larger than $\Phi_{4,\alpha}'(0)\sim\kappa^{\frac32}$, so that expansion along the fourth column reduces the leading part of the full determinant to a $3\times3$ minor. The exact identities $\Phi_{1,\alpha}'(0)=\Phi_{3,\alpha}'(0)=0$ then reduce this minor to the $2\times2$ Wronskian involving $\Phi_{1,\alpha}$ and $\Phi_{3,\alpha}$. After normalization, the remaining lower-order contributions are absorbed into a residual term. It therefore suffices to study the zero set of the reduced dispersion function $F=F_r+iF_i$, given by
\begin{align}\label{eq-F-intr}
  F(c_r,c_i,\alpha^2,\nu)=\frac{\Phi_{1,\alpha}(-1)}{\Phi_{1,\alpha}'(-1)}-  \frac{\Phi_{3,\alpha}(-1)}{\Phi_{3,\alpha}'(-1)}+Res(c_r,c_i,\alpha^2,\nu).
\end{align}
Thus, the leading dispersion relation is reduced to the wall matching condition between the first slow mode $\Phi_{1,\alpha}$ and the decaying fast mode $\Phi_{3,\alpha}$. This reduced form provides the basis for the subsequent analysis of the neutral curve and its parameter derivatives. 

\subsection{Eigenvalue curves and the transversal crossing condition}
\subsubsection{Main difficulties}
We next describe the main analytic difficulty in the study of the eigenvalue curve. For fixed $\alpha$, the eigenvalue curve $c(\nu)$ is determined by the dispersion relation
\begin{align*}
  F\left(c_r(\nu),c_i(\nu),\alpha^2,\nu\right)=0.
\end{align*}
In order to apply the Hopf bifurcation theory and prove the existence of
traveling waves, we need to verify the transversal crossing condition. This
requires derivative estimates for $c(\nu)$, such as \eqref{est-ci-nu-low} and \eqref{est-ci-nu-up}. These estimates are obtained from the implicit function theorem applied to the above dispersion relation, see \eqref{eq-dervi-c-nu}. Therefore, we need rather precise information on the Wronskian function $F\left(c_r(\nu),c_i(\nu),\alpha^2,\nu\right)$ and its parameter derivatives.

We explain the difficulty by considering the upper branch of the neutral curve. In this regime, it holds that $c_r\sim \varepsilon^{\frac{1}{5}}$. If we only keep the leading Rayleigh part $\phi_{1,\alpha}$ of $\Phi_{1,\alpha}$  and the leading Airy part $\phi_{A,1}$ of $\Phi_{3,\alpha}$, then the dispersion relation takes the form
\begin{align*}
  F(c_r,c_i,\alpha^2,\nu)=&\frac{\phi_{1,\alpha}(-1)}{\phi_{1,\alpha}'(-1)}-  \frac{\phi_{A,1}(-1)}{\phi_{A,1}'(-1)}+res\\
  =&\underbrace{\frac{-c_r-ic_i+\frac{4}{15}\alpha^2-\frac{8\pi}{15^2}\alpha^4i}{2}}_{\text{contribution from }\phi_{1,\alpha}} +\underbrace{\varepsilon^{\frac{1}{2}}c_r^{-\frac{1}{2}}e^{\frac{\pi}{4}i}}_{\text{contribution from }\phi_{A,1}}+res,
\end{align*}
where $res$ denotes the residual terms produced by the full Rayleigh--Airy
iteration.

The real part of the leading-order relation gives $c_r\sim\alpha^2$. For this purpose, it is enough to prove $|res|\ll c_r$. However, the imaginary part is much more delicate. On the neutral curve $c_i=0$, the imaginary part of the slow-mode contribution is of size $\alpha^4$,  while the imaginary part of the fast-mode contribution is of size $\varepsilon^{\frac{1}{2}}c_r^{-\frac{1}{2}}$. Balancing these two terms gives $c_r\sim \varepsilon^{\frac{1}{5}}$, which corresponds to the upper branch in Lin's classical asymptotic analysis \cite{LinCC1955}. To justify this balance rigorously, one needs the stronger estimate $|res|\ll c_r^2\sim \varepsilon^{\frac{2}{5}}$. In other words, the leading terms $\phi_{1,\alpha}$ and $\phi_{A,1}$ must indeed dominate the wall boundary values of the full modes  $\Phi_{1,\alpha}$ and $\Phi_{3,\alpha}$, and all correction terms generated by the Rayleigh--Airy iteration for the slow modes and by the Airy iteration for the fast modes must be sufficiently small at the boundary. 

The derivative estimates are even more delicate. To obtain the crossing condition, we need sharp bounds for derivatives of $F$, for instance $\partial_{c_r}F$, and we must show that these derivatives are still governed by the leading terms displayed above. For example, $\Im\pa_{c_r} \left(\frac{\phi_{A,1}(-1)}{\phi_{A,1}'(-1)}\right)\sim -\varepsilon^{\frac{1}{2}}c_r^{-\frac{3}{2}}$, the size of which is larger than $\Im\frac{\phi_{A,1}(-1)}{\phi_{A,1}'(-1)}\sim \varepsilon^{\frac{1}{2}}c_r^{-\frac{1}{2}}$ by a factor of order $c_r^{-1}\sim\varepsilon^{-\frac{1}{5}}$ on the upper branch. Since the residual $res$ consists of many terms produced by the Rayleigh--Airy iteration, a crude estimate would typically lose a much larger factor after differentiation, for instance of order $\varepsilon^{-\frac{1}{3}}$ or worse. Thus, even if $res$ itself is small enough, it does not automatically follow that its derivatives are also negligible. This is the main reason why a fine structural analysis of the residual terms is required.

\subsubsection{Refined estimates of the iteration terms}
To overcome the difficulties described above, we need a more refined analysis of the correction terms generated in the iteration. We illustrate the idea by considering the first slow mode. The first Airy correction to the Rayleigh solution is
\begin{align}\label{eq-psi1-intro}
  \psi^{[1]}_{1,\alpha}=-AirySolver \left(Diff(\phi_{1,\alpha})\right).
\end{align}
Recall that $AirySolver$ is constructed from the approximate
Green function \eqref{eq-Green-Airy}. The main difficulty in estimating
\eqref{eq-psi1-intro} comes from the nonlocal terms $a_1(x)$ and $a_2(x)$ in \eqref{eq-Green-Airy}, defined in \eqref{eq-a1} and \eqref{eq-a2}.  Unlike the local Airy terms, these nonlocal contributions do not enjoy
strong exponential decay. In \cite{GGN16adv}, it is sufficient to estimate such terms by upper bounds. In the present work, however, the derivative estimates required for the crossing condition force us to use more detailed structural properties of $a_1(x)$ and $a_2(x)$.

We first derive more precise expansions of the modified primitive Airy
functions in Lemma \ref{lem:pri-Airy-expan}. Based on these expansions, we decompose $a_1(x)$ as
\begin{align}\label{eq-a1-exp-intro}
  a_1(x)=a_{1,M}(x)+a_{1,SecM}(x)+a_{1,B}(x)+a_{1,R}(x),
\end{align}
Here the leading term is 
\begin{align*}
  a_{1,M}(x)=\frac{i\frac{e^{\frac{\pi}{2}i}}{2} \mathrm{Hi}(e^{-\frac{\pi}{2}i}\kappa \eta(x))}{\kappa\left(\eta_r'(x)\right)^{\frac{3}{2}}},
\end{align*}
where $\mathrm{Hi}$ is the Scorer function and $\eta$ is the Langer
transform. Similarly, we obtain an expansion of $a_2(x)$, whose leading term involves the second derivative of the Scorer function. These explicit leading-order formulas lead to slightly sharper
upper bounds for $a_1(x)$ and $a_2(x)$ than those used in
\cite{GGN16adv}. More importantly, they allow us to exploit the 
analytic properties of $\mathrm{Hi}$.

A typical difficult contribution in $\psi^{[1]}_{1,\alpha}$ can be reduced to an integral of the following form; see Section \ref{ssub:accurate_estimate_of_psi_1_alpha_a},
\begin{align*}
  |Int|\eqdef&\left|c_r\alpha^2\int^{2y_c+1}_{-1} \mathrm{Hi}''\left(e^{-\frac{\pi}{2}i}\varepsilon^{-\frac{1}{3}} \left(x-y_c-ic_i\right)\right)\frac{1}{x-\left(y_c+ic_i+ic_0\right)}dx\right|,
\end{align*}
where $y_c\approx-1+\frac{c_r}{2}$ denotes the location of the real critical layer such that $u_p(y_c)=c_r$. A direct estimate obtained by integrating the absolute value
of the integrand over the real interval gives only the rough bound
\begin{align*}
  |Int|\lesssim c_r\alpha^2 \left|\ln \varepsilon\right|\sim \varepsilon^{\frac{2}{5}}\left|\ln \varepsilon\right|.
\end{align*}
This is not sufficient, since the leading contribution to $F_i$ is of size $\alpha^4\sim \varepsilon^{\frac{2}{5}}$.

Instead, we use the analyticity of $\mathrm{Hi}$ and deform the contour from the real interval
\begin{align*}
  x\in [-1,2y_c+1]
\end{align*}
to the circular arc
\begin{align*}
  x=y_c+e^{i\theta}(y_c+1),\quad \theta\in [-\pi,0].
\end{align*}
The pole of the denominator is located at $x=y_c+i(c_i+c_0)$. Since $c_0+c_i>0$, this pole lies in the upper half-plane. Hence it is
outside the domain enclosed by the real interval and the lower circular arc, and the contour deformation is admissible. 

Using the asymptotic expansion of $\mathrm{Hi}''$ in the relevant sector,
see \eqref{eq-expan-Hi''}, we obtain
\begin{align*}
  |Int|=&\left|c_r\alpha^2\int^{0}_{-\pi} \frac{\mathrm{Hi}''\left(e^{-\frac{\pi}{2}i}\varepsilon^{-\frac{1}{3}} \left(e^{i\theta}(y_c+1)-ic_i\right)\right)ie^{i\theta}(y_c+1)}{ y_c+e^{i\theta}(y_c+1)  -\left(y_c+ic_i+ic_0\right)}d\theta\right|\lesssim \frac{c_r\alpha^2}{\left\langle \varepsilon^{-\frac{1}{3}} c_r \right\rangle^{3}}\ll \varepsilon^{\frac{2}{5}}.
\end{align*}
Thus this contribution is indeed an error term compared with the leading terms in the imaginary part of the dispersion relation. See Lemma \ref{lem-acc-psi-1-a} for details.

The explicit expansions of the nonlocal terms are also crucial for parameter derivatives. Away from the critical point, the same contour deformation argument gives the desired smallness for the differentiated correction terms. Near the critical point, however, this argument no longer gives the same gain. As shown in \eqref{est-acc-psi-1-a-cr-2}, the $L^\infty$-estimate of $\pa_{c_r}\psi^{[1]}_{1,\alpha}$ loses a factor of order $\varepsilon^{-\frac13}$ compared with that of $\psi^{[1]}_{1,\alpha}$. The key point is that this loss occurs only in a narrow region near $y_c$. Consequently, the boundary estimate at $y=-1$, see \eqref{est-acc-psi-1-a-cr-1}, and the $L^1$ estimate \eqref{est-acc-psi-1-a-cr-3} are strong enough to control the wall boundary value of the next Rayleigh correction $\pa_{c_r}\phi^{[1]}_{1,\alpha}$. Therefore, the differentiated correction terms remain lower order in the dispersion relation and do not affect the leading-order estimates needed for the transversal crossing condition. See Lemma \ref{lem-acc-psi-1-a-cr} for details.

Another difficulty comes from the boundary terms arising from the nonlocal part. For example, one such term is
\begin{align}\label{eq-def-a1B-intro}
  a_{1,B}(x)
  =&-\frac{\mathcal A_1(1,\kappa \eta(0))}
  {\kappa\left(\eta_r'(0)\right)^{\frac{3}{2}}}
  \mathrm{Ai}\left(e^{\frac{5\pi}{6}i}\kappa \eta(x)\right) +\frac{\mathcal A_2(1,\kappa \eta(-1))}
  {\kappa\left(\eta_r'(-1)\right)^{\frac{3}{2}}}
  \mathrm{Ai}\left(e^{\frac{\pi}{6}i}\kappa \eta(x)\right).
\end{align}
Away from the endpoints $0$ and $-1$, this boundary term has strong decay and is much smaller than the leading term $a_{1,M}(x)$. However, its parameter derivatives are more delicate. If one estimates, for instance,
\begin{align*}
  \pa_{c_r}\left(\mathcal A_2(1,\kappa \eta(-1))\mathrm{Ai}\left(e^{\frac{\pi}{6}i}\kappa \eta(x)\right)\right)
\end{align*}
by a direct product-rule bound, then
\begin{align*}
  &\left|\pa_{c_r}\left(\mathcal A_2(1,\kappa \eta(-1))\mathrm{Ai}\left(e^{\frac{\pi}{6}i}\kappa \eta(x)\right)\right)\right|\\
  \lesssim\;&
  \left|\pa_{c_r}\mathcal A_2(1,\kappa \eta(-1))\right|
  \left|\mathrm{Ai}\left(e^{\frac{\pi}{6}i}\kappa \eta(x)\right)\right|
  +\left|\mathcal A_2(1,\kappa \eta(-1))\right|
  \left|\pa_{c_r}\mathrm{Ai}\left(e^{\frac{\pi}{6}i}\kappa \eta(x)\right)\right|.
\end{align*}
This estimate is too large near $x=-1$.

The key point is that the boundary terms are not arbitrary Airy terms. They are formed by special pairs of Airy-type functions, one growing and one decaying. After taking parameter derivatives, the asymptotic expansions of these Airy functions reveal a cancellation structure within such pairs. This gives much better estimates than the direct product-rule bound. The precise estimates are given in Lemma \ref{lem-d-cr-a1B}.

This cancellation mechanism is also important for other parameter derivatives, such as $\pa_\nu$-estimates. Since the Airy-type functions used in this paper always appear in such paired growing--decaying combinations, this structure allows us to control the derivatives of the Airy terms throughout the argument.
 
\subsection{Organization of the paper}
The remainder of the paper is organized as follows. In Section 3, we study the Rayleigh problem. We construct the operator $RaySolver_{\alpha}$, build two linearly independent solutions \(\phi_{1,\alpha}\) and \(\phi_{2,\alpha}\) of the Rayleigh equation, and derive derivatives estimates for $\frac{\phi_{1,\alpha}(-1)}{\phi_{1,\alpha}'(-1)}$, which gives the Rayleigh contribution to the dispersion relation.

In Section 4, we study the Airy problem. We introduce the Langer transform, construct the approximate Green function, derive precise expansions of the associated Airy-type functions, and establish estimates for the modified and exact $AirySolver$ in suitable weighted spaces. These estimates will be used to control the viscous corrections in the Orr--Sommerfeld construction. In Section 5, we present the general strategy for constructing solutions of the Orr--Sommerfeld equation.

In Section 6, we construct the two slow modes. Starting from the Rayleigh solutions $\phi_{1,\alpha}$ and $\phi_{2,\alpha}$, we use the Rayleigh--Airy iteration to obtain Orr--Sommerfeld solutions $\Phi_{1,\alpha}$ and $\Phi_{2,\alpha}$. We then derive refined estimates for the correction terms and for the boundary expansion of $\Phi_{1,\alpha}$ at $y=-1$. This is one of the main technical parts of the paper, and it provides the slow-mode contribution to the dispersion relation.

In Section 7, we construct the two fast modes. We introduce two boundary-adapted approximate Green functions for the Airy equation and use the corresponding Airy iterations to construct $\Phi_{3,\alpha}$ and $\Phi_{4,\alpha}$. We also derive the boundary expansions of the fast modes, in particular the expansion of $\frac{\Phi_{3,\alpha}(-1)}{\Phi_{3,\alpha}'(-1)}$. 

In Section 8, we analyze the dispersion relation and prove the main spectral results. We first fix \(\varepsilon\) and locate the possible spectral points in the T--S eigen region \(\mathbb H\). We then study, respectively, the eigenvalue curve \(c(\nu)\) for fixed \(\alpha\) and the eigenvalue curve \(c(\alpha^2)\) for fixed \(\nu\). In particular, we obtain the key derivative estimates at the lower and upper branches of the neutral curve. 

Finally, in Section 9, using the spectral information obtained above, we prove the existence of traveling-wave solutions bifurcating from the plane Poiseuille flow.

\section{The Rayleigh equation}
In this section, we consider the following Rayleigh equation
\begin{align}\label{eq-Ray}
  \left(u_p-\hat c\right)\left(\pa_y^2-\alpha^2\right)\phi(y)-u_p''\phi(y)=f(y),\quad y\in[-1,0].
\end{align}
Here $\hat c=c_r+i\hat c_i$. Throughout this section, we assume that $|\hat c|\ll1$, $0<\hat c_i\lesssim c_r$, $\left|\ln \hat c_i\right|\sim \left|\ln \hat c\right|$, and $|\hat c|\ll \frac{1}{\left|\ln \hat c_i\right|}$. 

Note that we do not impose boundary conditions on \eqref{eq-Ray}.

We define the Rayleigh operator $Ray_\alpha$ by
\begin{align}\label{eq-op-Ray-alpha}
  Ray_\alpha \left(\phi\right)=\left(u_p-\hat c\right)\left(\pa_y^2-\alpha^2\right)\phi(y)-u_p''\phi(y).
\end{align}

\subsection{Homogeneous solutions for $\alpha=0$}
There are two linearly independent solutions of \eqref{eq-Ray} with $\alpha=0$ and $f=0$:
\begin{align}
  \phi_{1,0}(y)=\left(u_p-\hat c\right);\label{eq-phi10}\\
  \phi_{2,0}(y)=\left(u_p-\hat c\right)\int_0^y \frac{1}{\left(u_p-\hat c\right)^2}  dx.\label{eq-phi20}
\end{align}

A direct computation gives
\begin{align*}
  \phi_{2,0}(y)=\frac{y}{2(1-\hat c)}+\frac{u_p-\hat c}{4(1-\hat c)^{\frac{3}{2}}}\ln{\left(\frac{\sqrt{1-\hat c}+y}{\sqrt{1-\hat c}-y}\right)},
\end{align*}
\begin{align*}
  \phi_{2,0}'(y)=\frac{1+\phi_{2,0}(y)\phi_{1,0}'(y)}{\phi_{1,0}(y)}=\frac{1}{1-\hat c}-\frac{y}{2(1-\hat c)^{\frac{3}{2}}}\ln{\frac{(\sqrt{1-\hat c}+y)}{(\sqrt{1-\hat c}-y)}}.
\end{align*}
\begin{align*}
  \phi_{2,0}'''(y)=-\frac{1}{(1-\hat c)^{\frac{3}{2}}} \left( \frac{1}{(\sqrt{1-\hat c}+y)}+ \frac{1}{(\sqrt{1-\hat c}-y)}\right)-\frac{y}{2(1-\hat c)^{\frac{3}{2}}}\left(- \frac{1}{(\sqrt{1-\hat c}+y)^2}+ \frac{1}{(\sqrt{1-\hat c}-y)^2}\right).
\end{align*}
Here and below, the square root is chosen so that 
$\sqrt{1-\hat c}\to 1$ as $\hat c\to0$, and the logarithm is taken on the corresponding branch.

At the endpoints $y=-1$ and $y=0$, we have
\begin{equation}\label{eq-phi0-boundary}
  \begin{aligned}    
   \phi_{1,0}(-1)=-\hat c,\quad \phi_{1,0}'(-1)=2,\quad \phi_{1,0}(0)=1-\hat c,\quad \phi_{1,0}'(0)=0;\\ 
  \phi_{2,0}(-1)=\frac{-1}{2(1-\hat c)}+\frac{-\hat c}{4(1-\hat c)^{\frac{3}{2}}}\ln{\frac{(\sqrt{1-\hat c}-1)}{(\sqrt{1-\hat c}+1)}},\\ 
  \phi_{2,0}'(-1)=\frac{1+\phi_{2,0}(-1)\phi_{1,0}'(-1)}{\phi_{1,0}(-1)}=\frac{1}{1-\hat c}+\frac{1}{2(1-\hat c)^{\frac{3}{2}}}\ln{\frac{(\sqrt{1-\hat c}-1)}{(\sqrt{1-\hat c}+1)}},\\ 
  \phi_{2,0}(0)=0,\quad \phi_{2,0}'(0)=\frac{1}{1-\hat c}.   
  \end{aligned}
\end{equation}
\subsection{Non-homogeneous solutions for $\alpha=0$}
Using $\phi_{1,0}$ and $\phi_{2,0}$, we define the Green function $G_{R,0}(x,y)$ by
\begin{equation}
  G_{R,0}(x,y)=-\frac{1}{u_p(x)-\hat c}\left\{
    \begin{array}{ll}
      \phi_{1,0}(y)\phi_{2,0}(x),&x<y;\\ 
      \phi_{1,0}(x)\phi_{2,0}(y),&x>y,
    \end{array}
  \right.
\end{equation}
which satisfies
\begin{align*}
  Ray_0G_{R,0}(x,y)=\left(u_p(y)-\hat c\right) \pa_y^2G_{R,0}(x,y)-u_p''G_{R,0}(x,y)=\delta_x(y). 
\end{align*}
Then we define the solution operator for $Ray_0$ by
\begin{equation}\label{eq-op-solver-Ray0}
  \begin{aligned}    
  RaySolver_0(f)=&\int^0_{-1}G_{R,0}(x,y)f(x)dx \\
  =&-\phi_{1,0}(y)\int^y_{-1}\phi_{2,0}(x)\frac{f(x)}{u_p(x)-\hat c}dx-\phi_{2,0}(y)\int^0_{y}\phi_{1,0}(x)\frac{f(x)}{u_p(x)-\hat c}dx.    
  \end{aligned}
\end{equation}

We shall use the following weighted norms:
\begin{align*}
  \left\|f\right\|_{X_k}=&\sum_{j=0}^k\left\|(u_p-\hat c)^j\pa_y^j f\right\|_{L^\infty},\text{ for }k\ge1,\\
  \left\|f\right\|_{Y_k}=&\left\|f\right\|_{L^\infty}+\sum_{j=1}^k\left\|(u_p-\hat c)^{j-1}\pa_y^j f\right\|_{L^\infty},\text{ for }k\ge1.
\end{align*}
These norms are adapted to the possible degeneracy of $u_p-\hat c$ near the critical point.

\begin{lemma}\label{lem-Raysolver0}
  For $f\in X_2$, the following estimates hold:
  \begin{align*}
    \left\|RaySolver_0(f)\right\|_{Y_2}\lesssim&\left|\ln \hat c_i\right|\left\|f\right\|_{L^\infty},\\
    \left\|RaySolver_0(f)\right\|_{Y_3}\lesssim&\left|\ln \hat c_i\right| \left\|f\right\|_{L^\infty}+\left\|\left(u_p(y)-\hat c\right)\pa_yf\right\|_{L^\infty},\\
    \left\|RaySolver_0(f)\right\|_{Y_4}\lesssim&\left|\ln \hat c_i\right| \left\|f\right\|_{L^\infty}+\left\|\left(u_p(y)-\hat c\right)\pa_yf\right\|_{L^\infty}+\left\|\left(u_p(y)-\hat c\right)^2\pa_y^2f\right\|_{L^\infty},
  \end{align*}
  and
  \begin{align*}
    \pa_yRaySolver_0(f)(0)=0.
  \end{align*}
  If, in addition, $\pa_y f(0)=0$, then
  \begin{align}\label{eq-Ray-0-3-0}
    \pa_y^3RaySolver_0(f)(0)=0.
  \end{align}
\end{lemma}
\begin{proof}
 We first note that $|\hat c|\ll 1$ and
  \begin{align*}
    \left(u_p-\hat c\right)\ln{\left(\frac{\sqrt{1-\hat c}+y}{\sqrt{1-\hat c}-y}\right)}=\left(\sqrt{1-\hat c}+y\right)\left(\sqrt{1-\hat c}-y\right)\left(\ln{(\sqrt{1-\hat c}+y)}-\ln{(\sqrt{1-\hat c}-y)}\right).
  \end{align*}
  It follows that $\phi_{2,0}$ is uniformly bounded independently of $\hat c$.

  Recall that $y_c$ denotes the location of the real critical layer, defined by $u_p(y_c)-c_r=0$. For Poiseuille flow, we have
\begin{align}\label{eq-yc}
  y_c=-\sqrt{1-c_r}.
\end{align}
Since $|\hat c|\ll1$, we have $y_c\approx -1$ and hence $u_p'(y_c)=-2y_c\approx 2$.

  Then we have
  \begin{align*}
  \left|\phi_{1,0}(y)\int^y_{-1}\phi_{2,0}(x)\frac{f(x)}{u_p(x)-\hat c}dx\right|\lesssim \left|\int^y_{-1} \frac{1}{\left|u_p(x)-\hat c\right|}dx\right| \left\|f\right\|_{L^\infty}\lesssim \left|\ln \hat c_i\right|\left\|f\right\|_{L^\infty}.
\end{align*}
\begin{align*}
  \left|\phi_{2,0}(y)\int^0_{y}\phi_{1,0}(x)\frac{f(x)}{u_p(x)-\hat c}dx\right|=\left|\phi_{2,0}(y)\int^0_{y}f(x)dx\right|\lesssim  \left\|f\right\|_{L^\infty}.
\end{align*}
It follows that
\begin{align}\label{eq-Raysolver-0}
  \left\|RaySolver_0(f)\right\|_{L^\infty}\lesssim&\left|\ln \hat c_i\right|\left\|f\right\|_{L^\infty}.
\end{align}

From \eqref{eq-op-solver-Ray0}, we have
\begin{align*}
  \pa_y RaySolver_0(f)=-\phi_{1,0}'(y)\int^y_{-1}\phi_{2,0}(x)\frac{f(x)}{u_p(x)-\hat c}dx-\phi_{2,0}'(y)\int^0_{y}\phi_{1,0}(x)\frac{f(x)}{u_p(x)-\hat c}dx,
\end{align*}
and
\begin{align*}
  \left|\phi_{1,0}'(y)\int^y_{-1}\phi_{2,0}(x)\frac{f(x)}{u_p(x)-\hat c}dx\right|\lesssim \left|\int^y_{-1} \frac{1}{\left|u_p(x)-\hat c\right|}dx\right| \left\|f\right\|_{L^\infty}\lesssim \left\|f\right\|_{L^\infty}\left|\ln \hat c_i\right|,
\end{align*}
\begin{align*}
  \left|\phi_{2,0}'(y)\int^0_{y}\phi_{1,0}(x)\frac{f(x)}{u_p(x)-\hat c}dx\right|\lesssim \left\|\phi_{2,0}'\right\|_{L^\infty}\left\|f\right\|_{L^\infty}\lesssim \left\|f\right\|_{L^\infty}\left|\ln \hat c_i\right|.
\end{align*}
Therefore, we have
\begin{align}\label{eq-Raysolver-1}
  \left\|\pa_yRaySolver_0(f)\right\|_{L^\infty}\lesssim&\left|\ln \hat c_i\right|\left\|f\right\|_{L^\infty}.
\end{align}

Moreover, as $RaySolver_0(f)$ solves \eqref{eq-Ray} with $\alpha=0$, it holds that
\begin{align*}
  \left(u_p(y)-\hat c\right)\pa_{yy} RaySolver_0(f)=u_p''RaySolver_0(f)+f(y),
\end{align*}
and from \eqref{eq-Raysolver-0} that
\begin{align}\label{eq-Raysolver-2}
  \left\|\left(u_p(y)-\hat c\right)\pa_{yy} RaySolver_0(f)\right\|_{L^\infty}\lesssim&\left|\ln \hat c_i\right|\left\|f\right\|_{L^\infty}.
\end{align}

We next estimate the higher weighted derivatives.

Applying $\left(u_p(y)-\hat c\right)\pa_y$ to \eqref{eq-Ray} with $\alpha=0$, we have
\begin{align*}
  \left(u_p(y)-\hat c\right)u_p' \pa_y^2\phi+\left(u_p(y)-\hat c\right)^2\pa_y^3\phi-u_p''\left(u_p(y)-\hat c\right)\pa_y\phi=\left(u_p(y)-\hat c\right)\pa_yf.
\end{align*}
It follows that
\begin{align}\label{eq-wt-3-Ray}
  \left(u_p(y)-\hat c\right)^2\pa_y^3\phi=-\left(u_p(y)-\hat c\right)u_p' \pa_y^2\phi+u_p''\left(u_p(y)-\hat c\right)\pa_y\phi+\left(u_p(y)-\hat c\right)\pa_yf.
\end{align}
Combining this identity with \eqref{eq-Raysolver-1} and \eqref{eq-Raysolver-2}, we obtain
\begin{align*}
  \left\|(u_p-\hat c)^2\pa_{y}^3 RaySolver_0(f)\right\|_{L^\infty}\lesssim \left|\ln \hat c_i\right| \left\|f\right\|_{L^\infty}+\left\|\left(u_p(y)-\hat c\right)\pa_yf\right\|_{L^\infty}.
\end{align*}

Similarly, applying $\left(u_p(y)-\hat c\right) \pa_y$  to \eqref{eq-wt-3-Ray}, we have
\begin{align*}
  \left(u_p(y)-\hat c\right)^3\pa_y^4\phi=&-\left(u_p(y)-\hat c\right)(u_p')^2 \pa_y^2\phi-3\left(u_p(y)-\hat c\right)^2u_p' \pa_y^3\phi+u_p''\left(u_p(y)-\hat c\right)u_p'\pa_y\phi\\ 
  &+\left(u_p(y)-\hat c\right)u_p'\pa_yf+\left(u_p(y)-\hat c\right)^2\pa_y^2f.
\end{align*}
Therefore
\begin{align*}
  \left\|\left(u_p(y)-\hat c\right)^3\pa_y^4 RaySolver_0(f)\right\|_{L^\infty}\lesssim \left|\ln \hat c_i\right| \left\|f\right\|_{L^\infty}+\left\|\left(u_p(y)-\hat c\right)\pa_yf\right\|_{L^\infty}+\left\|\left(u_p(y)-\hat c\right)^2\pa_y^2f\right\|_{L^\infty}.
\end{align*}

From \eqref{eq-phi0-boundary}, we have
\begin{align*}
  \pa_yRaySolver_0(f)(0)=-\phi_{1,0}'(0)\int^0_{-1}\phi_{2,0}(x)\frac{f(x)}{u_p(x)-\hat c}dx=0.
\end{align*}
Moreover, differentiating the equation once more gives
\begin{align*}
  \pa_y^3RaySolver_0(f)=\frac{-u_p' \pa_y^2RaySolver_0(f)+u_p''\pa_yRaySolver_0(f)+\pa_yf}{u_p(y)-\hat c}.
\end{align*}
If $\pa_y f(0)=0$, then evaluating the above identity at $y=0$ and using
$u_p'(0)=0$ and $\pa_y RaySolver_0(f)(0)=0$, we obtain
\begin{align*}
  \pa_y^3RaySolver_0(f)(0)=0.
\end{align*}

Combining the above estimates, we complete the proof of this lemma.
\end{proof} 

\subsection{Non-homogeneous solution for $\alpha\neq0$}
We now turn to the case $\alpha\neq0$. Here we assume that $|\alpha|\ll\frac{1}{\left|\ln \hat c_i\right|}$. Without loss of generality, we assume that $\alpha>0$.
\begin{lemma}\label{lem-Raysolver-alpha}
  For $f\in X_2$, there exists a solution operator $RaySolver_\alpha$ such that $RaySolver_\alpha(f)$ solves \eqref{eq-Ray}. Moreover, the following estimates hold:
  \begin{align*}
    \left\|RaySolver_\alpha(f)\right\|_{Y_2}\lesssim&\left|\ln \hat c_i\right|\left\|f\right\|_{L^\infty},\\
    \left\|RaySolver_\alpha(f)\right\|_{Y_4}\lesssim&\left|\ln \hat c_i\right| \left\|f\right\|_{L^\infty}+\left\|\left(u_p(y)-\hat c\right)\pa_yf\right\|_{L^\infty}+\left\|\left(u_p(y)-\hat c\right)^2\pa_y^2f\right\|_{L^\infty},
  \end{align*}
  and
  \begin{align*}
    \pa_yRaySolver_\alpha(f)(0)=0.
  \end{align*}
  If, in addition, $\pa_y f(0)=0$, then
  \begin{align}\label{eq-Ray-alpha-3-0}
    \pa_y^3RaySolver_\alpha(f)(0)=0.
  \end{align}
\end{lemma}
\begin{proof}
We first regard the $\alpha^2$ term as a perturbation of $Ray_0$.  From \eqref{eq-op-Ray-alpha} and \eqref{eq-op-solver-Ray0}, we have
\begin{align*}
  Ray_\alpha \left(RaySolver_0(f) \right)=f-\alpha^2\left(u_p(y)-\hat c\right)RaySolver_0(f).
\end{align*}
Define
\begin{align*}
  S_{R,0}(f)=&RaySolver_0(f),\\ 
  S_{R,j}(f)=&RaySolver_0 \left(\alpha^2\left(u_p(y)-\hat c\right)S_{R,j-1}(f)\right).
\end{align*}
Using Lemma \ref{lem-Raysolver0}, we have
\begin{align*}
    \left\|S_{R,j}(f)\right\|_{L^\infty}\lesssim \left(|\ln \hat c_i|\alpha^2\right)^j|\ln \hat c_i|\left\|f\right\|_{L^\infty},
\end{align*}
and
\begin{align*}
  \left\|\pa_yS_{R,j}(f)\right\|_{L^\infty}\lesssim \left(|\ln \hat c_i|\alpha^2\right)\left\|S_{R,j-1}(f)\right\|_{L^\infty}\lesssim \left(|\ln \hat c_i|\alpha^2\right)^j|\ln \hat c_i|\left\|f\right\|_{L^\infty}.
\end{align*}
The same argument gives 
\begin{align*}
  \left\|S_{R,j}(f)\right\|_{Y_2}\lesssim \left(|\ln \hat c_i|\alpha^2\right)^j \left|\ln \hat c_i\right| \left\|f\right\|_{L^\infty},
\end{align*}
and
\begin{align*}
  \left\|S_{R,j}(f)\right\|_{Y_4}\lesssim \left(|\ln \hat c_i|\alpha^2\right)^j \left(\left|\ln \hat c_i\right| \left\|f\right\|_{L^\infty}+\left\|\left(u_p(y)-\hat c\right)\pa_yf\right\|_{L^\infty}+\left\|\left(u_p(y)-\hat c\right)^2\pa_y^2f\right\|_{L^\infty}\right). 
\end{align*}

We also note that, for $j\ge 2$, we only need to use the $L^\infty$ norm of $f$,
\begin{align*}
  \left\|S_{R,j}(f)\right\|_{Y_4}\lesssim \left(|\ln \hat c_i|\alpha^2\right)^j \left|\ln \hat c_i\right| \left\|f\right\|_{L^\infty}. 
\end{align*}

Recall that $\alpha^2|\ln \hat c_i|\ll1$. Hence the series $\sum_{j=0}^{\infty}S_{R,j}(f)$ converges in $Y_4$. By the recursive definition, its partial sums satisfy
\[
  Ray_\alpha\sum_{j=0}^{N}S_{R,j}(f)
  =
  f-\alpha^2(u_p-\hat c)S_{R,N}(f).
\]
Letting $N\to\infty$ and setting
\[
  RaySolver_\alpha(f)=\sum_{j=0}^{\infty}S_{R,j}(f),
\]
we obtain the desired solution operator for \eqref{eq-Ray}. The identities at $y=0$ follow from \eqref{eq-Ray-0-3-0} by induction.
\end{proof}
\subsection{Homogeneous solution for $\alpha\neq0$}
We next construct two linearly independent solutions of \eqref{eq-Ray} with $f=0$. Recall that $|\alpha|\ll\frac{1}{\left|\ln \hat c_i\right|}$. The idea is to use the Rayleigh solver for $\alpha=0$ to perturb the two solutions $\phi_{1,0}$ and $\phi_{2,0}$.

\begin{lemma}\label{lem-sol-Ray-hom}
There exist two solutions
  $\phi_{1,\alpha}(y)$ and $\phi_{2,\alpha}(y)$ of \eqref{eq-Ray} with $f=0$ such that
  \begin{align*}
    \left\|\phi_{1,\alpha}-\phi_{1,0}\right\|_{Y_4}
    \lesssim&
    \alpha^2\left|\ln \hat c_i\right|,\\
    \left\|\phi_{2,\alpha}-\phi_{2,0}\right\|_{Y_4}
    \lesssim&
    \alpha^2\left|\ln \hat c_i\right|.
  \end{align*}
  Moreover,
  \begin{align*}
    \pa_y\phi_{1,\alpha}(0)=\pa_y^3\phi_{1,\alpha}(0)=0,
  \end{align*}
  and
  \begin{align*}
    \pa_y\phi_{2,\alpha}(0)
    =
    \pa_y\phi_{2,0}(0)
    =
    \frac{1}{1-\hat c},
    \qquad
    \pa_y^3\phi_{2,\alpha}(0)
    =
    -\frac{2}{(1-\hat c)^2}
    +\frac{\alpha^2}{1-\hat c}.
  \end{align*}
\end{lemma}

\begin{proof}
  We define the iterates associated with $\phi_{1,0}$ by
  \begin{align*}
    \phi_{1,0}^{[0]}(y)=\phi_{1,0}(y),\qquad
    \phi_{1,0}^{[j]}(y)
    =
    RaySolver_0\left(
    \alpha^2\left(u_p(y)-\hat c\right)\phi_{1,0}^{[j-1]}
    \right),
    \qquad j\ge1.
  \end{align*}
  Similarly, we define
  \begin{align*}
    \phi_{2,0}^{[0]}(y)=\phi_{2,0}(y),\qquad
    \phi_{2,0}^{[j]}(y)
    =
    RaySolver_0\left(
    \alpha^2\left(u_p(y)-\hat c\right)\phi_{2,0}^{[j-1]}
    \right),
    \qquad j\ge1.
  \end{align*}
  We then set
  \begin{align*}
    \phi_{1,\alpha}(y)=\sum_{j=0}^{\infty}\phi_{1,0}^{[j]}(y),
    \qquad
    \phi_{2,\alpha}(y)=\sum_{j=0}^{\infty}\phi_{2,0}^{[j]}(y).
  \end{align*}
  The estimates in Lemma \ref{lem-Raysolver0}, together with the smallness of
  $\alpha^2|\ln \hat c_i|$, imply that these two series converge in $Y_4$. Moreover,
  \begin{align*}
    \left\|\sum_{j=1}^{\infty}\phi_{1,0}^{[j]}\right\|_{Y_4}
    \lesssim
    \alpha^2|\ln \hat c_i|,
    \qquad
    \left\|\sum_{j=1}^{\infty}\phi_{2,0}^{[j]}\right\|_{Y_4}
    \lesssim
    \alpha^2|\ln \hat c_i|.
  \end{align*}
  This gives the stated bounds for
  $\phi_{1,\alpha}-\phi_{1,0}$ and $\phi_{2,\alpha}-\phi_{2,0}$.

  The same telescoping argument as in Lemma \ref{lem-Raysolver-alpha} shows that
  both $\phi_{1,\alpha}$ and $\phi_{2,\alpha}$ solve \eqref{eq-Ray} with $f=0$.
  They are linearly independent for $\alpha$ sufficiently small, since they are small
  perturbations of the linearly independent pair $\phi_{1,0}$ and $\phi_{2,0}$.

  We next prove the identities at $y=0$. For the first mode, $\phi_{1,0}'(0)=0$ and  $\phi_{1,0}'''(0)=0$, then for the same reason as \eqref{eq-Ray-alpha-3-0}, we have
    \begin{align*}
    \pa_y\phi_{1,\alpha}(0)=\pa_y^3\phi_{1,\alpha}(0)=0.
  \end{align*}
  Similarly, 
  \begin{align*}
    \pa_y\phi_{2,\alpha}(0)
    =
    \pa_y\phi_{2,0}(0)
    =
    \frac{1}{1-\hat c},
  \end{align*}
  and
  \begin{align*}
    \pa_y^3\phi_{2,\alpha}(0)
    =\pa_y^3\phi_{2,0}(0)+\pa_y^3\phi_{2,0}^{[1]}(0)=
    -\frac{2}{(1-\hat c)^2}
    +\frac{\alpha^2}{1-\hat c}.
  \end{align*}
 
  This completes the proof.
\end{proof}
\begin{lemma} \label{lem-Ray-endpoint-expansion}
  For  the functions $\phi_{1,\alpha}(y)$ and $\phi_{2,\alpha}(y)$  constructed in Lemma \ref{lem-sol-Ray-hom}, it holds that
  \begin{align*}
    \left|\phi_{1,\alpha}'(-1)-\phi_{1,0}'(-1)\right|
    &\lesssim \alpha^2\left|\ln \hat c\right|,\\
    \left|\pa_{\hat c}\phi_{1,\alpha}(-1)-\pa_{\hat c}\phi_{1,0}(-1)\right|
    &\lesssim \alpha^2\left|\ln \hat c\right|,\\
    \left|\pa_{\hat c}\phi_{1,\alpha}'(-1)-\pa_{\hat c}\phi_{1,0}'(-1)\right|
    &\lesssim \frac{\alpha^2}{\left|\hat c\right|}.
  \end{align*}
  Moreover,
  \begin{align*}
    \Re \left(\frac{\phi_{1,\alpha}(-1)}{\phi_{1,\alpha}'(-1)}\right)
    =
    &-\frac{c_r}{2}
    +\frac{2}{15}\alpha^2
    -\frac{1}{3}c_r\alpha^2
    +\frac{4}{15^2}\alpha^4\ln \left|\frac{\hat c}{4}\right|\\
    &+O\left(\alpha^4+\alpha^4 \hat c\left|\ln \hat c\right|
    +\alpha^2 \hat c^2\left|\ln \hat c\right|\right),
  \end{align*}
  and
  \begin{align*}
    \Im \left(\frac{\phi_{1,\alpha}(-1)}{\phi_{1,\alpha}'(-1)}\right)
    =
    -\frac{\hat c_i}{2}
    -\frac{1}{3}\hat c_i\alpha^2
    +\frac{4}{15^2}\alpha^4\arg(-\hat c)
    +O \left(\alpha^4 \hat c\left|\ln \hat c\right|+\alpha^2 \hat c^2\left|\ln \hat c\right|+\alpha^6\left|\ln \hat c\right|\right).
  \end{align*}
  Consequently,
  \begin{align*}
    \pa_{c_r}\Re \left(\frac{\phi_{1,\alpha}(-1)}{\phi_{1,\alpha}'(-1)}\right)
    =
    \pa_{\hat c_i}\Im \left(\frac{\phi_{1,\alpha}(-1)}{\phi_{1,\alpha}'(-1)}\right)
    =
    -\frac{1}{2}
    -\frac{1}{3}\alpha^2
    +\frac{4}{15^2}\alpha^4
    \frac{c_r^{-1}}{1+\left(\frac{\hat c_i}{c_r}\right)^2}
    +O \left(\alpha^4\left|\ln \hat c\right|
    +\alpha^2 \hat c\left|\ln \hat c\right|\right),
  \end{align*}
  and
  \begin{align*}
    \pa_{\hat c_i}\Re \left(\frac{\phi_{1,\alpha}(-1)}{\phi_{1,\alpha}'(-1)}\right)
    =
    -\pa_{c_r}\Im \left(\frac{\phi_{1,\alpha}(-1)}{\phi_{1,\alpha}'(-1)}\right)
    =
    \frac{4}{15^2}\alpha^4
    \frac{\hat c_i c_r^{-2}}{1+\left(\frac{\hat c_i}{c_r}\right)^2}
    +O \left(\alpha^4\left|\ln \hat c\right|
    +\alpha^2 \hat c\left|\ln\hat  c\right|\right).
  \end{align*}
\end{lemma}

\begin{proof}
  We use the iterative representation introduced in the proof of
  Lemma \ref{lem-sol-Ray-hom}. We first estimate the imaginary part of
  $\phi_{2,0}$. Recall that
  \begin{align*}
    \phi_{2,0}(y)
    =
    \frac{y}{2(1-\hat c)}
    +
    \frac{\left(u_p-\hat c\right)}{4(1-\hat c)^{\frac{3}{2}}}
    \ln{\frac{u_p-\hat c}{(\sqrt{1-\hat c}-y)^2}}.
  \end{align*}
  Here
  \begin{align*}
    \ln{\frac{u_p-\hat c}{(\sqrt{1-\hat c}-y)^2}}
    =
    \ln{\left|\frac{u_p-\hat c}{(\sqrt{1-\hat c}-y)^2}\right|}
    +
    i\arg \left(\frac{u_p-\hat c}{(\sqrt{1-\hat c}-y)^2}\right).
  \end{align*}
  Recall that $\hat c_i\lesssim c_r$ and $y_c+1\sim c_r$. For $y\ge y_c+c_r$, we have
  \begin{align*}
    \left|\arg \left(\frac{u_p-\hat c}{(\sqrt{1-\hat c}-y)^2}\right)\right|
    \lesssim
    \arctan \left(\frac{\hat c_i}{u_p-c_r}\right),
  \end{align*}
  and therefore
  \begin{align*}
    \left|\Im \left(\phi_{2,0}(y)\right)\right|
    \lesssim
    \hat c_i\left|\ln c_r\right|.
  \end{align*}
  On the other hand, for $y< y_c+c_r$, we have
  $\left|u_p-\hat c\right|\lesssim c_r$, and hence
  \begin{align*}
    \left|\Im \left(\phi_{2,0}(y)\right)\right|
    \lesssim
    c_r+\hat c_i\left|\ln c_r\right|.
  \end{align*}
  Thus, for all $y\in[-1,0]$,
  \begin{align}\label{eq-est-phi20im-alpha}
    \left|\Im \left(\phi_{2,0}(y)\right)\right|
    \lesssim
    c_r+\hat c_i\left|\ln c_r\right|.
  \end{align}

  We next compute the first correction to $\phi_{1,0}$. By the definition, we have
\begin{align}\label{eq-phi-1-0-1}
  \phi_{1,0}^{[1]}(y)=-\alpha^2\phi_{1,0}(y)\int^y_{-1}\phi_{2,0}(x) \left(u_p(x)-\hat c\right)dx-\alpha^2\phi_{2,0}(y)\int^0_{y}\phi_{1,0}(x) \left(u_p(x)-\hat c\right)dx.
\end{align}
  Combining this formula with \eqref{eq-est-phi20im-alpha}, we obtain
  \begin{align}\label{eq-est-phi101RI}
    \begin{aligned}
      \left|\Re \left(\phi_{1,0}^{[1]}(y)\right)\right|
      &\lesssim \alpha^2,\qquad \left|\Im \left(\phi_{1,0}^{[1]}(y)\right)\right|
      &\lesssim
      \alpha^2c_r+\alpha^2\hat c_i\left|\ln c_r\right|.
    \end{aligned}
  \end{align}

  On the boundary $y=-1$, it holds that
  \begin{align*}
    \phi_{2,0}(-1)
    =
    -\frac{1}{2(1-\hat c)}
    -\frac{\hat c}{4(1-\hat c)^{\frac{3}{2}}}
    \ln{\frac{-\hat c}{(\sqrt{1-\hat c}+1)^2}}.
  \end{align*}
  Since $|\hat c|\ll1$, we have
\begin{align}\label{eq-ln-expan}
  \ln{\frac{-\hat c}{(\sqrt{1-\hat c}+1)^2}}=\ln {\left|\frac{\hat c}{4}\right|}+i \arg(-\hat c)+O(\hat c).
\end{align}
  Consequently,
  \begin{align*}
    \Re \left(\phi_{2,0}(-1)\right)
    =
    -\frac{1}{2}
    -\frac{c_r}{2}
    -\frac{c_r}{4}\ln \left|\frac{\hat c}{4}\right|
    +\frac{\hat c_i}{4}\arg(-\hat c)
    +O(\hat c^2|\ln\hat c|),
  \end{align*}
  and
  \begin{align*}
    \Im \left(\phi_{2,0}(-1)\right)
    =
    -\frac{\hat c_i}{2}
    -\frac{\hat c_i}{4}\ln \left|\frac{\hat c}{4}\right|
    -\frac{c_r}{4}\arg(-\hat c)
    +O(\hat c^2|\ln\hat  c|).
  \end{align*}
  Since
  \begin{align*}
    \int_{-1}^{0}\left(u_p-\hat c\right)^2\,dx
    =
    \frac{8}{15}-\frac{4}{3}\hat c+\hat c^2,
  \end{align*}
  we obtain
  \begin{align}\label{eq-phi101-at-minus-one}
    \phi_{1,0}^{[1]}(-1)
    =
    \alpha^2
    \bigg[
    &\frac{4}{15}
    -\frac{2}{5}c_r
    +\frac{2c_r}{15}\ln \left|\frac{\hat c}{4}\right|
    -\frac{2\hat c_i}{15}\arg(-\hat c) \nonumber\\
    &\quad
    +i\left(
    \frac{2c_r}{15}\arg(-\hat c)
    -\frac{2}{5}\hat c_i
    +\frac{2\hat c_i}{15}\ln \left|\frac{\hat c}{4}\right|
    \right)
    \bigg]
    +O(\alpha^2\hat c^2|\ln \hat c|).
  \end{align}

  We also need the corresponding expansion of the derivative. Recall that
  \begin{align*}
    \phi_{2,0}'(-1)
    =
    \frac{1}{1-\hat c}
    +
    \frac{1}{2(1-\hat c)^{\frac{3}{2}}}
    \ln{\frac{-\hat c}{(\sqrt{1-\hat c}+1)^2}}.
  \end{align*}
  By \eqref{eq-ln-expan},
  \begin{align*}
    \Re \left(\phi_{2,0}'(-1)\right)
    =
    1+\frac{1}{2}\ln \left|\frac{\hat c}{4}\right|
    +O(\hat c|\ln \hat c|),
  \end{align*}
  and
  \begin{align*}
    \Im \left(\phi_{2,0}'(-1)\right)
    =
    \frac{1}{2}\arg(-\hat c)
    +O(\hat c|\ln \hat c|).
  \end{align*}
  From
  \begin{align*}
    \phi_{1,0}^{[1]\prime}(-1)
    =
    -\alpha^2\phi_{2,0}'(-1)
    \int_{-1}^{0}\left(u_p(x)-\hat c\right)^2\,dx,
  \end{align*}
  we get
  \begin{align}\label{eq-phi101prime-at-minus-one}
    \Re\left(\phi_{1,0}^{[1]\prime}(-1)\right)
    =
    -\alpha^2\frac{8}{15}
    -\alpha^2\frac{4}{15}\ln \left|\frac{\hat c}{4}\right|
    +O(\alpha^2\hat c|\ln \hat c|),
  \end{align}
  and
  \begin{align}\label{eq-phi101prime-at-minus-one-im}
    \Im\left(\phi_{1,0}^{[1]\prime}(-1)\right)
    =
    -\alpha^2\frac{4}{15}\arg(-\hat c)
    +O(\alpha^2\hat c|\ln \hat c|).
  \end{align}

  Combining \eqref{eq-phi101-at-minus-one},
  \eqref{eq-phi101prime-at-minus-one}, and
  \eqref{eq-phi101prime-at-minus-one-im}, we get
\begin{align*}
  &\frac{\phi_{1,0}(-1)+\phi_{1,0}^{[1]}(-1)}{\phi_{1,0}'(-1)+\phi_{1,0}^{[1]\prime}(-1)}=\frac{-\hat c+\phi_{1,0}^{[1]}(-1)}{2+\phi_{1,0}^{[1]\prime}(-1)}=\frac{1}{2} \left( -\hat c+\phi_{1,0}^{[1]}(-1) \right) \left(1-\frac{1}{2}\phi_{1,0}^{[1]\prime}(-1)+O(\alpha^4\left|\ln \hat c\right| ) \right)\\
  =&\frac{-\hat c}{2}+\frac{2}{15}\alpha^2-\frac{1}{3}c_r\alpha^2+\frac{8}{15^2}\alpha^4+\frac{4}{15^2}\alpha^4\ln {\left|\frac{\hat c}{4}\right|}+i \left(-\frac{1}{3}\hat c_i\alpha^2+\frac{4}{15^2}\alpha^4\arg(-\hat c)\right)\\
  &+O \left(\alpha^4 \hat c\left|\ln \hat c\right|+\alpha^2 \hat c^2\left|\ln \hat c\right|+\alpha^6\left|\ln \hat c\right|\right).
\end{align*}

From the definition, we have
\begin{align*}
  \phi_{1,0}^{[2]}(y)=-\alpha^2\phi_{1,0}(y)\int^y_{-1}\phi_{2,0}(x)\phi_{1,0}^{[1]}(x)dx-\alpha^2\phi_{2,0}(y)\int^0_{y}\phi_{1,0}(x)\phi_{1,0}^{[1]}(x)dx,\\
  \phi_{1,0}^{[2]\prime}(y)=-\alpha^2\phi_{1,0}'(y)\int^y_{-1}\phi_{2,0}(x)\phi_{1,0}^{[1]}(x)dx-\alpha^2\phi_{2,0}'(y)\int^0_{y}\phi_{1,0}(x)\phi_{1,0}^{[1]}(x)dx.
\end{align*}
Similar to \eqref{eq-est-phi101RI}, one can check that
\begin{equation}\label{eq-est-phi102RI}
  \begin{aligned}    
  \left|\Re \left(\phi_{1,0}^{[2]}(y)\right)\right|\lesssim \alpha^4,\quad
  \left|\Im \left(\phi_{1,0}^{[2]}(y)\right)\right|\lesssim \alpha^4c_r+\alpha^4\hat c_i\left|\ln c_r\right|,\\
  \left|\Re \left(\phi_{1,0}^{[2]\prime}(y)\right)\right|\lesssim \alpha^4 \left(1+\left|\ln \left(\sqrt{1-\hat c}+y\right)\right|\right),\quad
  \left|\Im \left(\phi_{1,0}^{[2]\prime}(y)\right)\right|\lesssim \alpha^4 \left(1+\left|\ln \left(\sqrt{1-\hat c}+y\right)\right|\right).   
  \end{aligned}
\end{equation}
Moreover,
\begin{equation}\label{eq-est-phi10jRI}
  \begin{aligned}    
  \left|Re \left(\phi_{1,0}^{[j]}(y)\right)\right|\lesssim \alpha^{2j},\quad
  \left|Im \left(\phi_{1,0}^{[j]}(y)\right)\right|\lesssim \alpha^{2j}c_r+\alpha^{2j}c_i\left|\ln c_r\right|,\\
  \left|Re \left(\phi_{1,0}^{[j]\prime}(y)\right)\right|\lesssim \alpha^{2j}\left(1+\left|\ln \left(\sqrt{1-\hat c}+y\right)\right|\right),\quad
  \left|Im \left(\phi_{1,0}^{[j]\prime}(y)\right)\right|\lesssim \alpha^{2j}\left(1+\left|\ln \left(\sqrt{1-\hat c}+y\right)\right|\right).\\
  \end{aligned}
\end{equation}

  These bounds show that the higher-order corrections contribute only to the stated
  remainders. Therefore,
  \begin{align*}
    \Re \left(\frac{\phi_{1,\alpha}(-1)}{\phi_{1,\alpha}'(-1)}\right)
    =
    &-\frac{c_r}{2}
    +\frac{2}{15}\alpha^2
    -\frac{1}{3}c_r\alpha^2
    +\frac{4}{15^2}\alpha^4\ln \left|\frac{\hat c}{4}\right|\\
    &+O\left(\alpha^4+\alpha^4\hat c|\ln \hat c|^2
    +\alpha^2\hat c^2|\ln \hat c|\right),
  \end{align*}
  and
  \begin{align*}
   \Im \left(\frac{\phi_{1,\alpha}(-1)}{\phi_{1,\alpha}'(-1)}\right)=\frac{-\hat c_i}{2}-\frac{1}{3}\hat c_i\alpha^2+\frac{4}{15^2}\alpha^4\arg(-\hat c)+O \left(\alpha^4 \hat c\left|\ln \hat c\right|+\alpha^2 \hat c^2\left|\ln \hat c\right|+\alpha^6\left|\ln \hat c\right|\right).
  \end{align*}

  Finally, we justify the estimates involving $\pa_{\hat c}$. From the recursive construction,
  $\phi_{1,\alpha}(y)$ is analytic in $\hat c$. More precisely, by induction one may write
  \begin{align*}
    \phi_{1,0}^{[j]}(y)
    =
    \alpha^{2j}P_j(y,\hat c)
    +
    \alpha^{2j}Q_j(y,\hat c)
    \left(\sqrt{1-\hat c}+y\right)
    \ln \left(\sqrt{1-\hat c}+y\right)+err,
  \end{align*}
  where $P_j(y,\hat c)$ and $Q_j(y,\hat c)$ are smooth in $y$ and analytic in $\hat c$, and
  $\left|P_j(y,\hat c)\right|$, $\left|Q_j(y,\hat c)\right|$,
  $\left|\pa_{\hat c}P_j(y,\hat c)\right|$, and $\left|\pa_{\hat c}Q_j(y,\hat c)\right|$ are uniformly bounded.
  Here $err$ denotes a finite
sum of smoother logarithmic terms, each of which contains at least two powers of
$\sqrt{1-\hat c}+y$, namely terms of the form
\[
  \alpha^{2j}
  \left(\sqrt{1-\hat c}+y\right)^m
  \left(\ln\left(\sqrt{1-\hat c}+y\right)\right)^k,
  \qquad m\ge2,\quad k\ge0,
\]
multiplied by coefficients that are smooth in $y$, analytic in $\hat c$, and
uniformly bounded together with their $\hat c$-derivatives. These terms are
smoother at the logarithmic singularity $y=-\sqrt{1-\hat c}$ and are therefore
absorbed in the estimates below.
  In particular,
  \begin{align*}
    \left|\phi_{1,0}^{[j]}(y)\right|
    &\lesssim \alpha^{2j},\\
    \left|\phi_{1,0}^{[j]\prime}(y)\right|
    &\lesssim
    \alpha^{2j}
    \left(1+\left|\ln \left(\sqrt{1-\hat c}+y\right)\right|\right),\\
    \left|\pa_{\hat c}\phi_{1,0}^{[j]}(y)\right|
    &\lesssim
    \alpha^{2j}
    \left(1+\left|\ln \left(\sqrt{1-\hat c}+y\right)\right|\right),\\
    \left|\pa_{\hat c}\phi_{1,0}^{[j]\prime}(y)\right|
    &\lesssim
    \alpha^{2j}
    \frac{1}{\left|\sqrt{1-\hat c}+y\right|}.
  \end{align*}
  Evaluating these bounds at $y=-1$ gives
  \begin{align*}
    \left|\phi_{1,\alpha}'(-1)-\phi_{1,0}'(-1)\right|
    &\lesssim \alpha^2\left|\ln \hat c\right|,\\
    \left|\pa_{\hat c}\phi_{1,\alpha}(-1)-\pa_{\hat c}\phi_{1,0}(-1)\right|
    &\lesssim \alpha^2\left|\ln\hat  c\right|,\\
    \left|\pa_{\hat c}\phi_{1,\alpha}'(-1)-\pa_{\hat c}\phi_{1,0}'(-1)\right|
    &\lesssim \frac{\alpha^2}{|\hat c|}.
  \end{align*}

  Since the quotient
  \[
    \frac{\phi_{1,\alpha}(-1)}{\phi_{1,\alpha}'(-1)}
  \]
  is analytic in $\hat c$, the derivative identities follow from the above expansion and the
  Cauchy--Riemann relations. This gives
  \begin{align*}
    \pa_{c_r}\Re \left(\frac{\phi_{1,\alpha}(-1)}{\phi_{1,\alpha}'(-1)}\right)
    =
    \pa_{\hat c_i}\Im \left(\frac{\phi_{1,\alpha}(-1)}{\phi_{1,\alpha}'(-1)}\right)
    =
    -\frac{1}{2}
    -\frac{1}{3}\alpha^2
    +\frac{4}{15^2}\alpha^4
    \frac{c_r^{-1}}{1+\left(\frac{\hat c_i}{c_r}\right)^2}
    +O\left(\alpha^4|\ln \hat c|
    +\alpha^2\hat c|\ln \hat c|\right),
  \end{align*}
  and
  \begin{align*}
    \pa_{\hat c_i}\Re \left(\frac{\phi_{1,\alpha}(-1)}{\phi_{1,\alpha}'(-1)}\right)
    =
    -\pa_{c_r}\Im \left(\frac{\phi_{1,\alpha}(-1)}{\phi_{1,\alpha}'(-1)}\right)
    =
    \frac{4}{15^2}\alpha^4
    \frac{\hat c_i c_r^{-2}}{1+\left(\frac{\hat c_i}{c_r}\right)^2}
    +O\left(\alpha^4|\ln \hat c|
    +\alpha^2\hat c|\ln \hat c|\right).
  \end{align*}
  The proof is complete.
\end{proof}
\begin{remark}\label{rmk-Ray-endpoint-expansion}
  In the subregion $|\hat c_i|\ll |c_r|$, we have
  \begin{align*}
    \arg(-\hat c)
    =
    -\pi+\frac{\hat c_i}{c_r}+O\left(\left|\frac{\hat c_i}{c_r}\right|^2\right).
  \end{align*}
  Hence \eqref{eq-phi101-at-minus-one} reduces to
  \begin{align*}
    \phi_{1,0}^{[1]}(-1)
    =
    \alpha^2
    \bigg[
    &\frac{4}{15}
    -\frac{2}{5}c_r
    +\frac{2c_r}{15}\ln \left|\frac{\hat c}{4}\right|\\
    &\quad
    +i\left(
    -\frac{2\pi c_r}{15}
    -\frac{4}{15}\hat c_i
    +\frac{2\hat c_i}{15}\ln \left|\frac{\hat c}{4}\right|
    \right)
    \bigg]
    +O\left(\alpha^2\hat c^2|\ln \hat c|+\alpha^2 \frac{\hat c_i^2}{c_r}  \right),
  \end{align*}
  and
  \begin{align*}
    \phi_{1,0}^{[1]\prime}(-1)
    =
    -\alpha^2
    \left(
    \frac{8}{15}
    +\frac{4}{15}\ln \left|\frac{\hat c}{4}\right|
    -i\frac{4\pi}{15}
    \right)
    +O\left(\alpha^2\hat c|\ln \hat c|
    +\frac{\alpha^2\hat c_i}{c_r}\right).
  \end{align*}
  Consequently,
  \begin{align*}
    \Re \left(\frac{\phi_{1,\alpha}(-1)}{\phi_{1,\alpha}'(-1)}\right)
    =
    &-\frac{c_r}{2}
    +\frac{2}{15}\alpha^2
    -\frac{1}{3}c_r\alpha^2
    +\frac{4}{15^2}\alpha^4\ln \left|\frac{\hat c}{4}\right|\\
    &+O\left(\alpha^4+\alpha^4\hat c|\ln \hat c|
    +\alpha^2\hat c^2|\ln \hat c|
    +\frac{\alpha^4\hat c_i}{c_r}\right),
  \end{align*}
  and
  \begin{align*}
    \Im \left(\frac{\phi_{1,\alpha}(-1)}{\phi_{1,\alpha}'(-1)}\right)
    =
    -\frac{\hat c_i}{2}
    -\frac{1}{3}\hat c_i\alpha^2
    -\frac{4}{15^2}\pi\alpha^4
    +O\left(\alpha^4\hat c|\ln \hat c|
    +\alpha^2\hat c^2|\ln \hat c|
    +\frac{\alpha^4\hat c_i}{c_r}\right).
  \end{align*}
\end{remark}
 
\section{The Airy equation}
In this section, we study the Airy equation
\begin{align}\label{eq-Airy}
  i\varepsilon\pa_y^4\phi(y)+\left(u_p-c-2i\varepsilon\alpha^2\right)\pa_y^2\phi(y)=f(y),\quad y\in[-1,0].
\end{align}

Throughout the rest of this paper, we denote by $Airy_4$ the fourth-order operator
\begin{align}\label{op-airy-4}
  Airy_4=i\varepsilon\pa_y^4+\left(u_p-c-2i\varepsilon\alpha^2\right)\pa_y^2.
\end{align}

From here, we assume that $\varepsilon$ is sufficiently small and $c\in \mathbb H$, where $\mathbb H$ is the T--S eigen region defined in \eqref{eq-region-H}.
\subsection{Langer transform}
To study this problem, we follow the idea of introducing the Langer transform. 

We define
\begin{align}\label{eq-Langertran}
  \eta(y)=\eta_r(y)+i\eta_i,
\end{align}
where
\begin{align}\label{eq-Langertran-i}
  \eta_i=-\frac{c_i}{u_p'(y_c)},
\end{align}  
and
\begin{equation}\label{eq-Langertran-r}
  \eta_r(y)=\left\{
    \begin{array}{ll}
      \left(\frac{3}{2}\int^y_{y_c}\left(\frac{u_p(x)-c_r}{u_p'(y_c)}\right)^{\frac{1}{2}}dx\right)^{\frac{2}{3}},&y\ge y_c,\\ 
      -\left(\frac{3}{2}\int^{y_c}_y\left(\frac{c_r-u_p(x)}{u_p'(y_c)}\right)^{\frac{1}{2}}dx\right)^{\frac{2}{3}},&y\le y_c.
    \end{array}
  \right.
\end{equation}
Here $y_c=-\sqrt{1-c_r}$ is given in \eqref{eq-yc}.

Recall that the standard Airy function $\mathrm{Ai}(y)$ solves
\begin{align*}
  \pa_y^2\mathrm{Ai}(y)-y\mathrm{Ai}(y)=0.
\end{align*}
The motivation for introducing the Langer transform is that, $\eta_r$ satisfies
\begin{align}\label{eq-Langer-tran-moti}
  \left(\eta_r'(y)\right)^2\eta_r(y)=\frac{u_p-c_r}{u_p'(y_c)}.
\end{align}
Consequently, $\mathrm{Ai} \left(\eta_r(y)\right)$ is an approximate solution of the following generalized Airy equation,
\begin{align*}
  \pa_y^2\mathrm{g}(y)-\frac{u_p-c_r}{u_p'(y_c)}\mathrm{g}(y)=0.
\end{align*}

The Langer transform \eqref{eq-Langertran} has the following properties.
\begin{lemma}\label{lem-Langer}
   For $\eta$ defined in \eqref{eq-Langertran}, it holds that $\eta(y)$ is smooth in $y$, and $\pa_{c_r}\eta(y)$ is smooth in $y$, and
   \begin{align}\label{eq-Langer-prop1}
     \eta_r(y)=\left((y-y_c)^{\frac{3}{2}}+O(\left|y-y_c\right|^{\frac{5}{2}})\right)^{\frac{2}{3}}=(y-y_c)+O(\left|y-y_c\right|^2),
   \end{align}
   \begin{align}\label{eq-Langer-prop2}
    \eta_r(y_c)=0,\quad \eta_r'(y_c)=1,\quad \eta(-1)= -\frac{c}{2}+O(c^2).
  \end{align}  
     \begin{align}\label{eq-Langer-prop3}
     \eta_r'(y)>0,\quad \eta_r'(y)=O(1)\text{ for }y\in[-1,0],
  \end{align}   
    \begin{align}\label{eq-Langer-prop4}
     \pa_{c_r}\eta_r(y)=-\frac{1}{u_p'(y_c)}+O(\left|y-y_c\right|),
   \end{align}
   and
    \begin{align}\label{eq-Langer-prop5}
     \left(\pa_{c_r}\eta_r\right)(y_c)=-\frac{1}{u_p'(y_c)},\quad \left(\pa_{c_r}\eta_r'\right)(y_c)=\frac{2}{5 \left(u_p'(y_c)\right)^2}.
   \end{align}
 \end{lemma} 

\begin{proof}
For $y\ge y_c$, we have
\begin{align*}
  \left(\frac{u_p-c_r}{u_p'(y_c)}\right)^{\frac{1}{2}}=(y-y_c)^{\frac{1}{2}}+O(\left|y-y_c\right|^{\frac{3}{2}}),
\end{align*}
it follows that
\begin{align*}
  \frac{3}{2}\int^y_{y_c}\left(\frac{u_p(x)-c_r}{u_p'(y_c)}\right)^{\frac{1}{2}}dx=(y-y_c)^{\frac{3}{2}}+O(\left|y-y_c\right|^{\frac{5}{2}}).
\end{align*}
Therefore, by the definition of $\eta_r$ we have
\begin{align*}
  \eta_r(y)=\left((y-y_c)^{\frac{3}{2}}+O(\left|y-y_c\right|^{\frac{5}{2}})\right)^{\frac{2}{3}}=(y-y_c)+O(\left|y-y_c\right|^2).
\end{align*}
For $y\le y_c$ we have the same result. Actually as $u_p$ is smooth, $\eta_r$ is smooth and has an arbitrary-order Taylor expansion.

Moreover, one can easily check that
\begin{align}\label{eq-Langer-tran-deri}
  \left(\eta_r'(y)\right)^2\eta_r(y)=\frac{u_p-c_r}{u_p'(y_c)},\quad \eta_r''(y)=\frac{\frac{\eta_r(y)}{\eta_r'(y)}u_p'-\left(u_p-c_r\right)}{2\eta_r^2(y)u_p'(y_c)},
\end{align}
and
\begin{align*}
  \eta_r'(y_c)=1,\quad \eta_r''(y_c)=\frac{1}{5} \frac{u_p''(y_c)}{u_p'(y_c)}.
\end{align*}

From the definition of $\eta_r(y)$, we can see that $\eta_r(y)$ is monotonically increasing with respect to $y$, therefore, $\eta_r'(y)\ge0$. Then by \eqref{eq-Langer-tran-deri}, we have $\left|\eta_r'(y)\right|=O(1)$, which gives \eqref{eq-Langer-prop3}.

Next, we study the $c_r$-derivative of $\eta_r$. 

For $y\ge y_c$, we have
\begin{align*}
  \pa_{c_r}\eta_r(y)=&\pa_{c_r}\left(\frac{3}{2}\int^y_{y_c}\left(\frac{u_p(x)-c_r}{u_p'(y_c)}\right)^{\frac{1}{2}}dx\right)^{\frac{2}{3}}\\
  =&\left(\frac{3}{2}\int^y_{y_c}\left(\frac{u_p(x)-c_r}{u_p'(y_c)}\right)^{\frac{1}{2}}dx\right)^{-\frac{1}{3}}\int^y_{y_c}\pa_{c_r}\left(\frac{u_p(x)-c_r}{u_p'(y_c)}\right)^{\frac{1}{2}}dx.
\end{align*}
Recall that
\begin{align*}
  \frac{u_p(x)-c_r}{u_p'(y_c)}=\frac{1-x^2-c_r}{2\sqrt{1-c_r}},
\end{align*}
we have
\begin{align*}
  \pa_{c_r}\left(\frac{u_p(x)-c_r}{u_p'(y_c)}\right)^{\frac{1}{2}}=&\frac{\left(\frac{u_p(x)-c_r}{u_p'(y_c)}\right)^{\frac{1}{2}}}{4(1-c_r)}-\frac{\left(\frac{u_p(x)-c_r}{u_p'(y_c)}\right)^{-\frac{1}{2}}}{4\sqrt{1-c_r}},
\end{align*}
and then
\begin{align*}
  \pa_{c_r}\eta_r(y)=&\frac{\frac{\frac{2}{3}\eta_r^{\frac{3}{2}}(y)}{4(1-c_r)}-\frac{\int^y_{y_c}\left(\frac{u_p(x)-c_r}{u_p'(y_c)}\right)^{-\frac{1}{2}}dx}{4\sqrt{1-c_r}}}{\eta_r^{\frac{1}{2}}(y)}.
\end{align*}
It holds that
\begin{align*}
  \left(\frac{u_p(x)-c_r}{u_p'(y_c)}\right)^{-\frac{1}{2}}=\left(x-y_c\right)^{-\frac{1}{2}}+O(\left|x-y_c\right|^{\frac{1}{2}}),
\end{align*}
then we have
\begin{align*}
  \pa_{c_r}\eta_r(y_c)=-\frac{1}{u_p'(y_c)}.
\end{align*}

For $y\le y_c$, by using the same technical, we have
\begin{align*}
  \pa_{c_r}\eta_r(y)=&\frac{\frac{\frac{2}{3}\eta_r(y)\left(-\eta_r(y)\right)^{\frac{1}{2}}}{4(1-c_r)}-\frac{\int^{y_c}_y\left(\frac{c_r-u_p(x)}{u_p'(y_c)}\right)^{-\frac{1}{2}}dx}{4\sqrt{1-c_r}}}{\left(-\eta_r(y)\right)^{\frac{1}{2}}}.
\end{align*}

Therefore, we have
\begin{align*}
  \pa_{c_r}\eta_r(y)=-\frac{1}{u_p'(y_c)}+O(\left|y-y_c\right|),
\end{align*}
and
\begin{align*}
  \eta_r(-1)= -\frac{c_r}{2}+O(c_r^2).
\end{align*}

As $u$ and $\eta_r(y)$ are smooth, $\pa_{c_r}\eta_r(y)$ is also smooth.

Recall that $\eta_r'(y_c)=1$. We have 
\begin{align*}
  \left(\pa_{c_r}\eta_r'\right)(y_c)=\pa_{c_r}\left(\eta_r'(y_c)\right)-\eta_r''(y_c)\pa_{c_r}y_c=\frac{2}{5 \left(u_p'(y_c)\right)^2}=\frac{1}{10(1-c_r)}.
\end{align*}
By same technical, we can also get the value of $\pa_{c_r}\eta_r^{(n)} (y_c)$.

Then we complete the proof of this lemma.
\end{proof}

\subsection{Modified Airy function and approximate Green function}
We first introduce two modified Airy functions which are approximate solutions of
\begin{align*}
  i\varepsilon\pa_y^2A(y)+\left(u_p-c-2i\varepsilon\alpha^2\right)A(y)=0,\quad y\in[-1,0].
\end{align*}

We define
\begin{align}\label{eq-kappa}
  \kappa=\left(\frac{\varepsilon}{u_p'(y_c)}\right)^{-\frac{1}{3}},
\end{align}
and
\begin{align}\label{eq-modi-Airy}
   A_1(y)=\mathrm{Ai}(e^{\frac{\pi}{6}i}\kappa \eta(y)),\quad A_2(y)=\mathrm{Ai}(e^{\frac{5\pi}{6}i}\kappa \eta(y)).
\end{align}

Using \eqref{eq-Langer-tran-moti}, a direct computation gives, for $j=1,2$,
\begin{equation}\label{eq-airy2}
  \begin{aligned}    
  &i\varepsilon\pa_y^2 A_j(y)+\left(u_p(y)-c-2i\varepsilon\alpha^2\right) A_j(y)\\ 
  =&i\varepsilon\frac{\eta_r''(y)}{\eta_r'(y)}\pa_y A_j(y)+i \left(\left(\eta_r'(y)\right)^2-1\right)c_i A_j(y)-2i\varepsilon\alpha^2 A_j(y).    
  \end{aligned}
\end{equation}

From the properties of $\mathrm{Ai}(y)$, we can see that $A_1(y)$ and $A_2(y)$ are linearly independent solutions of \eqref{eq-airy2}, and their Wronskian is given by
\begin{align}\label{eq-Wron}
   A_1(y)\pa_yA_2(y)-\pa_yA_1(y)A_2(y)=\frac{1}{2\pi}\kappa \eta_r'(y).
\end{align}

Next, we construct an approximate Green function for the fourth-order Airy operator $Airy_4$.

Inspired by \cite{GGN16adv}, we define the following modified primitive Airy functions
\begin{align}
    A_1(1,y)=\int^y_0 \frac{1}{\left(\eta_r'(z)\right)^{\frac{1}{2}}}  A_1(z)dz,\quad  A_1(2,y)=\int^y_0  A_1(1,z)dz \label{eq-Airy-modi-1}\\ 
   A_2(1,y)=\int^y_{-1} \frac{1}{\left(\eta_r'(z)\right)^{\frac{1}{2}}} A_2(z)dz,\quad A_2(2,y)=\int^y_{-1} A_2(1,z)dz.\label{eq-Airy-modi-2}
\end{align}
and the following approximate Green function
\begin{equation}\label{eq-Green-Airy}
   G_A(x,y)=i \frac{2\pi}{\left(\eta_r'(x)\right)^{\frac{1}{2}}\kappa\varepsilon}\left\{
    \begin{array}{ll}
       A_1(2,y) A_2(x)+a_2(x)+a_1(x)(x-y_c),&x<y;\\ 
       A_1(x) A_2(2,y)+a_1(x)(y-y_c),&x>y.
    \end{array}
  \right.
\end{equation}
where
\begin{align}
  a_1(x)= A_1(1,x) A_2(x)- A_1(x) A_2(1,x),\label{eq-a1}\\ 
  a_2(x)= A_1(x) A_2(2,x)- A_1(2,x) A_2(x).\label{eq-a2}
\end{align}
With the above choice of $a_1(x)$ and $a_2(x)$, the functions $G_A(x,y)$, $\pa_yG_A(x,y)$, and $\pa_y^2G_A(x,y)$ are continuous across $y=x$.

A direct computation shows that
\begin{equation}\label{eq-Green-app}
  \begin{aligned}    
    &i\varepsilon\pa_y^4 G_A(x,y)+\left(u_p(y)-c-2i\varepsilon\alpha^2\right)\pa_y^2 G_A(x,y)\\ 
  =&\delta_x(y)+ i\varepsilon \left(\pa_y^2\left(\eta_r'(y)\right)^{-\frac{1}{2}}\right)\left(\eta_r'(y)\right)^{\frac{1}{2}}\pa_y^2 G_A(x,y)+i \left(\left(\eta_r'(y)\right)^2-1\right)c_i\pa_y^2 G_A(x,y)-2i\varepsilon\alpha^2\pa_y^2 G_A(x,y),
  \end{aligned}
\end{equation} 
where $\delta_x(y)$ is the Dirac delta distribution.

\begin{remark}
  The definition of the modified primitive Airy functions used here is slightly different from that in \cite{GGN16adv}. Here we have
\begin{align*}
  A_1(1,0)=A_1(2,0)=A_2(1,-1)=A_2(2,-1)=0,
\end{align*}
and then the Green function has some useful boundary properties. For example, we have
\begin{align*}
  \pa_yG_A(x,0)=0.
\end{align*}
\end{remark}

\subsection{Properties of the modified Airy functions}
We first recall some classical properties of the Airy functions.  

Let
\begin{align*}
  \mathrm{Ai}(0,z)=\mathrm{Ai}(z),\\
  \mathrm{Ai}(k,z)=\pa_z^{|k|} \mathrm{Ai}(z),\text{ for }k\in\mathbb Z, k\le-1,\\
  \mathrm{Ai}(k,z)=\int^z_{+\infty}\mathrm{Ai}(k-1,x)dx,\text{ for }k\in\mathbb Z, k\ge1,
\end{align*}
which are all analytic functions. The first primitive of Airy has the following properties, see \cite{Olver2010},
\begin{align*}
  \int^z_{+\infty}\mathrm{Ai}(x)dx=-\pi \left(\mathrm{Ai}(z)\mathrm{Gi}'(z)-\mathrm{Ai}'(z)\mathrm{Gi}(z)\right),\\
  \int^z_{+\infty}\mathrm{Ai}(x)dx=\int^{-\infty}_{+\infty}\mathrm{Ai}(x)dx+\int^z_{-\infty}\mathrm{Ai}(x)dx=-1+\pi \left(\mathrm{Ai}(z)\mathrm{Hi}'(z)-\mathrm{Ai}'(z)\mathrm{Hi}(z)\right),
\end{align*}
where $\mathrm{Gi}$ and $\mathrm{Hi}$ are Scorer functions. For given $\delta>0$ and big enough $M$, it holds that
\begin{align*}
  \mathrm{Gi}(z)=\frac{1}{\pi z}\sum^\infty_{j=0}\frac{\left(3j\right)!}{j!\left(3z^3\right)^j},\ \mathrm{Gi}'(z)=-\frac{1}{\pi z^2}\sum^\infty_{j=0}\frac{\left(3j+1\right)!}{j!\left(3z^3\right)^j},\text{ for }|z|\ge M,\ |\arg z|\le \frac{1}{3}\pi-\delta,\\
  \mathrm{Hi}(z)=-\frac{1}{\pi z}\sum^\infty_{j=0}\frac{\left(3j\right)!}{j!\left(3z^3\right)^j},\ \mathrm{Hi}'(z)=\frac{1}{\pi z^2}\sum^\infty_{j=0}\frac{\left(3j+1\right)!}{j!\left(3z^3\right)^j},\text{ for }|z|\ge M,\ |\arg -z|\le \frac{2}{3}\pi-\delta.
\end{align*}
Therefore, in the corresponding sectors, the representations in terms of
\(\mathrm{Ai}\), \(\mathrm{Gi}\), and \(\mathrm{Hi}\) yield the asymptotic
expansion of the first primitive \(\mathrm{Ai}(1,z)\). By induction, it is easy to check that 
\begin{align}\label{eq-Airy-k}
  \mathrm{Ai}(k,z)=\frac{z}{k-1}\mathrm{Ai}(k-1,z)-\frac{1}{k-1}\mathrm{Ai}(k-3,z), \text{ for }k\ge2.
\end{align}
Then we can get the asymptotic expansion for all $\mathrm{Ai}(k,z)$.

Combine the above analysis and the classical results from \cite{VS2010}, \cite{Olver2010}, and \cite{GGN16adv}, we have the following properties.
\begin{lemma}\label{lem:Airy-class-est} Let 
\begin{align*}
  S_z\eqdef\left\{z\Big||z|\ge M, |\arg z|\le \pi-\delta,\left|\arg (z)-\frac{1}{3}\pi\right|\ge \delta,\left|\arg (z)+\frac{1}{3}\pi\right|\ge \delta \right\}.
\end{align*}
For every small $\delta>0$, there exists a sufficiently large $M=M(\delta)$ such that, for all $z\in S_z$,
  \begin{align*}
  \mathrm{Ai}(k,z)=\frac{\left(-1\right)^{|k|}}{2\sqrt{\pi}}z^{-\frac{1}{4}-\frac{k}{2}}e^{-\frac{2}{3}z^{\frac{3}{2}}}\left(1+O(z^{-\frac{3}{2}})\right),\text{ for }k\in\mathbb Z,
\end{align*}
especially,
\begin{align*}
  \mathrm{Ai}'(z)=-\frac{z^{\frac{1}{4}}e^{-\frac{2}{3}z^{\frac{3}{2}}}}{2\sqrt\pi }\left(1+\frac{3\cdot 7z^{-\frac{3}{2}}}{144 }-\frac{5\cdot 7\cdot 9\cdot 13z^{-3}}{2\cdot144^2 }+O(|z|^{-\frac{9}{2}})\right)
\end{align*}
\begin{align*}
  \mathrm{Ai}(z)=\frac{z^{-\frac{1}{4}}e^{-\frac{2}{3}z^{\frac{3}{2}}}}{2\sqrt\pi}\left(1-\frac{3\cdot 5 z^{-\frac{3}{2}}}{144}+\frac{5\cdot 7\cdot 9\cdot 11z^{-3}}{2\cdot144^2}+O(|z|^{-\frac{9}{2}})\right)
\end{align*}
\begin{align*}
  \mathrm{Ai}(1,z)=-\frac{z^{-\frac{3}{4}}e^{-\frac{2}{3}z^{\frac{3}{2}}}}{2\sqrt\pi }\left(1+\left(\frac{3\cdot 7}{144}-1\right) z^{-\frac{3}{2}}+\left(\frac{3\cdot 5}{144}-\frac{5\cdot 7\cdot 9\cdot 13}{2\cdot144^2}+2\right)z^{-3}+O(|z|^{-\frac{9}{2}})\right),
\end{align*}
\begin{align*}
  \mathrm{Ai}(2,z)=\frac{z^{-\frac{5}{4}}e^{-\frac{2}{3}z^{\frac{3}{2}}}}{2\sqrt\pi}\left(1-\left(\frac{3\cdot 5}{144}+2\right)z^{-\frac{3}{2}}+O(|z|^{-3})\right).
\end{align*}
Furthermore, there hold asymptotic bounds:
\begin{align*}
  \left|\mathrm{Ai}(k,e^{\frac{\pi}{6}i}y)\right|\lesssim \left\langle y \right\rangle^{-\frac{k}{2}-\frac{1}{4}}e^{-\sqrt{2|y|}y/3},\text{ for }k\in\mathbb Z,y\in\mathbb R,\\
  \left|\mathrm{Ai}(k,e^{\frac{5\pi}{6}i}y)\right|\lesssim \left\langle y \right\rangle^{-\frac{k}{2}-\frac{1}{4}}e^{\sqrt{2|y|}y/3},\text{ for }k\in\mathbb Z,y\in\mathbb R.
\end{align*}
\end{lemma}
\begin{remark}\label{rmk:Airy-class-est}
  Throughout this paper, there exists a uniform $\delta_*>0$ such that, in every application of Lemma \ref{lem:Airy-class-est},  $|\arg z|\le \pi-\delta_*,\left|\arg (z)-\frac{1}{3}\pi\right|\ge \delta_*,\left|\arg (z)+\frac{1}{3}\pi\right|\ge \delta_*$. Indeed, we are mainly concerned with the cases where $\arg z$ is close to $\frac{\pi}{6}$ or $-\frac{5\pi}{6}$. Consequently, the constant $M$ in Lemma \ref{lem:Airy-class-est} can be chosen uniformly for all applications in this paper.
\end{remark}

The estimates in Lemma \ref{lem:Airy-class-est} show that $\mathrm{Ai}(e^{\frac{\pi}{6}i}y)$ decays exponentially as $y\to+\infty$, whereas $\mathrm{Ai}(e^{\frac{5\pi}{6}i}y)$ decays exponentially as $y\to-\infty$. Then we introduce the following primitive functions of $\mathrm{Ai}(e^{\frac{\pi}{6}i}y)$ and $\mathrm{Ai}(e^{\frac{5\pi}{6}i}y)$:
\begin{align}
  \mathcal A_1(1,y)=\int^y_{+\infty}\mathrm{Ai}(e^{\frac{\pi}{6}i}x)dx, \qquad \mathcal A_1(k,y)=\int^y_{+\infty}\mathcal A_1(k-1,x)dx,\label{eq-Airy-prim-1}\\
  \mathcal A_2(1,y)=\int^y_{-\infty}\mathrm{Ai}(e^{\frac{5\pi}{6}i}x)dx, \qquad \mathcal A_2(k,y)=\int^y_{-\infty}\mathcal A_2(k-1,x)dx.\label{eq-Airy-prim-2}
\end{align}
A direct calculation shows that
\begin{align}
  \mathcal A_1(k,y)=e^{-\frac{k\pi}{6}i}\mathrm{Ai}(k,e^{\frac{\pi}{6}i}y),\label{eq-Airy-prim-1-prop}\\
  \mathcal A_2(k,y)=e^{-\frac{5k\pi}{6}i}\mathrm{Ai}(k,e^{\frac{5\pi}{6}i}y).\label{eq-Airy-prim-2-prop}
\end{align}

Next, we derive explicit expressions for $A_1(1,y)$, $A_1(2,y)$, $A_2(1,y)$, and $A_2(2,y)$. Since these functions are defined by integral representations, integration by parts allows us to extract their leading-order contributions, which are expressed in terms of Airy functions composed with the Langer transformation. Similar techniques have been used in \cite{MWWZ2024} and \cite{BG23-LWR}.

\begin{lemma}\label{lem:pri-Airy-expan}
  It holds that
  \begin{align*}
  A_1(1,y)=&\frac{\mathcal A_1(1,\kappa \eta(y))}{\kappa\left(\eta_r'(y)\right)^{\frac{3}{2}}}-\frac{\mathcal A_1(1,\kappa \eta(0))}{\kappa\left(\eta_r'(0)\right)^{\frac{3}{2}}}+\frac{3}{2}\frac{\mathcal A_1(2,\kappa \eta(y))\eta_r''(y)}{\kappa^2\left(\eta_r'(y)\right)^{\frac{7}{2}}}-\frac{3}{2}\frac{\mathcal A_1(2,\kappa \eta(0))\eta_r''(0)}{\kappa^2\left(\eta_r'(0)\right)^{\frac{7}{2}}}\\
  &+\int^y_0\frac{3}{2}\frac{\mathcal A_1(2,\kappa \eta(z))\left( \frac{7}{2}\left(\eta_r''(z)\right)^2- \eta_r'(z)\eta_r'''(z)\right)}{\kappa^2\left(\eta_r'(z)\right)^{\frac{9}{2}}}dz.
\end{align*}
\begin{align*}
  A_1(2,y)
  =&\frac{\mathcal A_1(2,\kappa \eta(y)) }{\kappa^2\left(\eta_r'(y)\right)^{\frac{5}{2}}}+4\frac{\mathcal A_1(3,\kappa \eta(y))\eta_r''(y)}{\kappa^3\left(\eta_r'(y)\right)^{\frac{9}{2}}}-\frac{\mathcal A_1(2,\kappa \eta(0)) }{\kappa^2\left(\eta_r'(0)\right)^{\frac{5}{2}}}-4\frac{\mathcal A_1(3,\kappa \eta(0))\eta_r''(0)}{\kappa^3\left(\eta_r'(0)\right)^{\frac{9}{2}}}\\
  &+\int^y_04\frac{\mathcal A_1(3,\kappa \eta(z))\left(\frac{9}{2}\left(\eta_r''(z)\right)^2-\eta_r'(z)\eta_r'''(z)\right)  }{\kappa^3\left(\eta_r'(z)\right)^{\frac{11}{2}}}dz\\
  &-y \left(\frac{\mathcal A_1(1,\kappa \eta(0))}{\kappa\left(\eta_r'(0)\right)^{\frac{3}{2}}}+\frac{3}{2}\frac{\mathcal A_1(2,\kappa \eta(0))\eta_r''(0)}{\kappa^2\left(\eta_r'(0)\right)^{\frac{7}{2}}}\right)\\
  &+\int^y_0 \int^z_0\frac{3}{2}\frac{\mathcal A_1(2,\kappa \eta(x))\left( \frac{7}{2}\left(\eta_r''(x)\right)^2- \eta_r'(x)\eta_r'''(x)\right)}{\kappa^2\left(\eta_r'(x)\right)^{\frac{9}{2}}}dx dz
\end{align*}
\begin{align*}
  A_2(1,y)=&\frac{\mathcal A_2(1,\kappa \eta(y))}{\kappa\left(\eta_r'(y)\right)^{\frac{3}{2}}}-\frac{\mathcal A_2(1,\kappa \eta(-1))}{\kappa\left(\eta_r'(-1)\right)^{\frac{3}{2}}}+\frac{3}{2}\frac{\mathcal A_2(2,\kappa \eta(y))\eta_r''(y)}{\kappa^2\left(\eta_r'(y)\right)^{\frac{7}{2}}}-\frac{3}{2}\frac{\mathcal A_2(2,\kappa \eta(-1))\eta_r''(-1)}{\kappa^2\left(\eta_r'(-1)\right)^{\frac{7}{2}}}\\
  &+\int^y_{-1}\frac{3}{2}\frac{\mathcal A_2(2,\kappa \eta(z))\left( \frac{7}{2}\left(\eta_r''(z)\right)^2- \eta_r'(z)\eta_r'''(z)\right)}{\kappa^2\left(\eta_r'(z)\right)^{\frac{9}{2}}}dz.
\end{align*}
\begin{align*}
  A_2(2,y)=&\frac{\mathcal A_2(2,\kappa \eta(y)) }{\kappa^2\left(\eta_r'(y)\right)^{\frac{5}{2}}}+4\frac{\mathcal A_2(3,\kappa \eta(y))\eta_r''(y)}{\kappa^3\left(\eta_r'(y)\right)^{\frac{9}{2}}} -\frac{\mathcal A_2(2,\kappa \eta(-1)) }{\kappa^2\left(\eta_r'(-1)\right)^{\frac{5}{2}}}-4\frac{\mathcal A_2(3,\kappa \eta(-1))\eta_r''(-1)}{\kappa^3\left(\eta_r'(-1)\right)^{\frac{9}{2}}}\\
  &+\int^y_{-1}4\frac{\mathcal A_2(3,\kappa \eta(z))\left(\frac{9}{2}\left(\eta_r''(z)\right)^2-\eta_r'(z)\eta_r'''(z)\right)  }{\kappa^3\left(\eta_r'(z)\right)^{\frac{11}{2}}}dz \\
  &-(y+1) \left(  \frac{\mathcal A_2(1,\kappa \eta(-1))}{\kappa\left(\eta_r'(-1)\right)^{\frac{3}{2}}}+\frac{3}{2}\frac{\mathcal A_2(2,\kappa \eta(-1))\eta_r''(-1)}{\kappa^2\left(\eta_r'(-1)\right)^{\frac{7}{2}}}\right)\\
  &+\int^y_{-1} \int^z_{-1}\frac{3}{2}\frac{\mathcal A_2(2,\kappa \eta(x))\left( \frac{7}{2}\left(\eta_r''(x)\right)^2- \eta_r'(x)\eta_r'''(x)\right)}{\kappa^2\left(\eta_r'(x)\right)^{\frac{9}{2}}}dx dz.
\end{align*}
\end{lemma}
\begin{proof}
It follows from the definition of $A_1(1,y)$ and $\mathcal A_1(1,y)$ that
\begin{align*}
  A_1(1,y)=&\int^y_0 \frac{1}{\left(\eta_r'(z)\right)^{\frac{1}{2}}}\mathrm{Ai}(e^{\frac{\pi}{6}i}\kappa \eta(z))dz\\
  =&\int^y_0 \pa_z \left(\frac{\mathcal A_1(1,\kappa \eta(z))}{\kappa\left(\eta_r'(z)\right)^{\frac{3}{2}}}\right)+\frac{3}{2}\frac{\mathcal A_1(1,\kappa \eta(z))\eta_r''(z)}{\kappa\left(\eta_r'(z)\right)^{\frac{5}{2}}}dz\\
  =&\frac{\mathcal A_1(1,\kappa \eta(y))}{\kappa\left(\eta_r'(y)\right)^{\frac{3}{2}}}-\frac{\mathcal A_1(1,\kappa \eta(0))}{\kappa\left(\eta_r'(0)\right)^{\frac{3}{2}}}+\int^y_0\frac{3}{2}\frac{\mathcal A_1(1,\kappa \eta(z))\eta_r''(z)}{\kappa\left(\eta_r'(z)\right)^{\frac{5}{2}}}dz.
\end{align*}
By using integration by parts again, we have
\begin{align*}
  A_1(1,y)=&\frac{\mathcal A_1(1,\kappa \eta(y))}{\kappa\left(\eta_r'(y)\right)^{\frac{3}{2}}}-\frac{\mathcal A_1(1,\kappa \eta(0))}{\kappa\left(\eta_r'(0)\right)^{\frac{3}{2}}}+\frac{3}{2}\frac{\mathcal A_1(2,\kappa \eta(y))\eta_r''(y)}{\kappa^2\left(\eta_r'(y)\right)^{\frac{7}{2}}}-\frac{3}{2}\frac{\mathcal A_1(2,\kappa \eta(0))\eta_r''(0)}{\kappa^2\left(\eta_r'(0)\right)^{\frac{7}{2}}}\\
  &+\int^y_0\frac{3}{2}\frac{\mathcal A_1(2,\kappa \eta(z))\left( \frac{7}{2}\left(\eta_r''(z)\right)^2- \eta_r'(z)\eta_r'''(z)\right)}{\kappa^2\left(\eta_r'(z)\right)^{\frac{9}{2}}}dz.
\end{align*}

Based on the same idea, we have
\begin{align*}
  &\int^y_0 \frac{\mathcal A_1(1,\kappa \eta(z))}{\kappa\left(\eta_r'(z)\right)^{\frac{3}{2}}}+\frac{3}{2}\frac{\mathcal A_1(2,\kappa \eta(z))\eta_r''(z)}{\kappa^2\left(\eta_r'(z)\right)^{\frac{7}{2}}}dz\\
  =&\frac{\mathcal A_1(2,\kappa \eta(y)) }{\kappa^2\left(\eta_r'(y)\right)^{\frac{5}{2}}}- \frac{\mathcal A_1(2,\kappa \eta(0)) }{\kappa^2\left(\eta_r'(0)\right)^{\frac{5}{2}}}\\
  &+\frac{5}{2}\frac{\mathcal A_1(3,\kappa \eta(y))\eta_r''(y)}{\kappa^3\left(\eta_r'(y)\right)^{\frac{9}{2}}}- \frac{5}{2}\frac{\mathcal A_1(3,\kappa \eta(0))\eta_r''(0)}{\kappa^3\left(\eta_r'(0)\right)^{\frac{9}{2}}}+\int^y_0\frac{5}{2}\frac{\mathcal A_1(3,\kappa \eta(z))\left(\frac{9}{2}\left(\eta_r''(z)\right)^2-\eta_r'(z)\eta_r'''(z)\right)  }{\kappa^3\left(\eta_r'(z)\right)^{\frac{11}{2}}}dz\\
  &+\frac{3}{2}\frac{\mathcal A_1(3,\kappa \eta(y))\eta_r''(y)}{\kappa^3\left(\eta_r'(y)\right)^{\frac{9}{2}}}- \frac{3}{2}\frac{\mathcal A_1(3,\kappa \eta(0))\eta_r''(0)}{\kappa^3\left(\eta_r'(0)\right)^{\frac{9}{2}}}+\int^y_0\frac{3}{2}\frac{\mathcal A_1(3,\kappa \eta(z))\left(\frac{9}{2}\left(\eta_r''(z)\right)^2-\eta_r'(z)\eta_r'''(z)\right)  }{\kappa^3\left(\eta_r'(z)\right)^{\frac{11}{2}}}dz.
\end{align*}
It follows that
\begin{align*}
  &A_1(2,y)=\int^y_0 A_1(1,z)dz\\
  =&\frac{\mathcal A_1(2,\kappa \eta(y)) }{\kappa^2\left(\eta_r'(y)\right)^{\frac{5}{2}}}+4\frac{\mathcal A_1(3,\kappa \eta(y))\eta_r''(y)}{\kappa^3\left(\eta_r'(y)\right)^{\frac{9}{2}}}-\frac{\mathcal A_1(2,\kappa \eta(0)) }{\kappa^2\left(\eta_r'(0)\right)^{\frac{5}{2}}}-4\frac{\mathcal A_1(3,\kappa \eta(0))\eta_r''(0)}{\kappa^3\left(\eta_r'(0)\right)^{\frac{9}{2}}}\\
  &+\int^y_04\frac{\mathcal A_1(3,\kappa \eta(z))\left(\frac{9}{2}\left(\eta_r''(z)\right)^2-\eta_r'(z)\eta_r'''(z)\right)  }{\kappa^3\left(\eta_r'(z)\right)^{\frac{11}{2}}}dz-y \left(  \frac{\mathcal A_1(1,\kappa \eta(0))}{ \kappa\left(\eta_r'(0)\right)^{\frac{3}{2}}}+\frac{3}{2}\frac{\mathcal A_1(2,\kappa \eta(0))\eta_r''(0)}{\kappa^2\left(\eta_r'(0)\right)^{\frac{7}{2}}}\right)\\
  &+\int^y_0 \int^z_0\frac{3}{2}\frac{\mathcal A_1(2,\kappa \eta(x))\left( \frac{7}{2}\left(\eta_r''(x)\right)^2- \eta_r'(x)\eta_r'''(x)\right)}{\kappa^2\left(\eta_r'(x)\right)^{\frac{9}{2}}}dx dz.
\end{align*} 

The expansions of $A_2(1,y)$ and $A_2(2,y)$ can be derived in a similar manner, and we omit the details.
\end{proof}
 
We have the following properties.
\begin{lemma}\label{lem-Airy-Airy-est}
For any $-1<y\leq x<0$, we have that
\begin{align}\label{est-Airy-Airy-est}
\left|e^{-\f23(e^{\f{\pi i}{6}}\kappa\eta(x))^\f32} e^{\f23(e^{\f{\pi i}{6}}\kappa\eta(y))^\f32}\right|\lesssim e^{-\frac{1}{C}|\kappa\eta_r(x)-\kappa\eta_r(y)|(|\kappa\eta(x)|^\f12+|\kappa\eta(y)|^\f12)}.
\end{align}

\end{lemma}

\begin{proof}
Note that 
\begin{align*}
\left|e^{-\f23(e^{\f{\pi i}{6}}\kappa\eta(x))^\f32} e^{\f23(e^{\f{\pi i}{6}}\kappa\eta(y))^\f32}\right|=e^{-\f23 \Re \left((e^{\f{\pi i}{6}}\kappa\eta(x))^\f32-e^{\f{\pi i}{6}}\kappa\eta(y))^\f32)\right)}.
\end{align*}
Recall that $|\eta_i|\le 2M \frac{1}{\kappa}$. If $\min\{|\eta_r(x)|, |\eta_r(y)|\}\lesssim |\eta_i|$, then $\kappa\min\{|\eta_r(x)|, |\eta_r(y)|\}\lesssim 1$, and
\begin{align*}
  \left|e^{-\f23(e^{\f{\pi i}{6}}\kappa\eta(x))^\f32} e^{\f23(e^{\f{\pi i}{6}}\kappa\eta(y))^\f32}\right|\lesssim e^{-\frac{1}{C} \left( \kappa\max\{|\eta_r(x)|, |\eta_r(y)|\} \right)^{\frac{3}{2}}},
\end{align*}
then \eqref{est-Airy-Airy-est} is naturally true.  Therefore, in the rest of the proof, we assume that
\begin{align}\label{assume: (eta_i, eta_r)}
\f{|\eta_i|}{\min\{|\eta_r(x)|, |\eta_r(y)|\}}\leq \delta_0,
\end{align}
where $0<\delta_0<\frac{1}{4}$ is a small constant to be determined later.

As a result, the goal of this lemma is to prove that under \eqref{assume: (eta_i, eta_r)}, for $-1<y\leq x<0$
\begin{align*}
\Re (e^{\f{\pi i}{6}}\kappa\eta(x))^\f32-e^{\f{\pi i}{6}}\kappa\eta(y))^\f32)\gtrsim|\kappa\eta_r(x)-\kappa\eta_r(y)|(|\kappa\eta(x)|^\f12+|\kappa\eta(y)|^\f12).
\end{align*}
 Due to $\eta_r'\geq C^{-1}>0$, we know $\eta_r(x)$ is a monotonic increasing function. Thus we have $\eta_r(y)\leq \eta_r(x)$ for $y\leq x$. We divide the proof of this lemma into 3 cases.
 
{\bf Case 1: $\eta_r(y)\leq \eta_r(x)\leq 0$.} In this case, we have $|\eta_r(y)|\geq |\eta_r(x)|$ and $|\eta(y)|\geq |\eta(x)|$. We define angles
\begin{align*}
\theta(x)=\arctan(\f{\eta_i}{|\eta_r(x)|}),\quad \theta(y)=\arctan(\f{\eta_i}{|\eta_r(y)|}).
\end{align*}
For $\eta_i\geq 0$, we have $\theta(x)\geq \theta (y)\geq0$ and for $\eta_i< 0$, we have $\theta(x)\leq  \theta (y)< 0$. A direct calculation gives that
\begin{align*}
\arg((e^{\f{\pi i}{6}}\kappa\eta(x))^\f32)=\f{3}{4}\pi -\f32 \theta(x),\quad \arg((e^{\f{\pi i}{6}}\kappa\eta(y))^\f32)=\f{3}{4}\pi -\f32 \theta(y).
\end{align*}
 Then by \eqref{assume: (eta_i, eta_r)}, for $\eta_i\geq 0$, we get
\begin{align*}
 \cos\f{3\pi}{4}\le\cos(\arg((e^{\f{\pi i}{6}}\kappa\eta(y))^\f32))\leq \cos(\arg((e^{\f{\pi i}{6}}\kappa\eta(x))^\f32))< \cos\f{5\pi}{8}<0,
\end{align*}
and for $\eta_i< 0$, we get
\begin{align*}
\cos\f{7\pi}{8}\le\cos(\arg((e^{\f{\pi i}{6}}\kappa\eta(x))^\f32))\leq \cos(\arg((e^{\f{\pi i}{6}}\kappa\eta(y))^\f32))\leq \cos\f{3\pi}{4}<0.
\end{align*}

$\bullet$ For $\eta_i\geq 0$, we have
\begin{align*}
\Re (e^{\f{\pi i}{6}}\kappa\eta(x))^\f32-e^{\f{\pi i}{6}}\kappa\eta(y))^\f32)=&\cos(\arg((e^{\f{\pi i}{6}}\kappa\eta(x))^\f32))|\kappa\eta(x)|^\f32-\cos(\arg((e^{\f{\pi i}{6}}\kappa\eta(y))^\f32))|\kappa\eta(y)|^\f32\\
\geq &-\cos(\arg((e^{\f{\pi i}{6}}\kappa\eta(y))^\f32))(|\kappa\eta(y)|^\f32-|\kappa\eta(x)|^\f32)\\
\gtrsim &|\kappa\eta_r(x)-\kappa\eta_r(y)|(|\kappa\eta(x)|^\f12+|\kappa\eta(y)|^\f12).
\end{align*}

$\bullet$ For $\eta_i< 0$, we have
\begin{align*}
\Re (e^{\f{\pi i}{6}}\kappa\eta(x))^\f32-e^{\f{\pi i}{6}}\kappa\eta(y))^\f32)=&\cos(\arg((e^{\f{\pi i}{6}}\kappa\eta(x))^\f32))|\kappa\eta(x)|^\f32-\cos(\arg((e^{\f{\pi i}{6}}\kappa\eta(y))^\f32))|\kappa\eta(y)|^\f32\\
=&-\cos(\arg((e^{\f{\pi i}{6}}\kappa\eta(y))^\f32))(|\kappa\eta(y)|^\f32-|\kappa\eta(x)|^\f32)\\
&+(\cos(\arg((e^{\f{\pi i}{6}}\kappa\eta(x))^\f32))-\cos(\arg((e^{\f{\pi i}{6}}\kappa\eta(y))^\f32)))|\kappa\eta(x)|^\f32\\
=&I_1+I_2.
\end{align*}
For $I_1$, we use the argument above to have
\begin{align*}
I_1\geq -\cos\f{3\pi}{4}(|\kappa\eta(y)|^\f32-|\kappa\eta(x)|^\f32)\geq &C_*|\kappa\eta_r(x)-\kappa\eta_r(y)|(|\kappa\eta(x)|^\f12+|\kappa\eta(y)|^\f12),
\end{align*}
with some $C_*>0$.

For $I_2$, we first consider the case $|\kappa \eta(y)|\geq (4C_*^{-1}+1)^\f 23|\kappa\eta(x)|$, and we get
\begin{align*}
|I_2|\leq 2|\kappa\eta(x)|^\f32\leq \f12C_*(|\kappa\eta(y)|^\f32-|\kappa\eta(x)|^\f32).
\end{align*}
For $|\kappa \eta(x)|\leq |\kappa \eta(y)|\leq (4C_*^{-1}+1)^\f 23|\kappa\eta(x)|$, we use  \eqref{assume: (eta_i, eta_r)} to have
\begin{align*}
&|\cos(\arg((e^{\f{\pi i}{6}}\kappa\eta(x))^\f32))-\cos(\arg((e^{\f{\pi i}{6}}\kappa\eta(y))^\f32))|\\
&\le \frac{3}{2}|\theta(x)-\theta(y)|\lesssim|\f{\eta_i}{|\eta_r(x)|}-\f{\eta_i}{|\eta_r(y)|}|\lesssim \f{|\eta_i||\eta_r(x)-\eta_r(y)|}{|\eta_r(x)||\eta_r(y)|},
\end{align*}
which implies
\begin{align*}
|I_2|\lesssim \f{\kappa^\f32|\eta_i||\eta_r(x)-\eta_r(y)|}{|\eta_r(y)|^\f12}=\f{\kappa^\f32|\eta_i||\eta_r(x)-\eta_r(y)|}{|\eta_r(y)|}|\eta_r(y)|^\f12\lesssim \d_0|\kappa\eta_r(x)-\kappa\eta_r(y)||\kappa\eta(y)|^\f12.
\end{align*}
Thus for $|\kappa \eta(x)|\leq |\kappa \eta(y)|$, we get
\begin{align*}
|I_2|\leq \f12C_*(|\kappa\eta(y)|^\f32-|\kappa\eta(x)|^\f32)+\d_0|\kappa\eta_r(x)-\kappa\eta_r(y)||\kappa\eta(y)|^\f12.
\end{align*}
Combining with the estimate $I_1$ and taking $\d_0$ small enough, we obtain for $|\kappa \eta(x)|\leq |\kappa \eta(y)|$ that
\begin{align*}
\Re (e^{\f{\pi i}{6}}\kappa\eta(x))^\f32-e^{\f{\pi i}{6}}\kappa\eta(y))^\f32)\geq \f14 C_*|\kappa\eta_r(x)-\kappa\eta_r(y)|(|\kappa\eta(x)|^\f12+|\kappa\eta(y)|^\f12).
\end{align*}

{\bf Case 2: $\eta_r(y)\leq 0\leq \eta_r(x)$.} For this case, we have
\begin{align*}
\arg((e^{\f{\pi i}{6}}\kappa\eta(y))^\f32))=\f34\pi-\f 32 \theta(y),\quad \arg((e^{\f{\pi i}{6}}\kappa\eta(x))^\f32))=\f14\pi+\f 32 \theta(x),
\end{align*}
which deduces
\begin{align*}
\cos(\arg((e^{\f{\pi i}{6}}\kappa\eta(y))^\f32)))<\cos\f{5\pi}{8}<0,\quad \cos(\arg((e^{\f{\pi i}{6}}\kappa\eta(x))^\f32)))>\cos\f{3\pi}{8}>0.
\end{align*}
Then we have
\begin{align*}
\Re (e^{\f{\pi i}{6}}\kappa\eta(x))^\f32-e^{\f{\pi i}{6}}\kappa\eta(y))^\f32)=&\cos(\arg((e^{\f{\pi i}{6}}\kappa\eta(x))^\f32))|\kappa\eta(x)|^\f32-\cos(\arg((e^{\f{\pi i}{6}}\kappa\eta(y))^\f32))|\kappa\eta(y)|^\f32\\
\gtrsim& |\kappa\eta_r(x)-\kappa\eta_r(y)|(|\kappa\eta(x)|^\f12+|\kappa\eta(y)|^\f12).
\end{align*}

{\bf Case 3: $0\leq \eta_r(y)\leq \eta_r(x)$.} In this case, we have $|\eta_r(y)|\leq |\eta_r(x)|$ and $|\eta(y)|\leq |\eta(x)|$. Angles $\theta(x)$ and $\theta(y)$ are defined in Case 1. For $\eta_i\geq 0$, we have $\theta(y)\geq \theta (x)\geq0$ and for $\eta_i< 0$, we have $\theta(y)\leq  \theta (x)< 0$. A direct calculation gives that
\begin{align*}
\arg((e^{\f{\pi i}{6}}\kappa\eta(x))^\f32)=\f{1}{4}\pi +\f32 \theta(x),\quad \arg((e^{\f{\pi i}{6}}\kappa\eta(y))^\f32)=\f{1}{4}\pi +\f32 \theta(y).
\end{align*}
According to \eqref{assume: (eta_i, eta_r)} and $|\eta_r(y)|\leq |\eta_r(x)|$, we have $\f{|\eta_i|}{|\eta_r(x)|}\leq \d_0$.
Then for $\eta_i\geq 0$, we get
\begin{align*}
 \cos\f{\pi}{4}\ge\cos(\arg((e^{\f{\pi i}{6}}\kappa\eta(x))^\f32))\geq \cos(\arg((e^{\f{\pi i}{6}}\kappa\eta(y))^\f32))>\cos\f{3\pi}{8}>0,
\end{align*}
and for $\eta_i< 0$, we get
\begin{align*}
\cos\f{\pi}{8}\ge\cos(\arg((e^{\f{\pi i}{6}}\kappa\eta(y))^\f32))\geq \cos(\arg((e^{\f{\pi i}{6}}\kappa\eta(x))^\f32))> \cos\f{\pi}{4}>0.
\end{align*}

Then by using the same technique to Case 1, we obtain for $|\kappa \eta(y)|\leq |\kappa \eta(x)$ that
\begin{align*}
\Re (e^{\f{\pi i}{6}}\kappa\eta(x))^\f32-e^{\f{\pi i}{6}}\kappa\eta(y))^\f32)\geq \f14 C_*|\kappa\eta_r(x)-\kappa\eta_r(y)|(|\kappa\eta(x)|^\f12+|\kappa\eta(y)|^\f12).
\end{align*}

This finishes the proof of this lemma.
\end{proof}

From the definition of $\mathcal A_1(k,y)$ and $\mathcal A_2(k,y)$, and the expansion of Airy function in Lemma \ref{lem:Airy-class-est}, we can see that for $|\kappa\eta(y)|\ge M$, it holds that
\begin{align}
  \mathcal A_1(k,\kappa\eta(y))\sim \left\langle \kappa\eta(y) \right\rangle^{-\frac{1}{4}-\frac{1}{2}k}e^{-\frac{2}{3}\left(e^{\frac{\pi}{6}i}\kappa\eta(y)\right)^{\frac{3}{2}}}\left(1+O(\left\langle \kappa\eta(y) \right\rangle^{-\frac{3}{2}})\right),\label{eq-exp-MA1k}\\
  \mathcal A_2(k,\kappa\eta(y))\sim \left\langle \kappa\eta(y) \right\rangle^{-\frac{1}{4}-\frac{1}{2}k}e^{-\frac{2}{3}\left(e^{\frac{5\pi}{6}i}\kappa\eta(y)\right)^{\frac{3}{2}}}\left(1+O(\left\langle \kappa\eta(y) \right\rangle^{-\frac{3}{2}})\right).\label{eq-exp-MA2k}
\end{align}
Using these expansions, we derive the following estimates for the integration terms and the boundary terms in the in Lemma \ref{lem:pri-Airy-expan}.
\begin{lemma}\label{lem-Airy-int-est}
  It holds that
  \begin{align*}
      \int^0_y \left|\mathcal A_1(2,\kappa \eta(z))\right|dz  \lesssim \kappa^{-1}\left\langle \kappa\eta(y) \right\rangle^{-\frac{7}{4}} \left|e^{-\frac{2}{3}\left(e^{\frac{\pi}{6}i}\kappa\eta(y)\right)^{\frac{3}{2}}}\right|,\\
      \int^0_y \int^0_z \left|\mathcal A_1(2,\kappa \eta(x))\right| dx dz \lesssim \kappa^{-2}\left\langle \kappa\eta(y) \right\rangle^{-\frac{9}{4}} \left|e^{-\frac{2}{3}\left(e^{\frac{\pi}{6}i}\kappa\eta(y)\right)^{\frac{3}{2}}}\right|,\\
     \int^y_{-1} \left|\mathcal A_2(2,\kappa \eta(z))\right|dz \lesssim \kappa^{-1}\left\langle \kappa\eta(y) \right\rangle^{-\frac{7}{4}} \left|e^{-\frac{2}{3}\left(e^{\frac{5\pi}{6}i}\kappa\eta(y)\right)^{\frac{3}{2}}}\right|,\\
     \int^y_{-1} \int^z_{-1} \left|\mathcal A_2(2,\kappa \eta(x))\right| dx dz \lesssim \kappa^{-2}\left\langle \kappa\eta(y) \right\rangle^{-\frac{9}{4}} \left|e^{-\frac{2}{3}\left(e^{\frac{5\pi}{6}i}\kappa\eta(y)\right)^{\frac{3}{2}}}\right|,
  \end{align*}
  and
  \begin{align*}
    \left|\mathcal A_1(2,\kappa \eta(0))\right|+\left|\kappa y \mathcal A_1(1,\kappa \eta(0))\right|\lesssim \left\langle \kappa\eta(y) \right\rangle^{-\frac{5}{4}}\left|e^{-\frac{2}{3}\left(e^{\frac{\pi}{6}i}\kappa\eta(y)\right)^{\frac{3}{2}}}\right|,\\
    \left|\mathcal A_2(2,\kappa \eta(-1))\right|+\left|\kappa (y+1) \mathcal A_2(1,\kappa \eta(-1))\right|\lesssim \left\langle \kappa\eta(y) \right\rangle^{-\frac{5}{4}}\left|e^{-\frac{2}{3}\left(e^{\frac{5\pi}{6}i}\kappa\eta(y)\right)^{\frac{3}{2}}}\right|.
  \end{align*}
\end{lemma}
\begin{proof}
Since
\begin{align*}
|\mathcal A_1(2, \kappa\eta(z))|=|\mathrm{Ai}(2, e^{\frac{\pi i}{6}}\kappa\eta(z))|\lesssim \langle \kappa\eta(z) \rangle^{-\frac54} \left|e^{-\frac23(e^{\frac{\pi i}{6}}\kappa\eta(z))^\frac32}\right|,
\end{align*}
then by Lemma \ref{lem-Airy-Airy-est}, we have
\begin{align*}
\int_0^y \left|\mathcal A_1(2, \kappa\eta(z)) e^{\frac23(e^{\frac{\pi i}{6}}\kappa\eta(y))^\frac32}\right| dz\lesssim& \int_y^0\langle \kappa\eta(z) \rangle^{-\frac54} \left|e^{-\frac23(e^{\frac{\pi i}{6}}\kappa\eta(z))^\frac32} e^{\frac23(e^{\frac{\pi i}{6}}\kappa\eta(y))^\frac32}\right|dz\\
\lesssim&\int_y^0\langle \kappa\eta(z) \rangle^{-\frac54} e^{-C^{-1}|\kappa\eta_r(z)-\kappa\eta_r(y)|(|\kappa\eta(z)|^\frac12+|\kappa\eta(y)|^\frac12)}dz.
\end{align*}

If $|\kappa\eta_r(z)-\kappa\eta_r(y)|\leq 1$ and $|\kappa\eta(y)|\ge 1$, we have $\langle \kappa\eta(z) \rangle \sim\langle \kappa\eta(y) \rangle $ and
\begin{align*}
\int_y^0 \chi_{|\kappa\eta_r(z)-\kappa\eta_r(y)|\leq 1}\langle \kappa\eta(z) \rangle^{-\frac54} e^{-C^{-1}|\kappa\eta_r(z)-\kappa\eta_r(y)|(|\kappa\eta(z)|^\frac12+|\kappa\eta(y)|^\frac12)}dz\lesssim \kappa^{-1}\langle \kappa\eta(y) \rangle^{-\frac74}.
\end{align*}
If $|\kappa\eta_r(z)-\kappa\eta_r(y)|\leq 1$ and $|\kappa\eta(y)|<1$, we have $\langle \kappa\eta(y) \rangle\sim 1$, and
\begin{align*}
&\int_y^0 \chi_{|\kappa\eta_r(z)-\kappa\eta_r(y)|\leq 1}\langle \kappa\eta(z) \rangle^{-\frac54} e^{-C^{-1}|\kappa\eta_r(z)-\kappa\eta_r(y)|(|\kappa\eta(z)|^\frac12+|\kappa\eta(y)|^\frac12)}dz\\
\lesssim&\int_{z\in \left\{z| |\kappa\eta_r(z)-\kappa\eta_r(y)|\leq 1\right\}} dz \lesssim \kappa^{-1}\langle \kappa\eta(y) \rangle^{-\frac74}.
\end{align*}
For the rest part, $|\kappa\eta_r(z)-\kappa\eta_r(y)|>1$, we have
\begin{align*}
&\int_y^0\chi_{|\kappa\eta_r(z)-\kappa\eta_r(y)|> 1}\langle \kappa\eta(z) \rangle^{-\frac54} e^{-C^{-1}|\kappa\eta_r(z)-\kappa\eta_r(y)|(|\kappa\eta(z)|^\frac12+|\kappa\eta(y)|^\frac12)}dz\\
\lesssim& e^{-C^{-1} |\kappa\eta(y)|^\frac12}\int_y^0\langle \kappa\eta(z) \rangle^{-\frac54}dz\lesssim\kappa^{-1}e^{-C^{-1}|\kappa\eta(y)|^\frac12}\lesssim \kappa^{-1}\langle \kappa\eta(y) \rangle^{-\frac74}.
\end{align*}

For the second-order integral term, we write
\begin{align*}
  &\int^y_0 \int^z_0 \left|\mathcal A_1(2, \kappa\eta(x)) e^{\frac23(e^{\frac{\pi i}{6}}\kappa\eta(y))^\frac32}\right| dx dz\\
  =&\int^y_0 \left|e^{\frac23(e^{\frac{\pi i}{6}}\kappa\eta(y))^\frac32}e^{-\frac23(e^{\frac{\pi i}{6}}\kappa\eta(z))^\frac32}\right| \int^z_0  \left|\mathcal A_1(2, \kappa\eta(x)) e^{\frac23(e^{\frac{\pi i}{6}}\kappa\eta(z))^\frac32}\right| dx dz,
\end{align*}
Then by using the same technical, we can get the desired estimate. 

Next, we turn to the boundary terms. It is clear that
\begin{align*}
  \left|\kappa y \mathcal A_1(1,\kappa \eta(0))e^{\frac23(e^{\frac{\pi i}{6}}\kappa\eta(y))^\frac32}\right|\lesssim \kappa y\langle \kappa\eta(0) \rangle^{-\frac34} e^{-C^{-1}|\kappa\eta_r(0)-\kappa\eta_r(y)|(|\kappa\eta(0)|^\frac12+|\kappa\eta(y)|^\frac12)}.
\end{align*}
If $|\kappa\eta_r(0)-\kappa\eta_r(y)|\leq 1$, then $|y|\lesssim \kappa^{-1}$, and $\langle \kappa\eta(0) \rangle \sim\langle \kappa\eta(y) \rangle$. Thus for both cases $|\kappa\eta(y)|<1$ and $|\kappa\eta(y)|\ge1$, we have
\begin{align*}
  \kappa y\langle \kappa\eta(0) \rangle^{-\frac34}e^{-C^{-1}|\kappa y||\kappa\eta(0)|^\frac12}\lesssim\langle \kappa\eta(0) \rangle^{-\frac54} \lesssim\langle \kappa\eta(y) \rangle^{-\frac54}.
\end{align*}
If $|\kappa\eta_r(0)-\kappa\eta_r(y)|> 1$, then
\begin{align*}
  \left|\kappa y \mathcal A_1(1,\kappa \eta(0))e^{\frac23(e^{\frac{\pi i}{6}}\kappa\eta(y))^\frac32}\right|\lesssim \kappa |y|e^{-C^{-1}(|\kappa\eta(0)|^\frac12+|\kappa\eta(y)|^\frac12)}\lesssim \langle \kappa\eta(y) \rangle^{-\frac54}.
\end{align*}
The estimate for $\mathcal A_1(2,\kappa \eta(0))$ is easier.

The estimates for $\mathcal A_2$ related terms are the same.
\end{proof}

\subsection{Estimates for the modified Green function}
In this subsection, we derive estimates for the modified Green function $G_A$ defined in \eqref{eq-Green-Airy}. 

The Green function consists of two types of contributions. The first type consists of localized terms, such as $A_1(2,y)A_2(x)$ and $A_1(x)A_2(2,y)$, which enjoy exponential off-diagonal decay when $|x-y|$ is large. The second type consists of non-localized correction terms involving $a_1(x)$ and $a_2(x)$, which in general have only polynomial decay. Therefore, in order to obtain sharp estimates, we need to expand these non-localized terms more carefully.

Recall that
\begin{align*}
  a_1(x)= A_1(1,x) A_2(x)- A_1(x) A_2(1,x),\\ 
  a_2(x)= A_1(x) A_2(2,x)- A_1(2,x) A_2(x).
\end{align*}

By Lemma \ref{lem:pri-Airy-expan}, we can write 
\begin{align}\label{eq-a1-exp}
  a_1(x)=a_{1,M}(x)+a_{1,SecM}(x)+a_{1,B}(x)+a_{1,R}(x),
\end{align}
where
\begin{align}\label{eq-def-a1M}
  a_{1,M}(x)=\frac{\mathcal A_1(1,\kappa \eta(x))}{\kappa\left(\eta_r'(x)\right)^{\frac{3}{2}}}\mathrm{Ai}(e^{\frac{5\pi}{6}i}\kappa \eta(x))-\frac{\mathcal A_2(1,\kappa \eta(x))}{\kappa\left(\eta_r'(x)\right)^{\frac{3}{2}}}\mathrm{Ai}(e^{\frac{\pi}{6}i}\kappa \eta(x)),
\end{align}
\begin{align}\label{eq-def-a1SM}
  a_{1,SecM}(x)=\frac{3}{2}\frac{\eta_r''(x)}{\kappa^2\left(\eta_r'(x)\right)^{\frac{7}{2}}}\left(\mathcal A_1(2,\kappa \eta(x))\mathrm{Ai}(e^{\frac{5\pi}{6}i}\kappa \eta(x))-\mathcal A_2(2,\kappa \eta(x))\mathrm{Ai}(e^{\frac{\pi}{6}i}\kappa \eta(x))\right),
\end{align}
\begin{align}\label{eq-def-a1B}
  a_{1,B}(x)=-\frac{\mathcal A_1(1,\kappa \eta(0))}{\kappa\left(\eta_r'(0)\right)^{\frac{3}{2}}}\mathrm{Ai}(e^{\frac{5\pi}{6}i}\kappa \eta(x))+\frac{\mathcal A_2(1,\kappa \eta(-1))}{\kappa\left(\eta_r'(-1)\right)^{\frac{3}{2}}}\mathrm{Ai}(e^{\frac{\pi}{6}i}\kappa \eta(x)),
\end{align}
and
\begin{align}\label{eq-def-a1R}
  a_{1,R}(x)=a_1(x)-a_{1,M}(x)-a_{1,SecM}(x)-a_{1,B}(x).
\end{align}
Here $a_{1,M}$ is the leading main part, while $a_{1,SecM}$ is the secondary main part. These two terms will be used through their precise leading-order structures. The term $a_{1,B}$ contains the leading boundary contributions. Although it is not necessarily small in size, it enjoys exponential decay away from the corresponding boundary and is therefore harmless in the estimates below. All the remaining terms, including the secondary boundary contributions and the integral remainders in Lemma \ref{lem:pri-Airy-expan}, are collected in $a_{1,R}$.

Similarly, we write
\begin{align}\label{eq-a2-exp}
  a_2(x)=a_{2,M}(x)+a_{2,B}(x)+a_{2,R}(x),
\end{align}
where
\begin{align}\label{eq-def-a2M}
  a_{2,M}(x)=\mathrm{Ai}(e^{\frac{\pi}{6}i}\kappa \eta(x))\frac{\mathcal A_2(2,\kappa \eta(x))}{\kappa^2\left(\eta_r'(x)\right)^{\frac{5}{2}}} -\mathrm{Ai}(e^{\frac{5\pi}{6}i}\kappa \eta(x))\frac{\mathcal A_1(2,\kappa \eta(x))}{\kappa^2\left(\eta_r'(x)\right)^{\frac{5}{2}}},
\end{align}
\begin{equation}\label{eq-def-a2B}
  \begin{aligned}    
    a_{2,B}(x)=&-\mathrm{Ai}(e^{\frac{\pi}{6}i}\kappa \eta(x))\left( \frac{\mathcal A_2(2,\kappa \eta(-1))}{\kappa^2\left(\eta_r'(-1)\right)^{\frac{5}{2}}}+(x+1) \left(  \frac{\mathcal A_2(1,\kappa \eta(-1))}{\kappa\left(\eta_r'(-1)\right)^{\frac{3}{2}}}\right)\right)\\
 &+\mathrm{Ai}(e^{\frac{5\pi}{6}i}\kappa \eta(x)) \left(\frac{\mathcal A_1(2,\kappa \eta(0))}{\kappa^2\left(\eta_r'(0)\right)^{\frac{5}{2}}}+x \left(\frac{\mathcal A_1(1,\kappa \eta(0))}{\kappa\left(\eta_r'(0)\right)^{\frac{3}{2}}}\right)\right),    
  \end{aligned}
\end{equation}
and
\begin{align}\label{eq-def-a2R}
  a_{2,R}(x)=a_2(x)-a_{2,M}(x)-a_{2,B}(x).
\end{align}

The leading terms $a_{1,M}$, $a_{1,SecM}$, and $a_{2,M}$ are combinations of two rotated Airy functions and their primitives. These combinations can be reduced to the Scorer function $\mathrm{Hi}$ by the standard connection formulas for Airy and Scorer functions. The calculation of $a_{1,M}$ is closely related to the Airy-primitives identity used in Appendix A of \cite{CWZ2026}.
\begin{lemma}\label{lem-a1-a2-exp}
  It holds that
  \begin{align}\label{eq-exp-a1M}
    a_{1,M}(x)=-\frac{1}{2}\frac{\mathrm{Hi}(e^{-\frac{\pi}{2}i}\kappa \eta(x))}{\kappa\left(\eta_r'(x)\right)^{\frac{3}{2}}},
  \end{align}
  \begin{align}\label{eq-exp-a1SM}
    a_{1,SecM}(x)=-\frac{3i}{4}\frac{\eta_r''(x)}{\kappa^2\left(\eta_r'(x)\right)^{\frac{7}{2}}}\mathrm{Hi}''(e^{-\frac{\pi}{2}i}\kappa \eta(x)),
  \end{align}  
  and
  \begin{align}\label{eq-exp-a2M}
    a_{2,M}(x)=\frac{i}{2}\frac{1}{\kappa^2\left(\eta_r'(x)\right)^{\frac{5}{2}}}\mathrm{Hi}''(e^{-\frac{\pi}{2}i}\kappa \eta(x)).
  \end{align}
\end{lemma}
\begin{proof}
We first focus on the main part. It follows from the definition \eqref{eq-Airy-prim-1} and \eqref{eq-def-a1M} that
\begin{align*}
  a_{1,M}(x)=&\frac{\mathcal A_1(1,\kappa \eta(x))}{\kappa\left(\eta_r'(x)\right)^{\frac{3}{2}}}\mathrm{Ai}(e^{\frac{5\pi}{6}i}\kappa \eta(x))-\frac{\mathcal A_2(1,\kappa \eta(x))}{\kappa\left(\eta_r'(x)\right)^{\frac{3}{2}}}\mathrm{Ai}(e^{\frac{\pi}{6}i}\kappa \eta(x))\\
  =&\frac{\int^{\kappa \eta(x)}_{+\infty}\mathrm{Ai}(e^{\frac{\pi}{6}i}z)dz\mathrm{Ai}(e^{\frac{5\pi}{6}i}\kappa \eta(x))-\int^{\kappa \eta(x)}_{-\infty}\mathrm{Ai}(e^{\frac{5\pi}{6}i}z)dz\mathrm{Ai}(e^{\frac{\pi}{6}i}\kappa \eta(x))}{\kappa\left(\eta_r'(x)\right)^{\frac{3}{2}}}\\
  =&\frac{e^{-\frac{\pi}{6}i}\int^{e^{\frac{\pi}{6}i}\kappa \eta(x)}_{+\infty}\mathrm{Ai}(z)dz\mathrm{Ai}(e^{\frac{5\pi}{6}i}\kappa \eta(x))-e^{-\frac{\pi}{6}i}\int^{e^{\frac{\pi}{6}i}\kappa \eta(x)}_{-\infty}\mathrm{Ai}(e^{\frac{2\pi}{3}i}z)dz\mathrm{Ai}(e^{\frac{\pi}{6}i}\kappa \eta(x))}{\kappa\left(\eta_r'(x)\right)^{\frac{3}{2}}},
\end{align*}
by using the fact that $\mathrm{Ai}(e^{\frac{2\pi}{3}i}x)=\frac{e^{\frac{\pi}{3}i}}{2}\left(\mathrm{Ai}(x)-i\mathrm{Bi}(x)\right)$, we have
\begin{align*} 
  a_{1,M}(x)=&\frac{-\frac{e^{\frac{\pi}{6}i}}{2}\int^{+\infty}_{-\infty}\mathrm{Ai}(z)dz\mathrm{Ai}(e^{\frac{\pi}{6}i}\kappa \eta(x))}{\kappa\left(\eta_r'(x)\right)^{\frac{3}{2}}}\\ 
  &+\frac{-\frac{ie^{\frac{\pi}{6}i}}{2}\int^{e^{\frac{\pi}{6}i}\kappa \eta(x)}_{+\infty}\mathrm{Ai}(z)dz\mathrm{Bi}(e^{\frac{\pi}{6}i}\kappa \eta(x))+\frac{ie^{\frac{\pi}{6}i}}{2}\int^{e^{\frac{\pi}{6}i}\kappa \eta(x)}_{-\infty}\mathrm{Bi}(z)dz\mathrm{Ai}(e^{\frac{\pi}{6}i}\kappa \eta(x))}{\kappa\left(\eta_r'(x)\right)^{\frac{3}{2}}}.
\end{align*}
Recall that $\int^{0}_{-\infty}\mathrm{Bi}(x)dx=0$, $\int^{+\infty}_{-\infty}\mathrm{Ai}(x)dx=1$ and (see \cite{Olver2010})
\begin{align*}
  \mathrm{Gi}(y)=\mathrm{Bi}(y)\int^{+\infty}_{y}\mathrm{Ai}(x)dx+\mathrm{Ai}(y)\int^{y}_{0}\mathrm{Bi}(x)dx,\quad \mathrm{Gi}(y)=e^{\frac{\pi}{3}i}\mathrm{Hi}\left(e^{-\frac{2\pi}{3}i}y\right)-i\mathrm{Ai}(y).
\end{align*}
We have
\begin{align*}
  a_{1,M}(x)=&\frac{-\frac{e^{\frac{\pi}{6}i}}{2} \mathrm{Ai}(e^{\frac{\pi}{6}i}\kappa \eta(x))+\frac{ie^{\frac{\pi}{6}i}}{2}\mathrm{Gi}(e^{\frac{\pi}{6}i}\kappa \eta(x))}{\kappa\left(\eta_r'(x)\right)^{\frac{3}{2}}}=\frac{i\frac{e^{\frac{\pi}{2}i}}{2} \mathrm{Hi}(e^{-\frac{\pi}{2}i}\kappa \eta(x))}{\kappa\left(\eta_r'(x)\right)^{\frac{3}{2}}}.
\end{align*} 

Next, we study the secondary main part $a_{1,SecM}(x)$. Recall that
\begin{align*}
  a_{1,SecM}(x)=&\frac{3}{2}\frac{\eta_r''(x)}{\kappa^2\left(\eta_r'(x)\right)^{\frac{7}{2}}}\left(\mathcal A_1(2,\kappa \eta(x))\mathrm{Ai}(e^{\frac{5\pi}{6}i}\kappa \eta(x))-\mathcal A_2(2,\kappa \eta(x))\mathrm{Ai}(e^{\frac{\pi}{6}i}\kappa \eta(x))\right).
\end{align*}

It follows from \eqref{eq-Airy-k} and \eqref{eq-Airy-prim-2} that
\begin{align*}
  \mathcal A_1(2,y)=y\mathcal A_1(1,y)-e^{-\frac{\pi}{3}i}\mathrm{Ai}'(e^{\frac{\pi}{6}i}y),\\
  \mathcal A_2(2,y)=y\mathcal A_2(1,y)-e^{\frac{\pi}{3}i}\mathrm{Ai}'(e^{\frac{5\pi}{6}i}y),
\end{align*}
therefore, we have
\begin{align*}
  &\mathcal A_1(2,\kappa \eta(x))\mathrm{Ai}(e^{\frac{5\pi}{6}i}\kappa \eta(x))-\mathcal A_2(2,\kappa \eta(x))\mathrm{Ai}(e^{\frac{\pi}{6}i}\kappa \eta(x))\\
  =&\kappa \eta(x) \left(\mathcal A_1(1,\kappa \eta(x))\mathrm{Ai}(e^{\frac{5\pi}{6}i}\kappa \eta(x))-\mathcal A_2(1,\kappa \eta(x))\mathrm{Ai}(e^{\frac{\pi}{6}i}\kappa \eta(x))\right)\\
  &-e^{-\frac{\pi}{3}i}\mathrm{Ai}'(e^{\frac{\pi}{6}i}\kappa \eta(x))\mathrm{Ai}(e^{\frac{5\pi}{6}i}\kappa \eta(x))+e^{\frac{\pi}{3}i}\mathrm{Ai}'(e^{\frac{5\pi}{6}i}\kappa \eta(x))\mathrm{Ai}(e^{\frac{\pi}{6}i}\kappa \eta(x)).
\end{align*}
A direct calculation shows that
\begin{align*}
  &-e^{-\frac{\pi}{3}i}\mathrm{Ai}'(e^{\frac{\pi}{6}i}y)\mathrm{Ai}(e^{\frac{5\pi}{6}i}y)+e^{\frac{\pi}{3}i}\mathrm{Ai}'(e^{\frac{5\pi}{6}i}y)\mathrm{Ai}(e^{\frac{\pi}{6}i}y)\\
  =&e^{-\frac{\pi}{3}i} \left( \mathrm{Ai}(e^{\frac{\pi}{6}i}y)e^{\frac{2\pi}{3}i} \mathrm{Ai}'(e^{\frac{5\pi}{6}i}y) -\mathrm{Ai}'(e^{\frac{\pi}{6}i}y)\mathrm{Ai}(e^{\frac{5\pi}{6}i}y)\right)=e^{-\frac{\pi}{3}i}\frac{e^{-\frac{\pi}{6}i}}{2\pi}=\frac{e^{-\frac{\pi}{2}i}}{2\pi}.
\end{align*}
By the same technique used in deducing \eqref{eq-exp-a1M}, we have
\begin{align*}
  &\kappa \eta(x) \left(\mathcal A_1(1,\kappa \eta(x))\mathrm{Ai}(e^{\frac{5\pi}{6}i}\kappa \eta(x))-\mathcal A_2(1,\kappa \eta(x))\mathrm{Ai}(e^{\frac{\pi}{6}i}\kappa \eta(x))\right)=i\frac{e^{\frac{\pi}{2}i}}{2}\kappa \eta(x) \mathrm{Hi}(e^{-\frac{\pi}{2}i}\kappa \eta(x)).
\end{align*}
Recall that $\mathrm{Hi}$ satisfies
\begin{align*}
  \mathrm{Hi}''(z)-z\mathrm{Hi}(z)=\frac{1}{\pi}.
\end{align*}
We have
\begin{align*}
  a_{1,SecM}(x)=-i\frac{3}{4}\frac{\eta_r''(x)}{\kappa^2\left(\eta_r'(x)\right)^{\frac{7}{2}}}\mathrm{Hi}''(e^{-\frac{\pi}{2}i}\kappa \eta(x)).
\end{align*}

For the same reason, we get \eqref{eq-exp-a2M}.
\end{proof}
\begin{remark}\label{rmk-Hi''}
  Here we remark that, by the expansion of $\mathrm{Hi}(z)$, we can deduce the expansion of $\mathrm{Hi}''(z)$. For given $\delta>0$ and big enough $M$, it holds that
\begin{align}\label{eq-expan-Hi''}
  \mathrm{Hi}''(z)=-\frac{1}{\pi}\sum^\infty_{j=1}\frac{\left(3j\right)!}{j!\left(3z^3\right)^j},\text{ for }|z|\ge M,\ |\arg -z|\le \frac{2}{3}\pi-\delta.
\end{align}
By induction, we can derive the expansion of $\mathrm{Hi}^{(k)}(z)$ for $\forall k\in\mathbb Z_+$.
\end{remark}

Combining the above results, we get the following estimate for $a_1(x)$ and $a_2(x)$.
\begin{lemma}\label{lem-est-a1-a2}
  For $k=0,1,2$, where $a^{(k)}=\pa_x^k a$, it holds that
\begin{align*}
  \left|a_{1,M}^{(k)}(x)\right|\lesssim& \kappa^{-1+k}\left\langle \kappa \eta(x) \right\rangle^{-1-k},\\
  \left|a_{1,SecM}^{(k)}(x)\right|\lesssim& \kappa^{-2+k}\left\langle \kappa \eta(x) \right\rangle^{-3-k},\\
  \left|a_{1,B}^{(k)}(x)\right|\lesssim & \frac{e^{-\frac{1}{C}|\kappa\eta_r(x)-\kappa\eta_r(0)|(|\kappa\eta(x)|^\frac12+|\kappa\eta(0)|^\frac12)}}{\kappa^{1-k}\left\langle \kappa \eta(0) \right\rangle^{\frac{3}{4}}\left\langle \kappa \eta(x) \right\rangle^{\frac{1}{4}-\frac{1}{2}k}}+\frac{e^{-\frac{1}{C}|\kappa\eta_r(x)-\kappa\eta_r(-1)|(|\kappa\eta(x)|^\frac12+|\kappa\eta(-1)|^\frac12)}}{\kappa^{1-k}\left\langle \kappa \eta(-1) \right\rangle^{\frac{3}{4}}\left\langle \kappa \eta(x) \right\rangle^{\frac{1}{4}-\frac{1}{2}k}},\\
  \left|a_{1,R}^{(k)}(x)\right|\lesssim & \kappa^{-3+k}\left\langle \kappa \eta(x) \right\rangle^{-2+\frac{1}{2}k}\\
  &+\frac{e^{-\frac{1}{C}|\kappa\eta_r(x)-\kappa\eta_r(0)|(|\kappa\eta(x)|^\frac12+|\kappa\eta(0)|^\frac12)}}{\kappa^{2-k}\left\langle \kappa \eta(0) \right\rangle^{\frac{5}{4}}\left\langle \kappa \eta(x) \right\rangle^{\frac{1}{4}-\frac{1}{2}k}}+\frac{e^{-\frac{1}{C}|\kappa\eta_r(x)-\kappa\eta_r(-1)|(|\kappa\eta(x)|^\frac12+|\kappa\eta(-1)|^\frac12)}}{\kappa^{2-k}\left\langle \kappa \eta(-1) \right\rangle^{\frac{5}{4}}\left\langle \kappa \eta(x) \right\rangle^{\frac{1}{4}-\frac{1}{2}k}},\\
  \left|a_{2,M}^{(k)}(x)\right|\lesssim& \kappa^{-2+k}\left\langle \kappa \eta(x) \right\rangle^{-3-k},\\
  \left|a_{2,B}^{(k)}(x)\right|\lesssim & \frac{e^{-\frac{1}{C}|\kappa\eta_r(x)-\kappa\eta_r(0)|(|\kappa\eta(x)|^\frac12+|\kappa\eta(0)|^\frac12)}}{\kappa^{2-k}\left\langle \kappa \eta(0) \right\rangle^{\frac{5}{4}}\left\langle \kappa \eta(x) \right\rangle^{\frac{1}{4}-\frac{1}{2}k}}+\frac{e^{-\frac{1}{C}|\kappa\eta_r(x)-\kappa\eta_r(-1)|(|\kappa\eta(x)|^\frac12+|\kappa\eta(-1)|^\frac12)}}{\kappa^{2-k}\left\langle \kappa \eta(-1) \right\rangle^{\frac{5}{4}}\left\langle \kappa \eta(x) \right\rangle^{\frac{1}{4}-\frac{1}{2}k}},\\
  \left|a_{2,R}^{(k)}(x)\right|\lesssim & \kappa^{-3+k}\left\langle \kappa \eta(x) \right\rangle^{-2+\frac{1}{2}k}\\
  &+\frac{e^{-\frac{1}{C}|\kappa\eta_r(x)-\kappa\eta_r(0)|(|\kappa\eta(x)|^\frac12+|\kappa\eta(0)|^\frac12)}}{\kappa^{3-k}\left\langle \kappa \eta(0) \right\rangle^{\frac{7}{4}}\left\langle \kappa \eta(x) \right\rangle^{\frac{1}{4}-\frac{1}{2}k}}+\frac{e^{-\frac{1}{C}|\kappa\eta_r(x)-\kappa\eta_r(-1)|(|\kappa\eta(x)|^\frac12+|\kappa\eta(-1)|^\frac12)}}{\kappa^{3-k}\left\langle \kappa \eta(-1) \right\rangle^{\frac{7}{4}}\left\langle \kappa \eta(x) \right\rangle^{\frac{1}{4}-\frac{1}{2}k}}.
  \end{align*}
\end{lemma}
\begin{proof}
The results follow directly from Lemma \ref{lem-a1-a2-exp}, Lemma \ref{lem-Airy-Airy-est}, and Lemma \ref{lem-Airy-int-est}.  
\end{proof}
 
In the subsequent analysis, we will need to take the derivative of $a_{1}$ with respect to $c_r$. Among these terms, only the derivative of $a_{1,B}$ requires a careful and detailed examination. 
\begin{lemma}\label{lem-d-cr-a1B}
  It holds that
  \begin{equation}\label{est-d-cr-a1B}
    \left|\pa_{c_r}a_{1,B}^{(k)}(x)\right|\lesssim\left\{
      \begin{array}{ll}
        \frac{\kappa^{k}}{\left\langle \kappa c_r\right\rangle^{2-\frac{k}{2}}}\left|\ln c\right|,&\text{ for }-1\le x\le -1+\frac{C \left|\ln c\right|}{\kappa^{\frac{3}{2}}c_r^{\frac{1}{2}}},\\
        \varepsilon^2,&\text{ for }-1+\frac{C \left|\ln c\right|}{\kappa^{\frac{3}{2}}c_r^{\frac{1}{2}}}< x< -\frac{C \left|\ln c\right|}{\kappa^{\frac{3}{2}}},\\
        \kappa^{-2+\frac{3k}{2}}\left|\ln c\right|,&\text{ for }-\frac{C \left|\ln c\right|}{\kappa^{\frac{3}{2}}}\le x\le 0,
      \end{array}
    \right.\text{ for }k=0,1,2.
  \end{equation}
  where $C$ is a constant that sufficiently big.
\end{lemma}
\begin{proof}
Recall \eqref{eq-def-a1B} the definition of $a_{1,B}$. To study $\pa_{c_r}a_{1,B}$, it suffices to study $\pa_{c_r} \left(\mathcal A_2(1,\kappa \eta(-1))\mathrm{Ai}(e^{\frac{\pi}{6}i}\kappa \eta(x))  \right)$ and $\pa_{c_r} \left(\mathcal A_1(1,\kappa \eta(0))\mathrm{Ai}(e^{\frac{5\pi}{6}i}\kappa \eta(x))  \right)$.

From the \eqref{eq-Airy-prim-2-prop} and the expansion of Airy functions, we have
\begin{align*}
  &\pa_{c_r} \left(\mathcal A_2(1,\kappa \eta(-1))\mathrm{Ai}(e^{\frac{\pi}{6}i}\kappa \eta(x))\right)\\
  =& \left(\pa_{c_r} \left(\kappa \eta(-1)\right)\right)\mathrm{Ai}(e^{\frac{5\pi}{6}i}\kappa \eta(-1))\mathrm{Ai}(e^{\frac{\pi}{6}i}\kappa \eta(x))+\left(\pa_{c_r} \left(\kappa \eta(x)\right)\right)e^{-\frac{2\pi}{3}i}\mathrm{Ai}(1,e^{\frac{5\pi}{6}i}\kappa \eta(-1))\mathrm{Ai}'(e^{\frac{\pi}{6}i}\kappa \eta(x))\\
  =&\left(\left(e^{\frac{\pi}{6}i}\kappa \eta(-1)\right)^{\frac{1}{2}}\pa_{c_r} \left(e^{\frac{\pi}{6}i}\kappa \eta(-1)\right)-\left(e^{\frac{\pi}{6}i}\kappa \eta(x)\right)^{\frac{1}{2}}\pa_{c_r} \left(e^{\frac{\pi}{6}i}\kappa \eta(x)\right)\right)\mathcal A_2(1,\kappa \eta(-1))\mathrm{Ai}(e^{\frac{\pi}{6}i}\kappa \eta(x))\\
  &+\left(e^{\frac{\pi}{6}i}\kappa \eta(-1)\right)^{\frac{1}{2}}\pa_{c_r} \left(e^{\frac{\pi}{6}i}\kappa \eta(-1)\right)\mathcal A_2(1,\kappa \eta(-1))\mathrm{Ai}(e^{\frac{\pi}{6}i}\kappa \eta(x))\cdot O \left(\left\langle \kappa\eta(-1) \right\rangle^{-\frac{3}{2}}\right) \\
  &+\left(e^{\frac{\pi}{6}i}\kappa \eta(x)\right)^{\frac{1}{2}}\pa_{c_r} \left(e^{\frac{\pi}{6}i}\kappa \eta(x)\right)\mathcal A_2(1,\kappa \eta(-1))\mathrm{Ai}(e^{\frac{\pi}{6}i}\kappa \eta(x))\cdot O \left(\left\langle \kappa\eta(x) \right\rangle^{-\frac{3}{2}}\right).
\end{align*}

It holds that
\begin{align*}
  &\left(\left(e^{\frac{\pi}{6}i}\kappa \eta(x)\right)^{\frac{1}{2}}\pa_{c_r} \left(e^{\frac{\pi}{6}i}\kappa \eta(x)\right)-\left(e^{\frac{\pi}{6}i}\kappa \eta(-1)\right)^{\frac{1}{2}}\pa_{c_r} \left(e^{\frac{\pi}{6}i}\kappa \eta(-1)\right)\right)\\
  =& \kappa^{\frac{1}{2}}\pa_{c_r}\kappa \left(\left(e^{\frac{\pi}{6}i} \eta(x)\right)^{\frac{3}{2}}-\left(e^{\frac{\pi}{6}i} \eta(-1)\right)^{\frac{3}{2}}   \right) + \kappa^{\frac{3}{2}}e^{\frac{\pi}{6}i} \left(\left(e^{\frac{\pi}{6}i} \eta(x)\right)^{\frac{1}{2}}\pa_{c_r}\eta(x)-\left(e^{\frac{\pi}{6}i} \eta(-1)\right)^{\frac{1}{2}} \pa_{c_r}\eta(-1)  \right).
\end{align*}

Remark that $|\pa_{c_r}\kappa|\sim |\kappa|$, $\eta_r$ and  $\pa_{c_r}\eta(x)$ are smooth. If $|x+1|\le \frac{C \left|\ln c\right|}{\kappa^{\frac{3}{2}}c_r^{\frac{1}{2}}}$, with $C$ sufficiently large such that $\left|e^{-\frac{2}{3}\kappa^{\frac{3}{2}} \left(\left(e^{\frac{\pi}{6}i} \eta(-1+ \frac{C \left|\ln c\right|}{\kappa^{\frac{3}{2}}c_r^{\frac{1}{2}}})\right)^{\frac{3}{2}}-\left(e^{\frac{\pi}{6}i} \eta(-1)\right)^{\frac{3}{2}}   \right)}\right|\lesssim \varepsilon^3$,  we have $\left|\eta(x)\right|\sim c_r$, $\left|\eta(x)-\eta(-1)\right|\lesssim \frac{C \left|\ln c\right|}{\kappa^{\frac{3}{2}}c_r^{\frac{1}{2}}}$. It follows that
\begin{align}\label{est-lem-cr-a1}
  \left|\kappa^{\frac{1}{2}}\pa_{c_r}\kappa \left(\left(e^{\frac{\pi}{6}i} \eta(x)\right)^{\frac{3}{2}}-\left(e^{\frac{\pi}{6}i} \eta(-1)\right)^{\frac{3}{2}}   \right)\right|\lesssim \left|\ln c\right|^{\frac{3}{2}},
\end{align}
and
\begin{align*}
  \left|\kappa^{\frac{3}{2}}\frac{3}{2}e^{\frac{\pi}{6}i} \left(\left(e^{\frac{\pi}{6}i} \eta(x)\right)^{\frac{1}{2}}\pa_{c_r}\eta(x)-\left(e^{\frac{\pi}{6}i} \eta(-1)\right)^{\frac{1}{2}} \pa_{c_r}\eta(-1)  \right)\right|\lesssim \frac{1}{c_r}\left|\ln c\right|.
\end{align*}

If $|x+1|\ge \frac{C \left|\ln c\right|}{\kappa^{\frac{3}{2}}c_r^{\frac{1}{2}}}$, the exponential factor provides sufficient smallness. 

The remainder can be estimated in the same way.

Therefore, we have 
\begin{align*}
  \left|\pa_{c_r} \left(\frac{\mathcal A_2(1,\kappa \eta(-1))}{\kappa\left(\eta_r'(-1)\right)^{\frac{3}{2}}}\mathrm{Ai}(e^{\frac{\pi}{6}i}\kappa \eta(x))\right)\right|\lesssim  \frac{1}{\left\langle \kappa c_r\right\rangle^2} \left|\ln c\right|.
\end{align*}

By using the same technique, we also have
\begin{align*}
  \left|\pa_{c_r} \left(\frac{\mathcal A_1(1,\kappa \eta(0))}{\kappa\left(\eta_r'(0)\right)^{\frac{3}{2}}}\mathrm{Ai}(e^{\frac{5\pi}{6}i}\kappa \eta(x))\right)\right|\lesssim \frac{1}{\kappa^2}\left|\ln c\right|.
\end{align*}

Combining the above estimates, we get \eqref{est-d-cr-a1B} for $\pa_{c_r}a_{1,B}(x)$. The estimates for $\pa_{c_r}a_{1,B}'(x)$ and $\pa_{c_r}a_{1,B}''(x)$ can be derived by  using the same method. 

\end{proof}
\subsection{Estimates for the modified AirySolver}\label{sec-modified-airysolver}
Using the approximate Green function \eqref{eq-Green-Airy}, we define the modified Airy-solver by
\begin{align}\label{eq-modified-Airysolver}
  AirySolver_m(f)(y)=\int^0_{-1}G_A(x,y)f(x)dx.
\end{align}
By the definition of the modified primitive Airy functions, we can see that
\begin{align}\label{eq-boundary-modi-Airy-Solver}
  \pa_yAirySolver_m(f)(0)=0.
\end{align}

Let $\phi(y)=AirySolver_m(f)(y)$. By the definition of $G_A(x,y)$, we can see that $\phi(y)$ solves the following modified Airy equation
\begin{equation}\label{eq-phiapp4}
  \begin{aligned}    
  &i\varepsilon\pa_y^4\phi(y)+\left(u_p(y)-c-2i\varepsilon\alpha^2\right)\pa_y^2\phi(y)\\ 
  =&f(y)+ i\varepsilon \left(\pa_y^2\left(\eta_r'(y)\right)^{-\frac{1}{2}}\right)\left(\eta_r'(y)\right)^{\frac{1}{2}}\pa_y^2\phi(y)+i \left(\left(\eta_r'(y)\right)^2-1\right)c_i\pa_y^2\phi(y)-2i\varepsilon\alpha^2\pa_y^2\phi(y).    
  \end{aligned}
\end{equation}

As in \cite{GGN16adv}, we decompose the Green function into a localized part $G_{A,M}(x,y)$and a non-localized part $E(x,y)$,
\begin{align*}
  G_A(x,y)=G_{A,M}(x,y)+E(x,y),
\end{align*}
where
\begin{align*}
  G_{A,M}(x,y)=i \frac{2\pi}{\left(\eta_r'(x)\right)^{\frac{1}{2}}\kappa\varepsilon}\left\{
    \begin{array}{ll}
      A_1(2,y) A_2(x) ,&x<y;\\ 
       A_1(x) A_2(2,y) ,&x>y.
    \end{array}
  \right.
\end{align*}
\begin{align*}
  E(x,y)=i \frac{2\pi}{\left(\eta_r'(x)\right)^{\frac{1}{2}}\kappa\varepsilon}\left\{
    \begin{array}{ll}
       a_2(x)+a_1(x)(x-y_c),&x<y;\\ 
       a_1(x)(y-y_c),&x>y.
    \end{array}
  \right.
\end{align*}

Next we give some basic estimates for $AirySolver_m(f)(y)$. To state the estimates, we introduce the modified wave speed $\hat c=\hat c_r+i\hat c_i$, with $\hat c_r=c_r$, $\hat c_i=c_i+c_0$, and $c_0=\varepsilon^\frac{1}{2}$. Since $c\in\mathbb H$, we have $\hat c_i\ge \frac{1}{2}c_0$.

 The first three estimates \eqref{eq-est-AS-0}-\eqref{eq-est-AS-2} will be widely used in constructing the solution of \eqref{eq-Airy}, and the estimate \eqref{eq-est-AS-4} just shows the solution is well defined in suitable regularity space.
 
\begin{lemma}\label{lem-est-AS}
  For all $y\in[-1,0]$, the modified Airy solver satisfies
\begin{align}\label{eq-est-AS-0}
  \left|AirySolver_m(f)(y)\right| \lesssim \left(1+ \kappa c_r\right)|\ln c_0|\left\|\left(u_p-\hat c\right)f\right\|_{L^\infty}.
\end{align}
\begin{align}\label{eq-est-AS-1}
  \left|\left(u_p-\hat c\right)\pa_yAirySolver_m(f)(y)\right| \lesssim \left(1+ \kappa c_r\right)|\ln c_0|\left\|\left(u_p-\hat c\right)f\right\|_{L^\infty}.
\end{align}
\begin{align}\label{eq-est-AS-2}
  \left|\left(u_p-\hat c\right)^2\pa_y^2AirySolver_m(f)(y)\right| \lesssim |\ln c_0|\left\|\left(u_p-\hat c\right)f\right\|_{L^\infty}.
\end{align}
\begin{align}\label{eq-est-AS-3}
   \left|\left(u_p-\hat c\right)^2\pa_y^3AirySolver_m(f)(y)\right| \lesssim  \varepsilon^{-\frac{1}{2}}|\ln c_0|\left\|\left(u_p-\hat c\right)f\right\|_{L^\infty}.
\end{align}
\begin{align}\label{eq-est-AS-4}
   \left|\left(u_p-\hat c\right)\pa_y^4AirySolver_m(f)(y)\right| \lesssim  \varepsilon^{-1}|\ln c_0|\left\|\left(u_p-\hat c\right)f\right\|_{L^\infty}.
\end{align}
\end{lemma}
\begin{proof}
  Recall that $\frac{1}{\varepsilon}\sim \kappa^3$. From Lemma \ref{lem:pri-Airy-expan}, Lemma \ref{lem-Airy-Airy-est}, and Lemma \ref{lem-Airy-int-est}, we can see that
  \begin{align*}
    \left|G_{A,M}(x,y)\right|\lesssim  \frac{1}{\left\langle \kappa\eta(x) \right\rangle^{\frac{1}{4}}\left\langle \kappa\eta(y) \right\rangle^{\frac{5}{4}}}e^{-\frac{1}{C}|\kappa\eta_r(x)-\kappa\eta_r(y)|(|\kappa\eta_r(x)|^\f12+|\kappa\eta_r(y)|^\f12)}.
  \end{align*}
It follows that 
\begin{align*}
  \left|\int^0_{-1}G_{A,M}(x,y)f(x)dx\right|\lesssim  \int^0_{-1} \frac{1}{\left|u_p-\hat c\right|}dx\left\|\left(u_p-\hat c\right)f\right\|_{L^\infty}\lesssim |\ln c_0|\left\|\left(u_p-\hat c\right)f\right\|_{L^\infty}.
\end{align*}

For the non-local term, it follows from Lemma \ref{lem-est-a1-a2} that
\begin{align*}
  \frac{1}{\kappa\varepsilon}a_1(x)(x-y_c)\sim \frac{\kappa(x-y_c)}{\left\langle \kappa\eta(x) \right\rangle}\sim 1,\quad \frac{1}{\kappa\varepsilon}a_1(x)(y-y_c)\sim \frac{\kappa(y-y_c)}{\left\langle \kappa\eta(x) \right\rangle}\sim \frac{\kappa\eta_r(y)}{\left\langle \kappa\eta(x) \right\rangle},
\end{align*}
and
\begin{align*}
  \left|\frac{1}{\kappa\varepsilon}a_2(x)\right|\lesssim \left\langle \kappa\eta(x) \right\rangle^{-\frac{3}{2}}.
\end{align*}

For $y<x$, worst term is $\frac{1}{\kappa\varepsilon}a_1(x)(y-y_c) \sim \frac{\kappa\eta(y)}{\left\langle \kappa\eta(x) \right\rangle}$. If $y<y_c$ and $x$ close to $y_c$, this term could be big. Recall that $\eta_r(-1)=-\frac{c_r}{2}+O(c^2)$, we have
\begin{align*}
  \left|\frac{1}{\kappa\varepsilon}a_1(x)(y-y_c)\right|\lesssim\left|\frac{\kappa\eta(y)}{\left\langle \kappa\eta(x) \right\rangle}\right|\lesssim \kappa c_r.
\end{align*}

Then
\begin{align*}
  \left|\int^0_{-1}E(x,y)f(x)dx\right|\lesssim  \kappa c_r\int^0_{-1} \frac{1}{\left|u_p-\hat c\right|}dx\left\|\left(u_p-\hat c\right)f\right\|_{L^\infty}\lesssim \kappa c_r |\ln c_0|\left\|\left(u_p-\hat c\right)f\right\|_{L^\infty}.
\end{align*}
This gives \eqref{eq-est-AS-0}.
 
Next, we turn to $\left(u_p-\hat c\right)\pa_yAirySolver_m(f)(y)$.  Recall that
\begin{align*}
  \pa_yAirySolver_m(f)(y)=&i\frac{2\pi}{\kappa\varepsilon}\int^y_{-1} A_1(1,y) A_2(x) \frac{f(x)}{\left(\eta_r'(x)\right)^{\frac{1}{2}}}dx\\ 
  &+i\frac{2\pi}{\kappa\varepsilon}\int^0_{y}\left( A_1(x) A_2(1,y)+a_1(x)\right)\frac{f(x)}{\left(\eta_r'(x)\right)^{\frac{1}{2}}}dx.
\end{align*}

For the $E(x,y)$ part, there is only one term remain. Since
\begin{align*}
  \left|\left(u_p(y)-\hat c\right)a_1(x)\right|\sim \left|a_1(x)(y-y_c+ic_0)\right|,
\end{align*}
the estimate of this part is the same to \eqref{eq-est-AS-0}. 

For the $G_{A,M}(x,y)$ part, by using Lemma \ref{lem:pri-Airy-expan}-\ref{lem-Airy-int-est}, we have
\begin{align*}
  &\left|\frac{2\pi}{\kappa\varepsilon}\int^y_{-1} A_1(1,y) A_2(x) \frac{f(x)}{\left(\eta_r'(x)\right)^{\frac{1}{2}}}dx+\frac{2\pi}{\kappa\varepsilon}\int^0_{y} A_1(x) A_2(1,y)\frac{f(x)}{\left(\eta_r'(x)\right)^{\frac{1}{2}}}dx\right|\\
\lesssim&\int^0_{-1}\left|\frac{\kappa\left(u_p(y)-\hat c\right)}{\left\langle \kappa\eta(y) \right\rangle^{\frac{3}{4}}} \frac{1}{\left\langle \kappa\eta(x) \right\rangle^{\frac{1}{4}}} e^{-\frac{1}{C}\left|\kappa\eta_r(x)-\kappa\eta_r(y)  \right|\left( |\kappa\eta(y)|^{\frac{1}{2}}+|\kappa\eta(x)|^{\frac{1}{2}} \right) }\frac{1}{\left(u_p(x)-\hat c\right)}\frac{\left(u_p(x)-\hat c\right)f(x)}{\left(\eta_r'(x)\right)^{\frac{1}{2}}}\right|dx.
\end{align*}

We denote
\begin{align*}
  I_{GM,1}(x,y) = \frac{\kappa\left(u_p(y)-\hat c\right)}{\left\langle \kappa\eta(y) \right\rangle^{\frac{3}{4}}} \frac{1}{\left\langle \kappa\eta(x) \right\rangle^{\frac{1}{4}}} e^{-\frac{1}{C}\left|\kappa\eta_r(x)-\kappa\eta_r(y)  \right|\left( |\kappa\eta(y)|^{\frac{1}{2}}+|\kappa\eta(x)|^{\frac{1}{2}} \right) }\frac{1}{\left(u_p(x)-\hat c\right)}.
\end{align*}
Therefore, to get \eqref{eq-est-AS-1}, we only need to estimate
\begin{align*}
  \left|\int^0_{-1}I_{GM,1}(x,y)dx\right|.
\end{align*}
We divided the estimate to different cases.

Case 1. $|y-y_c|\ge \frac{1}{2}$.

Case 1.1. $|x-y|\le \frac{1}{4}$. For this case, it holds that 
\begin{align*}
  \left|\eta(x)\right|\sim\left|\eta(y)\right|\sim|x-y_c|\sim|y-y_c|\sim \left|u_p(x)-\hat c\right|\sim \left|u_p(y)-\hat c\right|\sim 1.
\end{align*}
 It follows that $\left|\kappa\eta_r(x)-\kappa\eta_r(y)  \right|\left( |\kappa\eta(y)|^{\frac{1}{2}}+|\kappa\eta(x)|^{\frac{1}{2}} \right)\sim\kappa^{\frac{3}{2}} \left|x-y\right|$. Therefore, for this case, we have
\begin{align*}
  &\left|\frac{\kappa\left(u_p(y)-\hat c\right)}{\left\langle \kappa\eta(y) \right\rangle^{\frac{3}{4}}} \frac{1}{\left\langle \kappa\eta(x) \right\rangle^{\frac{1}{4}}} e^{-\frac{1}{C}\left|\kappa\eta_r(x)-\kappa\eta_r(y)  \right|\left( |\kappa\eta(y)|^{\frac{1}{2}}+|\kappa\eta(x)|^{\frac{1}{2}} \right) }\frac{1}{\left(u_p(x)-\hat c\right)}\right|\lesssim e^{-\frac{1}{C} \kappa^{\frac{3}{2}} \left|x-y\right|}.
\end{align*}
After integration, we deduce
\begin{align*}
  \left|\int^0_{-1}I_{GM,1}(x,y)\chi_{\text{Case 1.1}}dx\right|\lesssim \kappa^{-\frac{3}{2}}.
\end{align*}

Case 1.2. $|x-y|\ge \frac{1}{4}$. For this case, $\left|\kappa\eta_r(x)-\kappa\eta_r(y)  \right|\left( |\kappa\eta(y)|^{\frac{1}{2}}+|\kappa\eta(x)|^{\frac{1}{2}} \right)\sim\kappa^{\frac{3}{2}}$, and
\begin{align*}
  e^{-\frac{1}{C}\left|\kappa\eta_r(x)-\kappa\eta_r(y)  \right|\left( |\kappa\eta(y)|^{\frac{1}{2}}+|\kappa\eta(x)|^{\frac{1}{2}} \right) }\sim e^{-\frac{1}{C}\kappa^{\frac{3}{2}}},
\end{align*}
which is extremely small.
 
Case 2. $\frac{2}{\kappa}\le|y-y_c|\le \frac{1}{2}$. 

Case 2.1. $\left|x-y_c\right|\ge \frac{1}{\kappa}$ and $\left|x-y\right|\le \frac{1}{\kappa}$. For this case, we have $\left|u_p(x)-\hat c\right|\sim \left|u_p(y)-\hat c\right|$ and $\left|\kappa\eta_r(x)-\kappa\eta_r(y)  \right|\left( |\kappa\eta(y)|^{\frac{1}{2}}+|\kappa\eta(x)|^{\frac{1}{2}} \right)\gtrsim\kappa\left|x-y \right|$. It follows that
\begin{align*}
  \left|\int^0_{-1}I_{GM,1}(x,y)\chi_{\text{Case 2.1}}dx\right|\lesssim \left|\int^0_{-1}\kappa e^{-\frac{1}{C}\kappa\left|x-y\right|}dx\right|\lesssim1.
\end{align*}

Case 2.2. $\left|x-y_c\right|\ge \frac{1}{\kappa}$ and $\left|x-y\right|>\frac{1}{\kappa}$. For this case, $|\kappa\eta_r(x)|\gtrsim 1$, $|\kappa\eta_r(y)|\gtrsim 1$, and the distance of $x$ and $y$ has lower bound, which gives $\left|\kappa\eta_r(x)-\kappa\eta_r(y)\right|\gtrsim \kappa |x-y|+1$. Then we have
\begin{align*}
  e^{-\frac{1}{C}\left|\kappa\eta_r(x)-\kappa\eta_r(y)  \right|\left( |\kappa\eta(y)|^{\frac{1}{2}}+|\kappa\eta(x)|^{\frac{1}{2}} \right) }\lesssim e^{- \frac{1}{C}\left( |\kappa\eta(y)|^{\frac{1}{2}}+|\kappa\eta(x)|^{\frac{1}{2}} \right) }e^{-\frac{1}{C}\kappa |x-y|}.
\end{align*}
The first exponential decay term can absorb all potential polynomial growth. Therefore, in this case, it holds that
\begin{align*}
  \left|I_{GM,1}(x,y)\right|=&\left|\frac{\kappa\left(u_p(y)-\hat c\right)}{\left\langle \kappa\eta(y) \right\rangle^{\frac{3}{4}}} \frac{1}{\left\langle \kappa\eta(x) \right\rangle^{\frac{1}{4}}} e^{-\frac{1}{C}\left|\kappa\eta_r(x)-\kappa\eta_r(y)  \right|\left( |\kappa\eta(y)|^{\frac{1}{2}}+|\kappa\eta(x)|^{\frac{1}{2}} \right) }\frac{1}{\left(u_p(x)-\hat c\right)}\right|\\ 
  \lesssim&\left|\left\langle \kappa\eta(y) \right\rangle^{\frac{1}{4}} e^{-\frac{1}{C}\left|\kappa\eta_r(x)-\kappa\eta_r(y)  \right|\left( |\kappa\eta(y)|^{\frac{1}{2}}+|\kappa\eta(x)|^{\frac{1}{2}} \right) }\frac{\kappa}{\kappa\left(u_p(x)-\hat c\right)}\right|\\
  \lesssim& \kappa e^{-\frac{1}{C}\kappa\left|x-y\right|}.
\end{align*}
Then the integration estimate is the same to Case 2.1.

Case 2.3. $\left|x-y_c\right|\le \frac{1}{\kappa}$. For this case, it holds that $\left|x-y\right|\ge \frac{1}{\kappa}$ and $\left\langle \kappa\eta(x) \right\rangle\sim 1$. Similar to the above case, we also have $\left|\kappa\eta_r(x)-\kappa\eta_r(y)\right|\gtrsim \kappa |x-y|+1$. It follows that
\begin{align*}
  \left|I_{GM,1}(x,y)\right|\lesssim  e^{-\frac{1}{C}\kappa\left|x-y \right| }\frac{1}{\left|u_p(x)-\hat c\right|}\lesssim \frac{1}{\left|u_p(x)-\hat c\right|}\lesssim \frac{1}{|x-y_c+ic_0|}.
\end{align*}
Then we deduce that
\begin{align*}
  \left|\int^0_{-1}I_{GM,1}(x,y)\chi_{\text{Case 2.3}}dx\right|\lesssim \left|\int^0_{-1}\frac{1}{|x-y_c+ic_0|}dx\right|\lesssim|\ln c_0|.
\end{align*}

Case 3. $|y-y_c|\le \frac{2}{\kappa}$.

Case 3.1. $\left|x-y_c\right|\ge \frac{4}{\kappa}$. For this case, it holds that  $\left|x-y\right|\ge \frac{2}{\kappa}$, $\left\langle \kappa\eta(y) \right\rangle\sim 1$, and $|\kappa\eta(x)|\gtrsim1$. Therefore, similar to Case 2.2, we have
\begin{align*}
  \left|\int^0_{-1}I_{GM,1}(x,y)\chi_{\text{Case 3.1}}dx\right|\lesssim \left|\int^0_{-1}\kappa e^{-\frac{1}{C}\kappa\left|x-y\right|}dx\right|\lesssim1.
\end{align*}

Case 3.2. $\left|x-y_c\right|\le \frac{4}{\kappa}$. In this case, we have $\left\langle \kappa\eta(y) \right\rangle\sim\left\langle \kappa\eta(x) \right\rangle\sim 1$, and
\begin{align*}
  \left|\int^0_{-1}I_{GM,1}(x,y)\chi_{\text{Case 3.2}}dx\right|\lesssim \left|\int^0_{-1}\frac{1}{\left(u_p(x)-\hat c\right)}dx\right|\lesssim|\ln c_0|.
\end{align*}

Combining the above estimates, we have
\begin{align*}
  \left|\left(u_p-\hat c\right)\pa_yAirySolver_m(f)(y)\right| \lesssim \left(1+ \kappa c_r\right)|\ln c_0|\left\|\left(u_p-\hat c\right)f\right\|_{L^\infty}.
\end{align*}

Next, we turn to $\left(u_p-\hat c\right)^2\pa_y^2AirySolver_m(f)(y)$. Recall that
\begin{align*}
  \pa_y^2AirySolver_m(f)(y)=&i\frac{2\pi}{\kappa\varepsilon}\int^y_{-1} \frac{A_1(y) A_2(x) f(x)}{\left(\eta_r'(y)\right)^{\frac{1}{2}}\left(\eta_r'(x)\right)^{\frac{1}{2}}}dx +i\frac{2\pi}{\kappa\varepsilon}\int^0_{y} \frac{A_1(x) A_2(y)f(x)}{\left(\eta_r'(y)\right)^{\frac{1}{2}}\left(\eta_r'(x)\right)^{\frac{1}{2}}}dx.
\end{align*}

Only $G_{A,M}(x,y)$ part remain. Similar to \eqref{eq-est-AS-1}, we have
\begin{align*}
  &\left|\left(u_p-\hat c\right)^2\pa_y^2AirySolver_m(f)(y)\right|\\ 
  \lesssim&\int^y_{-1}\left|\frac{\left|\kappa\left(u_p(y)-\hat c\right)\right|^2}{\left\langle \kappa\eta(y) \right\rangle^{\frac{1}{4}}} \frac{1}{\left\langle \kappa\eta(x) \right\rangle^{\frac{1}{4}}} e^{-\frac{1}{C}\left|\kappa\eta_r(x)-\kappa\eta_r(y)  \right|\left( |\kappa\eta(y)|^{\frac{1}{2}}+|\kappa\eta(x)|^{\frac{1}{2}} \right) }\frac{1}{\left(u_p(x)-\hat c\right)}\frac{\left(u_p(x)-\hat c\right)f(x)}{\left(\eta_r'(y)\right)^{\frac{1}{2}}\left(\eta_r'(x)\right)^{\frac{1}{2}}}\right|dx.
\end{align*}
We denote
\begin{align*}
  I_{GM,2}(x,y)=\frac{\left|\kappa\left(u_p(y)-\hat c\right)\right|^2}{\left\langle \kappa\eta(y) \right\rangle^{\frac{1}{4}}} \frac{1}{\left\langle \kappa\eta(x) \right\rangle^{\frac{1}{4}}} e^{-\frac{1}{C}\left|\kappa\eta_r(x)-\kappa\eta_r(y)  \right|\left( |\kappa\eta(y)|^{\frac{1}{2}}+|\kappa\eta(x)|^{\frac{1}{2}} \right) }\frac{1}{\left(u_p(x)-\hat c\right)}.
\end{align*}

Similar to \eqref{eq-est-AS-1}, we give the estimate of $\left|\int^0_{-1}I_{GM,2}(x,y)dx\right|$ in different cases. 
 
Case 1. $|y-y_c|\ge \frac{1}{2}$.

Case 1.1. $|x-y|\le \frac{1}{4}$. For this case, it holds that 
\begin{align*}
  \left|\int^0_{-1}I_{GM,2}(x,y)\chi_{\text{Case 1.1}}dx\right|\lesssim \int^0_{-1}\kappa^{\frac{3}{2}}e^{-\kappa^{\frac{3}{2}} \left|x-y\right|}dx\lesssim1.
\end{align*}

Case 1.2. $|x-y|\ge \frac{1}{4}$. Same to the one of \eqref{eq-est-AS-1}.

Case 2. $\frac{2}{\kappa}\le|y-y_c|\le \frac{1}{2}$. 

Case 2.1. $\left|x-y_c\right|\ge \frac{1}{\kappa}$ and $\left|x-y\right|\le \frac{1}{\kappa}$. Similar to the one of \eqref{eq-est-AS-1}, we have
\begin{align*}
  \left|\int^0_{-1}I_{GM,2}(x,y)\chi_{\text{Case 2.1}}dx\right|\lesssim \left|\int^0_{-1}\kappa |\kappa\eta(y)|^{\frac{1}{2}}e^{-\kappa\left|x-y \right||\kappa\eta(y)|^{\frac{1}{2}}}dx\right|\lesssim1.
\end{align*}

Case 2.2. $\left|x-y_c\right|\ge \frac{1}{\kappa}$ and $\left|x-y\right|>\frac{1}{\kappa}$. Similar to the one of \eqref{eq-est-AS-1}, we have
\begin{align*}
  \left|I_{GM,2}(x,y)\right|\lesssim \frac{\left|\kappa\left(u_p(y)-\hat c\right)\right|^2}{\left\langle \kappa\eta(y) \right\rangle^{\frac{1}{4}}} \frac{1}{\left\langle \kappa\eta(x) \right\rangle^{\frac{1}{4}}} e^{- \left( |\kappa\eta(y)|^{\frac{1}{2}}+|\kappa\eta(x)|^{\frac{1}{2}} \right) }e^{-\frac{1}{C}\kappa |x-y|}\frac{\kappa}{\kappa\left(u_p(x)-\hat c\right)}\lesssim \kappa e^{-\frac{1}{C}\kappa\left|x-y\right|}.
\end{align*}
The integration estimate follows directly. 

Case 2.3. $\left|x-y_c\right|\le \frac{1}{\kappa}$. Same to the one of \eqref{eq-est-AS-1}.

Case 3. $|y-y_c|\le \frac{2}{\kappa}$. Same to the one of \eqref{eq-est-AS-1}.

Combining the above estimates, we get \eqref{eq-est-AS-2}.

Next, we turn to $\left(u_p-\hat c\right)^2\pa_y^3AirySolver_m(f)(y)$. Recall that
\begin{align*}
   \pa_y^3AirySolver_m(f)(y)=&i\frac{2\pi}{\kappa\varepsilon}\int^y_{-1} \frac{A_1'(y) A_2(x) f(x)}{\left(\eta_r'(y)\right)^{\frac{1}{2}}\left(\eta_r'(x)\right)^{\frac{1}{2}}}dx +i\frac{2\pi}{\kappa\varepsilon}\int^0_{y} \frac{A_1(x) A_2'(y)f(x)}{\left(\eta_r'(y)\right)^{\frac{1}{2}}\left(\eta_r'(x)\right)^{\frac{1}{2}}}dx+\text{err}.
\end{align*}
Here the err terms consist of the terms when the $y$-derivative applied on $\frac{1}{\left(\eta_r'(y)\right)^{\frac{1}{2}}}$. And for these therms, the estimates is the same to \eqref{eq-est-AS-2}. So we only focus on the main part. Similar to \eqref{eq-est-AS-1}, we have
\begin{align*}
  &\left|\left(u_p-\hat c\right)^2\pa_y^3AirySolver_m(f)(y)\right|\\ 
  \lesssim&\int^y_{-1} \left|\kappa \frac{\left\langle \kappa\eta(y) \right\rangle^{\frac{9}{4}}}{\left\langle \kappa\eta(x) \right\rangle^{\frac{1}{4}}} e^{-\frac{1}{C}\left|\kappa\eta_r(x)-\kappa\eta_r(y)  \right|\left( |\kappa\eta(y)|^{\frac{1}{2}}+|\kappa\eta(x)|^{\frac{1}{2}} \right) }\frac{1}{\left(u_p(x)-\hat c\right)}\frac{\left(u_p(x)-\hat c\right)f(x)}{\left(\eta_r'(y)\right)^{\frac{1}{2}}\left(\eta_r'(x)\right)^{\frac{1}{2}}}\right|dx.
\end{align*}
We denote
\begin{align*}
  I_{GM,3}(x,y) = \kappa\frac{\left\langle \kappa\eta(y) \right\rangle^{\frac{9}{4}}}{\left\langle \kappa\eta(x) \right\rangle^{\frac{1}{4}}} e^{-\frac{1}{C}\left|\kappa\eta_r(x)-\kappa\eta_r(y)  \right|\left( |\kappa\eta(y)|^{\frac{1}{2}}+|\kappa\eta(x)|^{\frac{1}{2}} \right) }\frac{1}{\left(u_p(x)-\hat c\right)}.
\end{align*}
Similar to \eqref{eq-est-AS-1}, we give the estimate of $\left|\int^0_{-1}I_{GM,3}(x,y)dx\right|$ in different cases. 

Case 1. $|y-y_c|\ge \frac{1}{2}$.

Case 1.1. $|x-y|\le \frac{1}{4}$. For this case, it holds that 
\begin{align*}
  \left|\int^0_{-1}I_{GM,3}(x,y)\chi_{\text{Case 1.1}}dx\right|\lesssim \int^0_{-1}\kappa^{3}e^{-\frac{1}{C}\kappa^{\frac{3}{2}} \left|x-y\right|}dx\lesssim \varepsilon^{-\frac{1}{2}}.
\end{align*}

Case 1.2. $|x-y|\ge \frac{1}{4}$. Same to the one of \eqref{eq-est-AS-1}.

Case 2. $\frac{2}{\kappa}\le|y-y_c|\le \frac{1}{2}$. 

Case 2.1. $\left|x-y_c\right|\ge \frac{1}{\kappa}$ and $\left|x-y\right|\le \frac{1}{\kappa}$. Similar to the one of \eqref{eq-est-AS-1}, we have
\begin{align*}
  \left|\int^0_{-1}I_{GM,3}(x,y)\chi_{\text{Case 2.1}}dx\right|\lesssim \left|\int^0_{-1}\kappa^2 |\kappa\eta(y)|e^{-\frac{1}{C}\kappa\left|x-y \right||\kappa\eta(y)|^{\frac{1}{2}}}dx\right|\lesssim\varepsilon^{-\frac{1}{2}}.
\end{align*}

Case 2.2. $\left|x-y_c\right|\ge \frac{1}{\kappa}$ and $\left|x-y\right|>\frac{1}{\kappa}$. Similar to the one of \eqref{eq-est-AS-1}, and for this case, we have
\begin{align*}
  \left|\int^0_{-1}I_{GM,3}(x,y)\chi_{\text{Case 2.1}}dx\right|\lesssim \left|\int^0_{-1}\kappa^2 e^{-\frac{1}{C}\kappa\left|x-y \right| }dx\right|\lesssim \kappa.
\end{align*}

Case 2.3. $\left|x-y_c\right|\le \frac{1}{\kappa}$. Similar to the one of \eqref{eq-est-AS-1},
\begin{align*}
  \left|\int^0_{-1}I_{GM,3}(x,y)\chi_{\text{Case 2.3}}dx\right|\lesssim \left|\int^0_{-1}\frac{\kappa}{|x-y_c+ic_0|}dx\right|\lesssim \kappa|\ln c_0|.
\end{align*}

Case 3. $|y-y_c|\le \frac{2}{\kappa}$. 

Case 3.1. $\left|x-y_c\right|\ge \frac{4}{\kappa}$. Similar to the one of \eqref{eq-est-AS-1},
\begin{align*}
  \left|\int^0_{-1}I_{GM,3}(x,y)\chi_{\text{Case 3.1}}dx\right|\lesssim \left|\int^0_{-1}\kappa^2 e^{-\frac{1}{C}\kappa\left|x-y\right|}dx\right|\lesssim\kappa.
\end{align*}

Case 3.2. $\left|x-y_c\right|\le \frac{4}{\kappa}$. In this case, we have $\left\langle \kappa\eta(y) \right\rangle\sim\left\langle \kappa\eta(x) \right\rangle\sim 1$, and
\begin{align*}
  \left|\int^0_{-1}I_{GM,3}(x,y)\chi_{\text{Case 3.2}}dx\right|\lesssim \left|\int^0_{-1}\frac{\kappa}{\left(u_p(x)-\hat c\right)}dx\right|\lesssim\kappa|\ln c_0|.
\end{align*}

Last, we give the estimate to $\left(u_p-\hat c\right)\pa_y^4AirySolver_m(f)(y)$.

Recall that $AirySolver_m(f)$ satisfies \eqref{eq-phiapp4}. It is clear that 
\begin{align*}
  &\varepsilon \left|\left(u_p-\hat c\right)\pa_y^4AirySolver_m(f)(y)\right| \\ 
  \lesssim&  \left\|\left(u_p-\hat c\right)^2\pa_y^2AirySolver_m(f)(y)\right\|_{L^\infty}+  \left\|\left(u_p-\hat c\right)f\right\|_{L^\infty}+\varepsilon\left\|\left(u_p-\hat c\right)\pa_y^2AirySolver_m(f)(y)\right\|_{L^\infty}\\ 
  & +\hat c_i \left\|\left(u_p-\hat c\right)^2\pa_y^2AirySolver_m(f)(y)\right\|_{L^\infty}+\varepsilon\alpha^2\left\|\left(u_p-\hat c\right)\pa_y^2AirySolver_m(f)(y)\right\|_{L^\infty}.
\end{align*}
Remark that $\hat c_i\gtrsim \varepsilon^2$. By \eqref{eq-est-AS-2}, we deduce \eqref{eq-est-AS-4}.

This complete the proof of this lemma.

\end{proof}
In Lemma \ref{lem-est-AS}, the source term is allowed to have a critical-layer singularity of order $\left(u_p-\hat c\right)^{-1}$, and the estimates are stated in terms of the weighted norm $\left\|\left(u_p-\hat c\right)f\right\|_{L^\infty}$. If the source term is regular, then one obtains the following improved estimates.
\begin{lemma}\label{lem-est-AS-ns}
  For all $y\in[-1,0]$, it holds that
\begin{align}\label{eq-est-AS-0-ns}
  \left|AirySolver_m(f)(y)\right| \lesssim |\ln c_0| \left\| f\right\|_{L^\infty}.
\end{align}
\begin{align}\label{eq-est-AS-1-ns}
  \left|\left(u_p-\hat c\right)\pa_yAirySolver_m(f)(y)\right| \lesssim |\ln c_0| \left\| f\right\|_{L^\infty}.
\end{align} 
\end{lemma}
\begin{proof}
  From the proof of \eqref{eq-est-AS-0}, we can see that the only troublesome term is $\frac{1}{\kappa\varepsilon}a_1(x)(y-y_c)$. By using Lemma \ref{lem-est-a1-a2}, we have
  \begin{align*}
    \frac{1}{\kappa\varepsilon}\int^0_{y}\left| a_1(x)(y-y_c) f(x)\right| dx   \lesssim \int^0_{y}\left| \frac{\kappa(y-y_c)}{\left\langle \kappa\eta(x) \right\rangle} \right| dx \left\| f\right\|_{L^\infty} \lesssim |\ln c_0|\left\| f\right\|_{L^\infty}.
  \end{align*}
  Then we get \eqref{eq-est-AS-0-ns}. For the same reason, we also have \eqref{eq-est-AS-1-ns}.
\end{proof}

Next, we consider source terms of derivative type. More precisely, we estimate the modified Airy solver applied to $f''$:
\begin{align*}
  AirySolver_m(f'')(y)=\int^0_{-1} G_A(x,y)f''(x)dx.
\end{align*}

For later use, we introduce the following weighted norm:
\begin{align*}
  \mathcal W_0(y)=\min \left\{\varepsilon^{\frac{1}{3}}\left|u_p-\hat c\right|^3,\varepsilon \left|u_p-\hat c\right|\right\},
\end{align*}
and
\begin{align*}
  \left\|g\right\|_{\mathcal X}=\left\|\mathcal W_0(y)\pa_y^4g\right\|_{L^\infty}+\varepsilon^{\frac{1}{2}}\left\|\left(u_p-\hat c\right)^2\pa_y^3g\right\|_{L^\infty}+\left\|\left(u_p-\hat c\right)^2\pa_y^2g\right\|_{L^\infty}+\left\|\left(u_p-\hat c\right)\pa_yg\right\|_{L^\infty}+\left\|g\right\|_{L^\infty}.
\end{align*}
We have the following estimates.
\begin{lemma}\label{lem-est-ASS}
  For all $y\in[-1,0]$, the modified Airy solver satisfies
\begin{align}\label{eq-est-ASS-0}
  \left|AirySolver_m(f'')(y)\right| \lesssim \varepsilon^{-\frac{2}{3}}\left(1+ \kappa c_r\right)|\ln c_0| \left(\left\|\left(u_p-\hat c\right)f\right\|_{L^\infty}+\left\|\left(u_p-\hat c\right)^2f'\right\|_{L^\infty}\right),
\end{align}
\begin{align}\label{eq-est-ASS-1}
  \left|\left(u_p-\hat c\right)\pa_yAirySolver_m(f'')(y)\right| \lesssim \varepsilon^{-\frac{2}{3}}\left(1+ \kappa c_r\right)|\ln c_0|\left(\left\|\left(u_p-\hat c\right)f\right\|_{L^\infty}+\left\|\left(u_p-\hat c\right)^2f'\right\|_{L^\infty}\right),
\end{align}
\begin{align}\label{eq-est-ASS-2}
  \left|\left(u_p-\hat c\right)^2\pa_y^2AirySolver_m(f'')(y)\right| \lesssim \varepsilon^{-\frac{2}{3}}|\ln c_0|\left(\left\|\left(u_p-\hat c\right)f\right\|_{L^\infty}+\left\|\left(u_p-\hat c\right)^2f'\right\|_{L^\infty}\right),
\end{align}
\begin{align}\label{eq-est-ASS-3}
   \left|\left(u_p-\hat c\right)^2\pa_y^3AirySolver_m(f'')(y)\right| \lesssim  \varepsilon^{-1}|\ln c_0|\left(\left\|\left(u_p-\hat c\right)f\right\|_{L^\infty}+\left\|\left(u_p-\hat c\right)^2f'\right\|_{L^\infty}\right).
\end{align}
We also have
\begin{equation}\label{eq-est-ASS-4}
  \begin{aligned}    
       &\left|\mathcal W_0(y)\pa_y^4AirySolver_m(f'')(y)\right| \\
   \lesssim&  \varepsilon^{-\frac{2}{3}}|\ln c_0|\left(\left\|\left(u_p-\hat c\right)f\right\|_{L^\infty}+\left\|\left(u_p-\hat c\right)^2f'\right\|_{L^\infty}+\left\|\left(u_p-\hat c\right)^3f''\right\|_{L^\infty}\right).
  \end{aligned}
\end{equation}
Consequently,
 \begin{equation}\label{eq-est-ASS-5}
  \begin{aligned}    
      &\left\|AirySolver_m(f'')(y)\right\|_{\mathcal X}\\
      \lesssim& \varepsilon^{-\frac{2}{3}}\left(1+ \kappa c_r\right)|\ln c_0| \left(\left\|\left(u_p-\hat c\right)f\right\|_{L^\infty}+\left\|\left(u_p-\hat c\right)^2f'\right\|_{L^\infty}+\left\|\left(u_p-\hat c\right)^3f''\right\|_{L^\infty}\right).
  \end{aligned}
\end{equation}
\end{lemma}
\begin{proof}
  From the definition of the approximate Green function \eqref{eq-Green-Airy}, one can see that $\pa_xG_A(x,y)$ and $\pa_x^2G_A(x,y)$ are continuous. Then by integration by parts, we have
\begin{align*}
  AirySolver_m(f'')(y)=&\int^0_{-1}\pa_x^2G_A(x,y)f(x)dx-\pa_xG_A(0,y)f(0)+\pa_xG_A(-1,y)f(-1)\\
  &+ G_A(0,y)f'(0)- G_A(-1,y)f'(-1).
\end{align*}

First consider the boundary terms:

\textbf{Term $G_A(0,y)f'(0)$.} Remark that $y_c\sim -1$, we have $\eta(0)\sim\left|u_p(0)-\hat c\right|\sim 1$. Then by Lemma \ref{lem:pri-Airy-expan}-\ref{lem-est-a1-a2}, we have
\begin{align*}
  \left|G_A(0,y)\right|\lesssim  \kappa^2 \left|A_1(0)A_2(2,y)\right|+\kappa^2 \left|a_1(0)(y-y_c)\right|\lesssim \kappa^{-\frac{1}{4}}\left\langle \kappa\eta(y) \right\rangle^{-\frac{5}{4}}+|y-y_c|.
\end{align*}
It follows that
\begin{align*}
  \left|G_A(0,y)f'(0)\right|\lesssim \left\|\left(u_p-\hat c\right)^2f'\right\|_{L^\infty}.
\end{align*}

\textbf{Term $G_A(-1,y)f'(-1)$.} Remak that $\eta(-1)\sim\left|u_p(-1)-\hat c\right|\sim \left|1+y_c\right|\sim c_r$.  Similar to the previous case, we have
\begin{align*}
  \left|G_A(-1,y)\right|\lesssim& \kappa^2\left(\left|A_2(-1)A_1(2,y)\right|+\left|a_2(-1)\right|+\left|a_1(-1)(-1-y_c)\right|\right)\\
  \lesssim&\left\langle \kappa c_r \right\rangle^{-\frac{1}{4}}\left\langle \kappa\eta(y) \right\rangle^{-\frac{5}{4}}+\left\langle \kappa c_r \right\rangle^{-3}+\frac{\kappa \left|1+y_c\right|}{\left\langle \kappa c_r \right\rangle}\lesssim 1.
\end{align*}
Therefore, we deduce that
\begin{align*}
  \left|G_A(-1,y)f'(-1)\right|\lesssim \frac{1}{c_r^2} \left|u_p(-1)-\hat c\right|^2 \left|f'(-1)\right|\lesssim \frac{1}{c_r^2} \left\|\left(u_p-\hat c\right)^2f'\right\|_{L^\infty}.
\end{align*}

\textbf{Term $\pa_xG_A(0,y)f(0)$.} We have  
\begin{align*}
  \left|\pa_xG_A(0,y)\right|\lesssim& \kappa^2 \left|A_1'(0) A_2(2,y)\right|+\kappa^2\left|a_1'(0)(y-y_c)\right|\\
  \lesssim& \kappa^{1+\frac{1}{4}} \left\langle \kappa\eta(y) \right\rangle^{-\frac{5}{4}} +\frac{\kappa^2 \left|y-y_c\right|}{\left\langle \kappa \eta(0) \right\rangle^\frac{1}{2}}\lesssim \kappa^{\frac{3}{2}}.
\end{align*}
Therefore, we deduce that
\begin{align*}
  \left|\pa_xG_A(0,y)f(0)\right|\lesssim \kappa^{\frac{3}{2}} \left|(u_p(0)-\hat c)f(0)\right|\lesssim \kappa^{\frac{3}{2}}\left\|\left(u_p-\hat c\right)f\right\|_{L^\infty}.
\end{align*}

\textbf{Term $\pa_xG_A(-1,y)f(-1)$.}  We have
\begin{align*}
  \left|\pa_xG_A(-1,y)\right|\lesssim& \kappa^2\left(\left|A_2'(-1) A_1(2,y)\right|+\left|a_2'(-1)\right|+\left|a_1'(-1)(-1-y_c)\right|\right)\\
  \lesssim& \kappa \left\langle \kappa \eta(-1) \right\rangle^{\frac{1}{4}}\left\langle \kappa\eta(y) \right\rangle^{-\frac{5}{4}}+\kappa\left\langle \kappa \eta(-1) \right\rangle^{-1}+\frac{\kappa^2 \left|1+y_c\right|}{\left\langle \kappa \eta(-1) \right\rangle^\frac{1}{2}}\lesssim \kappa \left\langle \kappa c_r \right\rangle^{\frac{1}{2}}.
\end{align*}
Therefore, we deduce that
\begin{align*}
  \left|\pa_xG_A(-1,y)f(-1)\right|\lesssim \kappa \left\langle \kappa c_r \right\rangle^{\frac{1}{2}}\frac{1}{c_r} \left|u_p(-1)-\hat c\right| \left|f(-1)\right|\lesssim \frac{\kappa^2 }{\left\langle \kappa c_r \right\rangle^{\frac{1}{2}}}\left\|\left(u_p-\hat c\right)f\right\|_{L^\infty}.
\end{align*}

For the integration part, we have
\begin{align*}
  \int^0_{-1}\pa_x^2G_A(x,y)f(x)dx=&i\frac{2\pi}{\kappa\varepsilon}\int^y_{-1} \left(A_1(2,y) A_2''(x)+a_2''(x)+a_1''(x)(x-y_c)+2a_1'(x)\right) \frac{f(x)}{\left(\eta_r'(x)\right)^{\frac{1}{2}}}dx\\ 
  &+i\frac{2\pi}{\kappa\varepsilon}\int^0_{y}\left( A_1''(x) A_2(2,y)+a_1''(x)(y-y_c)\right)\frac{f(x)}{\left(\eta_r'(x)\right)^{\frac{1}{2}}}dx+\text{err}.
\end{align*}
Here the error terms consist of the terms when the $x$-derivative applied on $\frac{1}{\left(\eta_r'(x)\right)^{\frac{1}{2}}}$. And for these therms, the estimates is easier. 

We first focus on the $E(x,y)$ related terms. From Lemma \ref{lem-est-a1-a2}, we have
\begin{align*}
  &\frac{1}{\kappa\varepsilon} \int^y_{-1}\left(\left|a_2''(x)\right|+\left|a_1'(x)\right|\right) \left|f(x)\right|dx \\
  \lesssim&\frac{1}{\kappa\varepsilon} \int^y_{-1}\frac{1}{\left|u_p(x)-\hat c\right|}dx \left\|\left(u_p-\hat c\right)f\right\|_{L^\infty}\lesssim \varepsilon^{-\frac{2}{3}}|\ln c_0|\left\|\left(u_p-\hat c\right)f\right\|_{L^\infty}.
\end{align*}
\begin{align*}
  &\frac{1}{\kappa\varepsilon} \int^y_{-1}\left|a_1''(x)\left(x-y_c\right)f(x)\right|dx\\
  \lesssim&\frac{1}{\varepsilon} \int^y_{-1} \left(\left\langle \kappa\eta(x) \right\rangle^{-3}+e^{-\frac{1}{C}\kappa |x|}+e^{-\frac{1}{C}\kappa |x+1|}\right) \left|\frac{\left(x-y_c\right)}{\left(u_p(x)-\hat c\right)}\left(u_p(x)-\hat c\right)f(x)\right|dx
  \lesssim \varepsilon^{-\frac{2}{3}} \left\|\left(u_p-\hat c\right)f\right\|_{L^\infty}.
\end{align*}
For the term related to $a_1''(x)\left(y-y_c\right)$. If $y\ge y_c$, then we have $\left|\frac{\left(y-y_c\right)}{\left(u_p(x)-\hat c\right)}\right|\lesssim 1$, and the estimate is the same to the $a_1''(x)\left(x-y_c\right)$ case. If $y< y_c$, then we have $\left|\frac{\left(y-y_c\right)}{\left(u_p(x)-\hat c\right)}\right|\lesssim \left|\frac{c_r}{\left(u_p(x)-\hat c\right)}\right|$, and
\begin{align*}
  &\frac{1}{\kappa\varepsilon}\int^0_{y}\left|a_1''(x)\left(y-y_c\right)f(x)\right|dx\\
  \lesssim&\frac{1}{\varepsilon} \int^y_{-1}   \left|\frac{c_r}{\left(u_p(x)-\hat c\right)}\right|dx\left\|\left(u_p-\hat c\right)f\right\|_{L^\infty}  \lesssim \varepsilon^{-1}c_r |\ln c_0|\left\|\left(u_p-\hat c\right)f\right\|_{L^\infty}.
\end{align*}

Next, we turn to the $G_{A,M}(x,y)$ part. Similar to Lemma \ref{lem-est-AS}, we have
\begin{align*}
  &\left|\frac{2\pi}{\kappa\varepsilon}\int^y_{-1} A_1(2,y) A_2''(x) \frac{f(x)}{\left(\eta_r'(x)\right)^{\frac{1}{2}}}dx+\frac{2\pi}{\kappa\varepsilon}\int^0_{y} A_1''(x) A_2(2,y)\frac{f(x)}{\left(\eta_r'(x)\right)^{\frac{1}{2}}}dx\right|\\
\lesssim&\int^0_{-1}\left|\frac{1}{\varepsilon \kappa}\frac{\left\langle \kappa\eta(x) \right\rangle^{\frac{3}{4}}}{\left\langle \kappa\eta(y) \right\rangle^{\frac{5}{4}}}e^{-\frac{1}{C}\left|\kappa\eta_r(x)-\kappa\eta_r(y)  \right|\left( |\kappa\eta(y)|^{\frac{1}{2}}+|\kappa\eta(x)|^{\frac{1}{2}} \right) }\frac{1}{\left(u_p(x)-\hat c\right)}\frac{\left(u_p(x)-\hat c\right)f(x)}{\left(\eta_r'(x)\right)^{\frac{1}{2}}}\right|dx.
\end{align*}
Similar to Lemma \ref{lem-est-AS}, we denote
\begin{align*}
  I_{GMS,0}(x,y) = \frac{1}{\varepsilon \kappa}\frac{\left\langle \kappa\eta(x) \right\rangle^{\frac{3}{4}}}{\left\langle \kappa\eta(y) \right\rangle^{\frac{5}{4}}}e^{-\frac{1}{C}\left|\kappa\eta_r(x)-\kappa\eta_r(y)  \right|\left( |\kappa\eta(y)|^{\frac{1}{2}}+|\kappa\eta(x)|^{\frac{1}{2}} \right) }\frac{1}{\left(u_p(x)-\hat c\right)},
\end{align*}
and divided the estimate to different cases.

Case 1. $|y-y_c|\ge \frac{1}{2}$.

Case 1.1. $|x-y|\le \frac{1}{4}$. Similar to the estimate of \eqref{eq-est-AS-1}, we have
\begin{align*}
  &\left|\int^0_{-1}I_{GMS,0}(x,y)\chi_{\text{Case 1.1}}dx\right|\lesssim \int^0_{-1}\kappa^{\frac{3}{2}}e^{-\frac{1}{C} \kappa^{\frac{3}{2}} \left|x-y\right|}dx\lesssim 1.
\end{align*}

Case 1.2. $|x-y|\ge \frac{1}{4}$. The same to \eqref{eq-est-AS-1}.
 
Case 2. $\frac{2}{\kappa}\le|y-y_c|\le \frac{1}{2}$. 

Case 2.1. $\left|x-y_c\right|\ge \frac{1}{\kappa}$ and $\left|x-y\right|\le \frac{1}{\kappa}$. Similar to the estimate of \eqref{eq-est-AS-1}, we have
\begin{align*}
  &\left|\int^0_{-1}I_{GMS,0}(x,y)\chi_{\text{Case 2.1}}dx\right|\lesssim \int^0_{-1}\kappa^{3}e^{-\frac{1}{C} \kappa\left|x-y\right|}dx\lesssim \varepsilon^{-\frac{2}{3}}.
\end{align*}

Case 2.2. $\left|x-y_c\right|\ge \frac{1}{\kappa}$ and $\left|x-y\right|>\frac{1}{\kappa}$. For this case, similar to \eqref{eq-est-AS-1}, it holds that
\begin{align*}
  \left|I_{GMS,0}(x,y)\right| \lesssim& \varepsilon^{-1} e^{-\frac{1}{C}\kappa\left|x-y\right|}.
\end{align*}
Then the integration estimate is the same to Case 2.1.

Case 2.3. $\left|x-y_c\right|\le \frac{1}{\kappa}$. For this case, similar to \eqref{eq-est-AS-1}, it holds that 
\begin{align*}
  \left|I_{GMS,0}(x,y)\right|\lesssim \frac{1}{\varepsilon \kappa} e^{-\frac{1}{C}\kappa\left|x-y \right| }\frac{1}{\left|u_p(x)-\hat c\right|}.
\end{align*}
Then we deduce that
\begin{align*}
  \left|\int^0_{-1}I_{GMS,0}(x,y)\chi_{\text{Case 2.3}}dx\right| \lesssim \varepsilon^{-\frac{2}{3}}|\ln c_0|.
\end{align*}

Case 3. $|y-y_c|\le \frac{2}{\kappa}$.

Case 3.1. $\left|x-y_c\right|\ge \frac{4}{\kappa}$. For this case, similar to \eqref{eq-est-AS-1}, it holds that
\begin{align*}
  \left|\int^0_{-1}I_{GMS,0}(x,y)\chi_{\text{Case 3.1}}dx\right|\lesssim \left|\int^0_{-1}\frac{1}{\varepsilon }\frac{1}{\left\langle \kappa\eta(x) \right\rangle^{\frac{1}{4}}}  e^{-\frac{1}{C}\kappa\left|x-y\right|}dx\right|\lesssim \varepsilon^{-\frac{2}{3}}.
\end{align*}

Case 3.2. $\left|x-y_c\right|\le \frac{4}{\kappa}$. In this case, we have $\left\langle \kappa\eta(y) \right\rangle\sim\left\langle \kappa\eta(x) \right\rangle\sim 1$, and
\begin{align*}
  \left|\int^0_{-1}I_{GMS,0}(x,y)\chi_{\text{Case 3.2}}dx\right|\lesssim \varepsilon^{-\frac{2}{3}}|\ln c_0|.
\end{align*}

Combining the above estimates, we get \eqref{eq-est-ASS-0}. 

Next, we turn to $\left(u_p-\hat c\right)\pa_yAirySolver_m(f'')(y)$. 
\begin{align*}
  \pa_yAirySolver_m(f'')(y)=&\int^0_{-1}\pa_x^2\pa_yG_A(x,y)f(x)dx-\pa_x\pa_y G_A(0,y)f(0)+\pa_x\pa_y G_A(-1,y)f(-1)\\ 
  &+\pa_y G_A(0,y)f'(0)-\pa_y G_A(-1,y)f'(-1).
\end{align*}

First consider the boundary terms. Similar to \eqref{eq-est-ASS-0}, it holds that:

\textbf{Term $\pa_y G_A(0,y)f'(0)$.} We have
\begin{align*}
  \left|\pa_yG_A(0,y)\right|\lesssim  \kappa^2 \left|A_1(0)A_2(1,y)\right|+\kappa^2 \left|a_1(0)\right|\lesssim \kappa^{\frac{3}{4}}\left\langle \kappa\eta(y) \right\rangle^{-\frac{3}{4}}+1.
\end{align*}
It follows that
\begin{align*}
  \left|\left(u_p(y)-\hat c\right)\pa_y G_A(0,y)f'(0)\right|\lesssim \left\|\left(u_p-\hat c\right)^2f'\right\|_{L^\infty}.
\end{align*}

\textbf{Term $\pa_y G_A(-1,y)f'(-1)$.}  We have
\begin{align*}
  &\left|\left(u_p(y)-\hat c\right)\pa_y G_A(-1,y)\right|\lesssim \kappa^2 \left|\left(u_p(y)-\hat c\right)A_2(-1)A_1(1,y)\right| \\
  \lesssim& \left\langle \kappa c_r \right\rangle^{-\frac{1}{4}}\left\langle \kappa\eta(y) \right\rangle^{\frac{1}{4}}e^{-\frac{1}{C}\left|\kappa\eta_r(-1)-\kappa\eta_r(y)  \right|\left( |\kappa\eta(y)|^{\frac{1}{2}}+|\kappa\eta(-1)|^{\frac{1}{2}} \right) }.
\end{align*}
If $|y+1|\le \frac{1}{\kappa}$, then $\left\langle \kappa c_r \right\rangle^{-\frac{1}{4}}\left\langle \kappa\eta(y) \right\rangle^{\frac{1}{4}}\lesssim 1$. If $|y+1|\ge \frac{1}{\kappa}$, then we can use the exponential decay to absorb $\left\langle \kappa\eta(y) \right\rangle^{\frac{1}{4}}$. Therefore, we deduce that
\begin{align*}
  \left|\left(u_p(y)-\hat c\right)\pa_y G_A(-1,y)f'(-1)\right|\lesssim \frac{1}{ c_r^2} \left|u_p(-1)-\hat c\right|^2 \left|f'(-1)\right|\lesssim \frac{1}{c_r^{2}} \left\|\left(u_p-\hat c\right)^2f'\right\|_{L^\infty}.
\end{align*}

\textbf{Term $\pa_x\pa_y G_A(0,y)f(0)$.} We have 
\begin{align*}
  \left|\left(u_p(y)-\hat c\right)\pa_x\pa_yG_A(0,y)\right|\lesssim& \kappa^2 \left|\left(u_p(y)-\hat c\right)A_1'(0) A_2(1,y)\right|+\kappa^2\left|\left(u_p(y)-\hat c\right)a_1'(0)\right|\\
  \lesssim& \kappa   \left\langle \kappa \eta(0) \right\rangle^{\frac{1}{4}}\left\langle \kappa\eta(y) \right\rangle^{\frac{1}{4}} +\frac{\kappa^2 \left|y-y_c\right|}{\left\langle \kappa \eta(0) \right\rangle^\frac{1}{2}}\lesssim \kappa^{\frac{3}{2}}.
\end{align*}
Therefore, we deduce that
\begin{align*}
  \left|\left(u_p(y)-\hat c\right)\pa_x\pa_yG_A(0,y)f(0)\right|\lesssim  \kappa^{\frac{3}{2}}\left\|\left(u_p-\hat c\right)f\right\|_{L^\infty}.
\end{align*}

\textbf{Term $\pa_x\pa_y G_A(-1,y)f(-1)$.}  We have
\begin{align*}
  &\left|\left(u_p(y)-\hat c\right)\pa_x\pa_yG_A(-1,y)\right|\lesssim \kappa^2 \left|\left(u_p(y)-\hat c\right)A_2'(-1) A_1(1,y)\right| \\
  \lesssim& \kappa \left\langle \kappa \eta(-1) \right\rangle^{\frac{1}{4}}\left\langle \kappa\eta(y) \right\rangle^{\frac{1}{4}}e^{-\frac{1}{C}\left|\kappa\eta_r(-1)-\kappa\eta_r(y)  \right|\left( |\kappa\eta(y)|^{\frac{1}{2}}+|\kappa\eta(-1)|^{\frac{1}{2}} \right) } \lesssim \kappa\left\langle \kappa c_r \right\rangle^{\frac{1}{2}}.
\end{align*}
Therefore, we deduce that
\begin{align*}
  \left|\left(u_p(y)-\hat c\right)\pa_x\pa_yG_A(-1,y)f(-1)\right|\lesssim \frac{\kappa \left\langle \kappa c_r \right\rangle^{\frac{1}{2}}}{c_r} \left|u_p(-1)-\hat c\right| \left|f(-1)\right|\lesssim \frac{\kappa^2 }{\left\langle \kappa c_r \right\rangle^{\frac{1}{2}}}\left\|\left(u_p-\hat c\right)f\right\|_{L^\infty}.
\end{align*}

For the integration part, we have
\begin{align*}
  \int^0_{-1}\pa_x^2\pa_yG_A(x,y)f(x)dx=&i\frac{2\pi}{\kappa\varepsilon}\int^y_{-1} A_1(1,y) A_2''(x)  \frac{f(x)}{\left(\eta_r'(x)\right)^{\frac{1}{2}}}dx\\ 
  &+i\frac{2\pi}{\kappa\varepsilon}\int^0_{y}\left( A_1''(x) A_2(1,y)+a_1''(x) \right)\frac{f(x)}{\left(\eta_r'(x)\right)^{\frac{1}{2}}}dx+\text{err}.
\end{align*}
Here the err terms consist of the terms when the $x$-derivative applied on $\frac{1}{\left(\eta_r'(x)\right)^{\frac{1}{2}}}$. And for these therms, the estimates is easier. 

For the $E(x,y)$ related term, only $a_1''(x)$ rest. For the same estimate as the boundary term estimates for \eqref{eq-est-ASS-0}, we have 
\begin{align*}
  &\frac{1}{\kappa\varepsilon}\int^0_{y}\left|\left(u_p(y)-\hat c\right)a_1''(x) f(x)\right|dx  \lesssim \varepsilon^{-1}c_r |\ln c_0|\left\|\left(u_p-\hat c\right)f\right\|_{L^\infty}.
\end{align*}

Next, we turn to the $G_{A,M}(x,y)$ part. Similar to Lemma \ref{lem-est-AS}, we have
\begin{align*}
  &\left|\left(u_p(y)-\hat c\right)\right|\left|\frac{2\pi}{\kappa\varepsilon}\int^y_{-1} A_1(1,y) A_2''(x) \frac{f(x)}{\left(\eta_r'(x)\right)^{\frac{1}{2}}}dx+\frac{2\pi}{\kappa\varepsilon}\int^0_{y} A_1''(x) A_2(1,y)\frac{f(x)}{\left(\eta_r'(x)\right)^{\frac{1}{2}}}dx\right|\\
\lesssim&\int^0_{-1}\left|\frac{1}{\varepsilon \kappa} \left\langle \kappa\eta(y) \right\rangle^{\frac{1}{4}}\left\langle \kappa\eta(x) \right\rangle^{\frac{3}{4}} e^{-\frac{1}{C}\left|\kappa\eta_r(x)-\kappa\eta_r(y)  \right|\left( |\kappa\eta(y)|^{\frac{1}{2}}+|\kappa\eta(x)|^{\frac{1}{2}} \right) }\frac{1}{\left(u_p(x)-\hat c\right)}\frac{\left(u_p(x)-\hat c\right)f(x)}{\left(\eta_r'(x)\right)^{\frac{1}{2}}}\right|dx.
\end{align*}
To get \eqref{eq-est-ASS-1}, we denote
\begin{align*}
  I_{GMS,1}(x,y) = \frac{1}{\varepsilon \kappa}\left\langle \kappa\eta(y) \right\rangle^{\frac{1}{4}}\left\langle \kappa\eta(x) \right\rangle^{\frac{3}{4}}e^{-\frac{1}{C}\left|\kappa\eta_r(x)-\kappa\eta_r(y)  \right|\left( |\kappa\eta(y)|^{\frac{1}{2}}+|\kappa\eta(x)|^{\frac{1}{2}} \right) }\frac{1}{\left(u_p(x)-\hat c\right)},
\end{align*}
and divided the estimate to different cases. Here the estimates are very similar to \eqref{eq-est-ASS-0}, we omit the details.

Next, we turn to $\left(u_p-\hat c\right)^2\pa_y^2AirySolver_m(f'')(y)$. It holds that
\begin{align*}
  \pa_y^2AirySolver_m(f'')(y)=&-\int^0_{-1}\pa_x\pa_y^2G_A(x,y)f'(x)dx+\pa_y^2 G_A(0,y)f'(0)-\pa_y^2 G_A(-1,y)f'(-1)\\ 
  =&-\int_{[-1,0]\setminus [y_c-\frac{4}{\kappa},y_c+\frac{4}{\kappa}]}\pa_x\pa_y^2G_A(x,y)f'(x)dx+\int_{[y_c-\frac{4}{\kappa},y_c+\frac{4}{\kappa}]}\pa_x^2\pa_y^2G_A(x,y)f(x)dx\\
  &-\pa_x\pa_y^2G_A(y_c+\frac{4}{\kappa},y)f(y_c+\frac{4}{\kappa})+\pa_x\pa_y^2G_A(y_c-\frac{4}{\kappa},y)f(y_c-\frac{4}{\kappa})\\
  &+\pa_y^2 G_A(0,y)f'(0)-\pa_y^2 G_A(-1,y)f'(-1).
\end{align*}

We have
\begin{equation*}
  \pa_y^2G_A(x,y)=i \frac{2\pi}{\left(\eta_r'(y)\right)^{\frac{1}{2}}\left(\eta_r'(x)\right)^{\frac{1}{2}}\kappa\varepsilon}\left\{
    \begin{array}{ll}
       A_1(y)  A_2(x),&x<y;\\ 
       A_1(x) A_2(y),&x>y,
    \end{array}
  \right.
\end{equation*}
and
\begin{equation*}
  \pa_x\pa_y^2G_A(x,y)=i \frac{2\pi}{\left(\eta_r'(y)\right)^{\frac{1}{2}}\left(\eta_r'(x)\right)^{\frac{1}{2}}\kappa\varepsilon}\left\{
    \begin{array}{ll}
       A_1(y)  A_2'(x),&x<y;\\ 
       A_1'(x) A_2(y),&x>y,
    \end{array}
  \right.+\text{err}.
\end{equation*}
Here we remark that $\pa_x\pa_y^2G_A(x,y)$ is not continuous, and
\begin{equation}\label{eq-dx2dy2-green-Airy}
  \pa_x^2\pa_y^2G_A(x,y)=-i\frac{\delta_x(y)}{ \varepsilon}+i \frac{2\pi}{\left(\eta_r'(y)\right)^{\frac{1}{2}}\left(\eta_r'(x)\right)^{\frac{1}{2}}\kappa\varepsilon}\left\{
    \begin{array}{ll}
       A_1(y)  A_2''(x),&x<y;\\ 
       A_1''(x) A_2(y),&x>y,
    \end{array}
  \right.+\text{err}.
\end{equation}
Here the error terms consist of the terms when the $x$-derivative applied on $\frac{1}{\left(\eta_r'(x)\right)^{\frac{1}{2}}}$. And for these therms, the estimates is easier. 

First consider the boundary terms:

\textbf{Term $\pa_x\pa_y^2G_A(y_c+\frac{4}{\kappa},y)f(y_c+\frac{4}{\kappa})$.} We have
\begin{align*}
  &\left|\left(u_p(y)-\hat c\right)^2\pa_x\pa_y^2G_A(y_c+\frac{4}{\kappa},y)\right|\lesssim\kappa \left\langle \kappa\eta(y) \right\rangle^{\frac{7}{4}}  e^{-\frac{1}{C}\kappa\left|y_c+\frac{4}{\kappa}-y \right|\left( |\kappa\eta(y)|^{\frac{1}{2}}+1 \right) }.
\end{align*}
If $|y-y_c|\le \frac{8}{\kappa}$, then we have $\left\langle \kappa\eta(y) \right\rangle\sim 1$. If $|y-y_c|> \frac{8}{\kappa}$, then the exponential decay $e^{-\frac{1}{C}\kappa\left|y_c+\frac{4}{\kappa}-y \right|\left( |\kappa\eta(y)|^{\frac{1}{2}}+1 \right) }$ will absorb the polynomial growth. It follows that
\begin{align*}
  &\left|\left(u_p(y)-\hat c\right)^2\pa_x\pa_y^2G_A(y_c+\frac{4}{\kappa},y)f(y_c+\frac{4}{\kappa})\right|\lesssim \kappa^2 \left\|\left(u_p-\hat c\right)f\right\|_{L^\infty}.
\end{align*}

\textbf{Term $\pa_x\pa_y^2G_A(y_c-\frac{4}{\kappa},y)f(y_c-\frac{4}{\kappa})$.} The estimates it the same to $\pa_x\pa_y^2G_A(y_c+\frac{4}{\kappa},y)f(y_c+\frac{4}{\kappa})$.
 
\textbf{Term $\pa_y^2G_A(0,y)f'(0)$.} We have
\begin{align*}
  &\left|\left(u_p(y)-\hat c\right)^2 \pa_y^2G_A(0,y)\right|\lesssim\kappa^{-\frac{1}{4}} \left\langle \kappa\eta(y) \right\rangle^{\frac{7}{4}}  e^{-\frac{1}{C}\left|\kappa\eta_r(0)-\kappa\eta_r(y)  \right|\left( |\kappa\eta(y)|^{\frac{1}{2}}+|\kappa\eta(0)|^{\frac{1}{2}} \right) }.
\end{align*}
Therefore, we deduce that
\begin{align*}
  \left|\left(u_p(y)-\hat c\right)^2 \pa_y^2G_A(0,y)f'(0)\right|\lesssim \kappa^{\frac{3}{2}} \left\|\left(u_p-\hat c\right)^2f'\right\|_{L^\infty}.
\end{align*}

\textbf{Term $\pa_y^2G_A(-1,y)f'(-1)$.}  We have
\begin{align*}
  &\left|\left(u_p(y)-\hat c\right)^2 \pa_y^2G_A(-1,y)\right|\lesssim\left\langle \kappa c_r \right\rangle^{-\frac{1}{4}} \left\langle \kappa\eta(y) \right\rangle^{\frac{7}{4}}  e^{-\frac{1}{C}\left|\kappa\eta_r(-1)-\kappa\eta_r(y)  \right|\left( |\kappa\eta(y)|^{\frac{1}{2}}+|\kappa\eta(-1)|^{\frac{1}{2}} \right) }.
\end{align*}
Therefore, we deduce that
\begin{align*}
  \left|\left(u_p(y)-\hat c\right)^2 \pa_y^2G_A(-1,y)f'(-1)\right|\lesssim  \left\langle \kappa c_r \right\rangle^{\frac{3}{2}}\frac{1}{c_r^2} \left|u_p(-1)-\hat c\right|^2 \left|f'(-1)\right|\lesssim \frac{\kappa^2 }{\left\langle \kappa c_r \right\rangle^{\frac{1}{2}}}\left\|\left(u_p-\hat c\right)^2f'\right\|_{L^\infty}.
\end{align*}

Then we study $\left(u_p(y)-\hat c\right)^2\int_{[-1,0]\setminus [y_c-\frac{4}{\kappa},y_c+\frac{4}{\kappa}]}\pa_x\pa_y^2G_A(x,y)f'(x)dx$. It is to estimate 
\begin{align*}
  \int_{[-1,0]\setminus [y_c-\frac{4}{\kappa},y_c+\frac{4}{\kappa}]}\bigg|\kappa \left\langle \kappa\eta(y) \right\rangle^{\frac{7}{4}}\left\langle \kappa\eta(x) \right\rangle^{\frac{1}{4}}  &e^{-\frac{1}{C}\left|\kappa\eta_r(x)-\kappa\eta_r(y)  \right|\left( |\kappa\eta(y)|^{\frac{1}{2}}+|\kappa\eta(x)|^{\frac{1}{2}} \right) }\\
  &\qquad \qquad\cdot\frac{1}{\left(u_p(x)-\hat c\right)^2}\frac{\left(u_p(x)-\hat c\right)^2f'(x)}{\eta_r'(x)}\bigg|dx.
\end{align*}
We denote
\begin{align*}
  I_{GMS,2,1}(x,y) = \kappa \left\langle \kappa\eta(y) \right\rangle^{\frac{7}{4}}\left\langle \kappa\eta(x) \right\rangle^{\frac{1}{4}}  &e^{-\frac{1}{C}\left|\kappa\eta_r(x)-\kappa\eta_r(y)  \right|\left( |\kappa\eta(y)|^{\frac{1}{2}}+|\kappa\eta(x)|^{\frac{1}{2}} \right) }\frac{1}{\left(u_p(x)-\hat c\right)^2},
\end{align*}
and divided the estimate to different cases. Note that there is no need to establish estimate for Case 2.3 $\left(\frac{2}{\kappa}\le|y-y_c|\le \frac{1}{2}\text{ and }\left|x-y_c\right|\le \frac{1}{\kappa}\right)$ and Case 3.2 $\left(|y-y_c|\le \frac{2}{\kappa}\text{ and }\left|x-y_c\right|\le \frac{4}{\kappa}\right)$, and all other estimates are similar to the one of \eqref{eq-est-ASS-0}. We omit the details.

Next, we study $\left(u_p(y)-\hat c\right)^2\int_{[y_c-\frac{4}{\kappa},y_c+\frac{4}{\kappa}]}\pa_x^2\pa_y^2G_A(x,y)f(x)dx$. From \eqref{eq-dx2dy2-green-Airy}, we have
\begin{align*}
  &\left|\left(u_p(y)-\hat c\right)^2\int_{[y_c-\frac{4}{\kappa},y_c+\frac{4}{\kappa}]}\pa_x^2\pa_y^2G_A(x,y)f(x)dx\right|\\
  \lesssim& \left|\frac{1}{\varepsilon} \left(u_p(y)-\hat c\right)^2f(y)\chi_{y\in [y_c-\frac{4}{\kappa},y_c+\frac{4}{\kappa}]}\right| + \int_{[y_c-\frac{4}{\kappa},y_c+\frac{4}{\kappa}]}\left|I_{GMS,2,2}(x,y)\right| dx  \left\|\left(u_p-\hat c\right)f\right\|_{L^\infty},
\end{align*}
where
\begin{align*}
  I_{GMS,2,2}(x,y) = \kappa^2 \left\langle \kappa\eta(y) \right\rangle^{\frac{7}{4}}\left\langle \kappa\eta(x) \right\rangle^{\frac{3}{4}}  &e^{-\frac{1}{C}\left|\kappa\eta_r(x)-\kappa\eta_r(y)  \right|\left( |\kappa\eta(y)|^{\frac{1}{2}}+|\kappa\eta(x)|^{\frac{1}{2}} \right) }\frac{1}{\left(u_p(x)-\hat c\right)}.
\end{align*}

It is clear that
\begin{align*}
  \left|\frac{1}{\varepsilon} \left(u_p(y)-\hat c\right)^2f(y)\chi_{y\in [y_c-\frac{4}{\kappa},y_c+\frac{4}{\kappa}]}\right|\lesssim \kappa^2\left\|\left(u_p-\hat c\right)f\right\|_{L^\infty}.
\end{align*}

For the term related to $I_{GMS,2,2}(x,y)$. If $\left|y-y_c\right|\le \frac{6}{\kappa}$, then $\left\langle \kappa\eta(y) \right\rangle\sim\left\langle \kappa\eta(x) \right\rangle\sim1$; if $\left|y-y_c\right|> \frac{6}{\kappa}$, then the exponential decay will absorb $\left\langle \kappa\eta(y) \right\rangle^{\frac{7}{4}}$. Then we have
\begin{align*}
  \int_{[y_c-\frac{4}{\kappa},y_c+\frac{4}{\kappa}]}\left|I_{GMS,2,2}(x,y)\right| dx\lesssim\int_{[y_c-\frac{4}{\kappa},y_c+\frac{4}{\kappa}]}\frac{\kappa^2}{\left|u_p(x)-\hat c\right|} dx \lesssim \kappa^2|\ln c_0|.
\end{align*}
 
Combining the above estimates, we derive \eqref{eq-est-ASS-2}.

Next, we study $\left(u_p-\hat c\right)^2\pa_y^3AirySolver_m(f'')(y)$.
\begin{align*}
  \pa_y^3AirySolver_m(f'')(y)=&-\int^0_{-1}\pa_x\pa_y^3G_A(x,y)f'(x)dx+\pa_y^3 G_A(0,y)f'(0)-\pa_y^3 G_A(-1,y)f'(-1).
\end{align*}
Here we need to remark that
\begin{align*}
  &\pa_y^3G_A(x,y)\\
  =&   i \frac{2\pi}{\left(\eta_r'(y)\right)^{\frac{1}{2}}\left(\eta_r'(x)\right)^{\frac{1}{2}}\kappa\varepsilon}\left\{
    \begin{array}{ll}
       A_1'(y)  A_2(x),&x<y;\\ 
       A_1(x) A_2'(y),&x>y,
    \end{array}
  \right. +i \frac{2\pi \pa_y \left(\left(\eta_r'(y)\right)^{-\frac{1}{2}}\right)}{\left(\eta_r'(x)\right)^{\frac{1}{2}}\kappa\varepsilon}\left\{
    \begin{array}{ll}
       A_1(y)  A_2(x),&x<y;\\ 
       A_1(x) A_2(y),&x>y,
    \end{array}
  \right.
\end{align*}
which is not a continuous function. Therefore, 
\begin{align*}
  &\pa_x\pa_y^3G_A(x,y)\\
  =&i\frac{\delta_x(y)}{ \varepsilon}+ i \frac{2\pi}{\left(\eta_r'(y)\right)^{\frac{1}{2}}\left(\eta_r'(x)\right)^{\frac{1}{2}}\kappa\varepsilon}\left\{
    \begin{array}{ll}
       A_1'(y) A_2'(x),&x<y;\\ 
       A_1'(x) A_2'(y),&x>y,
    \end{array}
  \right.\\
  &+i \frac{2\pi \pa_x \left(\left(\eta_r'(x)\right)^{-\frac{1}{2}}\right)}{ \left(\eta_r'(x)\right)^{\frac{1}{2}}\kappa\varepsilon}\left\{
    \begin{array}{ll}
       A_1'(y)  A_2(x),&x<y;\\ 
       A_1(x) A_2'(y),&x>y,
    \end{array}
  \right. +i \frac{2\pi \pa_y \left(\left(\eta_r'(y)\right)^{-\frac{1}{2}}\right)}{\left(\eta_r'(x)\right)^{\frac{1}{2}}\kappa\varepsilon}\left\{
    \begin{array}{ll}
       A_1(y)  A_2'(x),&x<y;\\ 
       A_1'(x) A_2(y),&x>y,
    \end{array}
  \right.\\
  &+i \frac{2\pi \pa_y \left(\left(\eta_r'(y)\right)^{-\frac{1}{2}}\right) \pa_x \left(\left(\eta_r'(x)\right)^{-\frac{1}{2}}\right) }{ \kappa\varepsilon}\left\{
    \begin{array}{ll}
       A_1(y)  A_2(x),&x<y;\\ 
       A_1(x) A_2(y),&x>y.
    \end{array}
  \right.
\end{align*}
We note that, $\partial_x\partial_y^3 G_A(x,y)$ can be decomposed into the Dirac measure term $i\,\frac{\delta_x(y)}{\varepsilon}$ and a remainder which is a continuous function. This means that we could do extra integration by parts in $x$. We have
\begin{align*}
  &\pa_y^3AirySolver_m(f'')(y)\\
  =&-\int_{[-1,0]\setminus [y_c-\frac{4}{\kappa},y_c+\frac{4}{\kappa}]}\pa_x\pa_y^3G_A(x,y)f'(x)dx+\int_{[y_c-\frac{4}{\kappa},y_c+\frac{4}{\kappa}]}\pa_x^2\pa_y^3G_A(x,y)f(x)dx\\
  &-\pa_x\pa_y^3G_A(y_c+\frac{4}{\kappa},y)f(y_c+\frac{4}{\kappa})+\pa_x\pa_y^3G_A(y_c-\frac{4}{\kappa},y)f(y_c-\frac{4}{\kappa})\\
  &+\pa_y^3 G_A(0,y)f'(0)-\pa_y^3 G_A(-1,y)f'(-1)+i\frac{f'(y)}{ \varepsilon} .
\end{align*}
Here, by abuse of notation, $\pa_x\pa_y^3G_A(x,y)$ and $\pa_x^2\pa_y^3G_A(x,y)$ denote only the regular (function) part, excluding the measure-valued contribution.

For the Dirac measure part, it is clear that
\begin{align*}
  \left|\left(u_p(y)-\hat c\right)^2 \frac{f'(y) }{ \varepsilon} \right|\lesssim \varepsilon^{-1} \left\|\left(u_p-\hat c\right)^2f'\right\|_{L^\infty}.
\end{align*}

For the boundary terms, by using the same technique to \eqref{eq-est-ASS-2}, we can show that all the terms have upper bound
\begin{align*}
&\left|\left(u_p(y)-\hat c\right)^2\right| \Bigg(\left|\pa_x\pa_y^3G_A(y_c+\frac{4}{\kappa},y)f(y_c+\frac{4}{\kappa})\right|+\left|\pa_x\pa_y^3G_A(y_c-\frac{4}{\kappa},y)f(y_c-\frac{4}{\kappa})\right| \\
&\qquad \qquad\qquad\qquad\qquad\qquad\qquad\qquad+\left|\pa_y^3 G_A(0,y)f'(0)\right|+\left|\pa_y^3 G_A(-1,y)f'(-1)\right|\Bigg)\\
  \lesssim&\varepsilon^{-1} \left(\left\|\left(u_p-\hat c\right)f\right\|_{L^\infty}+\left\|\left(u_p-\hat c\right)^2f'\right\|_{L^\infty}\right).
\end{align*}

For the integration terms, by using the same technique to \eqref{eq-est-ASS-2}, we have
\begin{align*}
  \left|\left(u_p(y)-\hat c\right)^2\int_{[-1,0]\setminus [y_c-\frac{4}{\kappa},y_c+\frac{4}{\kappa}]}\pa_x\pa_y^3G_A(x,y)f'(x)dx\right|\lesssim \varepsilon^{-1} \left\|\left(u_p-\hat c\right)^2f'\right\|_{L^\infty},\\
  \left|\left(u_p(y)-\hat c\right)^2\int_{[y_c-\frac{4}{\kappa},y_c+\frac{4}{\kappa}]}\pa_x^2\pa_y^3G_A(x,y)f(x)dx\right|\lesssim \varepsilon^{-1} |\ln c_0|\left\|\left(u_p-\hat c\right)f\right\|_{L^\infty}.
\end{align*}

Combining the above estimates, we derive \eqref{eq-est-ASS-3}.

Let $\omega_{app}=\pa_y^2AirySolver_m(f'')(y)$, we have
\begin{align*}
  &i\varepsilon\pa_y^2\omega_{app}(y)+\left(u_p(y)-c-2i\varepsilon\alpha^2\right)\omega_{app}(y)\\ 
  =&\pa_y^2f(y)+ i\varepsilon \left(\pa_y^2\left(\eta_r'(y)\right)^{-\frac{1}{2}}\right)\left(\eta_r'(y)\right)^{\frac{1}{2}}\omega_{app}(y) +i \left(\left(\eta_r'(y)\right)^2-1\right) c_i \omega_{app}(y)-2i\varepsilon\alpha^2\omega_{app}(y).
\end{align*}
Then, recalling the definition of $\mathcal W_0(y)$, we deduce that
\begin{align*}
  \left\|\mathcal W_0(y)\pa_y^2\omega_{app}\right\|_{L^\infty}\lesssim&\left\|\left(u_p-\hat c\right)^2\omega_{app}\right\|_{L^\infty}+ \varepsilon\left\|\left(u_p-\hat c\right)\omega_{app}\right\|_{L^\infty}+\varepsilon^{-\frac{2}{3}} \left\|\left(u_p-\hat c\right)^3f''\right\|_{L^\infty}. 
\end{align*} 
Recall that $c_0\sim \varepsilon^{\frac{1}{2}}$. Then \eqref{eq-est-ASS-4} follows from \eqref{eq-est-ASS-2}, and \eqref{eq-est-ASS-5} follows from \eqref{eq-est-ASS-0}-\eqref{eq-est-ASS-4}.

\end{proof}

For regular differentiated source terms, we also have the following improved estimates.
\begin{lemma}\label{lem-est-ASS-ns}
  For all $y\in[-1,0]$, it holds that
\begin{align}\label{eq-est-ASS-0-ns}
  \left|AirySolver_m(f'')(y)\right| \lesssim \varepsilon^{-\frac{2}{3}}  \left(\left\| f\right\|_{L^\infty}+\left\|\left(u_p-\hat c\right)f'\right\|_{L^\infty}\right),
\end{align}
\begin{align}\label{eq-est-ASS-1-ns}
  \left|\left(u_p-\hat c\right)\pa_yAirySolver_m(f'')(y)\right| \lesssim \varepsilon^{-\frac{2}{3}}  \left(\left\| f\right\|_{L^\infty}+\left\|\left(u_p-\hat c\right)f'\right\|_{L^\infty}\right).
\end{align}
\end{lemma}
\begin{proof}
  Here, we only consider the case $c_r\gg \varepsilon^{\frac{1}{3}}$, as for the case $c_r\sim \varepsilon^{\frac{1}{3}}$ the estimate \eqref{eq-est-ASS-0} and \eqref{eq-est-ASS-1} are good enough. From the proof of \eqref{eq-est-ASS-0}, we can see that the only troublesome term is $\frac{1}{\kappa\varepsilon}a_1''(x)(y-y_c)$. Here the term with $a_{1,M}''$, $a_{1,SecM}''$, and $a_{1,B}''$ can be treated in the same way as Lemma \ref{lem-est-AS-ns}. For $a_{1,B}''$, by using Lemma \ref{lem-est-a1-a2}, we have
  \begin{align*}
    &\frac{1}{\kappa\varepsilon}\int^0_{y}\left| a_{1,B}''(x)(y-y_c) \right| dx \\
      \lesssim& \frac{|y-y_c|}{ \varepsilon} \int^0_{y}\frac{e^{-\frac{1}{C}|\kappa\eta_r(x)-\kappa\eta_r(-1)|(|\kappa\eta(x)|^\frac12+|\kappa\eta(-1)|^\frac12)}}{ \left\langle \kappa \eta(-1) \right\rangle^{\frac{3}{4}}\left\langle \kappa \eta(x) \right\rangle^{-\frac{3}{4}}}+\frac{e^{-\frac{1}{C}|\kappa\eta_r(x)-\kappa\eta_r(0)|(|\kappa\eta(x)|^\frac12+|\kappa\eta(0)|^\frac12)}}{ \left\langle \kappa \eta(0) \right\rangle^{\frac{3}{4}}\left\langle \kappa \eta(x) \right\rangle^{-\frac{3}{4}}}dx.
  \end{align*}
  For the first term, we split the estimate to two region. For $-1\le x\le y_c-\frac{y_c+1}{2}$, we have
  \begin{align*}
    \frac{\left\langle \kappa \eta(x) \right\rangle^{\frac{3}{4}}}{\left\langle \kappa \eta(-1) \right\rangle^{\frac{3}{4}}}\lesssim 1,
  \end{align*}
  and for $y_c-\frac{y_c+1}{2}\le x\le 0$, $e^{-\frac{1}{C}|\kappa\eta_r(x)-\kappa\eta_r(-1)|(|\kappa\eta(x)|^\frac12+|\kappa\eta(-1)|^\frac12)}$ is extremely small. For the second term, similar properties hold for $x$ close and away from $0$. 

  The estimates \eqref{eq-est-ASS-0-ns} and \eqref{eq-est-ASS-1-ns} follow directly.
\end{proof}

\subsection{Construction of the exact AirySolver}
In this subsection, we use an iteration argument to construct a solution operator for the non-homogeneous Airy equation \eqref{eq-Airy}.
 
\begin{lemma}\label{lem-real-Airy-Solver}
  For $f(y)$ such that $\left\|\left(u_p-\hat c\right)f\right\|_{L^\infty}< +\infty$, there exists a solution operator $AirySolver$ such that $\phi=AirySolver(f)$ solves \eqref{eq-Airy}. Moreover, $AirySolver(f)(y)$ satisfies the same estimates in Lemma \ref{lem-est-AS}.

  Moreover, it holds that
  \begin{align}\label{eq-boundary-Airy-Solver}
    \pa_yAirySolver(f)(0)=0.
  \end{align}
\end{lemma}
\begin{proof}
  In section \ref{sec-modified-airysolver}, we define a modified Airy-solver which satisfies \eqref{eq-phiapp4}. Recalling \eqref{op-airy-4}, we have
\begin{align*}
  Airy_4 \left(AirySolver_m(f)\right)=&f(y)+i\varepsilon \left(\pa_y^2\left(\eta_r'(y)\right)^{-\frac{1}{2}}\right)\left(\eta_r'(y)\right)^{\frac{1}{2}}\pa_y^2AirySolver_m(f)(y)\\
  &+i \left(\left(\eta_r'(y)\right)^2-1\right)c_i\pa_y^2AirySolver_m(f)(y)-2i\varepsilon\alpha^2\pa_y^2AirySolver_m(f)(y).
\end{align*}
We define
\begin{align*}
  S_{A,0}(f)=&AirySolver_m(f),\\
  S_{A,j}(f)=&-AirySolver_m\left(i\varepsilon \left(\pa_y^2\left(\eta_r'(y)\right)^{-\frac{1}{2}}\right)\left(\eta_r'(y)\right)^{\frac{1}{2}}\pa_y^2S_{A,j-1}(f)(y)\right.\\
  &\qquad\qquad\qquad\left.+i \left(\left(\eta_r'(y)\right)^2-1\right)c_i\pa_y^2S_{A,j-1}(f)(y)-2i\varepsilon\alpha^2\pa_y^2S_{A,j-1}(f)(y)\right).
\end{align*}

Recalling that $c_0=\varepsilon^{\frac{1}{2}}$, by using Lemma \ref{lem-est-AS}, we have
\begin{align*}
  &\left\| \left(u_p-\hat c\right)\varepsilon\left(\pa_y^2\left(\eta_r'(y)\right)^{-\frac{1}{2}}\right)\left(\eta_r'(y)\right)^{\frac{1}{2}}\pa_y^2S_{A,0}(f)(y)\right\|_{L^\infty}+\left\|\left(u_p-\hat c\right)\varepsilon\alpha^2\pa_y^2S_{A,0}(f)(y)\right\|_{L^\infty}\\
  \lesssim& \varepsilon^{\frac{1}{2}}\left\|\left(u_p-\hat c\right)^2 \pa_y^2S_{A,0}(f)(y)\right\|_{L^\infty}\lesssim \varepsilon^{\frac{1}{2}}|\ln c_0|\left\|\left(u_p-\hat c\right)f\right\|_{L^\infty},
\end{align*}
and
\begin{align*}
  &\left\| \left(u_p-\hat c\right)\left(\left(\eta_r'(y)\right)^2-1\right)c_i\pa_y^2S_{A,0}(f)(y)\right\|_{L^\infty} \lesssim \left|c_i\right|\left\|\left(u_p-\hat c\right)^2 \pa_y^2S_{A,0}(f)(y)\right\|_{L^\infty} .
\end{align*}
Then by using  Lemma \ref{lem-est-AS} again, we deduce that
\begin{align*}
  \left\|S_{A,1}(f)(y)\right\|_{X_1}\lesssim \kappa c_r(\varepsilon^{\frac{1}{2}}+|c_i|)|\ln c_0|^2\left\|\left(u_p-\hat c\right)f\right\|_{L^\infty},
\end{align*}
and
\begin{align*}
  \left\|\left(u_p-\hat c\right)^2 \pa_y^2S_{A,1}(f)(y)\right\|_{L^\infty}\lesssim (\varepsilon^{\frac{1}{2}}+|c_i|)|\ln c_0|^2\left\|\left(u_p-\hat c\right)f\right\|_{L^\infty}.
\end{align*}
Therefore, by iteration, we have
\begin{align}\label{eq-est-SA1}
  \left\|\left(u_p-\hat c\right)^2 \pa_y^2S_{A,j}(f)(y)\right\|_{L^\infty}\lesssim (\varepsilon^{\frac{1}{2}}+|c_i|)^j|\ln c_0|^{j+1}\left\|\left(u_p-\hat c\right)f\right\|_{L^\infty}.
\end{align}

We define
\begin{align*}
  \phi_{err}(y)=\sum^{\infty}_{j=1} S_{A,j}(f)(y),\quad  AirySolver(f)(y)= S_{A,0}(f)(y)+\phi_{err}(y).
\end{align*}
It follows from \eqref{eq-est-SA1} that
\begin{align}\label{eq-real-Airy-Solver-err}
  \left\|\phi_{err}(y)\right\|_{X_1}\lesssim \kappa c_r(\varepsilon^{\frac{1}{2}}+|c_i|)|\ln c_0|^2\left\|\left(u_p-\hat c\right)f\right\|_{L^\infty},
\end{align}
and then $AirySolver(f)(y)$ satisfies the same estimates to $AirySolver_m(f)(y)$ in the sense of \eqref{eq-est-AS-0} and \eqref{eq-est-AS-1}. By using the same techniques, one can easily check that $AirySolver(f)(y)$ satisfies all the same estimates as those in Lemma \ref{lem-est-AS}.

The boundary value \eqref{eq-boundary-Airy-Solver} follows from the definition of $AirySolver(f)$ and \eqref{eq-boundary-modi-Airy-Solver}.
\end{proof}

\begin{lemma}\label{lem-real-Airy-Solver-S}
  For $f(y)$ such that $\left\|\left(u_p-\hat c\right)f\right\|_{L^\infty} +\left\|\left(u_p-\hat c\right)^2f'\right\|_{L^\infty}+\left\|\left(u_p-\hat c\right)^3f''\right\|_{L^\infty}< +\infty$, there exists a solution operator $AirySolver$ such that $\phi=AirySolver(f'')$ solves
\begin{align*}
  i\varepsilon\pa_y^4\phi(y)+\left(u_p-c-2i\varepsilon\alpha^2\right)\pa_y^2\phi(y)=f''(y),\quad y\in[-1,0].
\end{align*}
and satisfies the same estimates as those in Lemma \ref{lem-est-ASS}.

  Moreover, it holds that
  \begin{align}\label{eq-boundary-Airy-Solver-S}
    \pa_yAirySolver(f'')(0)=0.
  \end{align}
\end{lemma}
\begin{proof}
  It holds that
  \begin{align*}
  &Airy_4 \left(AirySolver_m(f'')\right)\\
  =&f''(y)+i\varepsilon \left(\pa_y^2\left(\eta_r'(y)\right)^{-\frac{1}{2}}\right)\left(\eta_r'(y)\right)^{\frac{1}{2}}\pa_y^2AirySolver_m(f'')(y)\\
  &+i \left(\left(\eta_r'(y)\right)^2-1\right)c_i\pa_y^2AirySolver_m(f'')(y)-2i\varepsilon\alpha^2\pa_y^2AirySolver_m(f'')(y).
\end{align*}
Let 
\begin{align*}
  AirySolver(f'')(y)=AirySolver_m(f'')(y)+\phi_{err}(y),
\end{align*}
where $\phi_{err}(y)$ is constructed through $AirySolver$ that given in Lemma \ref{lem-real-Airy-Solver}, which satisfies
\begin{align*}
  Airy_4 \left(\phi_{err}\right)=&-i\varepsilon \left(\pa_y^2\left(\eta_r'(y)\right)^{-\frac{1}{2}}\right)\left(\eta_r'(y)\right)^{\frac{1}{2}}\pa_y^2AirySolver_m(f'')(y)\\
  &-i \left(\left(\eta_r'(y)\right)^2-1\right)c_i\pa_y^2AirySolver_m(f'')(y)+2i\varepsilon\alpha^2\pa_y^2AirySolver_m(f'')(y).
\end{align*}
From Lemma \ref{lem-est-ASS}, we have
\begin{align*}
  &\left\| \left(u_p-\hat c\right)\varepsilon\left(\pa_y^2\left(\eta_r'(y)\right)^{-\frac{1}{2}}\right)\left(\eta_r'(y)\right)^{\frac{1}{2}}\pa_y^2AirySolver_m(f'')(y)\right\|_{L^\infty}\\
  &+\left\|\left(u_p-\hat c\right)\varepsilon\alpha^2\pa_y^2AirySolver_m(f'')(y)\right\|_{L^\infty}\\
  \lesssim& \varepsilon^{\frac{1}{2}-\frac{2}{3}}|\ln c_0|\left(\left\|\left(u_p-\hat c\right)f\right\|_{L^\infty}+\left\|\left(u_p-\hat c\right)^2f'\right\|_{L^\infty}\right),
\end{align*}
and
\begin{align*}
  &\left\| \left(u_p-\hat c\right)\left(\left(\eta_r'(y)\right)^2-1\right)c_i\pa_y^2AirySolver_m(f'')(y)\right\|_{L^\infty}\\
  \lesssim& |c_i| \varepsilon^{-\frac{2}{3}}|\ln c_0|\left(\left\|\left(u_p-\hat c\right)f\right\|_{L^\infty}+\left\|\left(u_p-\hat c\right)^2f'\right\|_{L^\infty}\right).
\end{align*}

Therefore, by Lemma \ref{lem-real-Airy-Solver},  $\phi_{err}(y)$ has better estimates than $AirySolver_m(f'')$. This complete the proof of this lemma.

\end{proof}

\section{The Rayleigh--Airy iteration}
In this section, we describe the Rayleigh--Airy iteration scheme for constructing solutions to the Orr--Sommerfeld equation \eqref{eq-Orr-Sommerfeld}. Using the estimates developed in the previous sections, we first construct a solution operator for the inhomogeneous Orr--Sommerfeld equation. In the next section, we apply the same strategy to construct two slow modes.
\subsection{Iteration strategy}
Recall that $\hat c=c+ic_0$, where $c=c_r+ic_i$. Here $c_0=\varepsilon^{\frac{1}{2}}$ and $c_i\ge-\frac{1}{2}c_0$.

Recall the Orr-Sommerfeld operator
\begin{align*}
  Orr(\phi)=i\varepsilon\left(\pa_y^2-\alpha^2\right)^2\phi+\left(u_p-c\right)\left(\pa_y^2-\alpha^2\right)\phi-u_p''\phi.
\end{align*}
We use two decompositions of the Orr--Sommerfeld operator:
\begin{align}\label{eq-OS-decom}
  Orr(\phi)=Ray_{\alpha}(\phi)+Diff(\phi)=Airy_4(\phi)+Reg(\phi),
\end{align}
where
\begin{align}\label{eq-OS-decom-ray}
  Ray_{\alpha}(\phi)=\left(u_p-\hat c\right)\left(\pa_y^2-\alpha^2\right)\phi-u_p''\phi.
\end{align}
\begin{align}\label{eq-OS-decom-diff}
  Diff(\phi)=i\varepsilon\left(\pa_y^2-\alpha^2\right)^2\phi+ic_0\left(\pa_y^2-\alpha^2\right)\phi.
\end{align}
\begin{align}\label{eq-OS-decom-airy}
  Airy_4(\phi)=i\varepsilon\pa_y^4\phi+\left(u_p-c-2i\varepsilon\alpha^2\right)\pa_y^2\phi,
\end{align}
\begin{align}\label{eq-OS-decom-reg}
  Reg(\phi)= \left(-u_p''-\alpha^2\left(u_p-c \right)+i\varepsilon\alpha^4\right)\phi.
\end{align}
Here $Ray_\alpha$ contains the Rayleigh part with the regularized wave speed $\hat c$, while $Diff$ contains the viscous correction. The second decomposition separates the fourth-order Airy part from the regular lower-order terms.

To construct a solution to
\begin{align}\label{eq-inhome-OS}
  Orr(\phi)=f,
\end{align}
we first solve the Rayleigh equation and set
\begin{align*}
  \mathring \phi^{[0]}=RaySolver_{\alpha}(f).
\end{align*}
Substituting $\mathring \phi^{[0]}$ into the Orr--Sommerfeld operator gives
\begin{align*}
  Orr(\mathring \phi^{[0]})=f+Diff(\mathring \phi^{[0]}),
\end{align*}
Thus the Rayleigh approximation produces a residual term which is small but has low regularity. Then we solve the Airy equation and let
\begin{align*}
  \mathring \psi^{[1]}=AirySolver \left(-Diff(\mathring \phi^{[0]})\right).
\end{align*}
Then we have
\begin{align*}
  Orr(\mathring \phi^{[0]}+\mathring \psi^{[1]})=f+Reg(\mathring \psi^{[1]}).
\end{align*}
The AirySolver provides the required regularity gain. Therefore, $Reg(\mathring \psi^{[1]})$ has sufficient regularity and is smaller than the original source term. 

We denote
\begin{align}\label{op-Iter}
  Iter(f)=-Reg \left( AirySolver \left(Diff \left( RaySolver_{\alpha}(f) \right)\right) \right).
\end{align}
We will show that $Iter(f)$ is smaller than $f$ in a suitable norm. Hence the solution to \eqref{eq-inhome-OS} can be constructed by summing the Rayleigh--Airy iteration. This type of iteration was first introduced in \cite{GGN16duke,GGN16adv} and has proved useful in the study of problems related to T--S waves.

\subsection{Non-homogeneous solution}
In this subsection, we show the properties of the iteration operator $Iter$ given in \eqref{op-Iter}, and give the existence of solutions of \eqref{eq-inhome-OS}.
\begin{lemma}\label{lem-iter}
  For $f\in X_{2}$, it holds that 
  \begin{align*}
    \left\|Iter(f)\right\|_{\mathcal X}\lesssim \varepsilon^{\frac{1}{3}}\left(1+ \kappa c_r\right)|\ln c_0|^2\left\|f\right\|_{X_{2}}.
  \end{align*}
\end{lemma}
\begin{proof}
Let 
\begin{align*}
  \mathring \phi^{[0]}=RaySolver_{\alpha}(f).
\end{align*}
By Lemma \ref{lem-Raysolver-alpha}, we have
\begin{align*}
  \left\|\mathring \phi^{[0]}\right\|_{Y_{4}}\lesssim |\ln c_0|\left\|f\right\|_{X_{2}}.
\end{align*}

Since $\mathring \phi^{[0]}$ solves the non-homogeneous Rayleigh equation, we have
\begin{align*}
  \left(\pa_y^2-\alpha^2\right)\mathring \phi^{[0]}=\frac{f+u_p''\mathring \phi^{[0]}}{u_p-\hat c}.
\end{align*}
It follows that
\begin{align*}
  Orr(\mathring \phi^{[0]})=&f+Diff(\mathring \phi^{[0]})=f+i\varepsilon\left(\pa_y^2-\alpha^2\right)^2\mathring \phi^{[0]}+ic_0\left(\pa_y^2-\alpha^2\right)\mathring \phi^{[0]}\\ 
  =&f+i\varepsilon \pa_y^2 \left(\frac{f+u_p''\mathring \phi^{[0]}}{u_p-\hat c}\right)+\left(-i\varepsilon\alpha^2+ic_0\right)\left(\frac{f+u_p''\mathring \phi^{[0]}}{u_p-\hat c}\right),
\end{align*}
The term $Diff(\mathring\phi^{[0]})$ consists of a differentiated singular source and a lower-order singular source. We treat these two parts separately by using Lemma \ref{lem-real-Airy-Solver-S} and Lemma \ref{lem-real-Airy-Solver}, respectively. We then decompose the first Airy correction as
\begin{align*}
  \mathring \psi^{[1]}=\mathring \psi_{a}^{[1]}+\mathring \psi_{b}^{[1]},
\end{align*}
where
\begin{align*}
  \mathring \psi_{a}^{[1]}=-AirySolver \left(i\varepsilon \pa_y^2 \left(\frac{f+u_p''\mathring \phi^{[0]}}{u_p-\hat c}\right)\right),
\end{align*}
and
\begin{align*}
  \mathring \psi_{b}^{[1]}=-AirySolver \left(\left(-i\varepsilon\alpha^2+ic_0\right)\left(\frac{f+u_p''\mathring \phi^{[0]}}{u_p-\hat c}\right)\right).
\end{align*}
By Lemma \ref{lem-real-Airy-Solver} and Lemma \ref{lem-real-Airy-Solver-S}, we obtain
\begin{align*}
  \left\|\mathring \psi_{a}^{[1]}\right\|_{\mathcal X}\lesssim& \varepsilon^{\frac{1}{3}}\left(1+ \kappa c_r\right)|\ln c_0| \left(\left\|f+u_p''\mathring \phi^{[0]}\right\|_{L^\infty}+\left\|\left(u_p-\hat c\right)^2\pa_y \left(\frac{f+u_p''\mathring \phi^{[0]}}{u_p-\hat c}\right)\right\|_{L^\infty}\right) \\ 
  \lesssim& \varepsilon^{\frac{1}{3}}\left(1+ \kappa c_r\right)|\ln c_0| \left(\left\|f\right\|_{L^\infty}+\left\|\left(u_p-\hat c\right)f'\right\|_{L^\infty}+\left\|\mathring \phi^{[0]}\right\|_{L^\infty}+\left\|\left(u_p-\hat c\right)\mathring \phi^{[0]\prime}\right\|_{L^\infty}\right) \\ 
  \lesssim& \varepsilon^{\frac{1}{3}}\left(1+ \kappa c_r\right)|\ln c_0|^2\left\|f\right\|_{X_{2}}.
\end{align*}
Similarly,
\begin{align*}
  \left\|\mathring \psi_{b}^{[1]}\right\|_{\mathcal X}\lesssim& \left(\varepsilon\alpha^2+c_0\right)\left(1+ \kappa c_r\right)|\ln c_0| \left\|f+u_p''\mathring \phi^{[0]}\right\|_{L^\infty}\lesssim \left(\varepsilon\alpha^2+c_0\right)\left(1+ \kappa c_r\right)|\ln c_0|^2\left\|f\right\|_{X_{2}}.
\end{align*}

Recall that
\begin{align*}
  Iter(f)=&-Reg \left( AirySolver \left(Diff \left( RaySolver_{\alpha}(f) \right)\right) \right)\\
  =&Reg \left(\mathring \psi^{[1]}\right)=\left(-u_p''-\alpha^2\left(u-c \right)+i\varepsilon\alpha^4\right)\mathring \psi^{[1]}.
\end{align*}

Thus we have
\begin{align*}
  \left\|Iter(f)\right\|_{\mathcal X}\lesssim&  \varepsilon^{\frac{1}{3}}\left(1+ \kappa c_r\right)|\ln c_0|^2\left\|f\right\|_{X_{2}}.
\end{align*}
 
\end{proof}
 
\section{Slow modes}
In this section, we construct two solutions $\Phi_{1,\alpha}$ and  $\Phi_{2,\alpha}$ of the homogeneous Orr--Sommerfeld equation \eqref{eq-Orr-Sommerfeld}. They are obtained by applying the Rayleigh--Airy iteration scheme of the previous section, starting respectively from the Rayleigh solutions $\phi_{1,\alpha}$ and $\phi_{2,\alpha}$ given in Lemma \ref{lem-sol-Ray-hom}. Following \cite{GGN16adv}, we refer to these two solutions as the slow modes.

To analyze the relation between $c$ and $\nu$, or equivalently the zeros of the Wronskian in the dispersion relation studied in the next section, we need precise information on the boundary values of $\Phi_{1,\alpha}$ and $\Phi_{2,\alpha}$ at $y=-1$ and $y=0$. In particular, the leading contribution comes from the boundary value of $\Phi_{1,\alpha}$ at $y=-1$, together with its parameter derivatives with respect to $c_r$, $c_i$, and $\nu$.

We will show that the boundary behavior of $\Phi_{1,\alpha}$ is mainly determined by its Rayleigh component $\phi_{1,\alpha}$, while the viscous correction is of lower order. This agrees with the classical physical picture of T--S waves that the slow modes are perturbations of the Rayleigh modes.

\subsection{The slow modes  $\Phi_{1,\alpha}$ and  $\Phi_{2,\alpha}$ }
In this subsection, we give the construction of $\Phi_{1,\alpha}$ and  $\Phi_{2,\alpha}$, and provide a rough estimate for the above solutions.
\begin{lemma}\label{lem-rough-est-phi-psi}
  There exist $\Phi_{1,\alpha}$ and  $\Phi_{2,\alpha}$ which solve the homogeneous Orr--Sommerfeld equation \eqref{eq-Orr-Sommerfeld} such that
  \begin{align*}
    \left\|\Phi_{1,\alpha}-\phi_{1,\alpha}\right\|_{X_{2}}\lesssim \varepsilon^{\frac{1}{3}}\left(1+ \kappa c_r\right)|\ln c_0|^3,\quad \left\|\Phi_{2,\alpha}-\phi_{2,\alpha}\right\|_{X_{2}}\lesssim \varepsilon^{\frac{1}{3}}\left(1+ \kappa c_r\right)|\ln c_0|^3.
  \end{align*}
\end{lemma}
\begin{proof}

We first construct $\Phi_{1,\alpha}$. We remark that in the iteration strategy, the Rayleigh operator $Ray_{\alpha}$ is defined on $\hat c=c+ic_0$, therefore $\phi_{1,0}^{[0]}=u_p-\hat c$, and $\phi_{1,\alpha}=\sum^{\infty}_0 \phi_{1,0}^{[j]}(y)$ that constructed in Lemma \ref{lem-sol-Ray-hom} satisfies $Ray_{\alpha} \left(\phi_{1,\alpha}\right)=0$. 

It holds that
\begin{align*}
  Orr\left(\phi_{1,\alpha}\right)=&i\varepsilon\left(\pa_y^2-\alpha^2\right)^2\phi_{1,\alpha}+ic_0\left(\pa_y^2-\alpha^2\right)\phi_{1,\alpha}\\ 
  =&i\varepsilon \pa_y^2 \left(\frac{u_p''\phi_{1,\alpha}}{u_p-\hat c}\right)+\left(-i\varepsilon\alpha^2+ic_0\right)\left(\frac{u_p''\phi_{1,\alpha}}{u_p-\hat c}\right)
\end{align*}
Let
\begin{align*}
  \psi^{[1]}_{1,\alpha}=-AirySolver \left(Diff(\phi_{1,\alpha})\right)=\psi_{1,\alpha,a}^{[1]}+\psi_{1,\alpha,b}^{[1]},
\end{align*}
where
\begin{align*}
  \psi_{1,\alpha,a}^{[1]}=-AirySolver \left(i\varepsilon \pa_y^2 \left(\frac{u_p''\phi_{1,\alpha}}{u_p-\hat c}\right)\right),
\end{align*}
\begin{align*}
  \psi_{1,\alpha,b}^{[1]}=-AirySolver \left(\left(-i\varepsilon\alpha^2+ic_0\right)\left(\frac{u_p''\phi_{1,\alpha}}{u_p-\hat c}\right)\right),
\end{align*}
by Lemma \ref{lem-real-Airy-Solver} and Lemma \ref{lem-real-Airy-Solver-S} we have
\begin{equation}\label{est-psi1a-rough}
  \begin{aligned}    
     \left\|\psi_{1,\alpha,a}^{[1]}\right\|_{\mathcal X}\lesssim& \varepsilon^{\frac{1}{3}}\left(1+ \kappa c_r\right)|\ln c_0| \left(\left\|u_p''\phi_{1,\alpha}\right\|_{L^\infty}+\left\|\left(u_p-\hat c\right)^2\pa_y \left(\frac{u_p''\phi_{1,\alpha}}{u_p-\hat c}\right)\right\|_{L^\infty}\right) \\  
  \lesssim& \varepsilon^{\frac{1}{3}}\left(1+ \kappa c_r\right)|\ln c_0|^2, 
  \end{aligned}
\end{equation} 
and
\begin{align}\label{est-psi1b-rough}
  \left\|\psi_{1,\alpha,b}^{[1]}\right\|_{\mathcal X}\lesssim& \left(\varepsilon\alpha^2+c_0\right)\left(1+ \kappa c_r\right)|\ln c_0| \left\|u_p''\phi_{1,\alpha}\right\|_{L^\infty}\lesssim \left(\varepsilon\alpha^2+c_0\right)\left(1+ \kappa c_r\right)|\ln c_0|^2.
\end{align}

Then we have
\begin{align*}
  Orr\left(\phi_{1,\alpha}+\psi^{[1]}_{1,\alpha}\right)=Reg(\psi^{[1]}_{1,\alpha})=\left(-u_p''-\alpha^2\left(u_p-c \right)+i\varepsilon\alpha^4\right)\psi^{[1]}_{1,\alpha}.
\end{align*} 
Now we let $\phi_{1,\alpha}^{[0]}=\phi_{1,\alpha}$, and 
\begin{align*}
  \psi_{1,\alpha}^{[j]}=-AirySolver \left(Diff(\phi_{1,\alpha}^{[j-1]})\right),\quad \phi_{1,\alpha}^{[j]}=-RaySolver_{\alpha} \left(Reg(\psi^{[j]}_{1,\alpha})\right) \text{ for }j\ge 1.
\end{align*}
Then by Lemma \ref{lem-iter}, we have
\begin{align*}
  \left\|\phi_{1,\alpha}^{[j]}\right\|_{X_{2}}+\left\|\psi_{1,\alpha}^{[j]}\right\|_{X_{2}}\lesssim \varepsilon^{\frac{1}{3}j}\left(1+ \kappa c_r\right)^j|\ln c_0|^{j+2}.
\end{align*}

Then we define
\begin{align*}
  \Phi_{1,\alpha}=\phi_{1,\alpha} +\sum_{j=1}^{+\infty}\phi_{1,\alpha}^{[j]}+\psi_{1,\alpha}^{[j]}.
\end{align*}
It is clear that $\Phi_{1,\alpha}$ is well defined and $Orr \left(\Phi_{1,\alpha}\right)=0$.

By using the same technique, we can also construct $\Phi_{2,\alpha}$.
\end{proof}

\subsection{Expansion of $\Phi_{1,\alpha}$ at $y=-1$}
In this subsection, we derive refined estimates for $\Phi_{1,\alpha}$ and its parameter derivatives at the boundary $y=-1$. These boundary expansions will be used in Section \ref{sec-dispersion_relation} to derive the dispersion relation and to verify the transversal crossing condition along the neutral curve.

We show that the boundary value $\Phi_{1,\alpha}(-1)$ is mainly determined by the Rayleigh component $\phi_{1,\alpha}(-1)$. More precisely, we will prove estimates of the form
\begin{align*}
  \Phi_{1,\alpha}(-1)\approx\phi_{1,\alpha}(-1),\quad \pa_{c_r}\Phi_{1,\alpha}(-1)\approx\pa_{c_r}\phi_{1,\alpha}(-1).
\end{align*}
Near the neutral curve, one has $c_r\sim \alpha^2$ (see the analysis in Section \ref{sec-dispersion_relation}). In this regime,
\begin{align*}
  \left\|\phi_{1,\alpha}\right\|_{L^\infty}=O(1),\quad  \phi_{1,\alpha}(-1)=O(c_r),\qquad \Im \phi_{1,\alpha}(-1)=-c_i+O(c_r\alpha^2).
\end{align*}
Therefore, the correction terms $\phi_{1,\alpha}^{[j]}(-1)$ and $\psi_{1,\alpha}^{[j]}(-1)$, $j\ge1$, must be shown to be negligible compared with the leading contributions in the dispersion relation, for instance smaller than $\bigl|\Im \phi_{1,\alpha}(-1)\bigr|$. 

In addition, after taking derivatives with respect to $c_r$, $c_i$, and $\nu$, the corresponding correction terms must remain sufficiently small. These derivative estimates are needed to verify the transversal crossing condition on the neutral curve. 

The rough estimates, such as \eqref{est-psi1a-rough}, are not sufficient for this purpose. We therefore derive sharper estimates for $\psi_{1,\alpha}^{[1]}$ and $\phi_{1,\alpha}^{[1]}$, together with their parameter derivatives.
\subsubsection{Refined estimate of $\psi_{1,\alpha,a}^{[1]}$}\label{ssub:accurate_estimate_of_psi_1_alpha_a}
In this part, we derive estimates for $\psi_{1,\alpha,a}^{[1]}$, which is the singular component of $\psi_{1,\alpha}^{[1]}$.   

Recall Lemma \ref{lem-sol-Ray-hom} that $\phi_{1,\alpha}=\sum^{\infty}_{j=0} \phi_{1,0}^{[j]}(y)$ with $\phi_{1,0}^{[0]}=u_p(y)-\hat c$ and $u_p''=-2$. It holds that
\begin{align*}
  \pa_y^2 \left(\frac{u_p''\phi_{1,\alpha}}{u_p(y)-\hat c}\right)=-2\pa_y^2\left(1+\frac{\sum^{\infty}_{j=1} \phi_{1,0}^{[j]}(y)}{u_p(y)-\hat c}\right)=-2\pa_y^2\left(\frac{\sum^{\infty}_{j=1} \phi_{1,0}^{[j]}(y)}{u_p(y)-\hat c}\right).
\end{align*}
Therefore,
\begin{align*}
  \psi_{1,\alpha,a}^{[1]}=-AirySolver \left(i\varepsilon \pa_y^2 \left(\frac{u_p''\phi_{1,\alpha}}{u_p-\hat c}\right)\right)=2AirySolver \left(i\varepsilon \pa_y^2 \left(\frac{\sum^{\infty}_{j=1} \phi_{1,0}^{[j]}(y)}{u_p-\hat c}\right)\right).
\end{align*}

By Lemma \ref{lem-sol-Ray-hom}, we have
\begin{align*}
  \left\|\phi_{1,0}^{[1]}\right\|_{Y_4}\lesssim&\alpha^2\left|\ln c_0\right|.
\end{align*}
Hence one expects the estimate for $\psi_{1,\alpha,a}^{[1]}$ to improve upon \eqref{est-psi1a-rough} by an additional factor of $\alpha^2$. This improvement, however, is still not sufficient for the upper branch analysis. By using the precise expression of the non-local term $a_1$, we have the following estimate.
\begin{lemma}\label{lem-acc-psi-1-a}
  It holds that
  \begin{align}\label{est-acc-psi-1-a}
    \left\|\psi_{1,\alpha,a}^{[1]}\right\|_{L^\infty}\lesssim \varepsilon^{\frac{1}{3}}\alpha^2|\ln c_0|,
  \end{align}
  \begin{align}\label{est-acc-psi-1-a-dx}
    \left|\psi_{1,\alpha,a}^{[1]\prime}(-1)\right|\lesssim \varepsilon^{\frac{1}{3}} \frac{1}{c_r}\alpha^2|\ln c_0|.
  \end{align}
\end{lemma}
\begin{remark}
  If one directly applies \eqref{eq-est-ASS-0} in Lemma \ref{lem-est-ASS} in the regime $c_r\gg\varepsilon^{\frac{1}{3}}$, one could only obtain
  \begin{align*}
    \left\|\psi_{1,\alpha,a}^{[1]}\right\|_{L^\infty}\lesssim \varepsilon^{\frac{1}{3}}\alpha^2\left(1+ \kappa c_r\right)|\ln c_0|.
  \end{align*}
  This is weaker than \eqref{est-acc-psi-1-a}. On the upper branch, $c_r\sim \alpha^2\sim \varepsilon^{\frac{1}{5}}$, and the leading imaginary contribution to the Wronskian has size $\varepsilon^{\frac{2}{5}}$. Therefore, the rough estimate above gives an error of size
  \begin{align*}
    \varepsilon^{\frac{1}{3}}\alpha^2\left(1+ \kappa c_r\right)|\ln c_0|\sim \varepsilon^{\frac{2}{5}}|\ln c_0|,
  \end{align*}
  which is of the same order as the leading term up to a logarithmic factor and hence is not sufficient. This is why the sharper estimate \eqref{est-acc-psi-1-a} is needed.
\end{remark} 
\begin{proof}[Proof of Lemma \ref{lem-acc-psi-1-a}]
  Recalling the definition of expression of $\phi_{1,0}^{[1]}$ in \eqref{eq-phi-1-0-1}, we have
  \begin{align*}
  \frac{\phi_{1,0}^{[1]}(y)}{u_p(y)-\hat c}=&-\alpha^2\frac{\phi_{1,0} (y)}{u_p(y)-\hat c}\int^y_{-1}\phi_{2,0}(x) \left(u_p(x)-\hat c\right)dx-\alpha^2\frac{\phi_{2,0} (y)}{u_p(y)-\hat c}\int^0_{y}\phi_{1,0}(x) \left(u_p(x)-\hat c\right)dx\\
  =&-\alpha^2\int^y_{-1}\phi_{2,0}(x) \left(u_p(x)-\hat c\right)dx-\alpha^2\frac{\frac{y}{2(1-\hat c)}+\frac{u_p-\hat c}{4(1-\hat c)^{\frac{3}{2}}}\ln{\left(\frac{\sqrt{1-\hat c}+y}{\sqrt{1-\hat c}-y}\right)}}{\left(\sqrt{1-\hat c}+y\right)\left(\sqrt{1-\hat c}-y\right)}\int^0_{y}\left(u_p(x)-\hat c\right)^2dx.
\end{align*}
This term has one order singularity around $y_c$.

Let us define
\begin{align}\label{eq-yhc}
  \tilde y_{\hat c}=-\sqrt{1-\hat c}.
\end{align}
It is clear that $\left|\tilde y_{\hat c}+1\right|=O(|\hat c|)$ and $\Im(\tilde y_{\hat c})=\frac{1}{2}\hat c_i\left(1+O(|\hat c|)\right)>0$. Recalling the definition of $y_c$, we have
\begin{align*} 
  y_c-\tilde y_{\hat c}=\sqrt{1-\hat c}-\sqrt{1-c_r}=\frac{-i\hat c_i}{\sqrt{1-\hat c}+\sqrt{1-c_r}},
\end{align*}
therefore $\left|y_c-\tilde y_{\hat c}\right|\le |\hat c_i|$, and $\left|Re \left(y_c-\tilde y_{\hat c}\right)\right|\lesssim |\hat c_i|^2$.

The singularity of $\frac{\phi_{1,0}^{[1]}(y)}{u_p(y)-\hat c}$ can be represent by the following main part 
\begin{align*}
 \Upsilon_M=&\alpha^2\frac{y_c}{2(1-\hat c)\left(\sqrt{1-\hat c}+y\right)\left(\sqrt{1-\hat c}-\tilde y_{\hat c}\right)} \int^0_{y_c}\left(u_p(x)-\hat c\right)^2dx\\
 =&\alpha^2\frac{y_c}{4(1-\hat c)\tilde y_{\hat c} } \int^0_{y_c}\left(u_p(x)-\hat c\right)^2dx \frac{1}{y-\tilde y_{\hat c}}.
\end{align*}
We define
\begin{align*}
  \Upsilon_R=\frac{\phi_{1,0}^{[1]}(y)}{u_p(y)-\hat c}-\Upsilon_M.
\end{align*}
One can easily check that
\begin{align*}
  \left\|\Upsilon_R\right\|_{X_2}\lesssim \alpha^2|\ln c_0|.
\end{align*}
Then by Lemma \ref{lem-est-ASS-ns}, we have
\begin{align*}
   \left\|AirySolver_m(\varepsilon\Upsilon_R'')(y)\right\|_{L^\infty} \lesssim \varepsilon^{\frac{1}{3}}\alpha^2|\ln c_0|.
\end{align*}

Therefore, to prove \eqref{est-acc-psi-1-a}, we only need to estimate for $\varepsilon \alpha^2 AirySolver_m\left(\pa_y^2\frac{1}{ y-\tilde y_{\hat c}}\right)(y)$.

From the proof of \eqref{eq-est-ASS-0}, to estimate $AirySolver_m\left(\pa_y^2\frac{1}{ y-\tilde y_{\hat c}}\right)(y)$, the only troublesome term that generates extra $\kappa c_r$ is
\begin{align*}
  \frac{1}{\kappa\varepsilon}\int^0_{y}a_1''(x)\left(y-y_c\right)\frac{1}{x-\tilde y_{\hat c}}dx.
\end{align*}
Now we will use the precise information of the non-local term $a_1$ to improve the estimate \eqref{eq-est-ASS-0}.

Recall that we assume $c_r\gg \varepsilon^{\frac{1}{3}}$.

If $y\ge y_c$, then $\left|\frac{y-y_c}{x-\tilde y_{\hat c}}\right|\lesssim 1$, thus by \eqref{eq-yhc} and Lemma \ref{lem-est-a1-a2} we have
\begin{align*}
  &\frac{1}{\kappa\varepsilon}\left|\int^0_{y}a_1''(x)\left(y-y_c\right)\frac{1}{x-\tilde y_{\hat c}}dx\right|
  \lesssim\frac{1}{\kappa\varepsilon}\int^0_{y}\left|a_1''(x)\right| dx\lesssim  \varepsilon^{-\frac{2}{3}}.
\end{align*}
If $y_c-\frac{1}{\kappa}\le y\le y_c$, then $\left|y-y_c\right|\le \frac{1}{\kappa}$, and
\begin{align*}
  &\frac{1}{\kappa\varepsilon}\left|\int^0_{y} a_1''(x)\left(y-y_c\right)\frac{1}{x-\tilde y_{\hat c}}dx\right| \lesssim\varepsilon^{-\frac{2}{3}}|\ln c_0|.
\end{align*}

If $y\le y_c-\frac{1}{\kappa}$, then we write
\begin{align*}
  &\frac{1}{\kappa\varepsilon}\int^0_{y}a_1''(x)\frac{y-y_c}{x-\tilde y_{\hat c}}dx=\frac{1}{\kappa\varepsilon}\int^{2y_c+1}_{-1}a_1''(x)\frac{y-y_c}{x-\tilde y_{\hat c}}dx+\frac{1}{\kappa\varepsilon}\int^{0}_{2y_c+1}a_1''(x)\frac{y-y_c}{x-\tilde y_{\hat c}}dx-\frac{1}{\kappa\varepsilon}\int^y_{-1}a_1''(x)\frac{y-y_c}{x-\tilde y_{\hat c}}dx.
\end{align*}
In this case, it is clear that $\left|\frac{y-y_c}{x-\tilde y_{\hat c}}\right|\lesssim 1$ for $x\in [-1,y]\cup[2y_c+1,0]$, and
\begin{align*}
   \frac{1}{\kappa\varepsilon}\left|\int^{0}_{2y_c+1}a_1''(x)\left(y-y_c\right)\frac{1}{x-\tilde y_{\hat c}}dx\right|+\frac{1}{\kappa\varepsilon}\left|\int^y_{-1}a_1''(x)\left(y-y_c\right)\frac{1}{x-\tilde y_{\hat c}}dx\right|\lesssim\frac{1}{\kappa\varepsilon}\int^0_{y}\left|a_1''(x)\right| dx\lesssim  \varepsilon^{-\frac{2}{3}}.
\end{align*}
For the rest term, the integration on $[-1,2y_c+1]$, we will use precise expression of $a_1''$. Here it is easy to check that for $a_{1,SecM}$ and $a_{1,R}$ the estimates are good. 

Next, we give the estimate for the boundary term $a_{1,B}$. Recall that $y_c+1\sim c_r$ and $c_r\gg\varepsilon^{\frac{1}{3}}$. By Lemma \ref{lem-est-a1-a2}, for $y_c-\frac{y_c+1}{2}\le x\le 2y_c+1$, $\left|a_{1,B}\right|$ has strong exponential decay, and
\begin{align*}
  \left|a_{1,B}''(x)\frac{y-y_c}{\kappa\varepsilon}\frac{1}{x-\tilde y_{\hat c}}\right|\lesssim 1.
\end{align*}
For $-1\le x\le y_c-\frac{y_c+1}{2}$, we have $\left|\frac{y-y_c}{x-\tilde y_{\hat c}}\right|\lesssim 1$, $\frac{\left\langle \kappa \eta(x) \right\rangle^{\frac{3}{4}}}{\left\langle \kappa \eta(-1) \right\rangle^{\frac{3}{4}}}\lesssim 1$, and
\begin{align*}
  \left|a_{1,B}''(x)\frac{y-y_c}{\kappa\varepsilon}\frac{1}{x-\tilde y_{\hat c}}\right|\lesssim \frac{e^{-\frac{1}{C}|\kappa\eta_r(x)-\kappa\eta_r(0)|(|\kappa\eta(x)|^\frac12+|\kappa\eta(0)|^\frac12)}}{\varepsilon}.
\end{align*}
Therefore, it holds that
\begin{align*}
  \left|\frac{y-y_c}{\kappa\varepsilon}\int^{2y_c+1}_{-1}a_{1,B}''(x)\frac{1}{x-\tilde y_{\hat c}}dx\right|\lesssim \varepsilon^{-\frac{2}{3}} \left\langle \kappa c_r \right\rangle^{-\frac{1}{2}}.
\end{align*}

Next, we address the most subtle part: the contribution of the main term of the non-local term $a_{1}$ in the dangerous region, which reveals the special structure of the slow-mode solution. Recall that the main term of $a_1''(x)$ is $a_{1,M}''(x)$, and
\begin{align*}
  a_{1,M}''(x)=\frac{1}{2}\kappa\left(\eta_r'(x)\right)^{\frac{1}{2}} \mathrm{Hi}''(e^{-\frac{\pi}{2}i}\kappa \eta(x))+err,
\end{align*}
where $err=O(1)$ consists of the terms where derivative applied on $\eta_r'(x)$.

We write
\begin{align*}
  &\frac{y-y_c}{\kappa\varepsilon}\int^{2y_c+1}_{-1}\frac{1}{2}\kappa\left(\eta_r'(x)\right)^{\frac{1}{2}} \mathrm{Hi}''(e^{-\frac{\pi}{2}i}\kappa \eta(x))\frac{1}{x-\tilde y_{\hat c}}dx\\
  =&\frac{y-y_c}{2 \varepsilon}\int^{2y_c+1}_{-1} \left( \left(\eta_r'(x)\right)^{\frac{1}{2}} \mathrm{Hi}''(e^{-\frac{\pi}{2}i}\kappa \eta(x))-\mathrm{Hi}''\left(e^{-\frac{\pi}{2}i}\kappa (x-y_c-\frac{ic_i}{u_p'(y_c)})\right)\right)\frac{1}{x-\tilde y_{\hat c}}dx\\
  &+\frac{y-y_c}{2 \varepsilon}\int^{2y_c+1}_{-1} \mathrm{Hi}''\left(e^{-\frac{\pi}{2}i}\kappa (x-y_c-\frac{ic_i}{u_p'(y_c)})\right)\frac{1}{x-\tilde y_{\hat c}}dx\\
  =&I+II.
\end{align*}
For $I$, recalling \eqref{eq-Langertran}, Lemma \ref{lem-Langer}, and Remark \ref{rmk-Hi''}, we have $\eta=y-y_c+O(|y-y_c|^2)-\frac{ic_i}{u_p'(y_c)}$, $\eta_r'=1+O(|y-y_c|)$, and 
\begin{align*}
  &\left|\left(\eta_r'(x)\right)^{\frac{1}{2}} \mathrm{Hi}''(e^{-\frac{\pi}{2}i}\kappa \eta(x))-\mathrm{Hi}''\left(e^{-\frac{\pi}{2}i}\kappa (x-y_c-\frac{ic_i}{u_p'(y_c)})\right)\right|\\
  =&\left|\left(\left(\eta_r'(x)\right)^{\frac{1}{2}}-1\right) \mathrm{Hi}''(e^{-\frac{\pi}{2}i}\kappa \eta(x))+ \left(\mathrm{Hi}''(e^{-\frac{\pi}{2}i}\kappa \eta(x))-\mathrm{Hi}''\left(e^{-\frac{\pi}{2}i}\kappa (x-y_c-\frac{ic_i}{u_p'(y_c)})\right)\right)\right|\\
  =&\left|\left(\left(\eta_r'(x)\right)^{\frac{1}{2}}-1\right) \mathrm{Hi}''(e^{-\frac{\pi}{2}i}\kappa \eta(x))+ \int^{e^{-\frac{\pi}{2}i}\kappa \eta(x)}_{e^{-\frac{\pi}{2}i}\kappa (x-y_c-\frac{ic_i}{u_p'(y_c)})} \mathrm{Hi}'''(s)ds\right|\\
  \lesssim&|x-y_c|\left\langle \kappa\eta(x) \right\rangle^{-3}+\kappa|x-y_c|^2\left\langle \kappa\eta(x) \right\rangle^{-4}.
\end{align*}
It follows that
\begin{align*}
  \left|I\right|\lesssim& \frac{\left|y-y_c\right|}{\kappa\varepsilon}\int^{2y_c+1}_{-1} \left| \frac{x-y_c}{x-\tilde y_{\hat c}}\left\langle \kappa\eta(x) \right\rangle^{-3} \right|+\left| \frac{\kappa|x-y_c|^2}{x-\tilde y_{\hat c}}\left\langle \kappa\eta(x) \right\rangle^{-4} \right| dx\\
   \lesssim&\frac{\left|y-y_c\right|}{\kappa\varepsilon}\int^{2y_c+1}_{-1} \left\langle \kappa\eta(x) \right\rangle^{-3} dx\lesssim  \varepsilon^{-\frac{2}{3}}|y-y_c|.
\end{align*}
Next, for $II$, we use the fact that the integrand is analytic in $x$. Hence, we may deform the contour of integration from the real interval
\begin{align*}
  x\in [-1,2y_c+1]
\end{align*}
to the circular arc
\begin{align*}
  x=y_c+e^{i\theta}(y_c+1),\quad \theta\in [-\pi,0].
\end{align*}
Recall from \eqref{eq-yhc} that $\Im(\tilde y_{\hat c})>0$. It follows that the region enclosed by these two curves does not contain the singularity of $\frac{1}{x-\tilde y_{\hat c}}$.   Moreover, along the new contour, $\mathrm{Hi}''(z)$ admits the polynomial decay expansion described in \eqref{eq-expan-Hi''}, which will be crucial for the subsequent estimates.

By using the above properties, we have
\begin{align*}
  II=&\frac{y-y_c}{2 \varepsilon}\int^{2y_c+1}_{-1} \mathrm{Hi}''\left(e^{-\frac{\pi}{2}i}\kappa \left(x-y_c-\frac{ic_i}{u_p'(y_c)}\right)\right)\frac{1}{x-\tilde y_{\hat c}}dx\\
  =&\frac{y-y_c}{2 \varepsilon}\int^{0}_{-\pi} \mathrm{Hi}''\left(e^{-\frac{\pi}{2}i}\kappa \left(e^{i\theta}(y_c+1)-\frac{ic_i}{u_p'(y_c)}\right)\right)\frac{ie^{i\theta}(y_c+1)}{ y_c+e^{i\theta}(y_c+1)  -\tilde y_{\hat c}}d\theta.
\end{align*}
and
\begin{align*}
   \left|\mathrm{Hi}''\left(e^{-\frac{\pi}{2}i}\kappa \left(e^{i\theta}(y_c+1)-\frac{ic_i}{u_p'(y_c)}\right)\right)\frac{(y_c+1)}{ y_c+e^{i\theta}(y_c+1)  -\tilde y_{\hat c}}\right|\sim \left\langle \kappa c_r \right\rangle^{-3}\frac{y_c+1}{c_r}\sim\left\langle \kappa c_r \right\rangle^{-3}.
\end{align*}
It follows that
\begin{align*}
  \left|II\right|\lesssim \frac{\left|y-y_c\right|}{\varepsilon}\left\langle \kappa c_r \right\rangle^{-3}\lesssim \frac{c_r}{\varepsilon}\left\langle \kappa c_r \right\rangle^{-3}\lesssim \frac{1}{c_r^2}.
\end{align*}

Combining the above estimates, we have
\begin{align*}
   \left\|AirySolver_m(\varepsilon\Upsilon_M'')(y)\right\|_{L^\infty} \lesssim \varepsilon^{\frac{1}{3}}\alpha^2|\ln c_0|.
\end{align*}

Next, we study the derivative estimate for $y=-1$. For the same reason, here we only give the estimate for
\begin{align*}
  \varepsilon \alpha^2 \pa_yAirySolver_m\left(\pa_y^2\frac{1}{ y-\tilde y_{\hat c}}\right)(y)\Big|_{y=-1}.
\end{align*}

Recalling \eqref{eq-Airy-modi-2} the definition of $A_2(1,y)$, we have $A_2(1,-1)=0$. Therefore
\begin{align*}
  &\pa_yAirySolver_m\left(\pa_y^2\frac{1}{ y-\tilde y_{\hat c}}\right)(-1)= i \frac{2\pi}{\kappa\varepsilon} \int^0_{-1}a_1(x)\frac{1}{\left(\eta_r'(x)\right)^{\frac{1}{2}}}\pa_x^2 \frac{1}{x-\tilde y_{\hat c}}dx.
\end{align*}
For the main part, it holds that
\begin{align*}
   &\frac{1}{\kappa\varepsilon} \int^0_{-1}a_1(x)\pa_x^2 \frac{1}{x-\tilde y_{\hat c}}dx\\
   =&\frac{1}{\kappa\varepsilon} \int^0_{-1}a_1''(x) \frac{1}{x-\tilde y_{\hat c}}dx+\frac{-a_1(0) \frac{1}{(0-\tilde y_{\hat c})^2}+a_1(-1) \frac{1}{(-1-\tilde y_{\hat c})^2}-a_1'(0) \frac{1}{0-\tilde y_{\hat c}}+a_1'(-1) \frac{1}{-1-\tilde y_{\hat c}}}{\kappa\varepsilon}.
\end{align*}

For the boundary term, by using Lemma \ref{lem-est-a1-a2}, we have
\begin{align*}
  \left|\frac{-a_1(0) \frac{1}{(0-\tilde y_{\hat c})^2}+a_1(-1) \frac{1}{(-1-\tilde y_{\hat c})^2}-a_1'(0) \frac{1}{0-\tilde y_{\hat c}}+a_1'(-1) \frac{1}{-1-\tilde y_{\hat c}}}{\kappa\varepsilon}\right|\lesssim \frac{1}{\varepsilon}\frac{1}{ \left\langle \kappa c_r \right\rangle^{\frac{3}{2}}}.
\end{align*} 
 
For the integration term, we write
\begin{align*}
  \frac{1}{\kappa\varepsilon} \int^0_{-1}a_1''(x) \frac{1}{x-\tilde y_{\hat c}}dx=\frac{1}{\kappa\varepsilon} \int^{2y_c+1}_{-1}a_1''(x) \frac{1}{x-\tilde y_{\hat c}}dx+\frac{1}{\kappa\varepsilon} \int^0_{2y_c+1}a_1''(x) \frac{1}{x-\tilde y_{\hat c}}dx.
\end{align*}
By using Lemma \ref{lem-est-a1-a2}, one can easily check that
\begin{align*}
  \left|\frac{1}{\kappa\varepsilon} \int^0_{2y_c+1}a_1''(x) \frac{1}{x-\tilde y_{\hat c}}dx\right|\lesssim \frac{1}{\varepsilon}\frac{1}{ \left\langle \kappa c_r \right\rangle^{3}}.
\end{align*}
For $\frac{1}{\kappa\varepsilon} \int^{2y_c+1}_{-1}a_1''(x) \frac{1}{x-\tilde y_{\hat c}}dx$, similar to \eqref{est-acc-psi-1-a}, by using analyticity of $\mathrm{Hi}''$, we have
\begin{align*}
  \left|\frac{1}{\kappa\varepsilon} \int^{2y_c+1}_{-1}a_1''(x) \frac{1}{x-\tilde y_{\hat c}}dx\right|\lesssim \frac{1}{\varepsilon}\frac{1}{ \left\langle \kappa c_r \right\rangle^{2}}.
\end{align*}
Then by using the same technique to \eqref{est-acc-psi-1-a}, we have
\begin{align*}
   \left|\varepsilon \alpha^2 \pa_yAirySolver_m\left(\pa_y^2\frac{1}{ y-\tilde y_{\hat c}}\right)(-1)\right|\lesssim \varepsilon^{\frac{1}{3}} \frac{1}{c_r}\alpha^2|\ln c_0|,
\end{align*}
and then
\begin{align*}
  \left|  \pa_yAirySolver_m\left(i\varepsilon \pa_y^2 \left(\frac{\phi_{1,0}^{[1]}(y)}{u_p-\hat c}\right)\right)(-1)\right| \lesssim \varepsilon^{\frac{1}{3}} \frac{1}{c_r}\alpha^2|\ln c_0|.
\end{align*}

Following the proof of Lemma \ref{lem-real-Airy-Solver-S}, for the $\phi_{err}$ associated with $f''=i\varepsilon \pa_y^2 \left(\frac{\phi_{1,0}^{[1]}(y)}{u_p-\hat c}\right)$, we have
\begin{align}\label{est-acc-psi-1-a-err}
  \left\|\phi_{err}\right\|_{L^\infty}+\left\|\left(u_p-\hat c\right)\phi_{err}'\right\|_{L^\infty}\lesssim \varepsilon^{\frac{1}{3}}c_r\alpha^2|\ln c_0|.
\end{align}

Recall from \eqref{eq-est-phi10jRI} that $\phi_{1,0}^{[j]}$ shares the same structure as $\phi_{1,0}^{[1]}$, while enjoying better smallness, namely $\phi_{1,0}^{[j]} = O(\alpha^{2j})$. Even using the rough estimates in Lemma \ref{lem-real-Airy-Solver-S}, one can get 
\begin{equation}\label{est-acc-psi-1-a-j}
  \begin{aligned}    
    \left\|AirySolver \left(i\varepsilon \pa_y^2 \left(\frac{\sum^{\infty}_{j=2} \phi_{1,0}^{[j]}(y)}{u_p-\hat c}\right)\right)\right\|_{L^\infty}  \lesssim \left(1+ \kappa c_r\right)\varepsilon^{\frac{1}{3}}\alpha^4|\ln c_0|^2,\\
    \left|\pa_yAirySolver \left(i\varepsilon \pa_y^2 \left(\frac{\sum^{\infty}_{j=2} \phi_{1,0}^{[j]}(y)}{u_p-\hat c}\right)\right)(-1)\right| \lesssim \left(1+ \kappa c_r\right)\varepsilon^{\frac{1}{3}}\frac{1}{c_r}\alpha^4|\ln c_0|^2.
  \end{aligned}
\end{equation} 

The estimate \eqref{est-acc-psi-1-a} and \eqref{est-acc-psi-1-a-dx} follow directly. 

\end{proof}

\subsubsection{Refined estimates of $\pa_{c_r}\psi_{1,\alpha,a}^{[1]}$, $\pa_{c_i}\psi_{1,\alpha,a}^{[1]}$, and $\pa_{\nu}\psi_{1,\alpha,a}^{[1]}$} 

In this subsubsection, we study the parameter derivatives of $\psi_{1,\alpha,a}^{[1]}$. This is the most delicate part of the analysis. We derive estimates that are valid throughout the T--S eigen region $\mathbb H$. However, the level of precision obtained here is only needed near the upper branch neutral curve, namely $c_r\sim \varepsilon^{\frac{1}{5}}$. Away from this regime, weaker estimates would already be sufficient for our purposes.

\begin{lemma}\label{lem-acc-psi-1-a-cr}
  It holds that
  \begin{align}\label{est-acc-psi-1-a-cr-1}
    \left|\pa_{c_r}\psi_{1,\alpha,a}^{[1]}(-1)\right|\lesssim \frac{\alpha^2}{ \left\langle \kappa c_r \right\rangle^{\frac{3}{2}}}|\ln c_0|^2+\varepsilon^{-\frac{1}{6}}(c_r+\alpha^2)\alpha^2|\ln c_0|^2,
  \end{align}
    \begin{align}\label{est-acc-psi-1-a-cr-2}
    \left\|\pa_{c_r}\psi_{1,\alpha,a}^{[1]}\right\|_{L^\infty}\lesssim \alpha^2|\ln c_0|^2, 
  \end{align}
    \begin{align}\label{est-acc-psi-1-a-cr-3}
    \left\|\pa_{c_r}\psi_{1,\alpha,a}^{[1]}\right\|_{L^1}\lesssim  \frac{\alpha^2}{ \left\langle \kappa c_r \right\rangle^{\frac{3}{2}}}|\ln c_0|^2+\varepsilon^{-\frac{1}{6}}(c_r+\alpha^2)\alpha^2|\ln c_0|^2,
  \end{align}
    \begin{align}\label{est-acc-psi-1-a-cr-1-dx}
    \left|\pa_{c_r}\psi_{1,\alpha,a}^{[1]\prime}(-1)\right|\lesssim  \frac{\alpha^2}{c_r \left\langle \kappa c_r \right\rangle^{\frac{3}{2}}}|\ln c_0|^2+\varepsilon^{-\frac{1}{6}}\frac{c_r+\alpha^2}{c_r}\alpha^2|\ln c_0|^2.
  \end{align}
\end{lemma}
\begin{proof}
  Similar to Lemma \ref{lem-acc-psi-1-a}, the most troublesome part is 
  \begin{align*}
    \varepsilon \alpha^2\pa_{c_r} \left(AirySolver_m\left(\pa_y^2\frac{1}{ y-\tilde y_{\hat c}}\right)(y)\right).
  \end{align*}
Following the proof of \eqref{eq-est-ASS-0}, we check the $c_r$-derivative on each terms of $AirySolver_m\left(\pa_y^2\frac{1}{ y-\tilde y_{\hat c}}\right)(y)$.

First, we claim that 
\begin{align}\label{eq-d-cr-A11}
  \left|\pa_{c_r} A_1(1,y)\right|\lesssim \left\langle \kappa\eta(y) \right\rangle^{-\frac{1}{4}}e^{-\frac{2}{3}\frac{1}{\sqrt2}|\kappa\eta(y)|^{\frac{1}{2}}\kappa\eta_r(y)}.
\end{align} 
Recall \eqref{eq-kappa} the definition of $\kappa$, we have $\left|\pa_{c_r}\kappa\right|\lesssim \kappa$. From the definition of $\eta$ and Lemma \ref{lem-Langer}, we can see that $\left|\pa_{c_r}\eta_r\right|\lesssim 1$ and $\left|\pa_{c_r}\eta_i\right|\lesssim \left|c_i\right|$. Therefore, it holds that
\begin{align*}
  &\pa_{c_r} A_1(1,y)=\frac{\pa_{c_r}\left(\mathcal A_1(1,\kappa \eta(y))\right)}{\kappa\left(\eta_r'(y)\right)^{\frac{3}{2}}}+err\\
  =&\frac{\kappa \pa_{c_r}\left(\eta(y)\right) \mathrm{Ai}(e^{\frac{\pi}{6}i}\kappa \eta(y))}{\kappa\left(\eta_r'(y)\right)^{\frac{3}{2}}}+err=\frac{ \pa_{c_r}\left(\eta_r(y)\right) \mathrm{Ai}(e^{\frac{\pi}{6}i}\kappa \eta(y))}{ \left(\eta_r'(y)\right)^{\frac{3}{2}}}+err,
\end{align*} 
then \eqref{eq-d-cr-A11} follows directly. For the same reason, we have
\begin{align}\label{eq-d-cr-A1k}
  \left|\pa_{c_r} A_1(k+1,y)\right|\lesssim  \kappa^{-k}\left\langle \kappa\eta(y) \right\rangle^{-\frac{1}{4}-\frac{1}{2}k}e^{-\frac{2}{3}\frac{1}{\sqrt2}|\kappa\eta(y)|^{\frac{1}{2}}\kappa\eta_r(y)},
\end{align}
\begin{align}\label{eq-d-cr-A2k}
  \left|\pa_{c_r} A_2(k+1,y)\right|\lesssim  \kappa^{-k}\left\langle \kappa\eta(y) \right\rangle^{-\frac{1}{4}-\frac{1}{2}k}e^{\frac{2}{3}\frac{1}{\sqrt2}|\kappa\eta(y)|^{\frac{1}{2}}\kappa\eta_r(y)}.
\end{align}

With the above information, we check the boundary terms of $AirySolver_m\left(\pa_y^2\frac{1}{ y-\tilde y_{\hat c}}\right)(y)$.  Taking $f(x)=\frac{1}{x-\tilde y_{\hat c}}$, by \eqref{eq-yhc} the definition of $\tilde y_{\hat c}$, one can easily check that
\begin{align*}
  f(-1)\sim \frac{1}{c_r},\quad f'(-1)\sim \pa_{c_r}f(-1)\sim \frac{1}{c_r^2},\quad  \pa_{c_r}f'(-1)\sim \frac{1}{c_r^3},\quad  f(0)\sim f'(0)\sim \pa_{c_r}f(0)\sim \pa_{c_r}f'(0)\sim 1.
\end{align*}

Here we only give the estimate for the most delicate term $\pa_{c_r} \left(\pa_xG_A(-1,y)f(-1)\right)$, and the other terms can be treated in similar way.

By \eqref{eq-d-cr-A1k} and \eqref{eq-d-cr-A2k}, we have
\begin{align*}
  &\kappa^2\left|\pa_{c_r}\left(A_2'(-1)A_1(2,y)\right)\right| \\
  \lesssim&\kappa^2\left(\left\langle \kappa \eta(-1) \right\rangle^{\frac{3}{4}}\left\langle \kappa\eta(y) \right\rangle^{-\frac{5}{4}}+\left\langle \kappa \eta(-1) \right\rangle^{\frac{1}{4}}\left\langle \kappa\eta(y) \right\rangle^{-\frac{3}{4}}\right)e^{-\frac{1}{C}\left|\kappa\eta_r(-1)-\kappa\eta_r(y)  \right|\left( |\kappa\eta(y)|^{\frac{1}{2}}+|\kappa\eta(-1)|^{\frac{1}{2}} \right) }\\
   \lesssim&\kappa^2\left\langle \kappa c_r \right\rangle^{-\frac{1}{2}}.
\end{align*}

Next, we study the non-local term. It holds that
\begin{align*}
  \left|\pa_{c_r} \left(a_1'(-1) \left(-1-y_c\right)\right)\right|\lesssim \left|\pa_{c_r} \left(a_1'(-1)\right) \left(-1-y_c\right)\right|+\left|\left(a_1'(-1)\right)\right|.
\end{align*}
From the expansion of $a_1(x)$ and Lemma \ref{lem-est-a1-a2}, it is easy to see that
\begin{align*}
  \kappa^2\left(\left|\pa_{c_r} \left(a_{1,M}'(-1)\right)\right|+\left|\pa_{c_r} \left(a_{1,SecM}'(-1)\right)\right|+\left|\pa_{c_r} \left(a_{1,R}'(-1)\right)\right|\right)\lesssim \kappa^3\left\langle \kappa c_r \right\rangle^{-3}+\kappa^2\left\langle \kappa c_r \right\rangle^{-\frac{1}{2}}.
\end{align*}
For $a_{1,B}$, by using Lemma \ref{lem-d-cr-a1B}, we have
\begin{align*}
  \kappa^2 \left|\pa_{c_r} \left(a_{1,B}'(-1)\right)\right| \lesssim \kappa^3\left\langle \kappa c_r \right\rangle^{-\frac{3}{2}}|\ln c_0|.
\end{align*}
Therefore, it holds that
\begin{align*}
  \kappa^2\left|\pa_{c_r} \left(a_1'(-1) \left(-1-y_c\right)\right)\right|\lesssim \left(\kappa^3\left\langle \kappa c_r \right\rangle^{-3}+ \kappa^3\left\langle \kappa c_r \right\rangle^{-\frac{3}{2}}|\ln c_0|\right)c_r\lesssim \kappa^2\left\langle \kappa c_r \right\rangle^{-2}+ \kappa^2\left\langle \kappa c_r \right\rangle^{-\frac{1}{2}}|\ln c_0|.
\end{align*}

Recall the expansion of $a_2(x)$. By Lemma \ref{lem-a1-a2-exp} and Lemma \ref{lem-est-a1-a2}, we have
\begin{align*}
  \left|\kappa^2\pa_{c_r}a_2'(-1)\right|\lesssim \kappa^2\left\langle \kappa c_r \right\rangle^{-\frac{1}{2}}.
\end{align*}

Therefore, recalling $f(-1)\sim \frac{1}{c_r}$, we have
\begin{align*}
  \left|\pa_{c_r} \left(\pa_xG_A(-1,y)\right)f(-1)\right|\lesssim \left(\kappa^2\left\langle \kappa c_r \right\rangle^{-2}+ \kappa^2\left\langle \kappa c_r \right\rangle^{-\frac{1}{2}}|\ln c_0|\right)\frac{1}{c_r}\lesssim \frac{\kappa^3|\ln c_0|}{\left\langle \kappa c_r \right\rangle^{\frac{3}{2}}}.
\end{align*}

From the proof of \eqref{eq-est-ASS-0} and the fact that $\pa_{c_r}f(-1)\sim \frac{1}{c_r^2}$, we have
\begin{align*}
  \left| \pa_xG_A(-1,y)\pa_{c_r}f(-1)\right|\lesssim\frac{\kappa^3}{\left\langle \kappa c_r \right\rangle^{\frac{3}{2}}}.
\end{align*}
It follows that
\begin{align*}
  \left|\pa_{c_r} \left(\pa_xG_A(-1,y)f(-1)\right)\right|\lesssim \frac{\kappa^3|\ln c_0|}{\left\langle \kappa c_r \right\rangle^{\frac{3}{2}}}.
\end{align*}

Next, we study the integration part. We first study $\int^0_{-1}\left(\pa_{c_r}\pa_x^2G_A(x,y)\right)\frac{1}{x-\tilde y_{\hat c}}dx$.

For the non-local term part, we first study $\int^0_{-1}\left(\pa_{c_r}\pa_x^2E(x,y)\right)\frac{1}{x-\tilde y_{\hat c}}dx$. Here we only give the proof for the terms related to $\pa_{c_r}a_2''$ and $\pa_{c_r}\left(a_1''(x)\left(y-y_c\right)\right)$, the rest terms can be treated in the same way.

\paragraph{1) term $\frac{1}{\kappa\varepsilon} \int^y_{-1} \frac{\pa_{c_r}a_2''(x)}{x-\tilde y_{\hat c}} dx$.} 

Recalling the expansion of $a_2(x)$, using Lemma \ref{lem-est-a1-a2} and techniques in Lemma \ref{lem-d-cr-a1B}, one can easily check that
\begin{align*}
  \frac{1}{\kappa\varepsilon}\left|\int^y_{-1} \frac{\pa_{c_r}a_{2,B}''(x)+\pa_{c_r}a_{2,R}''(x)}{x-\tilde y_{\hat c}}dx\right| \lesssim\frac{1}{\varepsilon \left\langle \kappa c_r \right\rangle}|\ln c_0|.
\end{align*}

Recall \eqref{eq-exp-a2M} the expression of $a_2(x)$, we have
\begin{align*}
  \pa_{c_r}a_{2,M}''(x)=\frac{-1}{2\left(\eta_r'(x)\right)^{\frac{1}{2}}} \left(\eta(x)\pa_{c_r}\kappa +\kappa \pa_{c_r}\eta(x)\right)\mathrm{Hi}^{(5)}(e^{-\frac{\pi}{2}i}\kappa \eta(x))+err,
\end{align*}
where the term $err$ collects all contributions in which the $c_r$- or $x$-derivatives fall on $\eta_r'(x)$. One can easily check that
\begin{align*}
  \left|err\right|+ \left|\eta(x)\pa_{c_r}\kappa\mathrm{Hi}^{(5)}(e^{-\frac{\pi}{2}i}\kappa \eta(x)) \right|\lesssim 1.
\end{align*}
Therefore, for the same reason as Lemma \ref{lem-acc-psi-1-a}, to estimate $\frac{1}{\kappa\varepsilon} \int^y_{-1} \frac{\pa_{c_r}a_{2,M}''(x)}{x-\tilde y_{\hat c}} dx$ it suffice to estimate
\begin{align*}
  \frac{1}{\kappa\varepsilon} \int^y_{-1}\kappa \mathrm{Hi}^{(5)}(e^{-\frac{\pi}{2}i}\kappa \eta(x))\frac{1}{x-\tilde y_{\hat c}}dx.
\end{align*}

\paragraph{1.1) Case $-1< y\le 2y_c+1$.} When $y$ is close to $y_c$, it is not convenient to exploit the analyticity of $\mathrm{Hi}^{(5)}$. We therefore get the following rough estimate:
\begin{align*}
  &\frac{1}{\kappa\varepsilon}\left|\int^y_{-1}\kappa \mathrm{Hi}^{(5)}(e^{-\frac{\pi}{2}i}\kappa \eta(x))\frac{1}{x-\tilde y_{\hat c}}dx\right| \lesssim\frac{1}{\varepsilon}|\ln c_0|.
\end{align*}
This pointwise estimate is not optimal. In fact, it can be improved; for instance, as $y\to -1$, the integral above tends to $0$. However, for the purposes of the present proof, the bound stated above is already sufficient. Moreover, since in this case the interval $[-1,2y_c+1]$ is sufficiently short, 
we have the following $L^1$ estimate:
\begin{align*}
  \int_{-1}^0 \chi_{[-1,2y_c+1]}(y) \frac{1}{\kappa\varepsilon}\left|\int^y_{-1}\kappa \mathrm{Hi}^{(5)}(e^{-\frac{\pi}{2}i}\kappa \eta(x))\frac{1}{x-\tilde y_{\hat c}}dx\right| dy\lesssim  \frac{c_r}{\varepsilon}|\ln c_0|.
\end{align*}

\paragraph{1.2) Case $2y_c+1< y\le 0$.} For this case, we have
\begin{align*}
  &\frac{1}{\kappa\varepsilon} \int^y_{-1}\kappa \mathrm{Hi}^{(5)}(e^{-\frac{\pi}{2}i}\kappa \eta(x))\frac{1}{x-\tilde y_{\hat c}}dx\\
  =&\frac{1}{\kappa\varepsilon} \int^{2y_c+1}_{-1}\kappa \mathrm{Hi}^{(5)}(e^{-\frac{\pi}{2}i}\kappa \eta(x))\frac{1}{x-\tilde y_{\hat c}}dx+\frac{1}{\kappa\varepsilon} \int^{y}_{2y_c+1}\kappa \mathrm{Hi}^{(5)}(e^{-\frac{\pi}{2}i}\kappa \eta(x))\frac{1}{x-\tilde y_{\hat c}}dx.
\end{align*}
The second term
\begin{align*}
  \left|\frac{1}{\kappa\varepsilon} \int^{y}_{2y_c+1}\kappa \mathrm{Hi}^{(5)}(e^{-\frac{\pi}{2}i}\kappa \eta(x))\frac{1}{x-\tilde y_{\hat c}}dx\right|\lesssim \left|\frac{1}{\kappa\varepsilon} \int^{y}_{2y_c+1}\frac{\kappa^2}{\left\langle \kappa \eta(x) \right\rangle^6} dx\right|\lesssim \frac{1}{ \varepsilon\left\langle \kappa c_r \right\rangle^5}.
\end{align*}
For the integration on $[-1,2y_c+1]$, we write
\begin{align*}
   &\frac{1}{\kappa\varepsilon} \int^{2y_c+1}_{-1}\kappa \mathrm{Hi}^{(5)}(e^{-\frac{\pi}{2}i}\kappa \eta(x))\frac{1}{x-\tilde y_{\hat c}}dx\\
   =&\frac{1}{\varepsilon}\int^{2y_c+1}_{-1} \left(\mathrm{Hi}^{(5)}(e^{-\frac{\pi}{2}i}\kappa \eta(x))-\mathrm{Hi}^{(5)}\left(e^{-\frac{\pi}{2}i}\kappa (x-y_c-\frac{ic_i}{u_p'(y_c)})\right)\right)\frac{1}{x-\tilde y_{\hat c}}dx\\
  &+\frac{1}{\varepsilon}\int^{2y_c+1}_{-1} \mathrm{Hi}^{(5)}\left(e^{-\frac{\pi}{2}i}\kappa (x-y_c-\frac{ic_i}{u_p'(y_c)})\right)\frac{1}{x-\tilde y_{\hat c}}dx\\
  =&I+II.
\end{align*}
By using the same technique as Lemma \ref{lem-acc-psi-1-a}, we have
\begin{align*}
  |I|\lesssim \left|\frac{1}{\varepsilon}\int^{2y_c+1}_{-1}\frac{\kappa (x-y_c)}{\left\langle \kappa\eta(x) \right\rangle^6} dx\right|\lesssim \varepsilon^{-\frac{2}{3}},
\end{align*}
and
\begin{align*}
  |II|\lesssim \frac{1}{\varepsilon}\left|\int^{0}_{-\pi} \mathrm{Hi}^{(5)}\left(e^{-\frac{\pi}{2}i}\kappa (e^{i\theta}(y_c+1)-\frac{ic_i}{u_p'(y_c)})\right)\frac{ie^{i\theta}(y_c+1)}{ y_c+e^{i\theta}(y_c+1)  -\tilde y_{\hat c}}d\theta\right|\lesssim \frac{1}{ \varepsilon\left\langle \kappa c_r \right\rangle^6}.
\end{align*}

\paragraph{2) term $\frac{1}{\kappa\varepsilon} \int^0_{y} \frac{\pa_{c_r} \left(a_1''(x)\left(y-y_c\right)\right)}{x-\tilde y_{\hat c}} dx$.} 

It holds that
\begin{align*}
  \pa_{c_r}\left(a_1''(x)\left(y-y_c\right)\right)= \left(\pa_{c_r}a_1''(x)\right)\left(y-y_c\right)- a_1''\left(x\right)\frac{1}{2\sqrt{1-c_r}},
\end{align*}
here the second term can be treated in the same way to Lemma \ref{lem-acc-psi-1-a}, we only focus on the first term. 

\paragraph{2.1) Case $2y_c+1\le y\le 0$.} For this case, we have $\left|\frac{y-y_c}{x-\tilde y_{\hat c}}\right|\lesssim 1$, and
\begin{align*}
  \left|\frac{1}{\kappa\varepsilon}\int^0_{y} \pa_{c_r}a_{1,M}''(x) \frac{y-y_c}{x-\tilde y_{\hat c}}dx\right|\lesssim \left|\frac{1}{\kappa\varepsilon}\int^0_{2y_c+1}\kappa^2\left\langle \kappa\eta(x) \right\rangle^{-4} dx\right|\lesssim \frac{1}{\varepsilon}\frac{1}{ \left\langle \kappa c_r \right\rangle^{2}}.
\end{align*}
By using Lemma \ref{lem-d-cr-a1B}, we have
\begin{align*}
  \left|\frac{1}{\kappa\varepsilon}\int^0_{y} \pa_{c_r}a_{1,B}''(x) \frac{y-y_c}{x-\tilde y_{\hat c}}dx\right|\lesssim \frac{1}{\kappa\varepsilon} \left( \varepsilon^2+ \kappa\left|\ln c_0\right| \frac{C \left|\ln c_0\right|}{\kappa^{\frac{3}{2}}}\right)\lesssim \kappa^{\frac{3}{2}}\left|\ln c_0\right|^2.
\end{align*}
Then by the expansion of $a_1(x)$, one can easily check that
\begin{align*}
  \left|\frac{1}{\kappa\varepsilon}\int^0_{y} \left(\pa_{c_r}a_{1,SecM}''(x)+\pa_{c_r}a_{1,R}''(x)\right)\frac{y-y_c}{x-\tilde y_{\hat c}}dx\right|\lesssim   \kappa^2.
\end{align*}

\paragraph{2.2) Case $y_c-\frac{1}{\kappa}\le y\le 2y_c+1$.} For $y_c\le y\le 2y_c+1$, we have $\left|\frac{y-y_c}{x-\tilde y_{\hat c}}\right|\lesssim 1$, and for $y_c-\frac{1}{\kappa}\le y\le y_c$, we have $|y-y_c|\le\frac{1}{\kappa}$. Therefore
\begin{align*}
  \left|\frac{1}{\kappa\varepsilon}\int^0_{y} \pa_{c_r}a_{1,M}''(x) \frac{y-y_c}{x-\tilde y_{\hat c}}dx\right|\lesssim  \frac{1}{\varepsilon}|\ln c_0|.
\end{align*}
This estimate is not good, however, as the interval $[y_c-\frac{1}{\kappa},2y_c+1]$ is short, we have
\begin{align*}
  \int_{-1}^0 \chi_{[y_c-\frac{1}{\kappa},2y_c+1]}(y) \frac{1}{\kappa\varepsilon}\left|\int^0_{y}\left(\pa_{c_r}a_1''(x)\right)\frac{y-y_c}{x-\tilde y_{\hat c}}dx\right| dy\lesssim  \frac{c_r}{\varepsilon}|\ln c_0|.
\end{align*}

\paragraph{2.3) Case $-1\le y\le y_c-\frac{1}{\kappa}$.} Similar to Lemma  \ref{lem-acc-psi-1-a}, we write
\begin{align*}
  &\frac{1}{\kappa\varepsilon}\int^0_{y}\left(\pa_{c_r}a_1''(x)\right)\frac{y-y_c}{x-\tilde y_{\hat c}}dx\\
  =&\frac{1}{\kappa\varepsilon}\int^{2y_c+1}_{-1}\left(\pa_{c_r}a_1''(x)\right)\frac{y-y_c}{x-\tilde y_{\hat c}}dx +\frac{1}{\kappa\varepsilon}\int^{0}_{2y_c+1}\left(\pa_{c_r}a_1''(x)\right)\frac{y-y_c}{x-\tilde y_{\hat c}}dx-\frac{1}{\kappa\varepsilon}\int^y_{-1}\left(\pa_{c_r}a_1''(x)\right)\frac{y-y_c}{x-\tilde y_{\hat c}}dx.
\end{align*}
The second term can be treated as Case 2.1), and the last term can be treated as Case 2.2).

For the main term, the integration on $[-1,2y_c+1]$, similar to Lemma \ref{lem-acc-psi-1-a}, we use the analyticity of $\mathrm{Hi}^{(3)}$ to get
\begin{align*}
  \left|\frac{1}{\kappa\varepsilon}\int^{2y_c+1}_{-1}\left(\pa_{c_r}a_1''(x)\right)\frac{y-y_c}{x-\tilde y_{\hat c}}dx\right|\lesssim \frac{c_r}{\varepsilon} |\ln c_0|+\frac{1}{\varepsilon}\frac{1}{ \left\langle \kappa c_r \right\rangle^{3}}+\frac{1}{\varepsilon}\frac{1}{ \left\langle \kappa c_r \right\rangle^{\frac{3}{2}}}|\ln c_0|^2.
\end{align*}

For $y=-1$, $\int^0_{-1}\left(\pa_{c_r}\pa_x^2E(x,y)\right)\frac{1}{x-\tilde y_{\hat c}}dx$ only contains the $a_1''(x)\left(y-y_c\right)$ related term. From the above analysis, we can see that
\begin{align*}
  \left|\int^0_{-1}\left(\pa_{c_r}\pa_x^2E(x,-1)\right)\frac{1}{x-\tilde y_{\hat c}}dx\right|\lesssim \frac{c_r}{\varepsilon} |\ln c_0|+\frac{1}{\varepsilon}\frac{1}{ \left\langle \kappa c_r \right\rangle^{\frac{3}{2}}}|\ln c_0|+\frac{1}{\varepsilon}\frac{1}{ \left\langle \kappa c_r \right\rangle^{\frac{3}{2}}}|\ln c_0|^2.
\end{align*}

Next, we study $\int^0_{-1}\left(\pa_{c_r}\pa_x^2G_{A,M}(x,y)\right)\frac{1}{x-\tilde y_{\hat c}}dx$.

From \eqref{eq-d-cr-A1k} and \eqref{eq-d-cr-A2k}, we have
\begin{align*}
  &\left|\frac{2\pi}{\kappa\varepsilon}\int^y_{-1} \frac{\pa_{c_r}\left(A_1(2,y) A_2''(x)\right)}{\left(\eta_r'(x)\right)^{\frac{1}{2}}}\frac{1}{x-\tilde y_{\hat c}}dx+\frac{2\pi}{\kappa\varepsilon}\int^0_{y} \frac{\pa_{c_r}\left(A_1''(x) A_2(2,y)\right)}{\left(\eta_r'(x)\right)^{\frac{1}{2}}}\frac{1}{x-\tilde y_{\hat c}}  dx\right|\\
\lesssim&\frac{1}{\varepsilon}\left|\int^0_{-1}\left(\frac{\left\langle \kappa\eta(x) \right\rangle^{\frac{3}{4}}}{\left\langle \kappa\eta(y) \right\rangle^{\frac{3}{4}}} +\frac{\left\langle \kappa\eta(x) \right\rangle^{\frac{5}{4}}}{\left\langle \kappa\eta(y) \right\rangle^{\frac{5}{4}}} \right)  e^{-\frac{1}{C}\left|\kappa\eta_r(x)-\kappa\eta_r(y)  \right|\left( |\kappa\eta(y)|^{\frac{1}{2}}+|\kappa\eta(x)|^{\frac{1}{2}} \right) }\frac{1}{x-\tilde y_{\hat c}}dx\right|.
\end{align*}
Similar to Lemma \ref{lem-est-AS}, we write
\begin{align}\label{eq-I-GM,0,c_r}
  I_{GMS,0,c_r}(x,y) = \frac{1}{\varepsilon }\left(\frac{\left\langle \kappa\eta(x) \right\rangle^{\frac{3}{4}}}{\left\langle \kappa\eta(y) \right\rangle^{\frac{3}{4}}} +\frac{\left\langle \kappa\eta(x) \right\rangle^{\frac{5}{4}}}{\left\langle \kappa\eta(y) \right\rangle^{\frac{5}{4}}} \right)  e^{-\frac{1}{C}\left|\kappa\eta_r(x)-\kappa\eta_r(y)  \right|\left( |\kappa\eta(y)|^{\frac{1}{2}}+|\kappa\eta(x)|^{\frac{1}{2}} \right) }\frac{1}{x-\tilde y_{\hat c}},
\end{align}
and divided the estimate to different cases.

Case 1. $|y-y_c|\ge \frac{1}{2}$.

Case 1.1. $|x-y|\le \frac{1}{4}$. Similar to the estimate of \eqref{eq-est-AS-1}, we have
\begin{align*}
  &\left|\int^0_{-1}I_{GMS,0,c_r}(x,y)\chi_{\text{Case 1.1}}dx\right|\lesssim \int^0_{-1}\varepsilon^{-1}e^{-\frac{1}{C} \kappa^{\frac{3}{2}} \left|x-y\right|}dx\lesssim \varepsilon^{-\frac{1}{2}}.
\end{align*}

Case 1.2. $|x-y|\ge \frac{1}{4}$. The same to \eqref{eq-est-AS-1}.
 
Case 2. $\frac{2}{\kappa}\le|y-y_c|\le \frac{1}{2}$. 

Case 2.1. $\left|x-y_c\right|\ge \frac{1}{\kappa}$ and $\left|x-y\right|\le \frac{1}{\kappa}$. For this case, it holds that $\left|x-\tilde y_{\hat c}\right|\sim \left|y-y_{c}\right|\sim \left|\eta(y)\right|$. Similar to the estimate of \eqref{eq-est-AS-1}, we have
\begin{align*}
  &\left|\int^0_{-1}I_{GMS,0,c_r}(x,y)\chi_{\text{Case 2.1}}dx\right|\lesssim \int^0_{-1}\frac{1}{\varepsilon}\frac{1}{\left|\eta(y)\right|} e^{-\frac{1}{C}\kappa\left|x-y \right||\kappa\eta(y)|^{\frac{1}{2}}}dx\lesssim \varepsilon^{-\frac{1}{2}}|y-y_c|^{-\frac{3}{2}}.
\end{align*}
For $|y-y_c|\sim \frac{1}{\kappa}$, we have $\varepsilon^{-\frac{1}{2}}|y-y_c|^{-\frac{3}{2}}\sim \varepsilon^{-1}$. However, for the boundary $y=-1$, we have
\begin{align*}
  \varepsilon^{-\frac{1}{2}}|-1-y_c|^{-\frac{3}{2}}\sim \frac{1}{\varepsilon}\frac{1}{ \left\langle \kappa c_r \right\rangle^{\frac{3}{2}}}.
\end{align*}
And for the $L^1$ norm in $y$, we have
\begin{align*}
  \int_{-1}^0 \chi_{\frac{2}{\kappa}\le|y-y_c|\le \frac{1}{2}}(y) \varepsilon^{-\frac{1}{2}}|y-y_c|^{-\frac{3}{2}} dy\lesssim  \varepsilon^{-\frac{1}{2}} \kappa^{\frac{1}{2}}\lesssim \varepsilon^{-\frac{2}{3}}.
\end{align*}

Case 2.2. $\left|x-y_c\right|\ge \frac{1}{\kappa}$ and $\left|x-y\right|>\frac{1}{\kappa}$. For this case, similar to \eqref{eq-est-AS-1}, it holds that
\begin{align*}
  &\left|\int^0_{-1}I_{GMS,0,c_r}(x,y)\chi_{\text{Case 2.2}}dx\right|\lesssim \int^0_{-1}\frac{\kappa}{\varepsilon} e^{-\frac{1}{C}\kappa\left|x-y \right||\kappa\eta(y)|^{\frac{1}{2}}}e^{-\frac{1}{C}|\kappa\eta(y)|^{\frac{1}{2}}}dx\lesssim \frac{1}{\varepsilon|\kappa\eta(y)|^{\frac{1}{2}}} e^{-\frac{1}{C}|\kappa\eta(y)|^{\frac{1}{2}}}.
\end{align*}
For $y=-1$, the above value is extremely small. We also have
\begin{align*}
  \int_{-1}^0 \chi_{\frac{2}{\kappa}\le|y-y_c|\le \frac{1}{2}}(y) \frac{1}{\varepsilon|\kappa\eta(y)|^{\frac{1}{2}}} e^{-\frac{1}{C}|\kappa\eta(y)|^{\frac{1}{2}}} dy\lesssim  \varepsilon^{-1} \kappa^{-1}\lesssim \varepsilon^{-\frac{2}{3}}.
\end{align*}

 Case 2.3. $\left|x-y_c\right|\le \frac{1}{\kappa}$. For this case, similar to \eqref{eq-est-AS-1}, it holds that 
\begin{align*}
  &\left|\int^0_{-1}I_{GMS,0,c_r}(x,y)\chi_{\text{Case 2.3}}dx\right|\lesssim \int^0_{-1}\frac{1}{\varepsilon}\frac{1}{\left|x-\tilde y_{\hat c}\right|}  e^{- |\kappa\eta(y)|^{\frac{1}{2}}}dx\lesssim \frac{1}{\varepsilon} |\ln c_0|e^{-|\kappa\eta(y)|^{\frac{1}{2}}}.
\end{align*}
For $y=-1$, the above value is extremely small. We also have
\begin{align*}
  \int_{-1}^0 \chi_{\frac{2}{\kappa}\le|y-y_c|\le \frac{1}{2}}(y) \frac{1}{\varepsilon} |\ln c_0|e^{-|\kappa\eta(y)|^{\frac{1}{2}}} dy\lesssim \varepsilon^{-\frac{2}{3}}|\ln c_0|.
\end{align*}
 
 Case 3. $|y-y_c|\le \frac{2}{\kappa}$.

Case 3.1. $\left|x-y_c\right|\ge \frac{4}{\kappa}$. For this case, similar to \eqref{eq-est-AS-1}, it holds that
\begin{align*}
  &\left|\int^0_{-1}I_{GMS,0,c_r}(x,y)\chi_{\text{Case 3.1}}dx\right|\lesssim \int^0_{-1}\frac{\kappa}{\varepsilon} e^{-\kappa\left|x-y \right|}dx\lesssim \frac{1}{\varepsilon},
\end{align*}
and
\begin{align*}
  \int_{-1}^0 \chi_{|y-y_c|\le \frac{2}{\kappa}}(y) \frac{1}{\varepsilon} dy\lesssim \varepsilon^{-\frac{2}{3}}.
\end{align*}

Case 3.2. $\left|x-y_c\right|\le \frac{4}{\kappa}$. In this case, we have  
\begin{align*}
  &\left|\int^0_{-1}I_{GMS,0,c_r}(x,y)\chi_{\text{Case 3.2}}dx\right|\lesssim \int^0_{-1}\frac{1}{\varepsilon} \frac{1}{x-\tilde y_{\hat c}}dx\lesssim \frac{|\ln c_0|}{\varepsilon},
\end{align*}
and
\begin{align*}
  \int_{-1}^0 \chi_{|y-y_c|\le \frac{2}{\kappa}}(y) \frac{|\ln c_0|}{\varepsilon} dy\lesssim \varepsilon^{-\frac{2}{3}}|\ln c_0|.
\end{align*}

Next, we study $\int^0_{-1}\pa_x^2G_A(x,y)\pa_{c_r}\frac{1}{x-\tilde y_{\hat c}}dx$. 

By the definition of $\tilde y_{\hat c}$ in \eqref{eq-yhc}, we have
\begin{align}\label{eq-com-x-cr}
  \pa_{c_r}\frac{1}{x-\tilde y_{\hat c}}=\frac{\frac{1}{2\sqrt{1-\hat c}}}{\left(x-\tilde y_{\hat c}\right)^2}=-\pa_x\frac{\frac{1}{2\sqrt{1-\hat c}}}{x-\tilde y_{\hat c}}.
\end{align}

Then we have
\begin{align*}
  &\int^0_{-1}\pa_x^2G_A(x,y)\pa_{c_r}\frac{1}{x-\tilde y_{\hat c}}dx=-\int^0_{-1}\pa_x^2G_A(x,y)\pa_{x}\frac{\frac{1}{2\sqrt{1-\hat c}}}{x-\tilde y_{\hat c}}dx\\
  =&\int_{[-1,0]\setminus [y_c-\frac{y_c+1}{2},y_c+\frac{y_c+1}{2}]}\pa_x^2G_A(x,y)\frac{\frac{1}{2\sqrt{1-\hat c}}}{\left(x-\tilde y_{\hat c}\right)^2}dx+\int^{y_c+\frac{y_c+1}{2}}_{y_c-\frac{y_c+1}{2}}\pa_x^3G_A(x,y)\frac{\frac{1}{2\sqrt{1-\hat c}}}{x-\tilde y_{\hat c}}dx\\
  &-\left.\pa_x^2G_A(x,y) \frac{\frac{1}{2\sqrt{1-\hat c}}}{x-\tilde y_{\hat c}}\right|^{y_c+\frac{y_c+1}{2}}_{x=y_c-\frac{y_c+1}{2}}.
\end{align*}
Note that we avoid performing integration by parts over the entire interval $[-1,0]$, since the estimates for $a_{1,B}'''(x)$ at the endpoints $x=-1$ and $x=0$ are not good enough. However, on the upper branch, as the points $y_c+\frac{y_c+1}{2}$ and $y_c-\frac{y_c+1}{2}$ are separated from $-1$ and $0$, 
the term $a_{1,B}''$ exhibits exponential decay and is thus negligible. A straightforward verification yields the estimate
\begin{align}\label{eq-psi-1-a-cr-bd}
  \left|\left.\pa_x^2G_A(x,y) \frac{\frac{1}{2\sqrt{1-\hat c}}}{x-\tilde y_{\hat c}}\right|^{y_c+\frac{y_c+1}{2}}_{x=y_c-\frac{y_c+1}{2}}\right|\lesssim \varepsilon^{-\frac{2}{3}} \frac{1}{c_r}.
\end{align}
On the lower branch, one can directly do the intrgration by parts over $[-1,0]$, and we have the same estimate for the boundary terms as \eqref{eq-psi-1-a-cr-bd}.

We further remark that for $x\in [-1,0]\setminus \left[y_c-\frac{y_c+1}{2},y_c+\frac{y_c+1}{2}\right]$, it holds that
\begin{align*}
  \left|\frac{\frac{1}{2\sqrt{1-\hat c}}}{\left(x-\tilde y_{\hat c}\right)^2}\right|\lesssim \frac{1}{c_r}\frac{1}{x-\tilde y_{\hat c}},
\end{align*}
therefore by using the same technique of Lemma \ref{lem-acc-psi-1-a}, one can show that
\begin{align*}
  \left|\int_{[-1,0]\setminus [y_c-\frac{y_c+1}{2},y_c+\frac{y_c+1}{2}]}\pa_x^2G_A(x,y)\frac{\frac{1}{2\sqrt{1-\hat c}}}{\left(x-\tilde y_{\hat c}\right)^2}dx\right|\lesssim \varepsilon^{-\frac{2}{3}} \frac{1}{c_r}|\ln c_0|.
\end{align*}

The properties of $\partial_x^3 G_A(x,y)$ are very close to those of $\partial_{c_r}\partial_x^2 G_{A,M}(x,y)$.  Consequently, the term $\int^{y_c+\frac{y_c+1}{2}}_{y_c-\frac{y_c+1}{2}}\pa_x^3G_A(x,y)\frac{\frac{1}{2\sqrt{1-\hat c}}}{x-\tilde y_{\hat c}}dx$ admits the same estimate as  $\int^0_{-1}\left(\pa_{c_r}\pa_x^2G_{A,M}(x,y)\right)\frac{1}{x-\tilde y_{\hat c}}dx$.

Combining the above estimates, we have
\begin{align*}
  \varepsilon \alpha^2\pa_{c_r} \left(AirySolver_m\left(\pa_y^2\frac{1}{ y-\tilde y_{\hat c}}\right)(-1)\right)\lesssim c_r \alpha^2|\ln c_0|+ \frac{\alpha^2}{ \left\langle \kappa c_r \right\rangle^{\frac{3}{2}}}|\ln c_0|,
\end{align*}
\begin{align*}
  \left\|\varepsilon \alpha^2\pa_{c_r} \left(AirySolver_m\left(\pa_y^2\frac{1}{ y-\tilde y_{\hat c}}\right)(y)\right)\right\|_{L^\infty}\lesssim  \alpha^2|\ln c_0|,
\end{align*}
\begin{align*}
  \left\|\varepsilon \alpha^2\pa_{c_r} \left(AirySolver_m\left(\pa_y^2\frac{1}{ y-\tilde y_{\hat c}}\right)(y)\right)\right\|_{L^1}\lesssim c_r \alpha^2|\ln c_0|+ \frac{\alpha^2}{ \left\langle \kappa c_r \right\rangle^{\frac{3}{2}}}|\ln c_0|.
\end{align*}

Next, we study $\pa_{c_r}\left(\pa_yAirySolver_m\left(\pa_y^2\frac{1}{ y-\tilde y_{\hat c}}\right)(-1)\right)$. Similar to \eqref{est-acc-psi-1-a-dx}, there is only $a_1$ related terms. We have
\begin{align*}
  &\frac{1}{\kappa\varepsilon} \pa_{c_r}\int^0_{-1}a_1(x)\pa_x^2 \frac{1}{x-\tilde y_{\hat c}}dx\\
  =&\frac{1}{\kappa\varepsilon} \int^0_{-1}\pa_{c_r}a_1''(x) \frac{1}{x-\tilde y_{\hat c}}dx +\frac{1}{\kappa\varepsilon}\int_{[-1,0]\setminus [y_c-\frac{y_c+1}{2},y_c+\frac{y_c+1}{2}]}a_1''(x)\frac{\frac{1}{2\sqrt{1-\hat c}}}{\left(x-\tilde y_{\hat c}\right)^2}dx\\
  &+\frac{1}{\kappa\varepsilon}\int^{y_c+\frac{y_c+1}{2}}_{y_c-\frac{y_c+1}{2}}a_1'''(x)\frac{\frac{1}{2\sqrt{1-\hat c}}}{x-\tilde y_{\hat c}}dx -\frac{1}{\kappa\varepsilon}\left.a_1''(x) \frac{\frac{1}{2\sqrt{1-\hat c}}}{x-\tilde y_{\hat c}}\right|^{y_c+\frac{y_c+1}{2}}_{x=y_c-\frac{y_c+1}{2}}\\
  &+\pa_{c_r}\frac{-a_1(0) \frac{1}{(0-\tilde y_{\hat c})^2}+a_1(-1) \frac{1}{(-1-\tilde y_{\hat c})^2}-a_1'(0) \frac{1}{0-\tilde y_{\hat c}}+a_1'(-1) \frac{1}{-1-\tilde y_{\hat c}}}{\kappa\varepsilon}.
\end{align*} 

For the boundary term, the worst terms is $\pa_{c_r}\frac{a_1'(-1) \frac{1}{-1-\tilde y_{\hat c}}}{\kappa\varepsilon}$. By using Lemma \ref{lem-est-a1-a2} and Lemma \ref{lem-d-cr-a1B}, we have
\begin{align*}
  \left|\pa_{c_r}\frac{a_1'(-1) \frac{1}{-1-\tilde y_{\hat c}}}{\kappa\varepsilon}\right|\lesssim \frac{1}{\varepsilon} \frac{1}{\left\langle \kappa c_r \right\rangle^{\frac{3}{2}}} \frac{1}{c_r}\left|\ln c_0\right|+\frac{1}{\varepsilon} \frac{1}{\left\langle \kappa c_r \right\rangle^{\frac{3}{2}}}.
\end{align*}

For the intgration terms, the main difficulty comes from the following integral
\begin{align*}
  \frac{\kappa}{\varepsilon}\int^{y_c+\frac{y_c+1}{2}}_{y_c-\frac{y_c+1}{2}}\mathrm{Hi}'''\left(e^{-\frac{\pi}{2}i}\kappa \left(x-y_c-\frac{ic_i}{u_p'(y_c)}\right)\right) \frac{1}{x-\tilde y_{\hat c}}dx.
\end{align*}
As in \eqref{est-acc-psi-1-a-cr-1}, by using the analyticity of $\mathrm{Hi}'''$, we have
\begin{align*}
  \frac{\kappa}{\varepsilon}\int^{y_c+\frac{y_c+1}{2}}_{y_c-\frac{y_c+1}{2}}\mathrm{Hi}'''\left(e^{-\frac{\pi}{2}i}\kappa \left(x-y_c-\frac{ic_i}{u_p'(y_c)}\right)\right) \frac{1}{x-\tilde y_{\hat c}}dx\lesssim \frac{1}{c_r^4}.
\end{align*}
Therefore, one can easily check that all the integratrion terms have the following upper bound
\begin{align*}
 \frac{1}{\varepsilon} |\ln c_0|+\frac{1}{\varepsilon} \frac{1}{c_r}\frac{1}{ \left\langle \kappa c_r \right\rangle^{3}}.
\end{align*}

Combining the above estimates, we have
\begin{align*}
  \left|\varepsilon \alpha^2\pa_{c_r}\left(\pa_yAirySolver_m\left(\pa_y^2\frac{1}{ y-\tilde y_{\hat c}}\right)(-1)\right)\right|\lesssim \alpha^2|\ln c_0|+ \frac{\alpha^2}{c_r}\frac{1}{ \left\langle \kappa c_r \right\rangle^{\frac{3}{2}}}|\ln c_0|.
\end{align*}

Compared with Lemma \ref{lem-acc-psi-1-a}, we also need to take care that the $c_r$-derivative may act on $\phi_{err}$, where $\phi_{err}$ is introduced in Lemma \ref{lem-real-Airy-Solver-S} associated with $f''=i\varepsilon \pa_y^2 \left(\frac{\phi_{1,0}^{[1]}(y)}{u_p-\hat c}\right)$. From the previous analysis, one can see that, without any further refined analysis, taking a $c_r$-derivative leads to at most a loss of order $\varepsilon^{-\frac12}$ in the corresponding bounds. Therefore, from \eqref{est-acc-psi-1-a-err}, we have
\begin{align}\label{est-acc-psi-1-a-err-cr}
  \left\|\pa_{c_r}\phi_{err}\right\|_{L^\infty}+\left\|\left(u_p-\hat c\right)\pa_{c_r}\phi_{err}'\right\|_{L^\infty}\lesssim \varepsilon^{-\frac{1}{6}}c_r\alpha^2|\ln c_0|.
\end{align}

For the same reason, we deduce from \eqref{est-acc-psi-1-a-j} that
\begin{equation}\label{est-acc-psi-1-a-j-cr}
  \begin{aligned}    
    \left\|\pa_{c_r}AirySolver \left(i\varepsilon \pa_y^2 \left(\frac{\sum^{\infty}_{j=2} \phi_{1,0}^{[j]}(y)}{u_p-\hat c}\right)\right)\right\|_{L^\infty}  \lesssim \varepsilon^{-\frac{1}{6}}\alpha^4|\ln c_0|^2,\\
    \left|\pa_{c_r}\pa_yAirySolver \left(i\varepsilon \pa_y^2 \left(\frac{\sum^{\infty}_{j=2} \phi_{1,0}^{[j]}(y)}{u_p-\hat c}\right)\right)(-1)\right| \lesssim \varepsilon^{-\frac{1}{6}}\frac{1}{c_r}\alpha^4|\ln c_0|^2.
  \end{aligned}
\end{equation}  
This complete the estimate of this lemma.
\end{proof}

By using the similar techniques, we can also derive the estimates for $\pa_{c_i}\psi_{1,\alpha,a}^{[1]}$, $\pa_{\nu}\psi_{1,\alpha,a}^{[1]}$, and $\pa_{\alpha^2}\psi_{1,\alpha,a}^{[1]}$. 
\begin{lemma}\label{lem-acc-psi-1-a-ci}
  It holds that
  \begin{align}
    \left|\pa_{c_i}\psi_{1,\alpha,a}^{[1]}(-1)\right|\lesssim   \frac{\alpha^2}{ \left\langle \kappa c_r \right\rangle^{\frac{3}{2}}}|\ln c_0|^2+\varepsilon^{-\frac{1}{6}}(c_r+\alpha^2)\alpha^2|\ln c_0|^2,\label{est-acc-psi-1-a-ci-1}\\
    \left\|\pa_{c_i}\psi_{1,\alpha,a}^{[1]}\right\|_{L^\infty}\lesssim \alpha^2|\ln c_0|^2, \label{est-acc-psi-1-a-ci-2}\\ 
    \left\|\pa_{c_i}\psi_{1,\alpha,a}^{[1]}\right\|_{L^1}\lesssim  \frac{\alpha^2}{ \left\langle \kappa c_r \right\rangle^{\frac{3}{2}}}|\ln c_0|^2+\varepsilon^{-\frac{1}{6}}(c_r+\alpha^2)\alpha^2|\ln c_0|^2,\label{est-acc-psi-1-a-ci-3} \\
    \left|\pa_{c_i}\psi_{1,\alpha,a}^{[1]\prime}(-1)\right|\lesssim  \frac{\alpha^2}{c_r \left\langle \kappa c_r \right\rangle^{\frac{3}{2}}}|\ln c_0|^2+\varepsilon^{-\frac{1}{6}}\frac{c_r+\alpha^2}{c_r}\alpha^2|\ln c_0|^2.\label{est-acc-psi-1-a-ci-1-dx}
  \end{align}
  For fixed $\alpha$, we have 
    \begin{align}
    \left\|\pa_{\nu}\psi_{1,\alpha,a}^{[1]}\right\|_{L^\infty}\lesssim \frac{\varepsilon^{\frac{1}{3}}}{\nu} \alpha^2|\ln c_0|^2, \label{est-acc-psi-1-a-nu-2}\\ 
    \left|\pa_{\nu}\psi_{1,\alpha,a}^{[1]\prime}(-1)\right|\lesssim \frac{\varepsilon^{\frac{1}{3}}}{\nu} \frac{\alpha^2}{c_r}|\ln c_0|^2.\label{est-acc-psi-1-a-nu-1-dx}
  \end{align}
  For fixed $\nu$ or $\varepsilon$, we have
      \begin{align}
    \left\|\pa_{\alpha^2}\psi_{1,\alpha,a}^{[1]}\right\|_{L^\infty}\lesssim \varepsilon^{\frac{1}{3}}|\ln c_0|^2, \label{est-acc-psi-1-a-alpha-2}\\ 
    \left|\pa_{\alpha^2}\psi_{1,\alpha,a}^{[1]\prime}(-1)\right|\lesssim \varepsilon^{\frac{1}{3}}\frac{1}{c_r}|\ln c_0|^2.\label{est-acc-psi-1-a-alpha-1-dx}
  \end{align}
\end{lemma}
\begin{proof}
  The estimates for $\pa_{c_i}\psi_{1,\alpha,a}^{[1]}$ are obtained in the same
way as those for $\pa_{c_r}\psi_{1,\alpha,a}^{[1]}$.

We next study $\pa_{\nu}\psi_{1,\alpha,a}^{[1]}$ with $\alpha$ fixed. As in
Lemma \ref{lem-acc-psi-1-a-cr}, the main term to be considered is
\begin{align}\label{eq-acc-psi-1-a-nu-m}
  \pa_{\nu}
  AirySolver_m\left(
    \varepsilon\alpha^2\pa_y^2\frac{1}{y-\tilde y_{\hat c}}
  \right)(y).
\end{align}

  Recall that $\varepsilon=\frac{\nu}{\alpha}$ and $\kappa=\left(\frac{\varepsilon}{u_p'(y_c)}\right)^{-\frac{1}{3}}$. For fixed $\alpha$, we have
\begin{align*}
  \pa_\nu\varepsilon=\frac{1}{\nu}\varepsilon, \quad \pa_\nu\kappa=-\frac{1}{3}\frac{1}{\nu}\kappa.
\end{align*}
Therefore, when the $\nu$-derivative falls on the coefficient $i \frac{2\pi}{\left(\eta_r'(x)\right)^{\frac{1}{2}}\kappa\varepsilon}$ in \eqref{eq-Green-Airy}, the resulting estimate differs from that in
Lemma \ref{lem-acc-psi-1-a} only by an additional factor $\nu^{-1}$.

It remains to consider the case where the $\nu$-derivative falls on the Airy
factors in \eqref{eq-Green-Airy}. Since $\eta$ is independent of $\nu$, we have
\begin{align*}
  \pa_{\nu}\mathrm{Ai}(k,e^{\frac{\pi}{6}i}\kappa\eta(y))=-\frac{e^{\frac{\pi}{6}i}}{3\nu} \kappa\eta(y)\mathrm{Ai}(k-1,e^{\frac{\pi}{6}i}\kappa\eta(y))\sim \frac{\left\langle \kappa\eta(y) \right\rangle^{\frac{3}{2}}}{\nu}\left\langle \kappa\eta(y) \right\rangle^{-\frac{1}{4}-\frac{1}{2}k}e^{-\frac{2}{3}\frac{1}{\sqrt2}|\kappa\eta(y)|^{\frac{1}{2}}\kappa\eta_r(y)},\\
    \pa_{\nu}\mathrm{Hi}^{(j)}(e^{-\frac{\pi}{2}i}\kappa \eta(y))=-\frac{e^{-\frac{\pi}{2}i}}{3\nu} \kappa\eta(y)\mathrm{Hi}^{(j+1)}(e^{-\frac{\pi}{2}i}\kappa \eta(y))\sim \frac{1}{\nu}\frac{1}{\left\langle \kappa\eta(y) \right\rangle^{j+1}}.
\end{align*}
The discussion above shows that the main difficulty in estimating
$\pa_{c_r}\psi_{1,\alpha,a}^{[1]}$ comes from the nonlocal terms in $E(x,y)$,
for example from $\pa_{c_r}a_1''(x)$. In contrast, for
$\pa_{\nu}\psi_{1,\alpha,a}^{[1]}$, the main difficulty comes from the
$\nu$-dependence of the kernel $G_{A,M}$. A typical term is
\begin{align*}
  \partial_{\nu}\big(A_1(2,y) A_2''(x)\big)  = \frac{e^{-\frac{\pi}{3}i}\left(\eta_r'(x)\right)^2}{\left(\eta_r'(y)\right)^{\frac{5}{2}}}\partial_{\nu} \left(\mathcal A_1(2,\kappa \eta(y))\mathrm{Ai}''\left(e^{\frac{5\pi}{6}i}\kappa \eta(x)\right)\right)  + err.
\end{align*}
Here $err$ denotes terms that are easier to estimate. Since $\eta_r$ is
independent of $\nu$, applying the same argument as in
Lemma \ref{lem-d-cr-a1B} gives
\begin{align*}
  \left|\pa_{\nu}\left(A_1(2,y) A_2''(x)\right)\right|\lesssim& \left(\frac{\left\langle \kappa\eta(x) \right\rangle^{\frac{3}{4}}}{\left\langle \kappa\eta(y) \right\rangle^{\frac{3}{4}}} +\frac{\left\langle \kappa\eta(x) \right\rangle^{\frac{5}{4}}}{\left\langle \kappa\eta(y) \right\rangle^{\frac{5}{4}}} \right)\frac{\kappa \left| \eta_r(x)-\eta_r(y)\right| }{\nu}  e^{-\frac{1}{C}\left|\kappa\eta_r(x)-\kappa\eta_r(y)  \right|\left( |\kappa\eta(y)|^{\frac{1}{2}}+|\kappa\eta(x)|^{\frac{1}{2}} \right)}\\
  &+\frac{\left\langle \kappa\eta(x) \right\rangle^{\frac{3}{4}}}{\left\langle \kappa\eta(y) \right\rangle^{\frac{5}{4}}}\frac{1}{\nu}  e^{-\frac{1}{C}\left|\kappa\eta_r(x)-\kappa\eta_r(y)  \right|\left( |\kappa\eta(y)|^{\frac{1}{2}}+|\kappa\eta(x)|^{\frac{1}{2}} \right)}.
\end{align*}
Compared with \eqref{eq-I-GM,0,c_r}, this estimate contains an additional
factor $\kappa \left| \eta_r(x)-\eta_r(y)\right|$ which may appear problematic at first sight. However, this factor is harmless.
Indeed, if $|x-y|\le \kappa^{-1}$, then $\kappa \left| \eta_r(x)-\eta_r(y)\right| \lesssim 1$. On the other hand, if $|x-y|\ge \kappa^{-1}$, the exponential decay \newline $e^{-\frac{1}{C}\left|\kappa\eta_r(x)-\kappa\eta_r(y)  \right|\left( |\kappa\eta(y)|^{\frac{1}{2}}+|\kappa\eta(x)|^{\frac{1}{2}} \right) }$ dominates the polynomial growth coming from this extra factor.

Consequently, the main term in \eqref{eq-acc-psi-1-a-nu-m} satisfies estimates of the form \eqref{est-acc-psi-1-a-nu-2} and \eqref{est-acc-psi-1-a-nu-1-dx}.

More generally, the above analysis
shows that after taking one $\nu$-derivative, even under the rough estimates of the type in Lemma \ref{lem-est-ASS}, we lose at most a factor
$\frac{\kappa^{1/2}}{\nu}$. Therefore, by applying the same argument as in
Lemma \ref{lem-acc-psi-1-a-cr}, we obtain the desired estimate for
$\pa_{\nu}\psi_{1,\alpha,a}^{[1]}$.

Finally, the estimate for $\pa_{\alpha^2}\psi_{1,\alpha,a}^{[1]}$ with $\nu$
fixed can be derived in exactly the same way. If instead $\varepsilon$ is fixed,
the proof becomes simpler, and in fact one obtains better estimates.
\end{proof}

\subsubsection{Refined estimates of $\psi_{1,\alpha,b}^{[1]}$}
In this subsubsection, we derive estimates for $\psi_{1,\alpha,b}^{[1]}$. This term is less singular and hence less delicate than $\psi_{1,\alpha,a}^{[1]}$.
\begin{lemma}\label{lem-acc-psi-1-b}
   For all $y\in[-1,0]$, it holds that
  \begin{align}\label{est-acc-psi-1-b}
    \left|\psi_{1,\alpha,b}^{[1]}(y)\right|\lesssim c_0|\ln c_0| ,\quad  \left|\psi_{1,\alpha,b}^{[1]\prime}(-1)\right|\lesssim \frac{c_0}{c_r}|\ln c_0|.
  \end{align}
  Moreover,  the following refined pointwise estimate holds: 
  \begin{align}\label{est-acc-psi-1-b-2}
    \left|\psi_{1,\alpha,b}^{[1]}(y)\right|\lesssim c_0|\ln c_0||y-y_c|+\left(c_r+\alpha^2+\kappa c_r |c_i|\right)c_0|\ln c_0|^2+c_rc_0.
  \end{align}
\end{lemma}
\begin{proof}
    From the analysis in Section \ref{ssub:accurate_estimate_of_psi_1_alpha_a}, one can see that 
  \begin{align*}
    \frac{u_p''\phi_{1,\alpha}}{u_p-\hat c}=-2-2\frac{\phi_{1,0}^{[1]}(y)}{u_p(y)-\hat c}+err,\quad \text{ where }\left\|\phi_{1,0}^{[1]}(y)\right\|_{X_1}\lesssim \alpha^2|\ln c_0|.
  \end{align*} 
Therefore, by using Lemma \ref{lem-est-AS} and Lemma \ref{lem-est-AS-ns}, it is clear that
\begin{align*}
  \left\|AirySolver_m(c_0\frac{u_p''\phi_{1,\alpha}}{u_p-\hat c})(y)\right\|_{L^\infty}\lesssim c_0|\ln c_0|,\ \left\|\pa_yAirySolver_m(c_0\frac{u_p''\phi_{1,\alpha}}{u_p-\hat c})(-1)\right\|_{L^\infty}\lesssim \frac{c_0}{c_r}|\ln c_0|.
\end{align*}
Then by using the same technique in Lemma \ref{lem-acc-psi-1-a} we get \eqref{est-acc-psi-1-b}.

Next, we need more precise estimate, and refine the estimate of Lemma \ref{lem-est-AS-ns}.

The main term of $\frac{u_p''\phi_{1,\alpha}}{u_p-\hat c}$ is $-2$, then following the estimate of \eqref{eq-est-AS-0}, one can easily check that
\begin{align*}
  \left|\int^0_{-1}G_{A,M}(x,y)\cdot2dx\right|\lesssim  \frac{1}{\kappa}.
\end{align*}
Next, we turn to the non-local term $E(x,y)$. It follows from Lemma \ref{lem-est-a1-a2} that
\begin{align*}
  \left|\int^y_{-1}\frac{1}{\kappa\varepsilon}a_2(x)\cdot 2dx\right|\lesssim  \frac{1}{\kappa}.
\end{align*}
For the same reason to Lemma \ref{lem-est-AS-ns}, we have
\begin{align*}
  \left|\int^0_{y}\frac{1}{\kappa\varepsilon}a_1(x)(y-y_c)\cdot 2dx\right|\lesssim  |\ln c_0||y-y_c|.
\end{align*}
For the rest term $\int^y_{-1}\frac{1}{\left(\eta_r'(x)\right)^{\frac{1}{2}}\kappa\varepsilon}a_1(x)(x-y_c)\cdot 2dx$, by taking
\begin{align*}
  C^*=\int^{y_c}_{-1}i \frac{2\pi}{\left(\eta_r'(x)\right)^{\frac{1}{2}}\kappa\varepsilon}a_1(x)(x-y_c)\cdot 2dx.
\end{align*}
It is clear that $\left|C^*\right|\lesssim c_r$. Then we have
\begin{align*}
  \left| \int^y_{-1}i \frac{2\pi}{\left(\eta_r'(x)\right)^{\frac{1}{2}}\kappa\varepsilon}a_1(x)(x-y_c)\cdot 2dx-C^* \right|\lesssim |y-y_c|.
\end{align*}
Thus by following the technical of Lemma \ref{lem-acc-psi-1-a}, we have
\begin{align*}
  \left|AirySolver_m(c_0\frac{u_p''\phi_{1,\alpha}}{u_p-\hat c})(y)-C^*c_0\right|\lesssim c_0|\ln c_0||y-y_c|+c_0\alpha^2|\ln c_0|^2.
\end{align*}
Then, the estimate \eqref{est-acc-psi-1-b-2} follows from \eqref{eq-real-Airy-Solver-err}.

\end{proof}
\begin{lemma}\label{lem-acc-psi-1-b-cr}
   The following parameter-derivative estimates hold:
    \begin{align}\label{est-acc-psi-1-b-cr}
    \left\|\pa_{c_r}\psi_{1,\alpha,b}^{[1]}\right\|_{L^\infty}\lesssim \kappa c_r(\varepsilon^{\frac{1}{2}}+|c_i|)|\ln c_0|^2, \quad \left|\pa_{c_r}\psi_{1,\alpha,b}^{[1]\prime}(-1)\right|\lesssim \kappa (\varepsilon^{\frac{1}{2}}+|c_i|)|\ln c_0|^2,
  \end{align}
  and
      \begin{align}\label{est-acc-psi-1-b-ci}
    \left\|\pa_{c_i}\psi_{1,\alpha,b}^{[1]}\right\|_{L^\infty}\lesssim \kappa c_r(\varepsilon^{\frac{1}{2}}+|c_i|)|\ln c_0|^2, \quad \left|\pa_{c_i}\psi_{1,\alpha,b}^{[1]\prime}(-1)\right|\lesssim \kappa (\varepsilon^{\frac{1}{2}}+|c_i|)|\ln c_0|^2.
  \end{align}
  For fixed $\alpha$, it holds that
  \begin{align}\label{est-acc-psi-1-b-nu}
    \left\|\pa_{\nu}\psi_{1,\alpha,b}^{[1]}\right\|_{L^\infty}\lesssim \frac{c_0|\ln c_0|^2}{\nu}  , \quad \left|\pa_{\nu}\psi_{1,\alpha,b}^{[1]\prime}(-1)\right|\lesssim \frac{c_0|\ln c_0|^2}{\nu c_r}.
  \end{align}
  For fixed $\nu$ or $\varepsilon$, it holds that
    \begin{align}\label{est-acc-psi-1-b-al}
    \left\|\pa_{\alpha^2}\psi_{1,\alpha,b}^{[1]}\right\|_{L^\infty}\lesssim \frac{c_0|\ln c_0|^2}{\alpha^2} , \quad \left|\pa_{\alpha^2}\psi_{1,\alpha,b}^{[1]\prime}(-1)\right|\lesssim\frac{c_0|\ln c_0|^2}{\alpha^2 c_r} .
  \end{align}
\end{lemma}
\begin{proof}
We first study $\pa_{c_r}\psi_{1,\alpha,b}^{[1]}$. We again follow the argument used in Lemma \ref{lem-acc-psi-1-a-cr}. The worst contribution arises from the term $\partial_{c_r}\mathrm{AirySolver}_m(c_0)$, in particular, for $-1\le y\le 2y_c+1$,
\begin{align*}
  \left|  \frac{1}{\kappa\varepsilon}\int^0_{y}\left(\pa_{c_r}a_{1,M}(x)\right)\left(y-y_c\right)c_0 dx\right|\lesssim  \left|  \frac{\left|y-y_c\right|}{\kappa\varepsilon}\int^0_{y}\frac{c_0}{\left\langle \kappa\eta(x) \right\rangle^2}dx\right|\lesssim c_0\varepsilon^{-\frac{1}{3}}\left|y-y_c\right|\lesssim c_0\kappa c_r.
\end{align*}

For $\pa_{c_r}\psi_{1,\alpha,b}^{[1]\prime}(-1)$, similarly, the worst contribution is 
\begin{align*}
  \left|  \frac{1}{\kappa\varepsilon}\int^0_{-1}\left(\pa_{c_r}a_{1,M}(x)\right) c_0 dx\right|\lesssim  \left|  \frac{1}{\kappa\varepsilon}\int^0_{-1}\frac{c_0}{\left\langle \kappa\eta(x) \right\rangle^2}dx\right|\lesssim  c_0\kappa.
\end{align*}

Following the proof of Lemma \ref{lem-real-Airy-Solver}, for the $\phi_{err}$ associated with $f=c_0$, we have
\begin{align}\label{est-acc-psi-1-b-err}
  \left\|\phi_{err}\right\|_{L^\infty}+\left\|\left(u_p-\hat c\right)\phi_{err}'\right\|_{L^\infty}\lesssim \kappa c_r(\varepsilon^{\frac{1}{2}}+|c_i|)|\ln c_0|^2 c_0 .
\end{align}

Similar to Lemma \ref{lem-acc-psi-1-a-cr}, without any further refined analysis, taking a $c_r$-derivative leads to at most a loss of order $\varepsilon^{-\frac12}$ in the corresponding bounds. Then we have
\begin{align}\label{est-acc-psi-1-b-err-cr}
  \left\|\pa_{c_r}\phi_{err}\right\|_{L^\infty}+\left\|\left(u_p-\hat c\right)\pa_{c_r}\phi_{err}'\right\|_{L^\infty}\lesssim \kappa c_r(\varepsilon^{\frac{1}{2}}+|c_i|)|\ln c_0|^2.
\end{align}
For $|c_i|\sim \varepsilon^{\frac{1}{3}}$, we can see the above estimate is in the scale of $c_r|\ln c_0|^2$, which is still much smaller than $1$. However, we only concern the derivative estimate on the upper branch, and for that case, $c_i=0$, and the upper bound is in the scale of  $\varepsilon^{\frac{1}{6}} c_r|\ln c_0|^2$.

The other estimates can be proved in the same way.
\end{proof}

\subsubsection{Refined estimates of $\phi_{1,\alpha}^{[1]}$}
Using the estimates for $\psi^{[1]}_{1,\alpha}$ obtained in the previous two subsections, we now estimate the next Rayleigh correction $\phi_{1,\alpha}^{[1]}=-RaySolver_{\alpha} \left(Reg(\psi^{[1]}_{1,\alpha})\right)$. 
\begin{lemma}\label{lem-acc-phi-1}
  It holds that
  \begin{align}
    \left\|\phi_{1,\alpha}^{[1]}\right\|_{L^\infty}+\left\|\phi_{1,\alpha}^{[1]\prime}\right\|_{L^\infty} \lesssim(\varepsilon^{\frac{1}{3}}\alpha^2+c_0)|\ln c_0|^2, \label{est-acc-phi-1}\\
    \left|\pa_{c_r}\phi_{1,\alpha}^{[1]}(-1)\right|+\left|\pa_{c_r}\phi_{1,\alpha}^{[1]\prime}(-1)\right|  \lesssim \left( \frac{\alpha^2}{ \left\langle \kappa c_r \right\rangle^{\frac{3}{2}}} +\varepsilon^{-\frac{1}{6}}(c_r+\alpha^2)\alpha^2 +\kappa c_r(\varepsilon^{\frac{1}{2}}+|c_i|)\right)|\ln c_0|^3,\label{est-acc-phi-1-cr}\\
    \left|\pa_{c_i}\phi_{1,\alpha}^{[1]}(-1)\right|+\left|\pa_{c_i}\phi_{1,\alpha}^{[1]\prime}(-1)\right|  \lesssim \left( \frac{\alpha^2}{ \left\langle \kappa c_r \right\rangle^{\frac{3}{2}}} +\varepsilon^{-\frac{1}{6}}(c_r+\alpha^2)\alpha^2 +\kappa c_r(\varepsilon^{\frac{1}{2}}+|c_i|)\right)|\ln c_0|^3.\label{est-acc-phi-1-ci}
  \end{align}
  For fixed $\alpha$, it holds that
      \begin{align}
    \left|\pa_{\nu}\phi_{1,\alpha}^{[1]}(-1)\right|+\left|\pa_{\nu}\phi_{1,\alpha}^{[1]\prime}(-1)\right|\lesssim \frac{\varepsilon^{\frac{1}{3}}\alpha^2+c_0}{\nu} |\ln c_0|^2.\label{est-acc-phi-1-nu}
  \end{align}
  For fixed $\nu$ or $\varepsilon$, it holds that
      \begin{align}
    \left|\pa_{\alpha^2}\phi_{1,\alpha}^{[1]}(-1)\right|+\left|\pa_{\alpha^2}\phi_{1,\alpha}^{[1]\prime}(-1)\right|\lesssim\frac{\varepsilon^{\frac{1}{3}}\alpha^2+c_0}{\alpha^2} |\ln c_0|^2.\label{est-acc-phi-1-alpha}
  \end{align}
\end{lemma}
\begin{proof}
The estimate \eqref{est-acc-phi-1} follows directly from Lemma \ref{lem-Raysolver-alpha}, Lemma \ref{lem-acc-psi-1-a}, and Lemma \ref{lem-acc-psi-1-b}. 

 Next, we focus on the $c_r$-derivative. Recall the \eqref{eq-op-solver-Ray0} the definition of $RaySolver_0$, we have
  \begin{align*}
    RaySolver_0 \left(Reg(\psi^{[1]}_{1,\alpha})\right)(-1)=-\phi_{2,0}(-1)\int^0_{-1}Reg(\psi^{[1]}_{1,\alpha})(x)dx,\\
    \pa_y RaySolver_0 \left(Reg(\psi^{[1]}_{1,\alpha})\right)(-1)=-\phi_{2,0}'(-1)\int^0_{-1}Reg(\psi^{[1]}_{1,\alpha})(x)dx.
  \end{align*}
  By the definition of $\phi_{2,0}$, we have
\begin{align*}
  \left|\pa_{c_r}\phi_{2,0}(-1)\right|\lesssim |\ln c_0|,\quad \left|\pa_{c_r}\phi_{2,0}'(-1)\right|\lesssim \frac{1}{c_r}.
\end{align*}
It follows from  Lemma \ref{lem-acc-psi-1-a-cr} and Lemma \ref{lem-acc-psi-1-b-cr} that 
\begin{align*}
  &\left|\pa_{c_r}RaySolver_0 \left(Reg(\psi^{[1]}_{1,\alpha})\right)(-1)\right|+\left|\pa_{c_r}\pa_yRaySolver_0 \left(Reg(\psi^{[1]}_{1,\alpha})\right)(-1)\right|\\
  \lesssim&\left(\frac{\varepsilon^{\frac{1}{3}}\alpha^2+c_0}{c_r} +\frac{\alpha^2}{ \left\langle \kappa c_r \right\rangle^{\frac{3}{2}}}+ \varepsilon^{-\frac{1}{6}}(c_r+\alpha^2)\alpha^2+ \kappa c_r(\varepsilon^{\frac{1}{2}}+|c_i|)\right)|\ln c_0|^3.
\end{align*}

Recall Lemma \ref{lem-Raysolver-alpha} that $RaySolver_\alpha=\sum^{\infty}_{j=0} S_{R,j}$, and 
\begin{align*}
  S_{R,1}(Reg(\psi^{[1]}_{1,\alpha}))=&RaySolver_0 \left(\alpha^2\left(u_p(y)-\hat c\right)RaySolver_0 \left(Reg(\psi^{[1]}_{1,\alpha})\right)\right).
\end{align*}
The higher-order terms enjoy additional smallness. Moreover, after taking one $c_r$-derivative, we lose at most a factor of $\varepsilon^{-1/2}$. Hence it suffices to treat carefully the contribution coming from
$\psi^{[1]}_{1,\alpha,b}$. The potentially singular contribution is
\begin{align*}
  &\pa_{c_r}\left(\phi_{2,0}(-1)\int^0_{-1}\alpha^2\left(u_p(x)-\hat c\right)\phi_{1,0}(x)\int^x_{-1}\phi_{2,0}(z)\frac{Reg(\psi^{[1]}_{1,\alpha,b})(z)}{\left(u_p(z)-\hat c\right)}dz dx\right)\\
  =&\phi_{2,0}(-1)\int^0_{-1}\alpha^2\left(u_p(x)-\hat c\right)^2\int^x_{-1}Reg(\psi^{[1]}_{1,\alpha,b})(z)\pa_{c_r}\frac{\phi_{2,0}(z)}{\left(u_p(z)-\hat c\right)}dz dx+err\\
  =&\phi_{2,0}(-1)\int^0_{-1}\alpha^2\left(u_p(x)-\hat c\right)^2\int^x_{-1} \frac{\frac{-z}{1-\hat c}\psi^{[1]}_{1,\alpha,b}(z)}{\left(u_p(z)-\hat c\right)^2}dz dx+err,
\end{align*}
where $err$ denotes terms that can be estimated directly and are therefore
harmless. The main term above shows that when the $c_r$-derivative falls on
$\left(u_p(z)-\hat c\right)^{-1}$, it may produce an amplification at the
scale $1/c_0$. This is precisely the difficulty that requires the improved
estimate of the form \eqref{est-acc-psi-1-b-2}.

From \eqref{est-acc-psi-1-b-2} we can see that
\begin{align*}
  \left|\phi_{2,0}(-1)\int^0_{-1}\alpha^2\left(u_p(x)-\hat c\right)^2\int^x_{-1} \frac{\frac{-z}{1-\hat c}\psi^{[1]}_{1,\alpha,b}(z)}{\left(u_p(z)-\hat c\right)^2}dz dx\right|\lesssim \alpha^2 c_0 |\ln c_0|^2+\left(c_r+\alpha^2+\kappa c_r |c_i|\right)\alpha^2   |\ln c_0|^2.
\end{align*}
 
Combining the above estimates, we have
\begin{align*}
  \left|\pa_{c_r}S_{R,1}(Reg(\psi^{[1]}_{1,\alpha}))(-1)\right|\lesssim \left( \varepsilon^{-\frac{1}{6}}\alpha^2|\ln c_0|+ \frac{c_0}{c_r}+  c_0 |\ln c_0|^2+\left(c_r+\alpha^2+\kappa c_r |c_i|\right)  |\ln c_0|^2 \right)\alpha^2.
\end{align*}
Similarly, we also have
\begin{align*}
  \left|\pa_{c_r}S_{R,1}'(Reg(\psi^{[1]}_{1,\alpha}))(-1)\right|\lesssim \left( \varepsilon^{-\frac{1}{6}}\alpha^2|\ln c_0|+ \frac{c_0}{c_r}+  c_0 |\ln c_0|^2+\left(c_r+\alpha^2+\kappa c_r |c_i|\right)  |\ln c_0|^2 \right)\alpha^2 |\ln c_0|.
\end{align*}
For $S_{R,j}$ with $j>1$, we have better estimates. Then we deduce \eqref{est-acc-phi-1} and \eqref{est-acc-phi-1-cr}. The rest estimates can be deduced in similar way.

\end{proof}

\subsubsection{Boundary expansion of $\Phi_{1,\alpha}$}
Recall that $\Phi_{1,\alpha}
=\phi_{1,\alpha}
+\sum_{j=1}^{+\infty}\left(\phi_{1,\alpha}^{[j]}+\psi_{1,\alpha}^{[j]}\right)$. In the previous subsections, we obtained sharp estimates for the first correction terms $\phi_{1,\alpha}^{[1]}$ and $\psi_{1,\alpha}^{[1]}$. The same arguments give stronger bounds for $\phi_{1,\alpha}^{[j]}$ and $\psi_{1,\alpha}^{[j]}$ when $j>1$. Combining these bounds with Lemma \ref{lem-sol-Ray-hom}, we conclude that the leading boundary behavior of $\Phi_{1,\alpha}$ is determined by the Rayleigh component $\phi_{1,\alpha}$, up to the remainders estimated above.

Moreover, the corresponding remainder terms, together with their parameter derivatives, satisfy sufficiently strong bounds. These estimates will be used in the analysis of the neutral curve.

\subsection{Expansion of  $\frac{\Phi_{1,\alpha}(-1)}{\Phi_{1,\alpha}'(-1)}$}
In this subsection, we derive an expansion for $\frac{\Phi_{1,\alpha}(-1)}{\Phi_{1,\alpha}'(-1)}$. This expansion is one of the key steps in the analysis of the neutral curve.
\begin{lemma}\label{lem-res-ray}
  It holds that 
  \begin{align*}
    \frac{\Phi_{1,\alpha}(-1)}{\Phi_{1,\alpha}'(-1)}=\frac{\phi_{1,\alpha}(-1)}{\phi_{1,\alpha}'(-1)}+Res_{ray},
  \end{align*} 
  where the remainder $Res_{ray}$ satisfies
  \begin{align}\label{est-res-ray}
    \left|Res_{ray}\right|\lesssim \left(1+\frac{\alpha^2}{c_r}\right)\left(\varepsilon^{\frac{1}{3}}\alpha^2+c_0 \right)|\ln c_0|^2,
  \end{align}
  \begin{align}\label{est-res-ray-cr}
    \left|\pa_{c_r}Res_{ray}\right|\lesssim \left(1+\frac{\alpha^2}{c_r}\right)\left( \frac{\alpha^2}{ \left\langle \kappa c_r \right\rangle^{\frac{3}{2}}} +\varepsilon^{-\frac{1}{6}}(c_r+\alpha^2)\alpha^2 +\kappa c_r(\varepsilon^{\frac{1}{2}}+|c_i|)\right)|\ln c_0|^3,
  \end{align}
    \begin{align}\label{est-res-ray-ci}
    \left|\pa_{c_i}Res_{ray}\right|\lesssim \left(1+\frac{\alpha^2}{c_r}\right)\left( \frac{\alpha^2}{ \left\langle \kappa c_r \right\rangle^{\frac{3}{2}}} +\varepsilon^{-\frac{1}{6}}(c_r+\alpha^2)\alpha^2 +\kappa c_r(\varepsilon^{\frac{1}{2}}+|c_i|)\right)|\ln c_0|^3.
  \end{align}
  For fixed $\alpha$, it holds that
     \begin{align}\label{est-res-ray-nu}
    \left|\pa_{\nu}Res_{ray}\right|\lesssim \frac{\varepsilon^{\frac{1}{3}}\alpha^2+c_0}{\nu} |\ln c_0|^2\left(1+\frac{\alpha^2}{c_r}\right) .
  \end{align} 
    For fixed $\nu$ or $\varepsilon$, it holds that
     \begin{align}\label{est-res-ray-alpha}
    \left|\pa_{\alpha^2}Res_{ray}\right|\lesssim \frac{\varepsilon^{\frac{1}{3}}\alpha^2+c_0}{\alpha^2} |\ln c_0|^2\left(1+\frac{\alpha^2}{c_r}\right) .
  \end{align} 
\end{lemma}
\begin{proof}
  Recall that $\phi_{1,\alpha}(-1)=O(c_r+\alpha^2)$ and $\phi_{1,\alpha}'(-1)=O(1)$. We write
  \begin{align*}
    \frac{\Phi_{1,\alpha}(-1)}{\Phi_{1,\alpha}'(-1)}=&\frac{\phi_{1,\alpha}(-1)+\sum_{j=1}^{+\infty}\left(\phi_{1,\alpha}^{[j]}+\psi_{1,\alpha}^{[j]}\right)(-1)}{\phi_{1,\alpha}'(-1)+\sum_{j=1}^{+\infty}\left(\phi_{1,\alpha}^{[j]\prime}+\psi_{1,\alpha}^{[j]\prime}\right)(-1)}=\frac{\phi_{1,\alpha}(-1)}{\phi_{1,\alpha}'(-1)}+Res_{ray},
  \end{align*}
where
\begin{align*}
  Res_{ray}=O\left( \sum_{j=1}^{+\infty}\left(\phi_{1,\alpha}^{[j]}+\psi_{1,\alpha}^{[j]}\right)(-1) \right)  +O\left((c_r+\alpha^2)\sum_{j=1}^{+\infty}\left(\phi_{1,\alpha}^{[j]\prime}+\psi_{1,\alpha}^{[j]\prime}\right)(-1) \right).
\end{align*}
Therefore, \eqref{est-res-ray} follows from Lemma \ref{lem-acc-psi-1-a}, Lemma \ref{lem-acc-psi-1-b}, and Lemma \ref{lem-acc-phi-1}.

From Lemma \ref{lem-sol-Ray-hom} and Lemma \ref{lem-Ray-endpoint-expansion}, we have $\left|\pa_{c_r}\phi_{1,\alpha}(-1)\right|\lesssim 1$ and $\left|\pa_{c_r}\phi_{1,\alpha}'(-1)\right|\lesssim \frac{\alpha^2}{|\hat c|}$. By Lemma \ref{lem-acc-psi-1-a-cr}, Lemma \ref{lem-acc-psi-1-b-cr}, and Lemma \ref{lem-acc-phi-1}, we have deduce \eqref{est-res-ray-cr}.

The rest estimates can be get in the same way.
\end{proof}

\subsection{Value of $\Phi_{1,\alpha}$ at $y=0$}
In this subsection, we record the boundary estimates for $\Phi_{1,\alpha}$ at $y=0$.
\begin{lemma}\label{lem-Phi-0}
  It holds that
  \begin{align*}
    \Phi_{1,\alpha}'(0)=0,\quad \left|\Phi_{1,\alpha}'''(0)\right|\lesssim|\ln c_0|    .
  \end{align*}
\end{lemma}
\begin{proof}
  As $\Phi_{1,\alpha}$ is constructed from Rayleigh-Airy iteration, by the properties of $RaySolver_\alpha$ and $AirySolver$, it is clear that
\begin{align}\label{eq-est-phi-d-y0}
  \phi_{1,\alpha}^{[j]\prime}(0)=\psi_{1,\alpha}^{[j]\prime}(0)=0, \text{ for }j\in \mathbb Z_+.
\end{align}
Therefore,
\begin{align*}
  \Phi_{1,\alpha}'(0)=\phi_{1,\alpha}'(0)=0.
\end{align*}
Using Lemma \ref{lem-Raysolver0}, we deduce from \eqref{eq-est-phi-d-y0} that
\begin{align*}
  \phi_{1,\alpha}'''(0)=\phi_{1,\alpha}^{[j]\prime\prime\prime}(0)=0.
\end{align*}
It follows that
\begin{align*}
  \Phi_{1,\alpha}'''(0)=\sum_{j=1}^{+\infty} \psi_{1,\alpha}^{[j]\prime\prime\prime}(0) .
\end{align*}
By using Lemma \ref{lem-real-Airy-Solver}, Lemma \ref{lem-real-Airy-Solver-S}, one can easily check that
\begin{align*}
  \left|\psi_{1,\alpha,a}^{[1]\prime\prime\prime}(0)\right|\lesssim \alpha^2|\ln c_0|,\quad \left|\psi_{1,\alpha,b}^{[1]\prime\prime\prime}(0)\right|\lesssim |\ln c_0|.
\end{align*}
This finishes the proof of this lemma.
\end{proof}
We also have the following rough derivative estimates.
\begin{lemma}\label{lem-Phi1-0-cr}
   It holds that
  \begin{align*}
    \left|\pa_{c_r}\Phi_{1,\alpha}'''(0)\right|\lesssim \kappa^{\frac{3}{2}}|\ln c_0|,
  \end{align*}
    \begin{align*}
   \left|\pa_{c_i}\Phi_{1,\alpha}'''(0)\right|\lesssim \kappa^{\frac{3}{2}}|\ln c_0|.
  \end{align*}
    For fixed $\alpha$, it holds that
     \begin{align*}
   \left|\pa_{\nu}\Phi_{1,\alpha}'''(0)\right|\lesssim \frac{1}{\nu}\kappa^{\frac{1}{2}} |\ln c_0|.
  \end{align*} 
    For fixed $\nu$ or $\varepsilon$, it holds that
     \begin{align*}
    \left|\pa_{\alpha^2}\Phi_{1,\alpha}'''(0)\right|\lesssim \frac{1}{\alpha^2}\kappa^{\frac{1}{2}} |\ln c_0|.
  \end{align*} 
\end{lemma}
The proof follows from the same arguments as those used in Lemma \ref{lem-est-AS}, Lemma \ref{lem-est-ASS}, Lemma \ref{lem-acc-psi-1-a-cr}, and Lemma \ref{lem-acc-psi-1-a-ci}. In fact, sharper estimates can be obtained, but the rough bounds stated above are sufficient for the proof of the main theorem.

\subsection{Boundary values of $\Phi_{2,\alpha}$ at $y=-1$ and $y=0$}
In this subsection, we analyze the boundary values of $\Phi_{2,\alpha}$. 
Since $\Phi_{2,\alpha}$ plays a secondary role in the matching of the dispersion relation, the following rough bounds are sufficient for our purposes.

\begin{lemma}
  It holds that
  \begin{align}\label{est-Phi2--1}
    \left|\Phi_{2,\alpha}(-1)\right|\lesssim 1,\quad \left|\Phi_{2,\alpha}'(-1)\right|\lesssim |\ln c_0|,
  \end{align}
    \begin{align}\label{est-Phi2-0-ddd}
    \left|\Phi_{2,\alpha}'''(0)\right|\lesssim |\ln c_0|.
  \end{align}
    \begin{align}\label{est-Phi2-0-d}
     \Phi_{2,\alpha}'(0)=\frac{1}{1-\hat c}.
  \end{align}
\end{lemma}
\begin{proof}
Similar to $\Phi_{1,\alpha}$, we have the expansion for $\Phi_{2,\alpha}$ that
\begin{align*}
  \Phi_{2,\alpha}=\phi_{2,\alpha}+\sum_{j=1}^{+\infty}\left(\phi_{2,\alpha}^{[j]}+\psi_{2,\alpha}^{[j]}\right)
\end{align*}
From  Lemma \ref{lem-sol-Ray-hom}, it is clear that $\phi_{2,\alpha}(-1)\sim-1$ and $\phi_{2,\alpha}'(-1)\sim1$.  By using the same techniques to Lemma \ref{lem-acc-psi-1-a}, we can show that the boundary value of the rest terms in $\Phi_{2,\alpha}$ are sufficient small, then we get \eqref{est-Phi2--1}. 

By using Lemma \ref{lem-real-Airy-Solver} and Lemma \ref{lem-real-Airy-Solver-S}, one can easily check that 
\begin{align*}
  \left|(u_p-\hat c)^2\pa_y^3AirySolver \left(i\varepsilon \pa_y^2 \left(\frac{u_p''\phi_{2,\alpha}}{u_p-\hat c}\right)\right)\right|\lesssim |\ln c_0|, 
\end{align*}
\begin{align*}
  \left|(u_p-\hat c)^2\pa_y^3AirySolver \left(\left(-i\varepsilon\alpha^2+ic_0\right)\left(\frac{u_p''\phi_{2,\alpha}}{u_p-\hat c}\right)\right)\right|\lesssim |\ln c_0|.
\end{align*}
Recall Lemma \ref{lem-sol-Ray-hom} that $\left|\phi_{2,\alpha}'''(0)\right|\lesssim 1$. Then following the proof of Lemma \ref{lem-rough-est-phi-psi}, one can deduce $\eqref{est-Phi2-0-ddd}$.

For the same reason to Lemma \ref{lem-Phi-0}, we have
\begin{align*}
  \Phi_{2,\alpha}'(0)=\phi_{2,\alpha}'(0)=\phi_{2,0}'(0)=\frac{1}{1-\hat c}.
\end{align*}
\end{proof}
As in Lemma \ref{lem-Phi1-0-cr}, one obtains the following rough parameter-derivative estimates for $\Phi_{2,\alpha}$.
\begin{lemma}\label{lem-Phi2-cr}
  It holds that
  \begin{align*}
    \left|\pa_{c_r}\Phi_{2,\alpha}(-1)\right|+\left|\pa_{c_r}\Phi_{2,\alpha}'(-1)\right|+\left|\pa_{c_r}\Phi_{2,\alpha}'''(0)\right|\lesssim \kappa^{\frac{3}{2}}|\ln c_0|,
  \end{align*}
    \begin{align*}
    \left|\pa_{c_i}\Phi_{2,\alpha}(-1)\right|+\left|\pa_{c_i}\Phi_{2,\alpha}'(-1)\right|+\left|\pa_{c_i}\Phi_{2,\alpha}'''(0)\right|\lesssim \kappa^{\frac{3}{2}}|\ln c_0|.
  \end{align*}
    For fixed $\alpha$, it holds that
     \begin{align*}
    \left|\pa_{\nu}\Phi_{2,\alpha}(-1)\right|+\left|\pa_{\nu}\Phi_{2,\alpha}'(-1)\right|+\left|\pa_{\nu}\Phi_{2,\alpha}'''(0)\right|\lesssim \frac{1}{\nu}\kappa^{\frac{1}{2}} |\ln c_0|.
  \end{align*} 
    For fixed $\nu$ or $\varepsilon$, it holds that
     \begin{align*}
    \left|\pa_{\alpha^2}\Phi_{2,\alpha}(-1)\right|+\left|\pa_{\alpha^2}\Phi_{2,\alpha}'(-1)\right|+\left|\pa_{\alpha^2}\Phi_{2,\alpha}'''(0)\right|\lesssim \frac{1}{\alpha^2}\kappa^{\frac{1}{2}} |\ln c_0|.
  \end{align*} 
\end{lemma}

\section{Fast modes}
In this section, we construct two solutions $\Phi_{3,\alpha}$ and $\Phi_{4,\alpha}$ of the homogeneous Orr--Sommerfeld equation \eqref{eq-Orr-Sommerfeld}. They are obtained by perturbing the Airy profiles $\phi_{A,1}$ and $\phi_{A,2}$, where
\begin{align}\label{eq-unifi-prim-Airy}
\phi_{A,1}=\frac{ A_1(2,y)}{A_1(2,-1)},\qquad\phi_{A,2}=\frac{A_2(2,y)}{A_2(2,0)}.
\end{align}
Following \cite{GGN16adv}, we refer to these two solutions as the fast modes.

In contrast to the slow modes, the fast modes can be constructed using only the Airy iteration. We first construct the left fast mode $\Phi_{3,\alpha}$ starting from $\phi_{A,1}$. The construction of the right fast mode $\Phi_{4,\alpha}$ is analogous. We will show that the boundary behavior of $\Phi_{3,\alpha}$ is mainly determined by $\phi_{A,1}$, in agreement with the classical physical picture of fast modes in the T--S instability analysis.
\subsection{Construction of $\Phi_{3,\alpha}$}
In this subsection, we construct $\Phi_{3,\alpha}$ based on $\phi_{A,1}$.

By the normalization of $\phi_{A,1}$, we have
\begin{align*}
  \phi_{A,1}(-1)=1,\quad \phi_{A,1}(0)=0,\quad \phi_{A,1}'(0)=0.
\end{align*}
Formally, the Airy asymptotic expansion give
\begin{align*}
  \phi_{A,1}(y)\sim&  \left\langle \kappa c_r \right\rangle^{\frac{5}{4}}\left\langle \kappa\eta(y) \right\rangle^{-\frac{5}{4}}e^{-\frac{2}{3}\frac{1}{\sqrt2}|\kappa\eta_r(y)|^{\frac{1}{2}}\kappa\eta_r(y)+\frac{2}{3}\frac{1}{\sqrt2}|\kappa \eta_r(-1)|^{\frac{1}{2}}\kappa \eta_r(-1)},\\
  \pa_y\phi_{A,1}(y)\sim&  \kappa\left\langle \kappa c_r \right\rangle^{\frac{5}{4}}\left\langle \kappa\eta(y) \right\rangle^{-\frac{3}{4}}e^{-\frac{2}{3}\frac{1}{\sqrt2}|\kappa\eta_r(y)|^{\frac{1}{2}}\kappa\eta_r(y)+\frac{2}{3}\frac{1}{\sqrt2}|\kappa \eta_r(-1)|^{\frac{1}{2}}\kappa \eta_r(-1)},\\
  \pa_y^2\phi_{A,1}(y)\sim&  \kappa^2\left\langle \kappa c_r \right\rangle^{\frac{5}{4}}\left\langle \kappa\eta(y) \right\rangle^{-\frac{1}{4}}e^{-\frac{2}{3}\frac{1}{\sqrt2}|\kappa\eta_r(y)|^{\frac{1}{2}}\kappa\eta_r(y)+\frac{2}{3}\frac{1}{\sqrt2}|\kappa \eta_r(-1)|^{\frac{1}{2}}\kappa \eta_r(-1)}.
\end{align*}

We introduce a modified approximate Green function adapted to the left fast mode:
\begin{equation}\label{eq-Green-Airy-}
   G_{A-}(x,y)=i \frac{2\pi}{\left(\eta_r'(x)\right)^{\frac{1}{2}}\kappa\varepsilon}\left\{
    \begin{array}{ll}
       A_1(2,y) A_2(x),&x<y;\\ 
       A_1(x) A_2(2,y)-a_2(x)+a_1(x)(y-y_c)-a_1(x)(x-y_c),&x>y.
    \end{array}
  \right.
\end{equation}
We define the corresponding modified Airy solver by
\begin{align}\label{eq-modified-Airysolver-}
  AirySolver_{m-}(f)(y)=\int^0_{-1}G_{A-}(x,y)f(x)dx.
\end{align}
Compared to \eqref{eq-Green-Airy}, all nonlocal correction terms are placed on the side $x>y$.

From the definition of $A_1(2,y)$, we can easily check that
\begin{align*}
  AirySolver_{m-}(f)(0)=0,\quad \pa_yAirySolver_{m-}(f)(0)=0.
\end{align*}
We will use $AirySolver_{m-}$ to carry out the Airy iteration starting from $\phi_{A,1}$.

  We first introduce a weight function
\begin{align}\label{eq-weight-MW}
  \mathcal W(y)=\left|e^{\frac{1}{3}\left( \left(e^{\frac{\pi i}{6}}\kappa\eta(y)\right)^{\frac{3}{2}}-\left(e^{\frac{\pi i}{6}}\kappa\eta(-1)\right)^{\frac{3}{2}}\right)}\right|,
\end{align}
and the weighted norm
\begin{align*}
  \left\|f\right\|_{L^\infty_{\mathcal W}}=\left\|f\mathcal W\right\|_{L^\infty}.
\end{align*}
We note that the exponential decay of $\frac{1}{\mathcal W}$ is weaker than that of the Airy profile. The choice of $\mathcal W$ ensures that $\left\|\phi_{A,1}(y)\right\|_{L^\infty_{\mathcal W}}\lesssim 1$. 

We have the following estimate for $AirySolver_{m-}$.
\begin{lemma}\label{lem-AS--0}
  For any $s\ge0$, it holds that
  \begin{align}
     \left\|\left\langle \kappa\eta \right\rangle^{\frac{1}{2}-s}AirySolver_{m-}(f)\right\|_{L^\infty_{\mathcal W}}\lesssim \varepsilon^{\frac{1}{3}}\left\|\left\langle \kappa\eta \right\rangle^{-s}f\right\|_{L^\infty_{\mathcal W}},\label{est-AS--0}\\
     \left\|\left\langle \kappa\eta \right\rangle^{\frac{3}{2}-s}\pa_yAirySolver_{m-}(f)\right\|_{L^\infty_{\mathcal W}}\lesssim \left\|\left\langle \kappa\eta \right\rangle^{-s}f\right\|_{L^\infty_{\mathcal W}},\label{est-AS--1}\\
     \left\|\left\langle \kappa\eta \right\rangle^{1-s}\pa_y^2AirySolver_{m-}(f)\right\|_{L^\infty_{\mathcal W}}\lesssim \varepsilon^{-\frac{1}{3}} \left\|\left\langle \kappa\eta \right\rangle^{-s}f\right\|_{L^\infty_{\mathcal W}},\label{est-AS--2}\\
     \left\|\left\langle \kappa\eta \right\rangle^{\frac{1}{2}-s}\pa_y^3AirySolver_{m-}(f)\right\|_{L^\infty_{\mathcal W}}\lesssim \varepsilon^{-\frac{2}{3}} \left\|\left\langle \kappa\eta \right\rangle^{-s}f\right\|_{L^\infty_{\mathcal W}}\label{est-AS--3}.
  \end{align} 
\end{lemma}

\begin{proof}
  We first give the estimate for the $G_{A,M}(x,y)$ part. Recalling \eqref{eq-Airy-modi-1}, \eqref{eq-Airy-modi-2}, and Lemma \ref{lem-Airy-Airy-est}, we have
  \begin{align*}
    &\left|e^{\frac{1}{3}\left( \left(e^{\frac{\pi i}{6}}\kappa\eta(y)\right)^{\frac{3}{2}}-\left(e^{\frac{\pi i}{6}}\kappa\eta(-1)\right)^{\frac{3}{2}}\right)}\int_{-1}^0G_{A,M}(x,y)f(x)dx\right|\\
    \lesssim&\int_{-1}^0\left|\left\langle \kappa\eta(x) \right\rangle^{s}G_{A,M}(x,y)e^{\frac{1}{3}\left( \left(e^{\frac{\pi i}{6}}\kappa\eta(y)\right)^{\frac{3}{2}}-\left(e^{\frac{\pi i}{6}}\kappa\eta(x)\right)^{\frac{3}{2}}\right)}\right|dx\left\|\left\langle \kappa\eta \right\rangle^{-s}f\right\|_{L^\infty_{\mathcal W}}\\
    \lesssim&\int_{-1}^0 \frac{\left\langle \kappa\eta(x) \right\rangle^{s-\frac{1}{4}}}{\left\langle \kappa\eta(y) \right\rangle^{\frac{5}{4}}} e^{-\frac{1}{C}\left|\kappa\eta_r(x)-\kappa\eta_r(y)  \right|\left( |\kappa\eta(y)|^{\frac{1}{2}}+|\kappa\eta(x)|^{\frac{1}{2}} \right) }dx\left\|\left\langle \kappa\eta \right\rangle^{-s}f\right\|_{L^\infty_{\mathcal W}}.
  \end{align*} 
 We denote
\begin{align*}
  I_{\mathcal W,1}(x,y) =\frac{\left\langle \kappa\eta(x) \right\rangle^{s-\frac{1}{4}}}{\left\langle \kappa\eta(y) \right\rangle^{\frac{5}{4}}}e^{-\frac{1}{C}\left|\kappa\eta_r(x)-\kappa\eta_r(y)  \right|\left( |\kappa\eta(y)|^{\frac{1}{2}}+|\kappa\eta(x)|^{\frac{1}{2}} \right) }.
\end{align*}

\paragraph{1.1) Case $\left|y-y_c\right|\le \frac{1}{\kappa}$, $\left|x-y\right|\le \frac{2}{\kappa}$.} For this case $I_{\mathcal W,1}(x,y)\lesssim 1$, and 
\begin{align*}
  \int_{-1}^0I_{\mathcal W,1}(x,y)\chi_{\left|x-y\right|\le \frac{2}{\kappa}}(x)dx\lesssim \frac{1}{\kappa}.
\end{align*}

\paragraph{1.2) Case $\left|y-y_c\right|\le \frac{1}{\kappa}$, $\left|x-y\right|\ge \frac{2}{\kappa}$.} For this case, $|\kappa\eta(x)|\ge1$, then we can use the same technique as the Case 2.2 in the proof of \eqref{eq-est-AS-1} to get
\begin{align*}
  I_{\mathcal W,1}(x,y)\lesssim e^{-\frac{1}{C}\left|\kappa\eta_r(x)-\kappa\eta_r(y)  \right| },
\end{align*}
and 
\begin{align*}
  \int_{-1}^0I_{\mathcal W,1}(x,y)\chi_{\left|x-y\right|\ge \frac{2}{\kappa}}(x)dx\lesssim \frac{1}{\kappa}.
\end{align*}

\paragraph{2) Case $\left|y-y_c\right|\ge \frac{1}{\kappa}$.} For this case, we have $|\kappa\eta(y)|\ge1$, 
\begin{align*}
  I_{\mathcal W,1}(x,y)\lesssim \left\langle \kappa\eta(y) \right\rangle^{s-\frac{3}{2}}e^{-\frac{1}{C}\left|\kappa\eta_r(x)-\kappa\eta_r(y)  \right|\left( |\kappa\eta(y)|^{\frac{1}{2}}+|\kappa\eta(x)|^{\frac{1}{2}} \right) },
\end{align*}
 and 
\begin{align*}
  \int_{-1}^0I_{\mathcal W,1}(x,y) dx\lesssim \frac{1}{\kappa}\left\langle \kappa\eta(y) \right\rangle^{s-2}.
\end{align*}
 
Next, we turn to the non-local terms.  Similar to Lemma \ref{lem-est-AS}, here the only troublesome term is from $a_1(x)(y-y_c)$. Noting that the non-local terms only appear for $x>y$, we have
\begin{align*}
  &\left|e^{\frac{1}{3}\left( \left(e^{\frac{\pi i}{6}}\kappa\eta(y)\right)^{\frac{3}{2}}-\left(e^{\frac{\pi i}{6}}\kappa\eta(-1)\right)^{\frac{3}{2}}\right)}\int_{y}^0 \left| \frac{a_1(x)\left(y-y_c\right)}{\kappa\varepsilon}f(x)\right|dx\right|\\
  \lesssim &\int_{y}^0 \frac{\left\langle \kappa\eta(y) \right\rangle}{\left\langle \kappa\eta(x) \right\rangle^{1-s}}e^{-\frac{1}{C}\left|\kappa\eta_r(x)-\kappa\eta_r(y)  \right|\left( |\kappa\eta(y)|^{\frac{1}{2}}+|\kappa\eta(x)|^{\frac{1}{2}} \right) }dx\left\|\left\langle \kappa\eta \right\rangle^{-s}f\right\|_{L^\infty_{\mathcal W}}.
\end{align*}
 We denote
\begin{align*}
  I_{\mathcal W,2}(x,y) =\frac{\left\langle \kappa\eta(y) \right\rangle}{\left\langle \kappa\eta(x) \right\rangle^{1-s}}e^{-\frac{1}{C}\left|\kappa\eta_r(x)-\kappa\eta_r(y)  \right|\left( |\kappa\eta(y)|^{\frac{1}{2}}+|\kappa\eta(x)|^{\frac{1}{2}} \right) }.
\end{align*}
Then, by using the same technique to the $G_{A,M}$ part, we deduce that
\begin{align*}
  \int_{y}^0I_{\mathcal W,2}(x,y) dx\lesssim \frac{1}{\kappa}\left\langle \kappa\eta(y) \right\rangle^{s-\frac{1}{2}}.
\end{align*}
Then \eqref{est-AS--0} follows directly. 

For $\pa_y AirySolver_{m-}(f)$, the most troublesome term is also from the nonlocal term $a_1(x)$. We have
\begin{align*}
  &\left|e^{\frac{1}{3}\left( \left(e^{\frac{\pi i}{6}}\kappa\eta(y)\right)^{\frac{3}{2}}-\left(e^{\frac{\pi i}{6}}\kappa\eta(-1)\right)^{\frac{3}{2}}\right)}\int_{y}^0 \left| \frac{a_1(x)}{\kappa\varepsilon}f(x)\right|dx\right|\\
  \lesssim &\int_{y}^0 \frac{\kappa}{\left\langle \kappa\eta(x) \right\rangle^{1-s}}e^{-\frac{1}{C}\left|\kappa\eta_r(x)-\kappa\eta_r(y)  \right|\left( |\kappa\eta(y)|^{\frac{1}{2}}+|\kappa\eta(x)|^{\frac{1}{2}} \right) }dx\left\|\left\langle \kappa\eta \right\rangle^{-s}f\right\|_{L^\infty_{\mathcal W}}.
\end{align*}
This gives \eqref{est-AS--1}.

The estimate \eqref{est-AS--2} and \eqref{est-AS--3} can be deduced in the same way. 

\end{proof}
Next, we give the construction of the solution of the non-homogeneous Orr-Sommerfeld equation.

\begin{lemma}\label{lem-sol-nh-OS}
  For any $f\in L^\infty_{\mathcal W}$, there exists a solution $\Phi_f$ to \eqref{eq-inhome-OS} satisfying
  \begin{align}\label{est-sol-nh-OS}
    \left\|\Phi_{f}\right\|_{L^\infty_{\mathcal W}}\lesssim \varepsilon^{\frac{1}{3}}\left\|f\right\|_{L^\infty_{\mathcal W}},
  \end{align}
  and
  \begin{align}\label{est-sol-nh-b0}
    \pa_y\Phi_{f}(0)=0,\qquad \left|\pa_y^3\Phi_{f}(0)\right|\lesssim e^{-\kappa}.
  \end{align}
\end{lemma}
\begin{proof}
  By the definition of $AirySolver_{m-}$, we have
  \begin{align*}
  &Orr\left(AirySolver_{m-}(f)\right)\\
  =&f+\left(-u_p''-\alpha^2\left(u-c \right)+i\varepsilon\alpha^4\right)AirySolver_{m-}(f)\\ 
  &+i\varepsilon  \left(\pa_y^2\left(\eta_r'(y)\right)^{-\frac{1}{2}}\right)\left(\eta_r'(y)\right)^{\frac{1}{2}} \pa_y^2AirySolver_{m-}(f)-i2\varepsilon\alpha^2\pa_y^2AirySolver_{m-}(f)\\
  &+i c_i\left(\left(\eta_r'(y)\right)^2-1\right)  \pa_y^2AirySolver_{m-}(f).
\end{align*}
From Lemma \ref{lem-AS--0}, we have
\begin{align*}
  \left\|\left(-u_p''-\alpha^2\left(u-c \right)+i\varepsilon\alpha^4\right)AirySolver_{m-}(f)\right\|_{L^\infty_{\mathcal W}}\lesssim \varepsilon^{\frac{1}{3}}\left\|f\right\|_{L^\infty_{\mathcal W}},
\end{align*}
\begin{align*}
  \left\|\varepsilon  \left(\pa_y^2\left(\eta_r'(y)\right)^{-\frac{1}{2}}\right)\left(\eta_r'(y)\right)^{\frac{1}{2}} \pa_y^2AirySolver_{m-}(f)\right\|_{L^\infty_{\mathcal W}}\lesssim \varepsilon^{\frac{2}{3}}\left\|f\right\|_{L^\infty_{\mathcal W}},
\end{align*}
and
\begin{align*}
  &\left\|\left\langle \kappa\eta(y) \right\rangle^{-\frac{1}{2}}c_i\left(\left(\eta_r'(y)\right)^2-1\right)  \pa_y^2AirySolver_{m-}(f)\right\|_{L^\infty_{\mathcal W}}\\
  \lesssim& \left\| c_i \frac{\eta(y)}{\left\langle \kappa\eta(y) \right\rangle}\right\|_{L^\infty} \left\|\left\langle \kappa\eta(y) \right\rangle^{\frac{1}{2}}\pa_y^2AirySolver_{m-}(f)\right\|_{L^\infty_{\mathcal W}}\lesssim |c_i|  \left\|f\right\|_{L^\infty_{\mathcal W}}.
\end{align*}

Let
\begin{align*}
\Phi_{f}^{[0]}(y)=AirySolver_{m-}(f),\quad \Phi_{f}^{[1]}(y)=-AirySolver_{m-}\left(Orr\left(AirySolver_{m-}(f)\right)-f\right),\\
 \Phi_{f}^{[j]}(y)=-AirySolver_{m-}\left(Orr\left(\Phi_{f}^{[j-1]}\right)-Orr\left(\Phi_{f}^{[j-2]}\right)\right),\text{ for }j\ge2,
\end{align*}
then by Lemma \ref{lem-AS--0}, we have
\begin{align*}
  \left\|\Phi_{f}^{[j]}\right\|_{L^\infty_{\mathcal W}}\lesssim \left(|c_i|+ \varepsilon^{\frac{1}{3}}\right)^j\varepsilon^{\frac{1}{3}}.
\end{align*}

As we always assume that $|c_i|\lesssim \varepsilon^{\frac{1}{3}}$, therefore, let
\begin{align*}
  \Phi_{f}(y)=\sum_{j=0}^{+\infty}\Phi_{f}^{[j]}(y),
\end{align*}
we can see that $\Phi_{f}(y)$ is well defined and solves \eqref{eq-inhome-OS}, and satisfies \eqref{est-sol-nh-OS}.

The estimate \eqref{est-sol-nh-b0} follows directly from the definition of $AirySolver_{m-}$.
\end{proof}

We are now ready to construct $\Phi_{3,\alpha}$. We define
\begin{align*}
  \Phi_{3,\alpha}(y)=\sum_{j=0}^{+\infty}\Phi_{3,\alpha}^{[j]}(y),
\end{align*}
where
\begin{align*}
  \Phi_{3,\alpha}^{[0]}=\phi_{A,1}(y),\quad \Phi_{3,\alpha}^{[1]}(y)=-AirySolver_{m-}\left(Orr\left(\phi_{A,1}(y)\right)\right),\\
  \Phi_{3,\alpha}^{[j]}(y)=-AirySolver_{m-}\left(Orr\left(\Phi_{3,\alpha}^{[j-1]}\right)-Orr\left(\Phi_{3,\alpha}^{[j-2]}\right)\right),\text{ for }j\ge2.
\end{align*}
The estimates in Lemma \ref{lem-sol-nh-OS} imply that the above series converges in $L^\infty_{\mathcal W}$, and hence $\Phi_{3,\alpha}$ is well defined.

\subsection{Boundary values of $\Phi_{3,\alpha}$ at $y=-1$ and $y=0$}
In this subsection, we derive boundary estimates for $\Phi_{3,\alpha}$.
\begin{lemma}\label{lem-Phi3--1}
  It holds that
  \begin{align*}
    \left|\Phi_{3,\alpha}(-1)-\phi_{A,1}(-1)\right|\lesssim  \varepsilon^{\frac{1}{3}}|\kappa c_r|^{-\frac{1}{2}}+|c_i||\kappa c_r|^{\frac{3}{2}} ,\\
      \left|\pa_y\Phi_{3,\alpha}(-1)-\pa_y\phi_{A,1}(-1)\right|\lesssim |\kappa c_r|^{-\frac{3}{2}}+|c_i|\kappa|\kappa c_r|^{\frac{1}{2}},\\
    \pa_y\Phi_{3,\alpha}(0)=0,\quad \left|\pa_y^3\Phi_{3,\alpha}(0)\right|\lesssim e^{-\kappa}.
  \end{align*}
\end{lemma}
\begin{proof}

Recall that
\begin{align*}
  Orr\left(\phi_{A,1}\right)=&\left(-u_p''-\alpha^2\left(u-c \right)+i\varepsilon\alpha^4\right)\phi_{A,1}\\ 
  &+i\varepsilon  \left(\pa_y^2\left(\eta_r'(y)\right)^{-\frac{1}{2}}\right)\left(\eta_r'(y)\right)^{\frac{1}{2}} \pa_y^2\phi_{A,1}(y)-i2\varepsilon\alpha^2\pa_y^2\phi_{A,1}(y)\\
  &+i c_i\left(\left(\eta_r'(y)\right)^2-1\right)  \pa_y^2\phi_{A,1}(y).
\end{align*}
Here the potential troublesome terms are $u_p''\phi_{A,1}$ and $c_i\left(\left(\eta_r'(y)\right)^2-1\right)  \pa_y^2\phi_{A,1}(y)$. By using Lemma \ref{lem-AS--0}, we have
\begin{align*}
  \left|AirySolver_{m-}\left(\phi_{A,1}\right)(-1)\right|\lesssim \varepsilon^{\frac{1}{3}}|\kappa c_r|^{-\frac{1}{2}},\\
  \left|\pa_yAirySolver_{m-}\left(\phi_{A,1}\right)(-1)\right|\lesssim  |\kappa c_r|^{-\frac{3}{2}}.
\end{align*}

Note that
\begin{align*}
  \left|c_i\left(\left(\eta_r'(y)\right)^2-1\right)  \pa_y^2\phi_{A,1}(y)\right|\lesssim |c_i|\kappa \left\langle \kappa c_r \right\rangle^{\frac{5}{4}}\left\langle \kappa\eta(y) \right\rangle^{\frac{3}{4}}\frac{1}{\mathcal W(y)}.
\end{align*}
By using Lemma \ref{lem-AS--0}, we have
\begin{align*}
  \left|AirySolver_{m-}\left(c_i\left(\left(\eta_r'(y)\right)^2-1\right) \phi_{A,1}''(y)\right)(-1)\right|\lesssim |c_i||\kappa c_r|^{\frac{3}{2}},\\
  \left|\pa_yAirySolver_{m-}\left(c_i\left(\left(\eta_r'(y)\right)^2-1\right) \phi_{A,1}''(y)\right)(-1)\right|\lesssim |c_i|\kappa|\kappa c_r|^{\frac{1}{2}}.
\end{align*}
Therefore, we have 
\begin{align*}
  \left|\Phi_{3,\alpha}^{[1]}(-1)\right|\lesssim \varepsilon^{\frac{1}{3}}|\kappa c_r|^{-\frac{1}{2}}+|c_i||\kappa c_r|^{\frac{3}{2}},\quad  \left|\pa_y\Phi_{3,\alpha}^{[1]}(-1)\right|\lesssim |\kappa c_r|^{-\frac{3}{2}}+|c_i|\kappa|\kappa c_r|^{\frac{1}{2}},\\
  \left|\pa_y\Phi_{3,\alpha}^{[1]}(0)\right|=0,\quad \left|\pa_y^3\Phi_{3,\alpha}^{[1]}(0)\right|\lesssim e^{-\kappa}.
\end{align*}

By Lemma \ref{lem-sol-nh-OS} we can see that $\Phi_{3,\alpha}^{[j]}$ with $j\ge2$ have better estimates. Then the result of this lemma follows directly.
  
\end{proof}

The lemma shows that the boundary values of $\Phi_{3,\alpha}$ at $y=-1$ are determined, to leading order, by the Airy profile $\phi_{A,1}$. Next we will show that after taking derivative, the lower order terms are still small enough.
\begin{lemma}\label{lem-Phi3--1-cr}
   It holds that
  \begin{align}
    \left|\pa_{c_r}\Phi_{3,\alpha}(-1)-\pa_{c_r}\phi_{A,1}(-1)\right|\lesssim  |\kappa c_r|^{-\frac{1}{2}}+\varepsilon^{-\frac{1}{3}}|c_i||\kappa c_r|^{\frac{3}{2}}\left|\ln c_0\right|,\label{est-Phi3--1-cr-0}\\
      \left|\pa_{c_r}\pa_y\Phi_{3,\alpha}(-1)-\pa_{c_r}\pa_y\phi_{A,1}(-1)\right|\lesssim \varepsilon^{-\frac{1}{3}} |\kappa c_r|^{-\frac{3}{2}}+\varepsilon^{-\frac{2}{3}}|c_i||\kappa c_r|^{\frac{1}{2}}\left|\ln c_0\right|,\label{est-Phi3--1-cr-1}
  \end{align}
  and
    \begin{align}
    \left|\pa_{c_i}\Phi_{3,\alpha}(-1)-\pa_{c_i}\phi_{A,1}(-1)\right|\lesssim  |\kappa c_r|^{-\frac{1}{2}}+ |\kappa c_r|^{\frac{3}{2}}\left|\ln c_0\right|,\label{est-Phi3--1-ci-0}\\
      \left|\pa_{c_i}\pa_y\Phi_{3,\alpha}(-1)-\pa_{c_i}\pa_y\phi_{A,1}(-1)\right|\lesssim \varepsilon^{-\frac{1}{3}} |\kappa c_r|^{-\frac{3}{2}}+\varepsilon^{-\frac{1}{3}} |\kappa c_r|^{\frac{1}{2}}\left|\ln c_0\right|.\label{est-Phi3--1-ci-1}
  \end{align}
  For fixed $\alpha$, we have
      \begin{align}
    \left|\pa_{\nu}\Phi_{3,\alpha}(-1)-\pa_{\nu}\phi_{A,1}(-1)\right|\lesssim \frac{1}{\nu}\left(\varepsilon^{\frac{1}{3}}|\kappa c_r|^{-\frac{1}{2}}+|c_i||\kappa c_r|^{\frac{3}{2}} \right)\left|\ln c_0\right|,\label{est-Phi3--1-nu-0}\\
      \left|\pa_{\nu}\pa_y\Phi_{3,\alpha}(-1)-\pa_{\nu}\pa_y\phi_{A,1}(-1)\right|\lesssim \frac{1}{\nu}\left(|\kappa c_r|^{-\frac{3}{2}}+|c_i|\kappa|\kappa c_r|^{\frac{1}{2}}\right)\left|\ln c_0\right|.\label{est-Phi3--1-nu-1}
  \end{align}
    For fixed $\nu$ or $\varepsilon$, we have
      \begin{align}
    \left|\pa_{\alpha^2}\Phi_{3,\alpha}(-1)-\pa_{\alpha^2}\phi_{A,1}(-1)\right|\lesssim  \frac{1}{\alpha^2}\left(\varepsilon^{\frac{1}{3}}|\kappa c_r|^{-\frac{1}{2}}+|c_i||\kappa c_r|^{\frac{3}{2}} \right)\left|\ln c_0\right|,\label{est-Phi3--1-alpha-0}\\
      \left|\pa_{\alpha^2}\pa_y\Phi_{3,\alpha}(-1)-\pa_{\alpha^2}\pa_y\phi_{A,1}(-1)\right|\lesssim \frac{1}{\alpha^2}\left(|\kappa c_r|^{-\frac{3}{2}}+|c_i|\kappa|\kappa c_r|^{\frac{1}{2}}\right)\left|\ln c_0\right|.\label{est-Phi3--1-alpha-1}
  \end{align}
\end{lemma}
\begin{proof}
We only need to study the estimates for $\Phi_{3,\alpha}^{[1]}$. 

It holds that
\begin{align*}
  \pa_{c_r}AirySolver_{m-}\left(f\right)(y)=\int^0_{-1}\left(\pa_{c_r}G_{A-}(x,y)\right)f(x)dx+\int^0_{-1}G_{A-}(x,y)\pa_{c_r}f(x)dx.
\end{align*}
Combining the techniques used in Lemma \ref{lem-acc-psi-1-a-cr} and Lemma \ref{lem-AS--0}, one can easily check that
\begin{align*}
  \left.\int^0_{-1}\left(\pa_{c_r}G_{A-}(x,y)\right)\phi_{A,1}(x)dx\right|_{y=-1}\lesssim |\kappa c_r|^{-\frac{1}{2}},\\
  \left.\pa_y\int^0_{-1}\left(\pa_{c_r}G_{A-}(x,y)\right)\phi_{A,1}(x)dx\right|_{y=-1}\lesssim\varepsilon^{-\frac{1}{3}} |\kappa c_r|^{-\frac{3}{2}},
\end{align*}
and
\begin{align*}
  \left.\int^0_{-1}\left(\pa_{c_r}G_{A-}(x,y)\right)c_i\left(\left(\eta_r'(x)\right)^2-1\right) \phi_{A,1}''(x)dx\right|_{y=-1}\lesssim \varepsilon^{-\frac{1}{3}}|c_i||\kappa c_r|^{\frac{3}{2}}\left|\ln c_0\right|,\\
  \left.\pa_y\int^0_{-1}\left(\pa_{c_r}G_{A-}(x,y)\right)c_i\left(\left(\eta_r'(x)\right)^2-1\right) \phi_{A,1}''(x)dx\right|_{y=-1}\lesssim\varepsilon^{-\frac{2}{3}}|c_i||\kappa c_r|^{\frac{1}{2}}\left|\ln c_0\right|,
\end{align*}

By using the technique in Lemma \ref{lem-d-cr-a1B}, we have
\begin{align*}
  \left|\pa_{c_r}\phi_{A,1}(y)\right|\lesssim \frac{\left|\ln c_0\right|}{c_r}\frac{1}{\mathcal W},\quad \left|\pa_{c_r}\pa_y^2\phi_{A,1}(y)\right|\lesssim \kappa^3\left|\ln c_0\right|\frac{1}{\mathcal W},
\end{align*}
and then
\begin{align*}
  \pa_{c_r}\left(c_i\left(\left(\eta_r'(x)\right)^2-1\right) \phi_{A,1}''(y)\right)\lesssim |c_i|\kappa^2|\kappa c_r|\left|\ln c_0\right|\frac{1}{\mathcal W}.
\end{align*}

It follows of Lemma \ref{lem-AS--0} that
\begin{align*}
  \left|AirySolver_{m-}\left(\pa_{c_r}\phi_{A,1}\right)(-1)\right|\lesssim |\kappa c_r|^{-\frac{3}{2}}\left|\ln c_0\right|,\\
   \left|\pa_yAirySolver_{m-}\left(\pa_{c_r}\phi_{A,1}\right)(-1)\right|\lesssim \frac{1}{c_r}|\kappa c_r|^{-\frac{3}{2}}\left|\ln c_0\right|,
\end{align*}
and
\begin{align*}
  \left|AirySolver_{m-}\left(\pa_{c_r}\left(c_i\left(\left(\eta_r'(x)\right)^2-1\right) \phi_{A,1}''\right)\right)(-1)\right|\lesssim  |c_i|\kappa|\kappa c_r|^{\frac{1}{2}}\left|\ln c_0\right|,\\
   \left|\pa_yAirySolver_{m-}\left(\pa_{c_r}\left(c_i\left(\left(\eta_r'(x)\right)^2-1\right) \phi_{A,1}''\right)\right)(-1)\right|\lesssim |c_i|\kappa^2|\kappa c_r|^{-\frac{1}{2}}\left|\ln c_0\right|.
\end{align*}
Then \eqref{est-Phi3--1-cr-0} and \eqref{est-Phi3--1-cr-0} follows directly. The estimates \eqref{est-Phi3--1-ci-0} and \eqref{est-Phi3--1-ci-0} can be derived in the same way.

For fixed $\alpha$, by using the technique in Lemma \ref{lem-d-cr-a1B}, we have
\begin{align*}
  \left|\pa_{\nu}\phi_{A,1}(y)\right|\lesssim \frac{1}{\nu}\left|\phi_{A,1}(y)\right|+\frac{\kappa^{\frac{3}{2}}}{\nu} \left| \left(e^{\frac{\pi}{6}i} \eta(y)\right)^{\frac{3}{2}}-\left(e^{\frac{\pi}{6}i} \eta(-1)\right)^{\frac{3}{2}} \right|\left|\phi_{A,1}(y)\right|\lesssim \frac{\left|\ln c_0\right|}{\nu}\frac{1}{\mathcal W}.
\end{align*}
Then by using the similar techniques to \eqref{est-acc-psi-1-a-nu-2}, we get \eqref{est-Phi3--1-nu-0} and \eqref{est-Phi3--1-nu-1}. 

The estimates \eqref{est-Phi3--1-alpha-0} and \eqref{est-Phi3--1-alpha-1} can be deduced in the same way.

\end{proof}

Combining the above estimates, we see that the boundary behavior of $\Phi_{3,\alpha}$ at $y=-1$ is governed by $\phi_{A,1}$. At the opposite boundary $y=0$, the fast mode is exponentially small in the sense described above.

\subsection{Expansion of  $\frac{\Phi_{3,\alpha}(-1)}{\Phi_{3,\alpha}'(-1)}$}
In this subsection, we derive an expansion for $\frac{\Phi_{3,\alpha}(-1)}{\Phi_{3,\alpha}'(-1)}$.

\begin{lemma}\label{lem-res-airy}
  It holds that 
  \begin{align*}
    \frac{\Phi_{3,\alpha}(-1)}{\Phi_{3,\alpha}'(-1)}=\frac{A_1(2,-1)}{A_1(1,-1)}+Res_{airy}=&\frac{\mathcal A_1(2,\kappa \eta(-1))}{\kappa \eta_r'(-1)\mathcal A_1(1,\kappa \eta(-1))}+Res_{airy},
  \end{align*} 
  where the remainder $Res_{airy}$ satisfies
  \begin{align}\label{est-res-airy}
    \left|Res_{airy}\right|\lesssim \frac{ 1}{\kappa^2\left\langle \kappa c_r \right\rangle}+|c_i|c_r,
  \end{align}
  \begin{align}\label{est-res-airy-cr}
    \left|\pa_{c_r}Res_{airy}\right|\lesssim  \left(\frac{1}{\kappa \left\langle \kappa c_r \right\rangle}+|c_i|\kappa c_r\right)|\ln c_0|,
  \end{align}
    \begin{align}\label{est-res-airy-ci}
    \left|\pa_{c_i}Res_{airy}\right|\lesssim \left(\frac{1}{\kappa \left\langle \kappa c_r \right\rangle}+c_r\right)|\ln c_0|.
  \end{align}
  For fixed $\alpha$, it holds that
     \begin{align}\label{est-res-airy-nu}
    \left|\pa_{\nu}Res_{airy}\right|\lesssim \left(\frac{ 1}{\nu\kappa^2\left\langle \kappa c_r \right\rangle}+\frac{1}{\nu}|c_i|c_r\right)|\ln c_0|.
  \end{align} 
    For fixed $\nu$ or $\varepsilon$, it holds that
     \begin{align}\label{est-res-airy-alpha}
    \left|\pa_{\alpha^2}Res_{airy}\right|\lesssim \left(\frac{ 1}{\alpha^2\kappa^2\left\langle \kappa c_r \right\rangle}+\frac{1}{\alpha^2}|c_i|c_r\right)|\ln c_0| .
  \end{align} 
\end{lemma}
\begin{proof}
  Recall that $\phi_{A,1}(-1)=1$ and $\phi_{A,1}'(-1)=O \left(\kappa \left\langle \kappa c_r \right\rangle^{\frac{1}{2}}\right)$. We write
 \begin{align*}
    \frac{\Phi_{3,\alpha}(-1)}{\Phi_{3,\alpha}'(-1)}=&\frac{\phi_{A,1}(-1)+\sum_{j=1}^{+\infty}\Phi_{3,\alpha}^{[j]}(-1)}{\phi_{A,1}'(-1)+\sum_{j=1}^{+\infty}\Phi_{3,\alpha}^{[j]\prime}(-1)}=\frac{\phi_{A,1}(-1)}{\phi_{A,1}'(-1)}+Res_{airy,1},
  \end{align*}
where
\begin{align*}
  Res_{airy,1}=O\left( \frac{\sum_{j=1}^{+\infty}\Phi_{3,\alpha}^{[j]}(-1)}{\kappa \left\langle \kappa c_r \right\rangle^{\frac{1}{2}}} \right)  +O\left(\frac{\sum_{j=1}^{+\infty}\Phi_{3,\alpha}^{[j]\prime}(-1)}{\kappa^2 \left\langle \kappa c_r \right\rangle} \right).
\end{align*}

By using Lemma \ref{lem-Phi3--1}, we get
\begin{align*}
  \left|Res_{airy,1}\right|\lesssim \frac{ 1}{\kappa^2\left\langle \kappa c_r \right\rangle}+|c_i|c_r.
\end{align*} 

It is clear that $\pa_{c_r}\phi_{A,1}(-1)=0$ and $\pa_{c_r}\phi_{A,1}'(-1)=O\left(\frac{\kappa^2}{ \left\langle \kappa c_r \right\rangle^{\frac{1}{2}}}\right)$. Then by using Lemma \ref{lem-Phi3--1-cr}, one can easily check that the derivative estimates for $Res_{airy,1}$ satisfies \eqref{est-res-airy-cr}-\eqref{est-res-airy-alpha}.

Next, we give the analysis for the main term $\frac{\phi_{A,1}(-1)}{\phi_{A,1}'(-1)}$.

By Lemma \ref{lem:pri-Airy-expan}, we have
\begin{align*}
  \frac{\phi_{A,1}(-1)}{\phi_{A,1}'(-1)}=\frac{A_1(2,-1)}{A_1(1,-1)}=&\frac{\mathcal A_1(2,\kappa \eta(-1))}{\kappa \eta_r'(-1)\mathcal A_1(1,\kappa \eta(-1))}+Res_{airy,2},
\end{align*}
where
\begin{align*}
  Res_{airy,2}=O \left(\frac{\mathcal A_1(3,\kappa \eta(-1))}{\kappa^2\mathcal A_1(1,\kappa \eta(-1))}+ \left(\frac{\mathcal A_1(2,\kappa \eta(-1))}{\kappa \mathcal A_1(1,\kappa \eta(-1))}\right)^2 \right).
\end{align*}
Then by using the technique in Lemma \ref{lem-d-cr-a1B}, one can easily check that
\begin{align*}
  \left|Res_{airy,2}\right|\lesssim \frac{1}{\kappa^2\left\langle \kappa c_r \right\rangle},\qquad \left|\pa_{c_r}Res_{airy,2}\right|+\left|\pa_{c_i}Res_{airy,2}\right|\lesssim \frac{1}{\kappa\left\langle \kappa c_r \right\rangle^2},
\end{align*}
for fixed $\alpha$
\begin{align*}
  \left|\pa_{\nu}Res_{airy,2}\right|\lesssim \frac{1}{\nu\kappa^2\left\langle \kappa c_r \right\rangle},
\end{align*}
and for fixed $\nu$
\begin{align*}
  \left|\pa_{\alpha^2}Res_{airy,2}\right|\lesssim \frac{1}{\alpha^2\kappa^2\left\langle \kappa c_r \right\rangle}.
\end{align*}

Let 
\begin{align*}
  Res_{airy}=Res_{airy,1}+Res_{airy,2},
\end{align*}
we complete the proof of this lemma.

\end{proof}

 From  \eqref{eq-Airy-prim-1-prop}, it holds that 
\begin{align*}
  \frac{\mathcal A_1(2,\kappa \eta(-1))}{\kappa \eta_r'(-1)\mathcal A_1(1,\kappa \eta(-1))}=\frac{e^{-\frac{\pi}{6}i}\mathrm{Ai}(2,e^{\frac{\pi}{6}i}\kappa \eta(-1))}{\kappa \eta_r'(-1)\mathrm{Ai}(1,e^{\frac{\pi}{6}i}\kappa \eta(-1))}.
\end{align*}

When $\left|\kappa\eta(-1)\right|$ is sufficiently large, the Airy asymptotic expansion yield the following expansion of the main term.
\begin{lemma}\label{lem-Phi3-main--1}
  For $\left|\kappa\eta(-1)\right|\ge \max(M,100)$, where $M$ is the constant given in Remark \ref{rmk:Airy-class-est}, it holds that
  \begin{align}\label{est-Phi3-main--1}
    \frac{e^{-\frac{\pi}{6}i}\mathrm{Ai}(2,e^{\frac{\pi}{6}i}\kappa \eta(-1))}{\kappa \eta_r'(-1)\mathrm{Ai}(1,e^{\frac{\pi}{6}i}\kappa \eta(-1))}=-\frac{e^{\frac{\pi}{4}i}}{\kappa }\frac{1}{\left(\frac{\kappa c}{2}\right)^{\frac{1}{2}}} \left(1+Res_{airy,err}\right),
  \end{align}
  where $|Res_{airy,err}|\le \frac{1}{10}$.
\end{lemma}
\begin{proof}
    From the expansion in Lemma \ref{lem:Airy-class-est}, we have
  \begin{align*}
    \frac{e^{-\frac{\pi}{6}i}\mathrm{Ai}(2,e^{\frac{\pi}{6}i}\kappa \eta(-1))}{\kappa \eta_r'(-1)\mathrm{Ai}(1,e^{\frac{\pi}{6}i}\kappa \eta(-1))}=&-\frac{e^{-\frac{\pi}{6}i}}{\kappa \eta_r'(-1)}\frac{1}{\left(e^{\frac{\pi}{6}i}\kappa \eta(-1)\right)^{\frac{1}{2}}} \left(1+O \left(\frac{1}{\left\langle \kappa\eta(-1)\right\rangle^{\frac{3}{2}}}\right)  \right)\\
    =&-\frac{e^{\frac{\pi}{4}i}}{\kappa \eta_r'(-1)}\frac{1}{\left(-\kappa \eta(-1)\right)^{\frac{1}{2}}} \left(1+O \left(\frac{1}{\left\langle \kappa\eta(-1)\right\rangle^{\frac{3}{2}}}\right)  \right),
  \end{align*}
Recall Lemma \ref{lem-Langer}. It holds that
\begin{align*}
   \eta_r'(-1)=1+O(c_r),\quad \eta(-1)=-\frac{c}{2}+O(c_r^2).
\end{align*}
And \eqref{est-Phi3-main--1} follows directly.

\end{proof}

Next, we derive the parameter derivatives of the main Airy ratio.
\begin{lemma}\label{lem-Phi3-main--1-cr}
  It holds that
  \begin{align}\label{est-Phi3-main--1-cr-0}
    \pa_{c_r} \frac{\mathcal A_1(2,\kappa \eta(-1))}{\kappa \eta_r'(-1)\mathcal A_1(1,\kappa \eta(-1))}=-\frac{1}{2} \left(1-\frac{\mathrm{Ai}(e^{\frac{\pi}{6}i}\kappa \eta(-1))\mathrm{Ai}(2,e^{\frac{\pi}{6}i}\kappa \eta(-1))}{\left(\mathrm{Ai}(1,e^{\frac{\pi}{6}i}\kappa \eta(-1))\right)^2}\right)+O \left(c_r\right),
  \end{align}
  \begin{align}\label{est-Phi3-main--1-ci-0}
    \pa_{c_i} \frac{\mathcal A_1(2,\kappa \eta(-1))}{\kappa \eta_r'(-1)\mathcal A_1(1,\kappa \eta(-1))}=-i\frac{1}{2} \left(1-\frac{\mathrm{Ai}(e^{\frac{\pi}{6}i}\kappa \eta(-1))\mathrm{Ai}(2,e^{\frac{\pi}{6}i}\kappa \eta(-1))}{\left(\mathrm{Ai}(1,e^{\frac{\pi}{6}i}\kappa \eta(-1))\right)^2}\right)+O \left(c_r\right).
  \end{align}

  For $\kappa c_r\gg 1$ and $|c_i|\lesssim c_r$, it holds that
    \begin{align}\label{est-Phi3-main--1-cr}
     \pa_{c_r}\frac{e^{-\frac{\pi}{6}i}\mathrm{Ai}(2,e^{\frac{\pi}{6}i}\kappa \eta(-1))}{\kappa \eta_r'(-1)\mathrm{Ai}(1,e^{\frac{\pi}{6}i}\kappa \eta(-1))}=\frac{1}{4}\frac{e^{\frac{\pi}{4}i}}{ \left(\frac{\kappa c_r}{2}\right)^{\frac{3}{2}} } \left(1+O\left(c_r+\frac{|c_i|}{c_r}+\frac{1}{\left\langle \kappa c_r\right\rangle^{\frac{3}{2}}}\right)\right),
  \end{align}
     \begin{align}\label{est-Phi3-main--1-ci}
     \pa_{c_i}\frac{e^{-\frac{\pi}{6}i}\mathrm{Ai}(2,e^{\frac{\pi}{6}i}\kappa \eta(-1))}{\kappa \eta_r'(-1)\mathrm{Ai}(1,e^{\frac{\pi}{6}i}\kappa \eta(-1))}=i\frac{1}{4}\frac{e^{\frac{\pi}{4}i}}{ \left(\frac{\kappa c_r}{2}\right)^{\frac{3}{2}} } \left(1+O\left(c_r+\frac{|c_i|}{c_r}+\frac{1}{\left\langle \kappa c_r\right\rangle^{\frac{3}{2}}}\right)\right).
  \end{align} 
  For fixed $\alpha$, it holds that
      \begin{align}\label{est-Phi3-main--1-nu}
     \pa_{\nu}\frac{e^{-\frac{\pi}{6}i}\mathrm{Ai}(2,e^{\frac{\pi}{6}i}\kappa \eta(-1))}{\kappa \eta_r'(-1)\mathrm{Ai}(1,e^{\frac{\pi}{6}i}\kappa \eta(-1))}=-\frac{c_r}{4\nu}\frac{e^{\frac{\pi}{4}i}}{ \left(\frac{\kappa c_r}{2}\right)^{\frac{3}{2}} }\left(1+O\left(c_r+\frac{|c_i|}{c_r}+\frac{1}{\left\langle \kappa c_r\right\rangle^{\frac{3}{2}}}\right)\right).
  \end{align}
  For fixed $\nu$, it holds that
        \begin{align}\label{est-Phi3-main--1-alpha}
     \pa_{\alpha^2}\frac{e^{-\frac{\pi}{6}i}\mathrm{Ai}(2,e^{\frac{\pi}{6}i}\kappa \eta(-1))}{\kappa \eta_r'(-1)\mathrm{Ai}(1,e^{\frac{\pi}{6}i}\kappa \eta(-1))}=\frac{c_r}{8\alpha^2}\frac{e^{\frac{\pi}{4}i}}{ \left(\frac{\kappa c_r}{2}\right)^{\frac{3}{2}} }\left(1+O\left(c_r+\frac{|c_i|}{c_r}+\frac{1}{\left\langle \kappa c_r\right\rangle^{\frac{3}{2}}}\right)\right).
  \end{align}
\end{lemma}
\begin{proof}

Recall Lemma \ref{lem-Langer} that $\pa_{c_r}\eta(-1)=-\frac{1}{2}+O(c_r)$. We have
  \begin{align*}
    &\pa_{c_r}\frac{e^{-\frac{\pi}{6}i}\mathrm{Ai}(2,e^{\frac{\pi}{6}i}\kappa \eta(-1))}{\kappa \eta_r'(-1)\mathrm{Ai}(1,e^{\frac{\pi}{6}i}\kappa \eta(-1))}\\
    =&\frac{\pa_{c_r}\left(\kappa \eta(-1)\right)}{\kappa \eta_r'(-1)}\frac{\left(\mathrm{Ai}(1,e^{\frac{\pi}{6}i}\kappa \eta(-1))\right)^2-\mathrm{Ai}(e^{\frac{\pi}{6}i}\kappa \eta(-1))\mathrm{Ai}(2,e^{\frac{\pi}{6}i}\kappa \eta(-1))}{\left(\mathrm{Ai}(1,e^{\frac{\pi}{6}i}\kappa \eta(-1))\right)^2}+O \left( \frac{\mathrm{Ai}(2,e^{\frac{\pi}{6}i}\kappa \eta(-1))}{\kappa \mathrm{Ai}(1,e^{\frac{\pi}{6}i}\kappa \eta(-1))} \right)\\
    =&-\frac{1}{2} \left(1-\frac{\mathrm{Ai}(e^{\frac{\pi}{6}i}\kappa \eta(-1))\mathrm{Ai}(2,e^{\frac{\pi}{6}i}\kappa \eta(-1))}{\left(\mathrm{Ai}(1,e^{\frac{\pi}{6}i}\kappa \eta(-1))\right)^2}\right)\left(1+O(c_r)\right)+O \left( \frac{\mathrm{Ai}(2,e^{\frac{\pi}{6}i}\kappa \eta(-1))}{\kappa \mathrm{Ai}(1,e^{\frac{\pi}{6}i}\kappa \eta(-1))} \right).
  \end{align*}

From the expansion in Lemma \ref{lem:Airy-class-est}, we have
\begin{align*}
  \frac{\mathrm{Ai}(e^{\frac{\pi}{6}i}\kappa \eta(-1))\mathrm{Ai}(2,e^{\frac{\pi}{6}i}\kappa \eta(-1))}{\left(\mathrm{Ai}(1,e^{\frac{\pi}{6}i}\kappa \eta(-1))\right)^2}=1-\frac{1}{2} \frac{1}{\left(e^{\frac{\pi}{6}i}\kappa \eta(-1)\right)^{\frac{3}{2}}}+O\left( \frac{1}{\left\langle \kappa c_r\right\rangle^3} \right).
\end{align*}
Therefore, we have
\begin{align*}
  \pa_{c_r}\frac{e^{-\frac{\pi}{6}i}\mathrm{Ai}(2,e^{\frac{\pi}{6}i}\kappa \eta(-1))}{\kappa \eta_r'(-1)\mathrm{Ai}(1,e^{\frac{\pi}{6}i}\kappa \eta(-1))}=-\frac{1}{4}\frac{1}{\left(e^{\frac{\pi}{6}i}\kappa \eta(-1)\right)^{\frac{3}{2}}}+O\left( \frac{c_r}{\left\langle \kappa c_r\right\rangle^{\frac{3}{2}}}+ \frac{1}{\left\langle \kappa c_r\right\rangle^3}\right).
\end{align*}

Then for $\kappa c_r\gg 1$, and $|c_i|\lesssim c_r$, it holds that
\begin{align*}
  -\frac{1}{4}\frac{1}{\left(e^{\frac{\pi}{6}i}\kappa \eta(-1)\right)^{\frac{3}{2}}}=\frac{1}{4}\frac{e^{\frac{\pi}{4}i}}{ \left(\frac{\kappa c_r}{2}\right)^{\frac{3}{2}} } \left(1+O(c_r+\frac{|c_i|}{c_r})\right).
\end{align*}

Similarly, we have the $c_i$-derivative \eqref{est-Phi3-main--1-ci}.

For $\nu$-derivative, there is a little difference. For fixed $\alpha$, we have
  \begin{align*}
    &\pa_{\nu}\frac{e^{-\frac{\pi}{6}i}\mathrm{Ai}(2,e^{\frac{\pi}{6}i}\kappa \eta(-1))}{\kappa \eta_r'(-1)\mathrm{Ai}(1,e^{\frac{\pi}{6}i}\kappa \eta(-1))}\\
    =&\frac{\pa_{\nu}\left(\kappa \eta(-1)\right)}{\kappa \eta_r'(-1)}\frac{\left(\mathrm{Ai}(1,e^{\frac{\pi}{6}i}\kappa \eta(-1))\right)^2-\mathrm{Ai}(e^{\frac{\pi}{6}i}\kappa \eta(-1))\mathrm{Ai}(2,e^{\frac{\pi}{6}i}\kappa \eta(-1))}{\left(\mathrm{Ai}(1,e^{\frac{\pi}{6}i}\kappa \eta(-1))\right)^2}-\frac{\pa_{\nu}\kappa }{\kappa^2} \frac{e^{-\frac{\pi}{6}i}\mathrm{Ai}(2,e^{\frac{\pi}{6}i}\kappa \eta(-1))}{\eta_r'(-1)\mathrm{Ai}(1,e^{\frac{\pi}{6}i}\kappa \eta(-1))}\\
    =&-\frac{1}{3}\frac{1}{\nu}\frac{\eta(-1)}{\eta_r'(-1)}\frac{1}{2}\frac{1+O \left(\frac{1}{\left\langle \kappa c_r\right\rangle^{\frac{3}{2}}}\right)}{\left(e^{\frac{\pi}{6}i}\kappa \eta(-1)\right)^{\frac{3}{2}}} -\frac{1}{3}\frac{1}{\nu}\frac{e^{-\frac{\pi}{6}i}}{\kappa \eta_r'(-1)}\frac{1+O \left(\frac{1}{\left\langle \kappa c_r\right\rangle^{\frac{3}{2}}}\right)}{\left(e^{\frac{\pi}{6}i}\kappa \eta(-1)\right)^{\frac{1}{2}}} \\
    =&-\frac{1}{2}\frac{1}{\nu}\frac{\eta_r(-1)}{\eta_r'(-1)}\frac{1}{\left(e^{\frac{\pi}{6}i}\kappa \eta(-1)\right)^{\frac{3}{2}}}\left( 1+O \left(\frac{1}{\left\langle \kappa c_r\right\rangle^{\frac{3}{2}}}\right) \right)\\
    =&-\frac{c_r}{4\nu}\frac{e^{\frac{\pi}{4}i}}{ \left(\frac{\kappa c_r}{2}\right)^{\frac{3}{2}} } \left(1+O \left(c_r+\frac{|c_i|}{c_r}+\frac{1}{\left\langle \kappa c_r\right\rangle^{\frac{3}{2}}}\right)\right).
  \end{align*}
The $\alpha^2$-derivative can be deduced in the same way.

\end{proof}

\subsection{Construction of $\Phi_{4,\alpha}$}
In this subsection, we construct the right fast mode $\Phi_{4,\alpha}$ starting from the Airy profile $\phi_{A,2}$.

By the normalization of $\phi_{A,2}$, we have
\begin{align*}
  \phi_{A,2}(0)=1,\qquad \phi_{A,2}(-1)=0,\quad \phi_{A,2}'(-1)=0.
\end{align*}
Formally, the Airy asymptotic expansion give
\begin{align*}
  \phi_{A,2}(y)\sim& \kappa^{\frac{5}{4}}\left\langle \kappa\eta(y) \right\rangle^{-\frac{5}{4}}e^{\frac{2}{3}\frac{1}{\sqrt2}|\kappa\eta_r(y)|^{\frac{1}{2}}\kappa\eta_r(y)-\frac{2}{3}\frac{1}{\sqrt2}|\kappa \eta_r(0)|^{\frac{1}{2}}\kappa \eta_r(0)},\\
  \pa_y\phi_{A,2}(y)\sim & \kappa^{\frac{9}{4}} \left\langle \kappa\eta(y) \right\rangle^{-\frac{3}{4}}e^{\frac{2}{3}\frac{1}{\sqrt2}|\kappa\eta_r(y)|^{\frac{1}{2}}\kappa\eta_r(y)-\frac{2}{3}\frac{1}{\sqrt2}|\kappa \eta_r(0)|^{\frac{1}{2}}\kappa \eta_r(0)},\\
  \pa_y^2\phi_{A,2}(y)\sim & \kappa^{\frac{13}{4}} \left\langle \kappa\eta(y) \right\rangle^{-\frac{1}{4}}e^{\frac{2}{3}\frac{1}{\sqrt2}|\kappa\eta_r(y)|^{\frac{1}{2}}\kappa\eta_r(y)-\frac{2}{3}\frac{1}{\sqrt2}|\kappa \eta_r(0)|^{\frac{1}{2}}\kappa \eta_r(0)},\\
  \pa_y^3\phi_{A,2}(y)\sim & \kappa^{\frac{17}{4}} \left\langle \kappa\eta(y) \right\rangle^{\frac{1}{4}}e^{\frac{2}{3}\frac{1}{\sqrt2}|\kappa\eta_r(y)|^{\frac{1}{2}}\kappa\eta_r(y)-\frac{2}{3}\frac{1}{\sqrt2}|\kappa \eta_r(0)|^{\frac{1}{2}}\kappa \eta_r(0)}.
\end{align*}

As in the construction of $\Phi_{3,\alpha}$, we introduce an approximate Green function adapted to the right fast mode:
\begin{equation}\label{eq-Green-Airy+}
   G_{A+}(x,y)=i \frac{2\pi}{\left(\eta_r'(x)\right)^{\frac{1}{2}}\kappa\varepsilon}\left\{
    \begin{array}{ll}
       A_1(2,y) A_2(x)+a_2(x)-a_1(x)(y-y_c)+a_1(x)(x-y_c),&x<y;\\ 
       A_1(x) A_2(2,y),&x>y,
    \end{array}
  \right.
\end{equation}
and define the corresponding modified Airy solver by
\begin{align}\label{eq-modified-Airysolver+}
  AirySolver_{m+}(f)(y)=\int^0_{-1}G_{A+}(x,y)f(x)dx,
\end{align}
which satisfies
\begin{align*}
  AirySolver_{m+}(f)(-1)=0,\quad \pa_yAirySolver_{m+}(f)(-1)=0.
\end{align*}

We will use $AirySolver_{m+}$ to carry out the Airy iteration starting from $\phi_{A,2}$. Here we also introduce a weight function
\begin{align}\label{eq-weight-MW-+}
  \widetilde{\mathcal W}(y)=\left|e^{\frac{1}{3}\left( \left(e^{\frac{5\pi i}{6}}\kappa\eta(y)\right)^{\frac{3}{2}}-\left(e^{\frac{5\pi i}{6}}\kappa\eta(0)\right)^{\frac{3}{2}}\right)}\right|,
\end{align}
and the weighted norm
\begin{align*}
  \left\|f\right\|_{L^\infty_{\widetilde{\mathcal W}}}=\left\|f(y)\widetilde{\mathcal W}\right\|_{L^\infty}.
\end{align*}

The proof is identical to that of Lemma \ref{lem-AS--0}, and is omitted.
\begin{lemma}\label{lem-AS--0-+}
  For any $s\ge0$, it holds that
  \begin{align}
     \left\|\left\langle \kappa\eta \right\rangle^{\frac{1}{2}-s}AirySolver_{m+}(f)\right\|_{L^\infty_{\widetilde{\mathcal W}}}\lesssim \varepsilon^{\frac{1}{3}}\left\|\left\langle \kappa\eta \right\rangle^{-s}f\right\|_{L^\infty_{\widetilde{\mathcal W}}},\label{est-AS--0-+}\\
     \left\|\left\langle \kappa\eta \right\rangle^{\frac{3}{2}-s}\pa_yAirySolver_{m+}(f)\right\|_{L^\infty_{\widetilde{\mathcal W}}}\lesssim \left\|\left\langle \kappa\eta \right\rangle^{-s}f\right\|_{L^\infty_{\widetilde{\mathcal W}}},\label{est-AS--1-+}\\
     \left\|\left\langle \kappa\eta \right\rangle^{1-s}\pa_y^2AirySolver_{m+}(f)\right\|_{L^\infty_{\widetilde{\mathcal W}}}\lesssim \varepsilon^{-\frac{1}{3}} \left\|\left\langle \kappa\eta \right\rangle^{-s}f\right\|_{L^\infty_{\widetilde{\mathcal W}}},\label{est-AS--2-+}\\
     \left\|\left\langle \kappa\eta \right\rangle^{\frac{1}{2}-s}\pa_y^3AirySolver_{m+}(f)\right\|_{L^\infty_{\widetilde{\mathcal W}}}\lesssim \varepsilon^{-\frac{2}{3}} \left\|\left\langle \kappa\eta \right\rangle^{-s}f\right\|_{L^\infty_{\widetilde{\mathcal W}}}\label{est-AS--3-+}.
  \end{align} 
\end{lemma}

\begin{lemma}\label{lem-sol-nh-OS-+}
  For any $f\in L^\infty_{\widetilde{\mathcal W}}$, there exists a solution $\Phi_f$ to \eqref{eq-inhome-OS} satisfying
  \begin{align}\label{est-sol-nh-OS-+}
    \left\|\Phi_{f} \right\|_{L^\infty_{\widetilde{\mathcal W}}}\lesssim \varepsilon^{\frac{1}{3}}\left\|f\right\|_{L^\infty_{\widetilde{\mathcal W}}},
  \end{align}
  and
  \begin{align}\label{est-sol-nh-b0-+}
     \Phi_{f}(-1)=0,\qquad \pa_y\Phi_{f}(-1)=0.
  \end{align}
\end{lemma}

Then we define
\begin{align*}
  \Phi_{4,\alpha}(y)=\sum_{j=0}^{+\infty}\Phi_{4,\alpha}^{[j]}(y),
\end{align*}
where
\begin{align*}
  \Phi_{4,\alpha}^{[0]}=\phi_{A,2}(y),\quad \Phi_{4,\alpha}^{[1]}(y)=-AirySolver_{m+}\left(Orr\left(\phi_{A,2}(y)\right)\right),\\
  \Phi_{4,\alpha}^{[j]}(y)=-AirySolver_{m+}\left(Orr\left(\Phi_{4,\alpha}^{[j-1]}\right)-Orr\left(\Phi_{4,\alpha}^{[j-2]}\right)\right),\text{ for }j\ge2.
\end{align*}

\subsection{Boundary values of $\Phi_{4,\alpha}$ at $y=-1$ and $y=0$}
From the definition of $\Phi_{4,\alpha}(y)$ and $\phi_{A,2}(y)$, it is clear that $\Phi_{4,\alpha}(-1)=0$, and $\pa_y\Phi_{4,\alpha}(-1)=0$.

SAs for $\Phi_{3,\alpha}$, the boundary behavior of $\Phi_{4,\alpha}$ at $y=0$ is governed, to leading order, by the Airy profile $\phi_{A,2}$. The following estimates are obtained by the same argument as in Lemma \ref{lem-Phi3--1} and Lemma \ref{lem-Phi3--1-cr}.

\begin{lemma}\label{lem-Phi4--1}
  It holds that
  \begin{align*}
    \left|\pa_y\Phi_{4,\alpha}(0)-\pa_y\phi_{A,2}(0)\right|\lesssim  \varepsilon^{\frac{1}{2}}+|c_i|\varepsilon^{-\frac{1}{2}} ,\\
      \left|\pa_y^3\Phi_{4,\alpha}(0)-\pa_y^3\phi_{A,2}(0)\right|\lesssim  \varepsilon^{-\frac{1}{2}}+|c_i|\varepsilon^{-\frac{3}{2}}.
  \end{align*}
\end{lemma} 

\begin{lemma}\label{lem-Phi4--1-cr}
   It holds that
  \begin{align}
    \left|\pa_{c_r}\pa_y\Phi_{4,\alpha}(0)-\pa_{c_r}\pa_y\phi_{A,2}(0)\right|\lesssim  \varepsilon^{-\frac{1}{6}}+|c_i|\varepsilon^{-\frac{5}{6}}\left|\ln c_0\right|,\label{est-Phi4--1-cr-0}\\
      \left|\pa_{c_r}\pa_y^3\Phi_{4,\alpha}(0)-\pa_{c_r}\pa_y^3\phi_{A,2}(0)\right|\lesssim \varepsilon^{-\frac{5}{6}}+|c_i|\varepsilon^{-\frac{11}{6}}\left|\ln c_0\right|,\label{est-Phi4--1-cr-1}
  \end{align}
  and
    \begin{align}
    \left|\pa_{c_i}\pa_y\Phi_{4,\alpha}(0)-\pa_{c_i}\pa_y\phi_{A,2}(0)\right|\lesssim  \varepsilon^{-\frac{1}{6}} + \varepsilon^{-\frac{1}{2}}\left|\ln c_0\right|,\label{est-Phi4--1-ci-0}\\
      \left|\pa_{c_i}\pa_y^3\Phi_{4,\alpha}(0)-\pa_{c_i}\pa_y^3\phi_{A,2}(0)\right|\lesssim \varepsilon^{-\frac{1}{2}}+\varepsilon^{-\frac{3}{2}}\left|\ln c_0\right|.\label{est-Phi4--1-ci-1}
  \end{align}
  For fixed $\alpha$, we have
      \begin{align}
    \left|\pa_{\nu}\pa_y\Phi_{4,\alpha}(0)-\pa_{\nu}\pa_y\phi_{A,2}(0)\right|\lesssim \frac{1}{\nu}\left( \varepsilon^{\frac{1}{2}}+|c_i|\varepsilon^{-\frac{1}{2}}\right)\left|\ln c_0\right|,\label{est-Phi4--1-nu-0}\\
      \left|\pa_{\nu}\pa_y^3\Phi_{4,\alpha}(0)-\pa_{\nu}\pa_y^3\phi_{A,2}(0)\right|\lesssim \frac{1}{\nu}\left(\varepsilon^{-\frac{1}{2}}+|c_i|\varepsilon^{-\frac{3}{2}}\right)\left|\ln c_0\right|.\label{est-Phi4--1-nu-1}
  \end{align}
    For fixed $\nu$ or $\varepsilon$, we have
      \begin{align}
    \left|\pa_{\alpha^2}\pa_y\Phi_{4,\alpha}(0)-\pa_{\alpha^2}\pa_y\phi_{A,2}(0)\right|\lesssim  \frac{1}{\alpha^2}\left( \varepsilon^{\frac{1}{2}}+|c_i|\varepsilon^{-\frac{1}{2}}\right)\left|\ln c_0\right|,\label{est-Phi4--1-alpha-0}\\
      \left|\pa_{\alpha^2}\pa_y^3\Phi_{4,\alpha}(0)-\pa_{\alpha^2}\pa_y^3\phi_{A,2}(0)\right|\lesssim \frac{1}{\alpha^2}\left(\varepsilon^{-\frac{1}{2}}+|c_i|\varepsilon^{-\frac{3}{2}}\right)\left|\ln c_0\right|.\label{est-Phi4--1-alpha-1}
  \end{align}
\end{lemma}

\section{Dispersion relation}\label{sec-dispersion_relation}
In this section, we derive the dispersion relation for the even Orr--Sommerfeld modes subject to the no-slip boundary conditions. 

We look for a nontrivial
solution of the homogeneous Orr--Sommerfeld equation in the form
\begin{align*}
  \Phi(y)=B_1 \Phi_{1,\alpha}+B_2 \Phi_{2,\alpha}+B_3 \Phi_{3,\alpha}+B_4 \Phi_{4,\alpha}.
\end{align*}
Imposing the even boundary conditions \eqref{bc-even}, namely the no-slip
conditions at $y=-1$ and the symmetry conditions at $y=0$, gives the
following Wronskian condition:
\begin{align}\label{eq-wron-cond}
  W(c_r,c_i,\alpha,\nu)=\det\left(
    \begin{array}{cccc}
      \Phi_{1,\alpha}(-1)&\Phi_{2,\alpha}(-1)&\Phi_{3,\alpha}(-1)&\Phi_{4,\alpha}(-1)\\       
      \Phi_{1,\alpha}'(-1)&\Phi_{2,\alpha}'(-1)&\Phi_{3,\alpha}'(-1)&\Phi_{4,\alpha}'(-1)\\       
      \Phi_{1,\alpha}'(0)&\Phi_{2,\alpha}'(0)&\Phi_{3,\alpha}'(0)&\Phi_{4,\alpha}'(0)\\       
      \Phi_{1,\alpha}'''(0)&\Phi_{2,\alpha}'''(0)&\Phi_{3,\alpha}'''(0)&\Phi_{4,\alpha}'''(0)
    \end{array}
  \right)=0.
\end{align}
This equation is the eigenvalue dispersion relation. Its zeros determine the possible values of \(c=c_r+i c_i\).

We first fix $\varepsilon$ and look for two functions
$c_r(\alpha^2)$ and $c_i(\alpha^2)$ of $\alpha^2$ such that
\begin{align*}
  W(c_r(\alpha^2),c_i(\alpha^2),\alpha, \nu)=0.
\end{align*}
Since $\varepsilon$ is fixed, the dependence on $\nu$ will be suppressed in the following argument. Therefore, when no confusion can
arise, we abbreviate
$W(c_r(\alpha^2),c_i(\alpha^2),\alpha,\nu)$ as
$W(c_r(\alpha^2),c_i(\alpha^2),\alpha)$.

From the information in hand, we give the formal scale of each element in \eqref{eq-wron-cond}
\begin{align}\label{eq-matrix-scale}
  \left(
    \begin{array}{cccc}
      \Phi_{1,\alpha}(-1)&\Phi_{2,\alpha}(-1)&\Phi_{3,\alpha}(-1)&\Phi_{4,\alpha}(-1)\\       
      \Phi_{1,\alpha}'(-1)&\Phi_{2,\alpha}'(-1)&\Phi_{3,\alpha}'(-1)&\Phi_{4,\alpha}'(-1)\\       
      \Phi_{1,\alpha}'(0)&\Phi_{2,\alpha}'(0)&\Phi_{3,\alpha}'(0)&\Phi_{4,\alpha}'(0)\\       
      \Phi_{1,\alpha}'''(0)&\Phi_{2,\alpha}'''(0)&\Phi_{3,\alpha}'''(0)&\Phi_{4,\alpha}'''(0)
    \end{array}
  \right)\sim  \left(
    \begin{array}{cccc}
      -\hat c+\frac{4\alpha^2}{15}&1&1&0\\       
      1&|\ln c_0|&\kappa \left\langle \kappa c_r \right\rangle^{\frac{1}{2}}&0\\       
      0&\frac{1}{1-\hat c}&0&\kappa^{\frac{3}{2}}\\       
      \left|\ln c_0\right|&\left|\ln c_0\right|&e^{-\kappa}&\kappa^{\frac{9}{2}}
    \end{array}
  \right).
\end{align}
From the above matrix, we see that the eigenspace is at most one-dimensional.

We now reduce the Wronskian to its leading part. Expanding the determinant in \eqref{eq-wron-cond} along the fourth row and keeping the dominant contribution containing $\Phi_{4,\alpha}'''(0)$, we obtain
\begin{align*}
  W(c_r,c_i,\alpha^2)=&\Phi_{4,\alpha}'''(0)\det\left(
    \begin{array}{ccc}
      \Phi_{1,\alpha}(-1)&\Phi_{2,\alpha}(-1)&\Phi_{3,\alpha}(-1)\\       
      \Phi_{1,\alpha}'(-1)&\Phi_{2,\alpha}'(-1)&\Phi_{3,\alpha}'(-1)\\       
      \Phi_{1,\alpha}'(0)&\Phi_{2,\alpha}'(0)&\Phi_{3,\alpha}'(0)
    \end{array}
  \right)+Res_0(c_r,c_i,\alpha^2)\\
  =&-\frac{\Phi_{4,\alpha}'''(0)\Phi_{1,\alpha}'(-1)\Phi_{3,\alpha}'(-1)}{1-\hat c} \left(\frac{\Phi_{1,\alpha}(-1)}{\Phi_{1,\alpha}'(-1)}-  \frac{\Phi_{3,\alpha}(-1)}{\Phi_{3,\alpha}'(-1)}\right)+Res_0(c_r,c_i,\alpha^2),
\end{align*}
where $Res_0$ collects all the remaining cofactors and satisfies
\begin{align*}
  \left|Res_0(c_r,c_i,\alpha^2)\right|\lesssim \kappa^{\frac{5}{2}}\left\langle \kappa c_r \right\rangle^{\frac{1}{2}}\left|\ln c_0\right|.
\end{align*}
 
We therefore introduce the normalized dispersion function
\begin{align*}
  F(c_r,c_i,\alpha^2)=F_r(c_r,c_i,\alpha^2)+iF_i(c_r,c_i,\alpha^2) =-\frac{1-\hat c}{\Phi_{4,\alpha}'''(0)\Phi_{1,\alpha}'(-1)\Phi_{3,\alpha}'(-1)}W(c_r,c_i,\alpha^2).
\end{align*}
Equivalently,
\begin{align}\label{eq-F}
  F(c_r,c_i,\alpha^2)=\frac{\Phi_{1,\alpha}(-1)}{\Phi_{1,\alpha}'(-1)}-  \frac{\Phi_{3,\alpha}(-1)}{\Phi_{3,\alpha}'(-1)}+Res_1(c_r,c_i,\alpha^2),
\end{align}
with
\begin{align*}
  \left|Res_1(c_r,c_i,\alpha^2)\right|\lesssim \kappa^{-3}\left|\ln c_0\right|.
\end{align*} 
By Lemma \ref{lem-sol-Ray-hom}, Lemma \ref{lem-Ray-endpoint-expansion}, Lemma \ref{lem-res-ray}, Lemma \ref{lem-Phi1-0-cr}, Lemma \ref{lem-Phi2-cr}, Lemma \ref{lem-res-airy}, Lemma \ref{lem-Phi3-main--1-cr}, and Lemma \ref{lem-Phi4--1-cr}, we can see that the corresponding derivatives of $Res_1(c_r,c_i,\alpha^2)$ remain lower order compared
with the derivatives of the leading term in \eqref{eq-F}.

Therefore, it suffices to study the zero points of $F$. For each fixed \(\varepsilon\), the following result
gives the existence of the upper and lower neutral branches and localizes the
possible values of \(c\) for each \(\alpha\). This will be the key step in the
proof of the main theorem.
\begin{lemma}\label{lem-sol-curve}
  For $\varepsilon$ sufficiently small, there exists a unique solution curve $\left(c_r(\alpha^2), c_i(\alpha^2)\right)$ in the region $\mathbb H$ such that 
  \begin{align*}
  F(c_r(\alpha^2),c_i(\alpha^2),\alpha^2)=0.
\end{align*}
Moreover, there exist $\widetilde{\mathcal K}_-(\varepsilon)\sim\ 
    \mathcal K_-(\varepsilon)\sim 
    \mathcal K_+(\varepsilon)\sim 
    \widetilde{\mathcal K}_+(\varepsilon)\sim1$ depending continuously on $\varepsilon$, such that $\widetilde {\mathcal K}_-(\varepsilon)< {\mathcal K}_-(\varepsilon)$, ${\mathcal K}_+(\varepsilon)<\widetilde {\mathcal K}_+(\varepsilon)$,
\begin{align*}
  c_i \left({\mathcal K}_-(\varepsilon)\varepsilon^{\frac{1}{3}} \right)=0,\qquad c_i \left( \widetilde {\mathcal K}_-(\varepsilon)\varepsilon^{\frac{1}{3}} \right)=-\frac{c_0}{2},  
\end{align*}
and
\begin{align*}
 c_i \left({\mathcal K}_+(\varepsilon)\varepsilon^{\frac15} \right)=0,\qquad c_i \left(\widetilde {\mathcal K}_+(\varepsilon)\varepsilon^{\frac15} \right)=-\frac{c_0}{2}.
\end{align*}
Along this curve, we have $c_r(\alpha^2)\sim \alpha^2$. In addition, in the intermediate regime $\varepsilon^{\frac{1}{3}}\ll \alpha^2\ll \varepsilon^{\frac{1}{5}}$, one has $c_i (\alpha^2)\sim \frac{1}{\kappa^{\frac{3}{2}}c_r^{\frac{1}{2}}}$. 
\end{lemma}
\begin{proof}
  We first localize the possible triples $\left(c_r, c_i, \alpha^2)\right)$ for which $F(c_r,c_i,\alpha^2)=0$. The discussion is divided according to the size of \(c_r\). 

  By Lemma \ref{lem-sol-Ray-hom}, Lemma \ref{lem-Ray-endpoint-expansion}, and Lemma \ref{lem-res-ray}, we have
\begin{align}
  \Re \left(\frac{\Phi_{1,\alpha}(-1)}{\Phi_{1,\alpha}'(-1)}\right)=&\frac{-c_r}{2}+\frac{2}{15}\alpha^2+O\left(\alpha^2 c+\alpha^4\left|\ln c_0\right|+\left(1+\frac{\alpha^2}{c_r}\right)\left(\varepsilon^{\frac{1}{3}}\alpha^2+c_0 \right)|\ln c_0|^2\right),\label{expan-Phi1-Re}\\
  \Im \left(\frac{\Phi_{1,\alpha}(-1)}{\Phi_{1,\alpha}'(-1)}\right)=&\frac{-c_i}{2}-\frac{1}{3}c_i\alpha^2+\frac{4}{15^2}\alpha^4\arg(-c)\label{expan-Phi1-Im}\\
  &+O \left(\alpha^4 c\left|\ln c_0\right|+\alpha^2 c^2\left|\ln c_0\right|+\left(1+\frac{\alpha^2}{c_r}\right)\left(\varepsilon^{\frac{1}{3}}\alpha^2+c_0 \right)|\ln c_0|^2\right).\nonumber
\end{align}

\paragraph{Step 1: Localization of possible zeros.}

\paragraph{1) Case $c_r\ge \max(M,100)\varepsilon^{\frac{1}{3}}$.} By Lemma \ref{lem-res-airy} and Lemma \ref{lem-Phi3-main--1}, we have
\begin{align}\label{expan-Phi3}
   \left(\frac{\Phi_{3,\alpha}(-1)}{\Phi_{3,\alpha}'(-1)}\right)=&-\frac{e^{\frac{\pi}{4}i}}{\kappa }\frac{1}{\left(\frac{\kappa c}{2}\right)^{\frac{1}{2}}} \left(1+Res_{airy,err}\right) +O \left( \frac{ 1}{\kappa^2\left\langle \kappa c_r \right\rangle}+|c_i|c_r\right).
\end{align}
Consequently,
\begin{align*}
  \left|\Re\left(\frac{\Phi_{3,\alpha}(-1)}{\Phi_{3,\alpha}'(-1)}\right)\right|+\left|Res_1(c_r,c_i,\alpha^2)\right|\lesssim\kappa^{-\frac{3}{2}}c_r^{-\frac{1}{2}}+ \kappa^{-3}\left|\ln c_0\right|\ll |c_r|.
\end{align*}
Therefore, if $F_r(c_r,c_i,\alpha^2)=0$, the leading terms in
\eqref{expan-Phi1-Re} must balance. Hence
\begin{align}\label{balance-cr-alpha}
  \frac{c_r}{2}\approx \frac{2}{15}\alpha^2.
\end{align}
In particular, in this range any zero of \(F\) satisfies $c_r\sim \alpha^2$. 

We now study the imaginary part under the balance $\frac{c_r}{2}\approx\frac{2}{15}\alpha^2$. We divided the problem into two cases.
\paragraph{1.1) Case $c_r\ll \varepsilon^{\frac{1}{5}}$.} In this regime, 
\begin{align*}
  \left|\frac{e^{\frac{\pi}{4}i}}{\kappa }\frac{1}{\left(\frac{\kappa c}{2}\right)^{\frac{1}{2}}}\right|\gg c_r^2\sim \alpha^4.
\end{align*}
Suppose first that
\begin{align*}
  \left| \Im \left( \frac{e^{\frac{\pi}{4}i}}{\kappa }\frac{1}{\left(\frac{\kappa c}{2}\right)^{\frac{1}{2}}} \right) \right|\ll \frac{1}{\kappa^{\frac{3}{2}}|c|^{\frac{1}{2}}}.
\end{align*}
Then the phase of $\left(\kappa c\right)^{\frac{1}{2}}$ must be close to \(\pi/4\),  or
equivalently \(c_i\gg c_r\). In this situation the term $-\frac{c_i}{2}$ dominates the imaginary part of \(F\), and we obtain
\begin{align*}
  \left|F_i(c_r,c_i,\alpha^2)+\frac{c_i}{2}\right|\ll|c_i|.
\end{align*}
Thus \(F_i(c_r,c_i,\alpha^2)\neq 0\), and hence such a point cannot be a zero
of \(F\).

It follows that, for \(F_i(c_r,c_i,\alpha^2)=0\), it is necessary that
\begin{align*}
  \left| \Im \left( \frac{e^{\frac{\pi}{4}i}}{\kappa }\frac{1}{\left(\frac{\kappa c}{2}\right)^{\frac{1}{2}}} \right) \right|\sim \frac{1}{\kappa^{\frac{3}{2}}|c|^{\frac{1}{2}}}.
\end{align*}
In particular, \(|c_i|\lesssim c_r\). Under this condition all the remaining
terms in \(F_i\) are lower order, and therefore
\begin{align*}
  \left|F_i(c_r,c_i,\alpha^2)-\Im \left( \frac{e^{\frac{\pi}{4}i}}{\kappa }\frac{1}{\left(\frac{\kappa c}{2}\right)^{\frac{1}{2}}} \left(1+Res_{airy,err}\right) \right)+\frac{c_i}{2}\right|\ll \left| \Im \left( \frac{e^{\frac{\pi}{4}i}}{\kappa }\frac{1}{\left(\frac{\kappa c}{2}\right)^{\frac{1}{2}}} \right) \right|.
\end{align*}
Thus, at a zero of $F_i(c_r,c_i,\alpha^2)$,
\begin{align*}
  \frac{c_i}{2}\approx\Im \left( \frac{e^{\frac{\pi}{4}i}}{\kappa }\frac{1}{\left(\frac{\kappa c}{2}\right)^{\frac{1}{2}}}\left(1+Res_{airy,err}\right) \right).
\end{align*}
Since in the present regime $\kappa c_r\ge \max(M,100)$, this gives $|c_i|\le \frac{1}{\kappa}\le\frac{c_r}{100}$, and
\begin{align*}
  c_i\sim \frac{1}{\kappa^{\frac{3}{2}}c_r^{\frac{1}{2}}}.
\end{align*}

\paragraph{1.2) Case $c_r\sim \varepsilon^{\frac{1}{5}}$.} Write $c_r=C_{c_r,+}\varepsilon^{\frac{1}{5}}$. In the transition from the previous regime, \(C_{c_r,+}\) grows from
\(o(1)\) to \(O(1)\). By the same analysis to the previous case, we have $|c_i|\lesssim\frac{1}{\kappa^{\frac{3}{2}}c_r^{\frac{1}{2}}}\ll c_r$. Therefore,
\begin{align*}
  \frac{4}{15^2}\alpha^4\arg(-c)=-\frac{4\pi}{15^2}\alpha^4+O \left(\alpha^4 \frac{|c_i|}{c_r}\right).
\end{align*}
Recalling the definition of $Res_{airy,err}$ in Lemma \ref{lem-Phi3-main--1},  we also have
\begin{align*}
  \Im \left( \frac{e^{\frac{\pi}{4}i}}{\kappa }\frac{1}{\left(\frac{\kappa c}{2}\right)^{\frac{1}{2}}}\left(1+Res_{airy,err}\right) \right)= \frac{1}{\kappa^{\frac{3}{2}}c_r^{\frac{1}{2}}}+O \left(\frac{|c_i|}{\kappa^{\frac{3}{2}}c_r^{\frac{3}{2}}}+\frac{c_r^{\frac{1}{2}}}{\kappa^{\frac{3}{2}}}+\frac{1}{\kappa^3c_r^2}\right).
\end{align*}

Together with the balance $\frac{c_r}{2}\approx\frac{2}{15}\alpha^2$, this implies
\begin{align*}
  \alpha^4\sim \frac{1}{\kappa^{\frac{3}{2}}c_r^{\frac{1}{2}}}\sim \varepsilon^{\frac{2}{5}}.
\end{align*}

Using $\kappa=\varepsilon^{-\frac{1}{3}}2^{\frac{1}{3}} \left(1+O(c_r)\right)$, we obtain
\begin{equation}\label{eq-exp-Fi-upb}
  \begin{aligned}    
      F_i(c_r,c_i,\alpha^2)=&\frac{-c_i}{2}-\frac{4\pi}{15^2}\alpha^4+\frac{1}{\kappa^{\frac{3}{2}}c_r^{\frac{1}{2}}}+o\left(\varepsilon^{\frac{2}{5}}\right)=\frac{-c_i}{2}-\frac{\pi}{4}c_r^2+\frac{1}{\sqrt2}\varepsilon^{\frac{1}{2}}c_r^{-\frac{1}{2}}+o\left(\varepsilon^{\frac{2}{5}}\right)\\
  =& \frac{-c_i}{2}-\frac{\pi}{4}C_{c_r,+}^2\varepsilon^{\frac{2}{5}}+\frac{C_{c_r,+}^{-\frac{1}{2}}}{\sqrt2}\varepsilon^{\frac{2}{5}}+o\left(\varepsilon^{\frac{2}{5}}\right).
  \end{aligned}
\end{equation} 
Therefore, a necessary condition for \(F_i=0\)  in this transition regime is $|c_i|\lesssim \varepsilon^{\frac{2}{5}}$.

\paragraph{2) Case $ c_r\le \max(M,100)\varepsilon^{\frac{1}{3}}$.}
Let $c_r=C_{c_r,-}\varepsilon^{\frac{1}{3}}$, where $C_{c_r,-}$ decreases from $\max(M,100)$. 
From the analysis in Case 1, at the point $\kappa c_r=\max(M,100)$ we have $c_i\le \frac{c_r}{100}$.

We claim that, as  $C_{c_r,-}$ decreases, no zero of \(F\) can have $c_i\ge\max(M,100)\varepsilon^{\frac{1}{3}}$. 
Indeed, if this were the case, then by Lemma  \ref{lem-Phi3-main--1},
\begin{align*}
  \left| \Im \left( \frac{e^{\frac{\pi}{4}i}}{\kappa }\frac{1}{\left(\frac{\kappa c}{2}\right)^{\frac{1}{2}}} \right) \right|
  \le \frac{1}{\kappa^{\frac{3}{2}}c_i^{\frac{1}{2}}}
  <\frac{c_i}{100}.
\end{align*}
The Airy contribution is then too small to balance the leading term $-\frac{c_i}{2}$ in $F_i$. Hence $F_i(c_r,c_i,\alpha^2)\neq 0$, a
contradiction. Thus, in this lower branch regime, any zero of $F$ must satisfy $c_i\lesssim \varepsilon^{\frac13}.$ 

\paragraph{Step 2: Construction of the solution curve near the lower branch.}

We now construct the solution curve
\[
  \alpha^2\mapsto (c_r(\alpha^2),c_i(\alpha^2)).
\]
We start from the lower branch, where \(c_r\sim \varepsilon^{\frac13}\).
Following the approach of C.C. Lin \cite{LinCC1955}, we introduce the
Tietjens function and the Hankel function which are analytic functions:
\begin{align*}
  \mathrm{Ti}\left(-\kappa \eta(-1)\right)=\frac{ \mathcal A_1(2,\kappa \eta(-1)) } {\kappa \eta(-1)\mathcal A_1(1,\kappa \eta(-1))},
\end{align*}
and
\begin{align*}
  \mathrm{Han}\left(-\kappa \eta(-1)\right)= \left(1-Ti \left(-\kappa \eta(-1)\right)\right)^{-1}=\left(1-\frac{ \mathcal A_1(2,\kappa \eta(-1)) } {\kappa \eta(-1)\mathcal A_1(1,\kappa \eta(-1))}\right)^{-1}.
\end{align*}

Then we write
\begin{align*}
  F(c_r,c_i,\alpha^2)=&-\frac{\eta(-1)}{\eta_r'(-1)}-\frac{c}{2}+\frac{\eta(-1)}{\eta_r'(-1)} \left(1-\frac{\mathcal A_1(2,\kappa \eta(-1))}{\kappa \eta(-1)\mathcal A_1(1,\kappa \eta(-1))}\right)+\frac{2}{15}\alpha^2+O\left(\alpha^2 c+c_0|\ln c_0|^2+c^2\right)\\
  =&-\frac{c}{2\mathrm{Han}\left( \frac{\kappa c}{2}\right)}+\frac{2}{15}\alpha^2+O\left(\alpha^2 c+c_0|\ln c_0|^2+c^2\right),
\end{align*}
where we use the fact that $-\frac{\eta(-1)}{\eta_r'(-1)}-\frac{c}{2}=O(c^2)$.

We first set \(c_i=0\). According to the numerical data in
\cite{Miles1960}, the condition \(F_i=0\) can hold only when
\[
  \frac{\kappa c_r}{2}\approx 2.3,
\]
where
\[
  \mathrm{Han}_i(2.3)\approx 0,
  \qquad
  \mathrm{Han}_r(2.3)\approx 2.294.
\]
Then $F_r(c_r,0,\alpha^2)=0$ gives
\begin{align*}
  c_r\approx 0.6117\alpha^2\approx 3.65\varepsilon^{\frac{1}{3}}.
\end{align*}
We now seek a function \(\widetilde c_r(\alpha^2)\) satisfying $F_r(\tilde c_r(\alpha^2),0,\alpha^2)=0$ in a neighborhood of the point  $\alpha^2\approx 5.967$. Recall that
\begin{align*}
  F_r(c_r,0,\alpha^2)=-\frac{c_r}{2}\frac{\mathrm{Han}_r\left(\frac{\kappa c_r}{2}\right)}{\mathrm{Han}_r^2\left(\frac{\kappa c_r}{2}\right)+\mathrm{Han}_i^2\left(\frac{\kappa c_r}{2}\right)}+\frac{2}{15}\alpha^2+err,
\end{align*}
where $err=O\left(\alpha^2 c+c_0|\ln c_0|^2+c^2\right)$. By Lemma \ref{lem-Ray-endpoint-expansion}, Remark \ref{rmk-Ray-endpoint-expansion}, Lemma \ref{lem-res-ray}, and Lemma \ref{lem-res-airy},
\begin{align*}
  \left|\pa_{c_r}err\right|\lesssim \varepsilon^{\frac{1}{3}}\left|\ln c_0\right|^3.
\end{align*}
Using $\mathrm{Han}_r'(2.3)\approx -0.6242$ and $\mathrm{Han}_i'(2.3)\approx 1.213$, we obtain around $\alpha^2\approx 5.967\varepsilon^{\frac{1}{3}}$ that
\begin{align*}
  \pa_{c_r}F_r(c_r,0,\alpha^2)\approx-0.49 .
\end{align*}
Similarly,
\begin{align*}
  \pa_{\alpha^2}F_r(c_r,0,\alpha^2)\approx \frac{2}{15}.
\end{align*}
Therefore, by the implicit function theorem, there exists a unique function $\tilde c_r(\alpha^2)$ near $\alpha^2\approx 5.967\varepsilon^{\frac{1}{3}}$ such that
\begin{align*}
  F_r(\widetilde c_r(\alpha^2),0,\alpha^2)=0.
\end{align*}
Moreover, near $\alpha^2\approx 5.967\varepsilon^{\frac{1}{3}}$,
\begin{align*}
  \pa_{\alpha^2}\widetilde c_r(\alpha^2)\approx 0.27.
\end{align*}
Since \(\mathrm{Han}_i'(2.3)\approx 1.213\), the quantity $F_i(\widetilde c_r(\alpha^2),0,\alpha^2)$ changes sign as \(\alpha^2\) varies. Hence there exists ${\mathcal K}_-(\varepsilon)\approx 5.967$ such that
\begin{align*}
  \left.F_i \left(\tilde c_r(\alpha^2),0,\alpha^2\right)\right|_{\alpha^2=K_{\alpha^2,-}(\varepsilon)\varepsilon^{\frac{1}{3}}}=0.
\end{align*}
This gives the lower neutral point.

\paragraph{Step 3: Extension by the implicit function theorem.}

We next extend the solution curve from the lower neutral point. By Lemma \ref{lem-Ray-endpoint-expansion}, Remark \ref{rmk-Ray-endpoint-expansion}, Remark \ref{rmk-Ray-endpoint-expansion}, Lemma \ref{lem-res-ray}, Lemma \ref{lem-res-airy}, and Lemma \ref{lem-Phi3-main--1-cr} we have
\begin{align*}
  \pa_{c_r}F(c_r,c_i,\alpha^2)=-\frac{1}{2} \frac{\mathrm{Ai}(e^{\frac{\pi}{6}i}\kappa \eta(-1))\mathrm{Ai}(2,e^{\frac{\pi}{6}i}\kappa \eta(-1))}{\left(\mathrm{Ai}(1,e^{\frac{\pi}{6}i}\kappa \eta(-1))\right)^2}+O \left( \left(c_r+\alpha^2+\frac{\alpha^4}{c_r}\right)|\ln c_0|^3\right),\\
  \pa_{c_i}F(c_r,c_i,\alpha^2)=-i\frac{1}{2} \frac{\mathrm{Ai}(e^{\frac{\pi}{6}i}\kappa \eta(-1))\mathrm{Ai}(2,e^{\frac{\pi}{6}i}\kappa \eta(-1))}{\left(\mathrm{Ai}(1,e^{\frac{\pi}{6}i}\kappa \eta(-1))\right)^2}+O \left( \left(c_r+\alpha^2+\frac{\alpha^4}{c_r}\right)|\ln c_0|^3\right),
\end{align*}
Therefore, 
\begin{align*}
\det \left(
    \begin{array}{cc}
      \pa_{c_r}F_r&\pa_{c_i}F_r\\       
      \pa_{c_r}F_i &\pa_{c_i}F_i
    \end{array}
  \right)=\frac{1}{4} \left|\frac{\mathrm{Ai}(e^{\frac{\pi}{6}i}\kappa \eta(-1))\mathrm{Ai}(2,e^{\frac{\pi}{6}i}\kappa \eta(-1))}{\left(\mathrm{Ai}(1,e^{\frac{\pi}{6}i}\kappa \eta(-1))\right)^2}\right|^2+O \left( \left(c_r+\alpha^2+\frac{\alpha^4}{c_r}\right)|\ln c_0|^3\right).
\end{align*}
For $\left|\kappa c\right|\ge M$, by Lemma \ref{lem:Airy-class-est}, it is clear that $\det \left(
    \begin{array}{cc}
      \pa_{c_r}F_r&\pa_{c_i}F_r\\       
      \pa_{c_r}F_i &\pa_{c_i}F_i
    \end{array}
  \right)\approx \frac{1}{4}$.  For the case $\left|\kappa c\right|< M$ we use the specific value of the Airy quotient
\begin{align*}
  AR(z)=\frac{\mathrm{Ai}(z)\mathrm{Ai}(2,z)}{\mathrm{Ai}^2(1,z)}.
\end{align*}
The plot of $\left|AR(z)\right|$ shows that this quotient does not vanish in the
relevant region:
\begin{figure}[htbp]
    \centering
    \begin{subfigure}{0.48\textwidth}
        \centering
        \includegraphics[width=\textwidth]{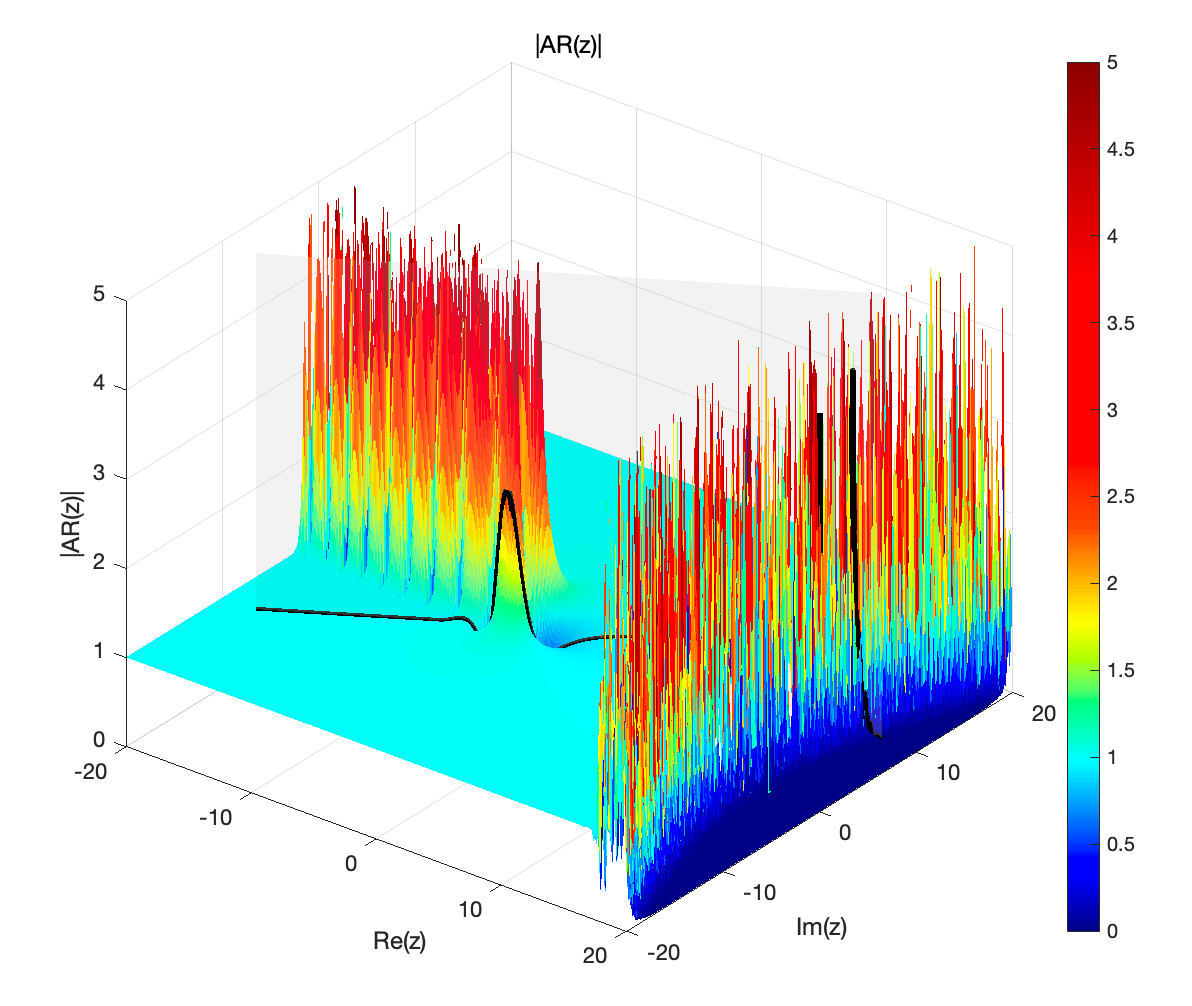}
        \label{fig:small-range}
    \end{subfigure}
    \hfill 
    \begin{subfigure}{0.48\textwidth}
        \centering
        \includegraphics[width=\textwidth]{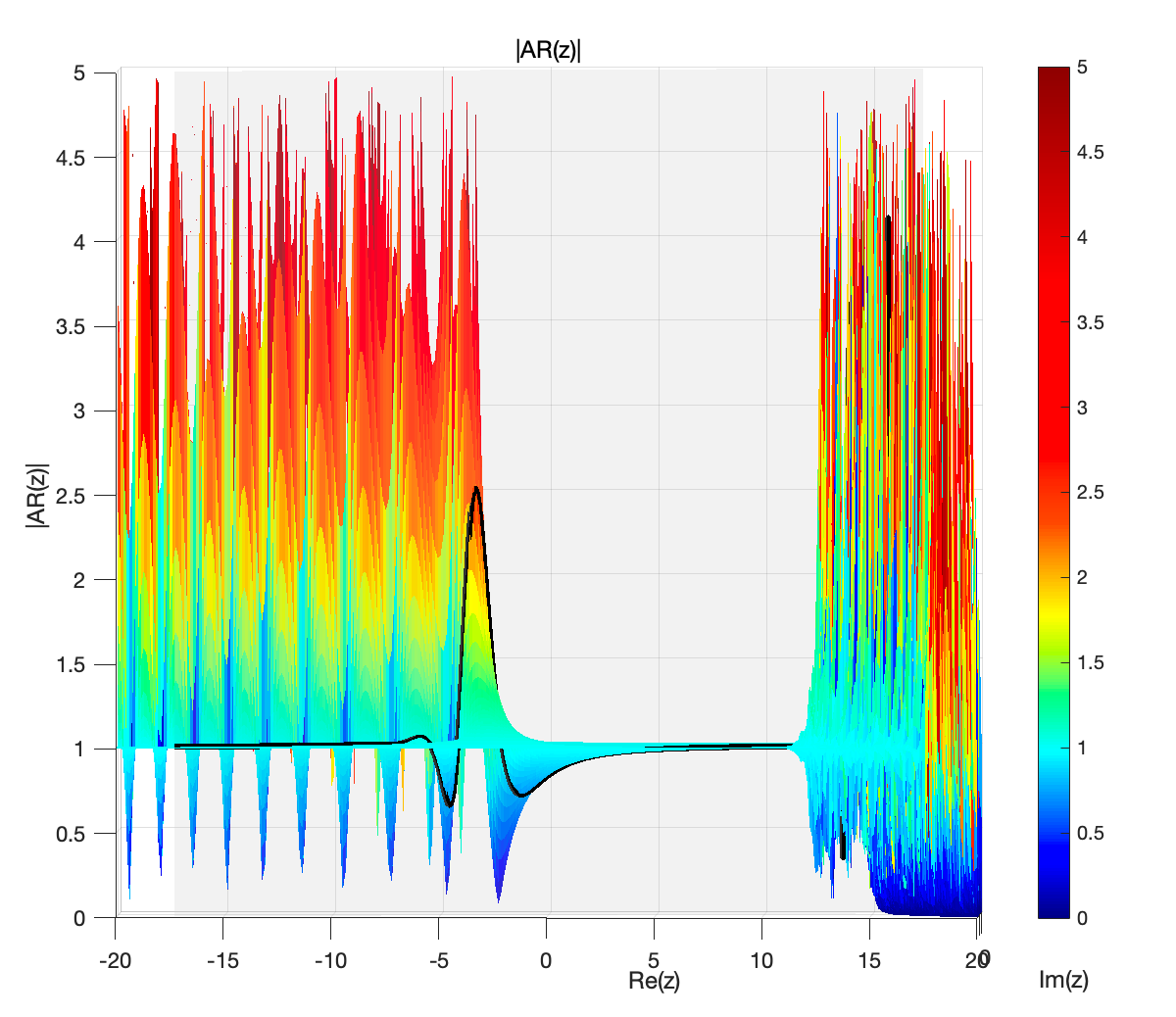}
        \label{fig:large-range}
    \end{subfigure}
    \caption{Plot of $\left|AR(z)\right|$. Black curve is the value for $\arg (z)=-\frac{5}{6}\pi$.}
    \label{fig:comparison}
\end{figure}

Consequently, for \((c_r,c_i,\alpha^2)\in\mathbb H\), 
\begin{align*}
  \left|\frac{\mathrm{Ai}(e^{\frac{\pi}{6}i}\kappa \eta(-1))\mathrm{Ai}(2,e^{\frac{\pi}{6}i}\kappa \eta(-1))}{\left(\mathrm{Ai}(1,e^{\frac{\pi}{6}i}\kappa \eta(-1))\right)^2}\right|\sim1.
\end{align*}
The implicit function theorem
therefore extends the solution curve $\left(c_r(\alpha^2), c_i(\alpha^2)\right)$ from the lower neutral point. 

Moreover, by Lemma \ref{lem-Ray-endpoint-expansion}, Remark \ref{rmk-Ray-endpoint-expansion}, Lemma \ref{lem-res-ray}, Lemma \ref{lem-res-airy}, and Lemma \ref{lem-Phi3-main--1-cr} we have 
\begin{align*}
  \pa_{\alpha^2}c_r=&-\frac{\pa_{c_i}F_i\pa_{\alpha^2}F_r-\pa_{c_i}F_r\pa_{\alpha^2}F_i}{\det \left(
    \begin{array}{ll}
      \pa_{c_r}F_r&\pa_{c_i}F_r\\       
      \pa_{c_r}F_i &\pa_{c_i}F_i
    \end{array}
  \right)}=\frac{\frac{1}{15} \Re\frac{\mathrm{Ai}(e^{\frac{\pi}{6}i}\kappa \eta(-1))\mathrm{Ai}(2,e^{\frac{\pi}{6}i}\kappa \eta(-1))}{\left(\mathrm{Ai}(1,e^{\frac{\pi}{6}i}\kappa \eta(-1))\right)^2}+O \left( \left(c_r+\alpha^2+\frac{\alpha^4}{c_r}\right)|\ln c_0|^3\right)}{\det \left(
    \begin{array}{ll}
      \pa_{c_r}F_r&\pa_{c_i}F_r\\       
      \pa_{c_r}F_i &\pa_{c_i}F_i
    \end{array}
  \right)}.
\end{align*}
For $\left|\kappa\eta(-1)\right|\ge M$, this gives $\pa_{\alpha^2}c_r(\alpha^2)\approx \frac{4}{15}$. For \(|\kappa\eta(-1)|<M\), the sign is determined by the real part of
\(AR(z)\). The plot of \(\Re AR(z)\) shows that it remains positive in the
relevant region:
\begin{figure}[htbp]
    \centering
    \begin{subfigure}{0.48\textwidth}
        \centering
        \includegraphics[width=\textwidth]{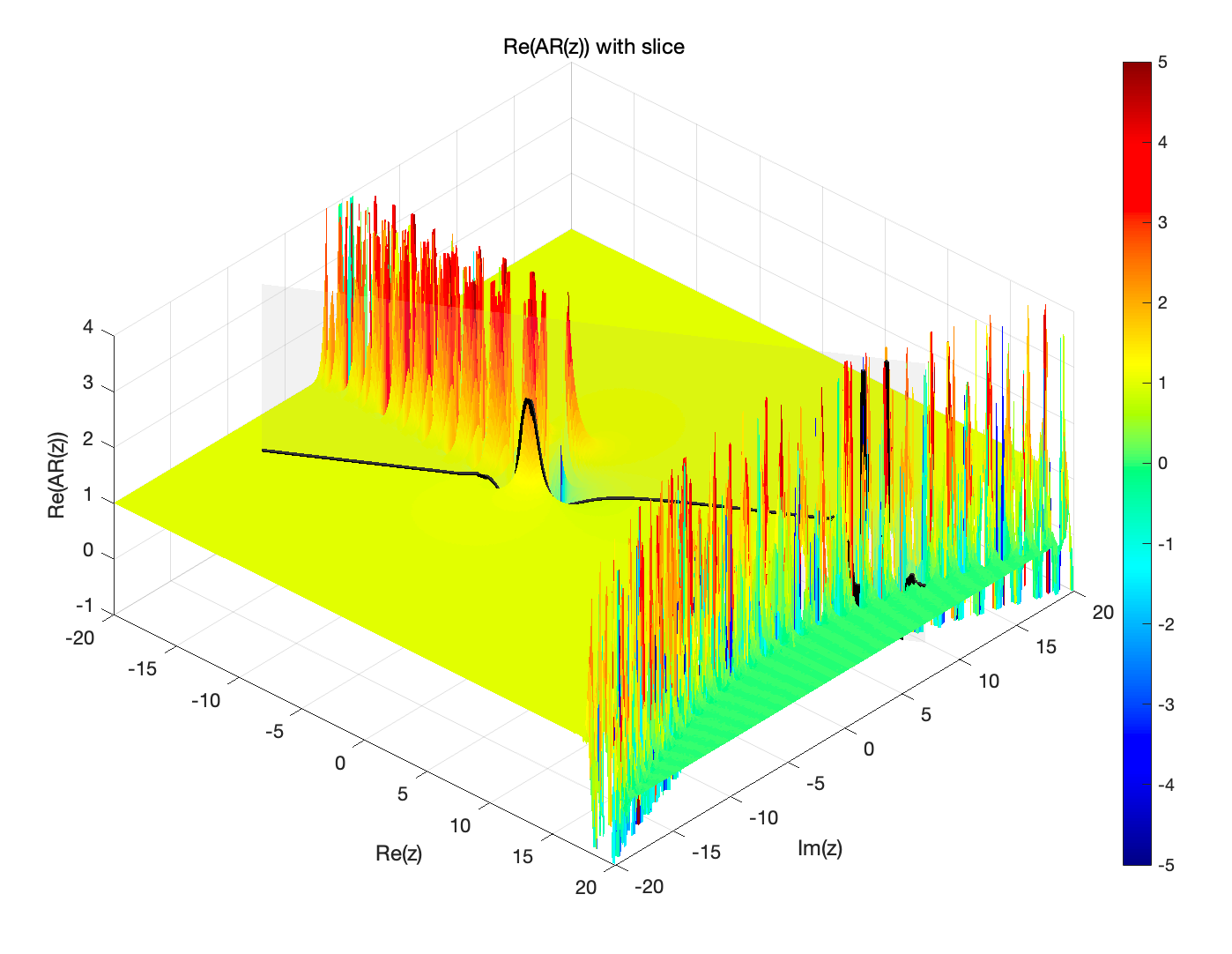}
        \label{fig:small-range}
    \end{subfigure}
    \hfill 
    \begin{subfigure}{0.48\textwidth}
        \centering
        \includegraphics[width=\textwidth]{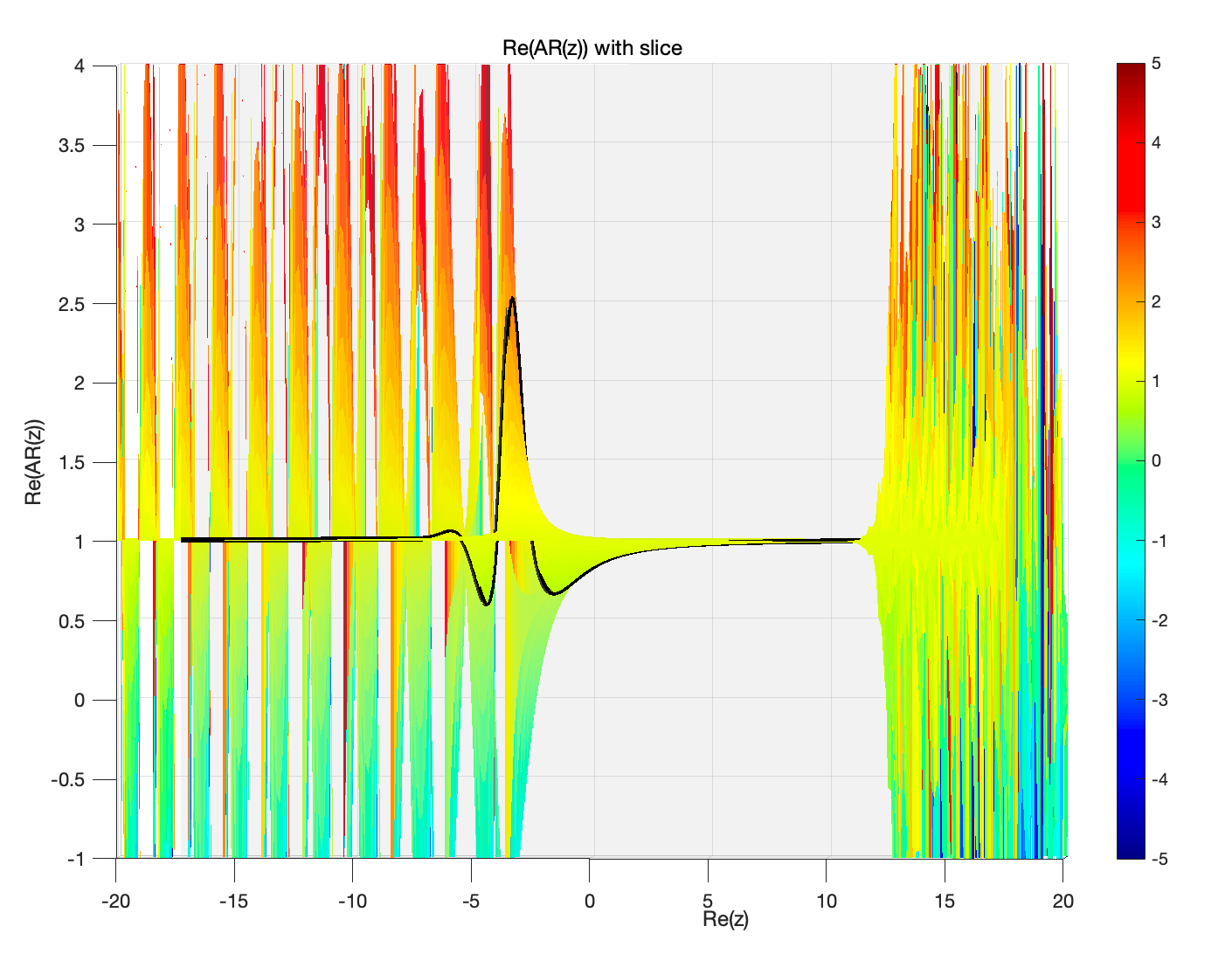}
        \label{fig:large-range}
    \end{subfigure}
    \caption{Plot of $\Re AR(z)$. Black curve is the value for $\arg (z)=-\frac{5}{6}\pi$.}
    \label{fig:comparison}
\end{figure}

Hence, along the solution curve,
\begin{align}\label{cr-monotone}
  \pa_{\alpha^2}c_r(\alpha^2)>0,
  \qquad
  \pa_{\alpha^2}c_r(\alpha^2)=O(1).
\end{align}

We now identify the part of the curve lying below the lower neutral point.
Since the Hankel function is analytic, for \(c_i=-c_0/2\) we have
\begin{align*}
  \mathrm{Han}\left( \frac{\kappa c}{2}\right)=\mathrm{Han}\left( \frac{\kappa c_r}{2}\right)\left(1+O(\frac{c_0}{c_r})\right).
\end{align*}
Together with \(\mathrm{Han}_i'(2.3)\approx 1.213\), this implies that there
exists $\widetilde {\mathcal K}_-(\varepsilon)$ such that $\widetilde {\mathcal K}_-(\varepsilon)\sim1$, $\widetilde {\mathcal K}_-(\varepsilon)< {\mathcal K}_-(\varepsilon)$, $c_i \left(\widetilde K_{\alpha^2,-}(\varepsilon)\varepsilon^{\frac{1}{3}}\right)=-\frac{c_0}{2}$.

As \(\alpha^2\) increases from $\mathcal K_-(\varepsilon)\varepsilon^{\frac13}$, the
monotonicity \eqref{cr-monotone} implies that \(c_r(\alpha^2)\) increases.
Combining this with the localization obtained in Step 1, we find that, in the
intermediate regime $\varepsilon^{\frac{1}{3}}\ll\alpha^2\ll \varepsilon^{\frac{1}{5}}$, one has
\begin{align*}
  c_r \left( \alpha^2\right)\sim\alpha^2,\quad c_i\sim \varepsilon^{\frac{1}{2}} |\alpha|^{-1}.
\end{align*}

Finally, consider the upper branch where $\alpha^2\sim \varepsilon^{\frac{1}{5}}$. At a zero of \(F_i \left( c_r(\alpha^2), c_i(\alpha^2), \alpha^2\right)\), the expansion \eqref{eq-exp-Fi-upb} gives 
\begin{align*}
  \frac{-c_i(\alpha^2)}{2}-\frac{4\pi}{15^2}\alpha^4+\frac{1}{\kappa^{\frac{3}{2}}c_r(\alpha^2)^{\frac{1}{2}}}+o\left(\varepsilon^{\frac{2}{5}}\right)=0.
\end{align*}
Since \(\pa_{\alpha^2}c_r(\alpha^2)\approx \frac{4}{15}\) in this regime, the leading terms above have a simple crossing. Hence there exist $\mathcal K_+(\varepsilon)\sim \widetilde{\mathcal K}_+(\varepsilon)\sim 1$ and $\widetilde {\mathcal K}_+(\varepsilon)>{\mathcal K}_+(\varepsilon)$ such that
\begin{align*}
 c_i \left({\mathcal K}_+(\varepsilon)\varepsilon^{\frac15} \right)=0,\qquad c_i \left(\widetilde {\mathcal K}_+(\varepsilon)\varepsilon^{\frac15} \right)=-\frac{c_0}{2}.
\end{align*}
The point $\left(c_r \left(K_{\alpha^2,+}(\varepsilon)\varepsilon^{\frac15}\right),0  \right)$ is the upper neutral point. 

\paragraph{Step 4: Uniqueness of the solution curve.}
It remains to prove uniqueness. Suppose, to the contrary, that there is another
solution curve  $\left(\mathring c_r(\alpha^2), \mathring c_i(\alpha^2)\right)$ in \(\mathbb H\). By the non-degeneracy estimate above and the implicit
function theorem, the curve is locally unique near any of its points. Moreover,
the same computation as in \eqref{cr-monotone} gives $\pa_{\alpha^2}\mathring c_r(\alpha^2)>0$ and $\pa_{\alpha^2}\mathring c_r(\alpha^2)=O(1)$. Hence, for \(\alpha^2\gg \varepsilon^{\frac13}\), we must have $\mathring c_r(\alpha^2)\gg \varepsilon^{\frac{1}{3}}$. 

By the localization result in Step 1, any zero in this range satisfies
\[
  \mathring c_r(\alpha^2)\sim \alpha^2,
  \qquad
  |\mathring c_i(\alpha^2)|
  \lesssim
  \frac{1}{\kappa^{\frac32}\mathring c_r^{\frac12}}(\alpha^2).
\]

Choose \(\alpha^2\) in the intermediate regime $\varepsilon^{\frac13}\ll \alpha^2\ll \varepsilon^{\frac15}$, and suppose that $c(\alpha^2)\neq \mathring c(\alpha^2)$. 

Since both points are zeros of \(F\), we have
\begin{align*}
  &
  \left.
  \frac{\Phi_{1,\alpha}(-1)}{\Phi_{1,\alpha}'(-1)}
  \right|_{c=c(\alpha^2)}
  -
  \left.
  \frac{\Phi_{1,\alpha}(-1)}{\Phi_{1,\alpha}'(-1)}
  \right|_{c=\mathring c(\alpha^2)}
  \\
  &=
  \left.
  \left(
    \frac{\Phi_{3,\alpha}(-1)}{\Phi_{3,\alpha}'(-1)}
    -Res_1(c_r,c_i,\alpha^2)
  \right)
  \right|_{c=c(\alpha^2)}
  -
  \left.
  \left(
    \frac{\Phi_{3,\alpha}(-1)}{\Phi_{3,\alpha}'(-1)}
    -Res_1(c_r,c_i,\alpha^2)
  \right)
  \right|_{c=\mathring c(\alpha^2)}.
\end{align*}
However, by Lemma \ref{lem-Ray-endpoint-expansion}, Remark \ref{rmk-Ray-endpoint-expansion}, Lemma \ref{lem-res-ray}, Lemma \ref{lem-res-airy}, and Lemma \ref{lem-Phi3-main--1-cr}, we have
\begin{align*}
  \left.\frac{\Phi_{1,\alpha}(-1)}{\Phi_{1,\alpha}'(-1)}\right|_{c=c(\alpha^2)}-\left.\frac{\Phi_{1,\alpha}(-1)}{\Phi_{1,\alpha}'(-1)}\right|_{c=\mathring c(\alpha^2)}=O \left(c(\alpha^2)-\mathring c(\alpha^2)\right),
\end{align*}
and 
\begin{align*}
  \left.\left(\frac{\Phi_{3,\alpha}(-1)}{\Phi_{3,\alpha}'(-1)}-Res_1(c_r,c_i,\alpha^2)\right)\right|_{c=c(\alpha^2)}-\left.\left(\frac{\Phi_{3,\alpha}(-1)}{\Phi_{3,\alpha}'(-1)}-Res_1(c_r,c_i,\alpha^2)\right)\right|_{c=\mathring c(\alpha^2)}=o \left(c(\alpha^2)-\mathring c(\alpha^2)\right).
\end{align*}
This contradiction gives the uniqueness of the solution curve. 
\end{proof}

Next, we prove the main theorem.
\begin{proof}[Proof of Theorem \ref{thm1}]
    In Lemma \ref{lem-sol-curve}, the parameter $\varepsilon$ is fixed, and the solution curve is parametrized by $\alpha^2$. We now fix $\alpha>0$ sufficiently small. Then $\varepsilon=\frac{\nu}{\alpha}$ becomes a function of $\nu$.

We first translate the result of Lemma \ref{lem-sol-curve} into the fixed-$\alpha$ formulation. We look for  a solution curve
  \[
    \nu\mapsto (c_r(\nu),c_i(\nu))
  \]
  such that
  \begin{align*}
    F(c_r(\nu),c_i(\nu),\alpha^2,\nu)=0.
  \end{align*}
We emphasize that this curve has a different meaning from the curve $\alpha^2\mapsto (c_r(\alpha^2),c_i(\alpha^2))$ constructed in  Lemma \ref{lem-sol-curve}.

 We first show that, for fixed sufficiently small $\alpha$, the lower neutral point exists. By Lemma \ref{lem-sol-curve}, for each fixed $\varepsilon$ there exists a lower neutral point at
\[
  \alpha^2=\mathcal K_-(\varepsilon)\varepsilon^{\frac13},
\]
where $\mathcal K_-(\varepsilon)\sim 1$ is continuous in $\varepsilon$. More
precisely, there exist
\[
  c_r\sim \varepsilon^{\frac13},
  \qquad
  c_i=0,
\]
such that
\begin{align*}
  F\left(
    c_r,
    c_i,
    \mathcal K_-(\varepsilon)\varepsilon^{\frac13},
    \nu
  \right)=0,
  \qquad
  \nu=\alpha\varepsilon,
\end{align*}
where $\alpha^2=\mathcal K_-(\varepsilon)\varepsilon^{\frac13}$ in the
fixed-$\varepsilon$ formulation.

We now fix $\alpha$ and regard
\[
  \varepsilon=\frac{\nu}{\alpha}
\]
as a function of $\nu$. To locate the lower neutral point in the fixed-$\alpha$
formulation, we compare the two curves in the $(\varepsilon,\alpha^2)$-plane:
\[
  \left(\frac{\nu}{\alpha},\alpha^2\right)
  \quad\text{and}\quad
  \left(
    \frac{\nu}{\alpha},
    \mathcal K_-\left(\frac{\nu}{\alpha}\right)
    \left(\frac{\nu}{\alpha}\right)^{\frac13}
  \right).
\]
Since $\mathcal K_-(\varepsilon)$ is continuous and comparable to $1$, by
varying $\nu$ there exists $\nu_{*,-}$ such that these two curves intersect,
namely
\begin{align}\label{eq-lower-neutral-fixed-alpha}
  \alpha^2
  =
  \mathcal K_-\left(\frac{\nu_{*,-}}{\alpha}\right)
  \left(\frac{\nu_{*,-}}{\alpha}\right)^{\frac13}.
\end{align}
Equivalently,
\begin{align*}
  \nu_{*,-}
  =
  \left[
    \mathcal K_-\left(\frac{\nu_{*,-}}{\alpha}\right)
  \right]^{-3}
  \alpha^7
  \eqdef
  J_{\nu,-}(\alpha)\alpha^7.
\end{align*}
Since $\mathcal K_-(\varepsilon)\sim 1$, we have $J_{\nu,-}(\alpha)\sim 1$.

Thus the lower neutral point is located at
\begin{align*}
  \nu=\nu_{*,-}=J_{\nu,-}(\alpha)\alpha^7.
\end{align*} 
At this point,
\begin{align*}
  c_r(\nu_{*,-})  \sim  \left(\frac{\nu_{*,-}}{\alpha}\right)^{\frac13}  \sim  \alpha^2\sim \nu_{*,-}^{\frac27},  \qquad  c_i(\nu_{*,-})=0.
\end{align*}

Then, for $\alpha^{11}\ll\nu\ll \alpha^7$, we have
\begin{align*}
  \left(\frac{\nu}{\alpha}\right)^{\frac{1}{3}}\ll\alpha^2\ll\left(\frac{\nu}{\alpha}\right)^{\frac{1}{5}}.
\end{align*}
By the result of Lemma \ref{lem-sol-curve}, we have
\begin{align*}
  c_r(\nu)\sim \alpha^2,\qquad c_i(\nu)\sim \nu^{\frac{1}{2}}\alpha^{-\frac{3}{2}}>0.
\end{align*}
Similar to Lemma \ref{lem-sol-curve}, at $\alpha^{11}\ll\nu\ll \alpha^7$, using contradiction argument, we can show the uniqueness of the solution curve $\left(c_r(\nu), c_i(\nu)\right)$.

Similar to the lower neutral point, there exists $\nu_{*,+}$ such that
\begin{align*}
  \alpha^2=K_{\alpha^2,+}(\frac{\nu_{*,+}}{\alpha}),
\end{align*}
and
  \begin{align*}
     \nu_{*,+}=K_{\alpha^2,+}^{-5}\left(\frac{\nu_{*,+}}{\alpha}\right)\alpha^{11}\eqdef J_{\nu,+}\left(\alpha\right)\alpha^{11}.
  \end{align*} 
Then we get the upper neutral point where
\begin{align*}
  c_r(\nu_{*,+})\sim \alpha^2\sim \nu_{*,+}^{\frac{2}{11}},\qquad c_i\left(\nu_{*,+}\right)=0.
\end{align*}

Last, we show that there exists only one lower neutral point and one upper neutral point, and prove the transversal crossing condition.

For fixed $\alpha$, the implicit function theorem gives, along the solution curve,
\begin{align}\label{eq-dervi-c-nu}
 \pa_{\nu}c_r=-\frac{\pa_{c_i}F_i\pa_{\nu}F_r-\pa_{c_i}F_r \pa_{\nu}F_i }{\det \left(
    \begin{array}{ll}
      \pa_{c_r}F_r&\pa_{c_i}F_r\\       
      \pa_{c_r}F_i &\pa_{c_i}F_i
    \end{array}
  \right)},\qquad    \pa_{\nu}c_i=-\frac{- \pa_{c_r}F_i\pa_{\nu}F_r+\pa_{c_r}F_r \pa_{\nu}F_i }{\det \left(
    \begin{array}{ll}
      \pa_{c_r}F_r&\pa_{c_i}F_r\\       
      \pa_{c_r}F_i &\pa_{c_i}F_i
    \end{array}
  \right)}.
\end{align}

We first analyze the lower branch $\nu\sim \alpha^7$. In this regime, we use the formulation
\begin{align*}
  F(c_r,c_i,\alpha^2,\nu)=&-\frac{\eta(-1)}{\eta_r'(-1)}-\frac{c}{2}+\frac{\eta(-1)}{\eta_r'(-1)}\frac{1}{\mathrm{Han}\left( -\kappa \eta(-1)\right)} +\frac{2}{15}\alpha^2 +O\left(\alpha^2 c+c_0|\ln c_0|^2+c^2\right).
\end{align*}
Here we fix $c_0=\alpha^3$. Around the lower branch, it is clear that $c_0\sim \left(\frac{\nu}{\alpha}\right)^{\frac{1}{2}}$. The value of $c_0$ will not influence the solution curve $\left(c_r(\nu), c_i(\nu)\right)$. As the different choice of $c_0$ just influence the construction of the linearly independent solutions, but not change whether there exists eigenfunction.

At the lower neutral point, where $c_i=0$, we have $-\kappa \eta(-1)\approx 2.3$, and
\begin{align*}
  \mathrm{Han}_i\left(-\kappa \eta(-1)\right)=o(1),\ \mathrm{Han}_r\left(-\kappa \eta(-1)\right)\approx2.29,\ \mathrm{Han}_r'\left(-\kappa \eta(-1)\right)\approx-0.62,\ \mathrm{Han}_i'\left(-\kappa \eta(-1)\right)\approx1.21.
\end{align*}

 By Lemma \ref{lem-Ray-endpoint-expansion}, Remark \ref{rmk-Ray-endpoint-expansion}, Lemma \ref{lem-res-ray}, Lemma \ref{lem-res-airy}, and Lemma \ref{lem-Phi3-main--1-cr} we have
\begin{align*}  
  \pa_{\nu}F_r\approx -\frac{2^{\frac{1}{3}}}{12} c_r^2\nu^{-\frac{4}{3}}\alpha^{\frac{1}{3}}\Re \frac{\mathrm{Han}'(-\kappa \eta(-1))}{\left(\mathrm{Han}\left(-\kappa \eta(-1)\right)\right)^2}\approx -\frac{1}{3} \frac{1}{\kappa\nu} \left(\frac{\kappa c_r}{2}\right)^2\Re \frac{\mathrm{Han}'(-\kappa \eta(-1))}{\left(\mathrm{Han}\left(-\kappa \eta(-1)\right)\right)^2}\sim \nu^{-\frac{5}{7}}.
\end{align*}
\begin{align*}  
  \pa_{\nu}F_i\approx-\frac{2^{\frac{1}{3}}}{12} c_r^2\nu^{-\frac{4}{3}}\alpha^{\frac{1}{3}}\Im \frac{\mathrm{Han}'(-\kappa \eta(-1))}{\left(\mathrm{Han}\left(-\kappa \eta(-1)\right)\right)^2}\approx -\frac{1}{3} \frac{1}{\kappa\nu} \left(\frac{\kappa c_r}{2}\right)^2\Im \frac{\mathrm{Han}'(-\kappa \eta(-1))}{\left(\mathrm{Han}\left(-\kappa \eta(-1)\right)\right)^2}\sim-\nu^{-\frac{5}{7}}.
\end{align*}
\begin{align*}  
  \pa_{c_r}F_r\approx \Re \left(-\frac{1}{2} \frac{1}{\mathrm{Han}\left(-\kappa \eta(-1)\right)}+\frac{c_r\kappa}{4}\frac{\mathrm{Han}'(-\kappa \eta(-1))}{\left(\mathrm{Han}\left(-\kappa \eta(-1)\right)\right)^2}\right)\sim -1,
\end{align*}
\begin{align*}  
  \pa_{c_r}F_i\approx\frac{c_r\kappa}{4}\Im\frac{\mathrm{Han}'(-\kappa \eta(-1))}{\left(\mathrm{Han}\left(-\kappa \eta(-1)\right)\right)^2}\sim 1,
\end{align*}
\begin{align*}  
  \pa_{c_i}F_r\approx-\frac{c_r\kappa}{4}\Im\frac{\mathrm{Han}'(-\kappa \eta(-1))}{\left(\mathrm{Han}\left(-\kappa \eta(-1)\right)\right)^2}\sim -1,
\end{align*}
\begin{align*}  
  \pa_{c_i}F_i\approx \Re \left(-\frac{1}{2}\frac{1}{\mathrm{Han}\left(-\kappa \eta(-1)\right)}+\frac{c_r\kappa}{4}\frac{\mathrm{Han}'(-\kappa \eta(-1))}{\left(\mathrm{Han}\left(-\kappa \eta(-1)\right)\right)^2}\right)\sim -1.
\end{align*}
From the above results, at $c_i=0$, we have
\begin{align*}
  & - \pa_{c_r}F_i\pa_{\nu}F_r+\pa_{c_r}F_r \pa_{\nu}F_i\\
  \approx& \left(-\frac{1}{2}\Re \frac{1}{\mathrm{Han}\left(-\kappa \eta(-1)\right)}\right)\left(-\frac{1}{3} \frac{1}{\kappa\nu} \left(\frac{\kappa c_r}{2}\right)^2\Im \frac{\mathrm{Han}'(-\kappa \eta(-1))}{\left(\mathrm{Han}\left(-\kappa \eta(-1)\right)\right)^2}\right)\sim\nu^{-\frac{5}{7}}.
\end{align*}
Therefore, at $\nu=\nu_{*,-}$, where $c_i=0$, it holds that, 
\begin{align*}
  \pa_{\nu}c_i\sim -\nu^{-\frac{5}{7}}\sim-\alpha^{-5}.
\end{align*}
The same estimate remains valid in a small neighborhood of the lower branch. Hence $c_i$ crosses zero transversely and is strictly decreasing with respect to $\nu$ near the lower neutral point.  This proves that the lower neutral point is unique. It also shows that there exists $\widetilde J_{\nu,-}(\alpha)>J_{\nu,-}(\alpha)$ such that
  \begin{align*}
    -\frac12\alpha^3 =c_i\left(\widetilde J_{\nu,-}(\alpha)\alpha^7\right)<c_i(\nu)<c_i\left(J_{\nu,-}(\alpha)\alpha^7\right)=0, \text{ for }    J_{\nu,-}(\alpha)\alpha^7 <\nu <\widetilde J_{\nu,-}(\alpha)\alpha^7.
  \end{align*}
  Similarly, from \eqref{eq-dervi-c-nu}, on the lower branch, we also have
  \begin{align*}
    \pa_{\nu}c_r \sim \nu^{-\frac57} \sim \alpha^{-5}.
  \end{align*} 

We now analyze the upper branch where $\nu\sim \alpha^{11}$. Similar to the lower branch, here we fixed $c_0=\alpha^5\sim \left(\frac{\nu}{\alpha}\right)^{\frac{1}{2}}$. Recall \eqref{expan-Phi1-Re}, \eqref{expan-Phi1-Im}, and \eqref{expan-Phi3}. By Lemma \ref{lem-Ray-endpoint-expansion}, Remark \ref{rmk-Ray-endpoint-expansion}, Lemma \ref{lem-res-ray}, Lemma \ref{lem-res-airy}, and Lemma \ref{lem-Phi3-main--1-cr},  at $\nu=\nu_{*,+}$, where $c_i=0$, we have
\begin{align*}  
  \pa_{\nu}F_r\approx \frac{c_r}{2\nu}\frac{1}{ \left(\kappa c_r\right)^{\frac{3}{2}} }\sim\nu^{-\frac{7}{11}},\quad \pa_{\nu}F_i\approx \frac{c_r}{2\nu}\frac{1}{ \left(\kappa c_r\right)^{\frac{3}{2}} }\sim\nu^{-\frac{7}{11}},
\end{align*}
\begin{align*}  
  \pa_{c_r}F_r\approx -\frac{1}{2},\qquad \pa_{c_r}F_i\approx-\frac{1}{2}\frac{1}{ \left(\kappa c_r\right)^{\frac{3}{2}} }\sim- \nu^{\frac{2}{11}},
\end{align*}
\begin{align*}  
  \left|\pa_{c_i}F_r\right|\lesssim c_r |\ln c_0|\lesssim \nu^{\frac{2}{11}} |\ln \nu|,\qquad \pa_{c_i}F_i\approx-\frac{1}{2}.
\end{align*} 
Substituting the above estimates into \eqref{eq-dervi-c-nu}, we obtain
\begin{align*}
  \pa_{\nu}c_r\sim \nu^{-\frac{7}{11}}\sim \alpha^{-7},\qquad \pa_{\nu}c_i\sim \nu^{-\frac{7}{11}}\sim \alpha^{-7}.
\end{align*}

Thus $c_i$ crosses zero transversely and is strictly increasing with respect  to $\nu$ near the upper neutral point. Therefore the upper neutral point is  unique. Moreover, similar to the lower branch, on the upper branch, there exists $\widetilde J_{\nu,+}<J_{\nu,+}\left(\alpha\right)$ such that
  \begin{align*}
  -\frac{\alpha^5}{2}=c_i \left( \widetilde J_{\nu,+}(\alpha)\alpha^{11} \right)< c_i\left(\nu\right)< c_i \left(J_{\nu,+}(\alpha)\alpha^{11} \right)=0, \text{ for }\widetilde J_{\nu,+}(\alpha)\alpha^{11}<\nu< J_{\nu,+}(\alpha)\alpha^{11}.
\end{align*}

Then we finishes the proof of this theorem.
\end{proof}

\begin{proof}[Proof of Theorem \ref{thm2}]
  The proof of Theorem \ref{thm2} is the same to the proof of Theorem \ref{thm1}. Here we only give the derivative estimates on the lower and upper branch.

  For fixed $\nu$, there exists $J_{\alpha,-}\left(\nu\right)$ and $J_{\alpha,+}\left(\nu\right)$ such that for $\alpha^2=J_{\alpha,-}\left(\nu\right)\nu^{\frac{2}{7}}$ and $\alpha^2=J_{\alpha,+}\left(\nu\right)\nu^{\frac{2}{11}}$ reach the lower neutral point and the upper  neutral point respectively.
  It remains to record the signs of the crossings at the lower and upper neutral points. Along the solution curve, the implicit function theorem gives
  \begin{align}\label{eq-deriv-c-alpha2}
    \pa_{\alpha^2}c_r
    =&-\frac{\pa_{c_i}F_i\pa_{\alpha^2}F_r-\pa_{c_i}F_r\pa_{\alpha^2}F_i}{\det \left(
    \begin{array}{cc}
      \pa_{c_r}F_r&\pa_{c_i}F_r\\
      \pa_{c_r}F_i&\pa_{c_i}F_i
    \end{array}
    \right)},\quad
    \pa_{\alpha^2}c_i
    =-\frac{-\pa_{c_r}F_i\pa_{\alpha^2}F_r+\pa_{c_r}F_r\pa_{\alpha^2}F_i}{\det \left(
    \begin{array}{cc}
      \pa_{c_r}F_r&\pa_{c_i}F_r\\
      \pa_{c_r}F_i&\pa_{c_i}F_i
    \end{array}
    \right)}.
  \end{align}
 On the lower branch, $\alpha^2\sim \nu^{\frac{2}{7}}$, and $c_r\approx 0.6117\alpha^2$. By using the properties of Hankel function, we have
  \begin{align*}  
  \pa_{\alpha^2}F_r\approx \frac{2}{15} +\frac{1}{6} \frac{1}{\kappa\alpha^2} \left(\frac{\kappa c_r}{2}\right)^2\Re \frac{\mathrm{Han}'(-\kappa \eta(-1))}{\left(\mathrm{Han}\left(-\kappa \eta(-1)\right)\right)^2}\approx \frac{2}{15}-0.014\sim1,
\end{align*}
\begin{align*}  
  \pa_{\alpha^2}F_i\approx\frac{1}{6} \frac{1}{\kappa\alpha^2} \left(\frac{\kappa c_r}{2}\right)^2\Im \frac{\mathrm{Han}'(-\kappa \eta(-1))}{\left(\mathrm{Han}\left(-\kappa \eta(-1)\right)\right)^2}\sim 1.
\end{align*} 
Then it is clear that
\begin{align*}
  \left|\pa_{\alpha^2}c_r\right|\sim1,\qquad \pa_{\alpha^2}c_i\sim 1.
\end{align*}

On the upper branch $\alpha^2\sim \nu^{\frac{2}{11}}$, it holds that
  \begin{align*}  
  \pa_{\alpha^2}F_r\approx \frac{2}{15},\qquad \pa_{\alpha^2}F_i\approx -\frac{8\pi}{15^2}\alpha^2-\frac{c_r}{4\alpha^2}\frac{1}{ \left(\kappa c_r\right)^{\frac{3}{2}} }\sim -\alpha^2.
\end{align*}
Then we have
\begin{align*}
  \pa_{\alpha^2}c_r\sim 1,\qquad \pa_{\alpha^2}c_i\sim -\alpha^2\sim -\nu^{\frac{2}{11}}.
\end{align*}
\end{proof}
 \begin{remark} \label{rmk-eigenfunction}
The matrix \eqref{eq-matrix-scale} also gives the formal relative sizes of the coefficients $B_j$. From the third row of \eqref{eq-matrix-scale}, the terms involving $B_2$ and $B_4$ must balance, and therefore 
\begin{align*}
 B_2:B_4\sim \kappa^{\frac{3}{2}}:1. 
 \end{align*} 
 Having fixed this balance, the dominant contribution in the fourth row comes from the $B_4$-term. The only remaining term capable of balancing it is the $B_1$-term, which yields 
 \begin{align*} 
 B_1:B_4\sim \kappa^{\frac{9}{2}}:|\ln c_0|. 
 \end{align*} 
 Finally, the second row gives the balance between the $B_1$- and $B_3$-terms: 
 \begin{align*} 
 B_1:B_3\sim \kappa \left\langle \kappa c_r \right\rangle^{\frac{1}{2}}:1. 
 \end{align*} 
 Consequently, 
 \begin{align*} 
 B_1:B_2:B_3:B_4 \sim \kappa^{\frac{9}{2}}: \kappa^{\frac{3}{2}}|\ln c_0|: \frac{\kappa^{\frac{7}{2}}}{\left\langle \kappa c_r \right\rangle^{\frac{1}{2}}}: |\ln c_0|. 
 \end{align*} 
 Thus, at the level of leading-order balance, the leading-order neutral mode is dominated by the slow mode $\Phi_{1,\alpha}$.
\end{remark}
\begin{remark}\label{rmk-odd-Wron}
If one considers the odd modes, then the boundary conditions are
  \begin{align*}
    \Phi(-1)=\Phi'(-1)=\Phi(0)=\Phi''(0)=0.
  \end{align*}
  By using the same technique in this paper, one can see that the corresponding boundary matrix for  $\Phi_{1,\alpha}, \Phi_{2,\alpha}, \Phi_{3,\alpha}, \Phi_{4,\alpha}$ has the scale
  \begin{align}\label{eq-matrix-scale-odd}
  \left(
    \begin{array}{cccc}
      \Phi_{1,\alpha}(-1)&\Phi_{2,\alpha}(-1)&\Phi_{3,\alpha}(-1)&\Phi_{4,\alpha}(-1)\\       
      \Phi_{1,\alpha}'(-1)&\Phi_{2,\alpha}'(-1)&\Phi_{3,\alpha}'(-1)&\Phi_{4,\alpha}'(-1)\\       
      \Phi_{1,\alpha}(0)&\Phi_{2,\alpha}(0)&\Phi_{3,\alpha}(0)&\Phi_{4,\alpha}(0)\\       
      \Phi_{1,\alpha}''(0)&\Phi_{2,\alpha}''(0)&\Phi_{3,\alpha}''(0)&\Phi_{4,\alpha}''(0)
    \end{array}
  \right)\sim  \left(
    \begin{array}{cccc}
      -\hat c+\frac{4\alpha^2}{15}&1&1&0\\       
      1&|\ln c_0|&\kappa \left\langle \kappa c_r \right\rangle^{\frac{1}{2}}&0\\       
      1&\alpha^2+c_r&0&1\\       
      \left|\ln c_0\right|&\left|\ln c_0\right|&e^{-\kappa}&\kappa^{3}
    \end{array}
  \right).
\end{align}
Thus the corresponding Wronskian could not be zero in the T--S eigen region $\mathbb H$.
\end{remark}

\section{Traveling waves bifurcation}
In this section, we prove Theorem \ref{thm-bif}. 
For completeness, and to make the bifurcation argument self-contained, 
we include the details of the proof. The proof is based on the 
Lyapunov--Schmidt reduction framework used in \cite{Kagei2019}, 
adapted to the present incompressible setting.

We first reformulate the traveling-wave problem as a time-independent system.  We then recall the spectral structure of the critical Orr--Sommerfeld  mode and construct the corresponding projections. Finally, we perform the Lyapunov--Schmidt reduction and 
solve the reduced nonlinear system by a contraction argument.
\subsection{Traveling-wave formulation}
We rewrite \eqref{NS-pert-nonlinear} as 
\begin{align}\label{eq-eigen-nonlinear-1}
 \pa_t\omega+L_\nu\omega=\mathcal N(\nu,\omega),
\end{align}
where $L_\nu\omega=u_p\pa_x \omega-u_p''\pa_x\phi-\nu\Delta\omega$, and $\mathcal N(\nu,\omega)=-u\cdot\nabla\omega$.

Our aim is to find an $x$-periodic traveling-wave solution with period $\frac{2\pi}{\alpha^{[0]}}$ of the form
\begin{align*}
  \phi(t,x,y)=\tilde\phi(x-c t,y),
\end{align*}
where $\tilde\phi(x,y)$ is defined on $\mathbb T_{\frac{2\pi}{\alpha^{[0]}}}\times[-1,1]$ and $c\in \mathbb R$. 

Substituting this ansatz into \eqref{eq-eigen-nonlinear-1}, we obtain
\begin{align}\label{eq-eigen-nonlinear-2}
  \mathcal L_{c,\nu}\omega=\mathcal N(\nu,\omega),
\end{align}
where
\begin{align*}
  \mathcal L_{c,\nu}=L_\nu-c\pa_x.
\end{align*}
Equivalently, this equation reads
\begin{align*}
  -c\pa_x\omega+u_p\pa_x \omega-u_p''\pa_x\phi-\nu\Delta\omega=-u\cdot\nabla\omega.
\end{align*}

It is natural to perform Galilean coordinate transformation and reduce the problem to the following system independent of $t$:
\begin{align}\label{eq-reformulation-nonlin}
  -c\pa_x\tilde\omega+u_p\pa_x \tilde\omega-u_p''\pa_x\tilde\phi-\nu\Delta\tilde\omega=-\tilde u\cdot\nabla\tilde\omega,
\end{align}
where $\tilde u= \left(-\pa_y \tilde\phi,\pa_x\tilde\phi\right)$, and $\tilde \omega= \Delta\tilde\phi$.

The corresponding system for velocity is 
\begin{align}\label{eq-reformulation-nonlin-u}
  -c\pa_x\tilde u+u_p\pa_x \tilde u+\tilde u^{(2)}u_p'\bm e_1-\nu\Delta \tilde u+\nabla\tilde p=-\tilde u\cdot \nabla\tilde u.
\end{align}

For notational simplicity, we drop the tildes and write 
$(\phi,\omega,u,p)$ instead of $(\tilde\phi,\tilde\omega,\tilde u,\tilde p)$.

The solution of \eqref{eq-reformulation-nonlin-u} satisfies the non-slip boundary condition
\begin{align*}
  u(x,\pm1)=0.
\end{align*}
We write $\omega(x,y)=\sum_{\alpha=k\alpha^{[0]},k\in\mathbb Z}\hat\omega(\alpha,y)e^{ik\alpha^{[0]}x}$, $\phi(x,y)=\sum_{\alpha=k\alpha^{[0]},k\in\mathbb Z}\hat\phi(\alpha,y)e^{ik\alpha^{[0]}x}$. The boundary condition for $\alpha\neq0$ is
\begin{align}\label{bc-tw-1}
  \hat\phi(\alpha,\pm1)=0,\qquad\hat\phi'(\alpha,\pm1)=0.
\end{align}
and for  $\alpha=0$ is
\begin{align}\label{bc-tw-2}
  \hat\phi(0,-1)=\hat\phi'(0,\pm1)=\pa_y^2\hat\phi(0,1)-\pa_y^2\hat\phi(0,-1)=0.
\end{align}
\begin{remark}
The zero Fourier mode $\alpha=0$ corresponds to the shear-flow component. Since the vorticity formulation is obtained by taking the curl of the velocity equation, it may lose the information of a possible constant streamwise forcing. Therefore, for the velocity formulation \eqref{eq-reformulation-nonlin-u}, we need to impose an additional compatibility condition to ensure that no constant external force remains. More precisely, we require that the pressure $p(x,y)$ is periodic in $x$ and satisfies $\int_{\mathbb T_{\frac{2\pi}{\alpha^{[0]}}}} \pa_xp(x,y) d x=0$. 

For a traveling-wave solution, integrating the streamwise momentum equation in $x$ gives
\begin{align*}
  \nu \int_{\mathbb T_{\frac{2\pi}{\alpha^{[0]}}}}\pa_y^3\phi(x,y)  dx=-\int_{\mathbb T_{\frac{2\pi}{\alpha^{[0]}}}}\pa_y\phi(x,y)\pa_x\pa_y\phi(x,y)-\pa_x\phi(x,y)\pa_y^2\phi(x,y) dx.
\end{align*}
Integrating the right-hand side further in $y$, and using integration by parts gives
\begin{align*}
  &\int_{\mathbb T_{\frac{2\pi}{\alpha^{[0]}}}\times[-1,1]}\pa_y\phi(x,y)\pa_x\pa_y\phi(x,y)-\pa_x\phi(x,y)\pa_y^2\phi(x,y) dxdy\\
  =&\int_{\mathbb T_{\frac{2\pi}{\alpha^{[0]}}}\times[-1,1]}2\pa_y\phi(x,y)\pa_x\pa_y\phi(x,y) dxdy=\int_{\mathbb T_{\frac{2\pi}{\alpha^{[0]}}}\times[-1,1]}\pa_x \left(\pa_y\phi(x,y)\right)^2 dxdy=0.
\end{align*}
Thus, in order to exclude a constant forcing of the form $C\bm e_1$, it suffices to impose the zero-mode compatibility condition
\begin{align*}
  \int_{[-1,1]}\pa_y^3\hat\phi(0,y)  dy= \pa_y^2\hat\phi(0,1)-\pa_y^2\hat\phi(0,-1)=0.
\end{align*}
Finally, for the zero mode we do not need to impose both values of $\hat\phi(0,y)$ at the boundary. Indeed, adding a constant to the stream function $\hat\phi(0,y)$ does not change the velocity field and does not affect any term in \eqref{eq-reformulation-nonlin}. To fix this indeterminacy, we impose the normalization $\hat\phi(0,-1)=0$.
\end{remark}

\subsection{Spectrum properties of the Orr-Sommerfeld operator}
\subsubsection{The critical eigenspace}
By the assumptions of Theorem \ref{thm-bif}, $0$ is an isolated semisimple eigenvalue of $\mathcal L_{c_r^{[0]},\nu^{[0]}}$. The null space is
\begin{align*}
  N \left(\mathcal L_{c_r^{[0]},\nu^{[0]}}\right)=\text{span} (\omega_+,\omega_-), \quad \omega_-=\overline{\omega_+}.
\end{align*}
where
\begin{align*}
  \omega_+(x,y)=  \Delta \left(\phi_{\nu^{[0]},\alpha^{[0]}}(y)e^{i\alpha^{[0]}x}\right)=\left(\pa_y^2-\left(\alpha^{[0]}\right)^2\right)\phi_{\nu^{[0]},\alpha^{[0]}}(y)e^{i\alpha^{[0]}x}, 
\end{align*}
and $\phi_{\nu^{[0]},\alpha^{[0]}}(y)=\phi_{\nu^{[0]},\alpha^{[0]},r}(y)+i\phi_{\nu^{[0]},\alpha^{[0]},i}(y)$ denotes the Orr--Sommerfeld eigenfunction at the neutral point $\left(\nu^{[0]},\alpha^{[0]},c_r^{[0]}\right)$ obtained in Theorem \ref{thm1}. This eigenfunction is constructed in the previous sections as a nontrivial linear combination of the four modes $\Phi_{1,\alpha^{[0]}},\Phi_{2,\alpha^{[0]}},\Phi_{3,\alpha^{[0]}}$, and $\Phi_{4,\alpha^{[0]}}$ satisfying
\begin{align*}
  \phi_{\nu^{[0]},\alpha^{[0]}}(\pm1)=\pa_y\phi_{\nu^{[0]},\alpha^{[0]}}(\pm1)=0.
\end{align*}
We denote
\begin{align*}
  \omega_{\nu^{[0]},\alpha^{[0]}}(y)= \omega_{\nu^{[0]},\alpha^{[0]},r}(y)+i\omega_{\nu^{[0]},\alpha^{[0]},i}(y)=\left(\pa_y^2-\left(\alpha^{[0]}\right)^2\right)\phi_{\nu^{[0]},\alpha^{[0]}}(y).
\end{align*}

We define the following real eigenfunction
\begin{align}
  \omega_1(x,y)=\Re \left( \omega_{\nu^{[0]},\alpha^{[0]}}(y)e^{i\alpha^{[0]} x}\right)= \omega_{\nu^{[0]},\alpha^{[0]},r}(y)\cos \left(\alpha^{[0]} x\right)- \omega_{\nu^{[0]},\alpha^{[0]},i}(y)\sin \left(\alpha^{[0]} x\right),\label{eq-eigen-omega1}\\
  \omega_2(x,y)=\Im \left( \omega_{\nu^{[0]},\alpha^{[0]}}(y)e^{i\alpha^{[0]} x}\right)= \omega_{\nu^{[0]},\alpha^{[0]},r}(y)\sin \left(\alpha^{[0]} x\right)+ \omega_{\nu^{[0]},\alpha^{[0]},i}(y)\cos \left(\alpha^{[0]} x\right).\label{eq-eigen-omega2}
\end{align}

\subsubsection{The adjoint Orr--Sommerfeld eigenfunctions}
We now derive the adjoint Orr--Sommerfeld eigenfunctions. To this end, we first recall the linearized velocity operator around the shear flow $\left(u_p,0\right)^\top$: 
\begin{align*}
  L_{u_p}(u)=-\nu\Delta u +P_{L}\left(\left(u_p(y),0\right)^\top\cdot\nabla u+u\cdot\nabla \left(u_p(y),0\right)^\top\right),
\end{align*}
where $P_{L}$ is the Leray projection.

Taking the $L^2$ inner product of $\left(L_{u_p}+\lambda\right)u$ with a divergence-free test function $u^*$, we have
\begin{align*}
  \left\langle (L_{u_p}+\lambda)u,u^*\right\rangle=\left\langle u,(L_{u_p}^*+\bar \lambda)u^*\right\rangle.
\end{align*}
Writing $u=\left(-\pa_y\phi,\pa_x\phi\right)^\top$ and $u^*=\left(-\pa_y\phi^*,\pa_x\phi^*\right)^\top$, we obtain
\begin{align*}
  \left\langle (L_{u_p}+\lambda)u,u^*\right\rangle=&\left\langle (L_{u_p}+\lambda) \left(\begin{array}{l}
      -\pa_y \phi\\
      \pa_x \phi
    \end{array}\right),\left(\begin{array}{l}
      -\pa_y \phi^*\\
      \pa_x \phi^*
    \end{array}\right)\right\rangle\\
    =&\left\langle -\pa_y(L_{u_p}^{(1)}+\lambda)\pa_y \phi-\pa_x(L_{u_p}^{(2)}+\lambda)\pa_x \phi,\phi^*\right\rangle.
\end{align*}
Taking Fourier transform in $x$, we have
\begin{align*}
  &\mathcal F_x \left( -\pa_y(L_{u_p}^{(1)}-\lambda)\pa_y \phi-\pa_x(L_{u_p}^{(2)}-\lambda)\pa_x \phi \right)\\
  =&- \left(\lambda \left(\pa_y^2-\alpha^2\right)\phi+u_pi\alpha\left(\pa_y^2-\alpha^2\right)\phi-u_p''i\alpha\phi-\nu\left(\pa_y^2-\alpha^2\right)^2\phi\right).
\end{align*}
Thus the adjoint problem can be studied at the level of the Orr--Sommerfeld operator. Motivated by the energy pairing, it is natural to pair the vorticity $\omega$ with the adjoint stream function $\phi^*$.

For the Orr-Sommerfeld operator
\begin{align*}
  Orr_{\alpha,\nu}(\omega)=i\varepsilon\left(\pa_y^2-\alpha^2\right)\omega+u_p\omega-u_p''\Delta_\alpha^{-1}\omega =i\varepsilon\left(\pa_y^2-\alpha^2\right)^2\phi+u_p\left(\pa_y^2-\alpha^2\right)\phi-u_p''\phi.
\end{align*}
Direct calculation shows that, for $\phi^*$ such that
\begin{align}\label{eq-ad-phi}
  \phi^*(-1)=\phi^{*\prime}(-1)=\phi^{*\prime}(0)=\phi^{*\prime\prime\prime}(0)=0,
\end{align}
we have
\begin{equation}\label{eq-derive-ad-OS}
  \begin{aligned}    
  &\int_{-1}^0 \left( Orr_{\alpha,\nu}-c\right)(\omega)\overline{\phi^*}dy\\
   =&\int_{-1}^0 \left(i\varepsilon\left(\pa_y^2-\alpha^2\right)^2\phi+\left(u_p-c\right)\left(\pa_y^2-\alpha^2\right)\phi-u_p''\phi\right)\overline{\phi^*}dy\\
  =&\int_{-1}^0  i\varepsilon\left(\pa_y^2-\alpha^2\right)\phi\cdot \overline{\left(\pa_y^2-\alpha^2\right)\phi^*}+\left(\pa_y^2-\alpha^2\right)\phi\cdot \overline{\left(u_p-\bar c\right)\phi^*}-\phi \overline{u_p''\phi^*}dy\\
  =&\int_{-1}^0  \phi\cdot \overline{\left(-i\varepsilon\left(\pa_y^2-\alpha^2\right)^2\phi^*+\left(\pa_y^2-\alpha^2\right)\left(\left(u_p-\bar c\right)\phi^*\right)-u_p''\phi^*\right)}dy\\
  =&\int_{-1}^0 \phi\cdot\overline{\left(Orr_{\alpha,\nu}^*-\bar c\right)(\omega^*)}dy. 
  \end{aligned}
\end{equation}
Therefore, the adjoint Orr-Sommerfeld operator is 
\begin{align*}
  Orr_{\alpha,\nu}^*(\omega^*)=&-i\varepsilon\left(\pa_y^2-\alpha^2\right)\omega^*+\Delta_\alpha \left(u_p\Delta_\alpha^{-1}\omega^*\right)-u_p''\Delta_\alpha^{-1}\omega^*\\
  =&-i\varepsilon\left(\pa_y^2-\alpha^2\right)^2\phi^*+\left(\pa_y^2-\alpha^2\right)(u_p\phi^*)-u_p''\phi^*.
\end{align*}

If $c$ is a simple eigenvalue of $Orr_{\alpha,\nu}$, it follows from the
standard spectral theory of linear operators (see \cite{Kato1966}) that $\bar c$ is a simple eigenvalue of the adjoint operator $Orr_{\alpha,\nu}^*$. As the eigenfunction $\phi_{\nu,\alpha}(y)$ we got in the T-S region is even in $y\in[-1,1]$, from \eqref{eq-ad-phi} and \eqref{eq-derive-ad-OS} it is clear that the adjoint eigenfunction $\phi_{\nu,\alpha}^*(y)$ is also even in $y\in[-1,1]$.
 
Accordingly, we normalize the adjoint eigenfunction by the vorticity--stream-function pairing
\begin{align}\label{eq-choice-phi*}
  \int_{-1}^1  \left(\pa_y^2-\alpha^2\right) \phi_{\nu,\alpha}(y)\cdot \overline{ \phi_{\nu,\alpha}^*(y)} dy=1.
\end{align}
It follows that
\begin{align}\label{eq-choice-phi*-r}
\int_{-1}^1  \left(\pa_y^2-\alpha^2\right)\phi_{\nu,\alpha,r}(y)\cdot \phi_{\nu,\alpha,r}^*(y)+\left(\pa_y^2-\alpha^2\right)\phi_{\nu,\alpha,i}(y)\cdot \phi_{\nu,\alpha,i}^*(y) dy=1,\\
  \int_{-1}^1  \left(\pa_y^2-\alpha^2\right)\phi_{\nu,\alpha,r}(y)\cdot \phi_{\nu,\alpha,i}^*(y)-\left(\pa_y^2-\alpha^2\right)\phi_{\nu,\alpha,i}(y)\cdot \phi_{\nu,\alpha,r}^*(y) dy=0.\label{eq-choice-phi*-i}
\end{align}

\subsubsection{Derivative of the eigenvalue curve $c(\nu)$}
In Theorem \ref{thm1}, we have provided the value for $\pa_{\nu}c(\nu)$ at the neutral point, now we show that these values can be represented by the eigenfunction and the adjoint eigenfunction. This has been pointed out in \cite{JS1972,Iooss1972}.

\begin{lemma}\label{lem-der-nu-c-phi}
Fix $\alpha>0$, and let $\left(c_r(\nu),c_i(\nu)\right)$ be the eigenvalue curve obtained in Theorem \ref{thm1}. Then
\begin{align}
  \pa_\nu c_i=\int_{-1}^1 \frac{1}{\alpha}\left(\pa_y^2-\alpha^2\right)^2 \phi_{\nu,\alpha,r}\cdot   \phi_{\nu,\alpha,r}^*(y)+\frac{1}{\alpha}\left(\pa_y^2-\alpha^2\right)^2 \phi_{\nu,\alpha,i}\cdot   \phi_{\nu,\alpha,i}^*(y) dy,
\end{align}
and
\begin{align}
  \pa_\nu c_r=\int_{-1}^1 \frac{1}{\alpha}\left(\pa_y^2-\alpha^2\right)^2 \phi_{\nu,\alpha,r}\cdot   \phi_{\nu,\alpha,i}^*(y)-\frac{1}{\alpha}\left(\pa_y^2-\alpha^2\right)^2 \phi_{\nu,\alpha,i}\cdot  \phi_{\nu,\alpha,r}^*(y) dy.
\end{align}
\end{lemma}
\begin{proof}
  
For fixed $\alpha>0$, $\phi_{\nu,\alpha}$ satisfies
\begin{align*}
  i \frac{\nu}{\alpha}\left(\pa_y^2-\alpha^2\right)^2 \phi_{\nu,\alpha}+\left(u_p-c\right)\left(\pa_y^2-\alpha^2\right) \phi_{\nu,\alpha}-u_p'' \phi_{\nu,\alpha}=0.
\end{align*}
Taking $\nu$-derivative on the above equation, we have
\begin{align*}
  &i \frac{1}{\alpha}\left(\pa_y^2-\alpha^2\right)^2 \phi_{\nu,\alpha}+i \frac{\nu}{\alpha}\left(\pa_y^2-\alpha^2\right)^2\pa_\nu \phi_{\nu,\alpha}\\
  &-\pa_\nu c\left(\pa_y^2-\alpha^2\right) \phi_{\nu,\alpha}+\left(u_p-c\right)\left(\pa_y^2-\alpha^2\right)\pa_\nu \phi_{\nu,\alpha}-u_p''\pa_\nu \phi_{\nu,\alpha}=0,
\end{align*}
then we do inner product with $ \phi_{\nu,\alpha}^*(y)$ to get
\begin{align*}
  &\int_{-1}^1  i \frac{1}{\alpha}\left(\pa_y^2-\alpha^2\right)^2 \phi_{\nu,\alpha}\cdot \overline{ \phi_{\nu,\alpha}^*(y)} dy-\int_{-1}^1  \pa_\nu c\left(\pa_y^2-\alpha^2\right) \phi_{\nu,\alpha}\cdot \overline{ \phi_{\nu,\alpha}^*(y)} dy\\
  &+\int_{-1}^1 \left(Orr_{\alpha,\nu}-c\right) \left(\Delta_\alpha\pa_\nu \phi_{\nu,\alpha}\right)\cdot \overline{ \phi_{\nu,\alpha}^*(y)} dy=0.
\end{align*}
Note that $\bar c$ is the eigenvalue for $Orr_{\alpha,\nu}^*$, it holds that
\begin{align*}
  \int_{-1}^1 \left(Orr_{\alpha,\nu}-c\right) \left(\Delta_\alpha\pa_\nu \phi_{\nu,\alpha}\right)\cdot \overline{ \phi_{\nu,\alpha}^*(y)} dy=\int_{-1}^1\pa_\nu \phi_{\nu,\alpha}(y)\cdot \overline{\left(Orr_{\alpha,\nu}^*-\bar c\right) \left(\Delta_\alpha\phi_{\nu,\alpha}^*\right)} dy=0.
\end{align*}
From the above calculation and \eqref{eq-choice-phi*}, we have 
\begin{align*}
  \int_{-1}^1  i \frac{1}{\alpha}\left(\pa_y^2-\alpha^2\right)^2 \phi_{\nu,\alpha}\cdot \overline{ \phi_{\nu,\alpha}^*(y)} dy=\int_{-1}^1  \pa_\nu c\left(\pa_y^2-\alpha^2\right) \phi_{\nu,\alpha}\cdot \overline{ \phi_{\nu,\alpha}^*(y)} dy=\pa_\nu c=\pa_\nu c_r+i\pa_\nu c_i.
\end{align*}
It follows that
\begin{align*}
  \Re \left(\int_{-1}^1  i \frac{1}{\alpha}\left(\pa_y^2-\alpha^2\right)^2 \phi_{\nu,\alpha}\cdot \overline{ \phi_{\nu,\alpha}^*(y)} dy\right)=\pa_\nu c_r,\\
  \Im \left(\int_{-1}^1  i \frac{1}{\alpha}\left(\pa_y^2-\alpha^2\right)^2 \phi_{\nu,\alpha}\cdot \overline{ \phi_{\nu,\alpha}^*(y)} dy\right)=\pa_\nu c_i.
\end{align*}
The results of this lemma follows directly.
\end{proof}
\subsubsection{Spectral projections}
We define the adjoint operator $\mathcal L_{c,\nu}^*$ through the duality relation
\begin{align*}
  \int_{\mathbb T_{\frac{2\pi}{\alpha^{[0]}}}\times[-1,1]} \mathcal L_{c,\nu} \left( \Delta \phi \right)\cdot\bar \phi^* dxdy=\int_{\mathbb T_{\frac{2\pi}{\alpha^{[0]}}}\times[-1,1]} \phi\cdot \overline{\mathcal L_{c,\nu}^* \left( \Delta \phi^* \right)} dxdy.
\end{align*}

For $\left(c_r^{[0]},\nu^{[0]}\right)$, we define the following real dual eigenfunction
\begin{align}
  \phi_1^*(x,y)= \phi_{\nu^{[0]},\alpha^{[0]},r}^*(y)\cos \left(\alpha^{[0]} x\right)- \phi_{\nu^{[0]},\alpha^{[0]},i}^*(y)\sin \left(\alpha^{[0]} x\right),\\
  \phi_2^*(x,y)= \phi_{\nu^{[0]},\alpha^{[0]},r}^*(y)\sin \left(\alpha^{[0]} x\right)+ \phi_{\nu^{[0]},\alpha^{[0]},i}^*(y)\cos \left(\alpha^{[0]} x\right),
\end{align}
which satisfy,
\begin{align*}
  \mathcal L_{c_r^{[0]},\nu^{[0]}}^* \left(\Delta\phi_1^* \right)=\mathcal L_{c_r^{[0]},\nu^{[0]}}^* \left(\Delta\phi_2^* \right)=0.
\end{align*}

By \eqref{eq-eigen-omega1}, \eqref{eq-eigen-omega2}, and \eqref{eq-choice-phi*}, one can easily check that 
\begin{align*}
  &\int_{\mathbb T_{\frac{2\pi}{\alpha^{[0]}}}\times[-1,1]} \omega_1(x,y)\phi_1^*(x,y) dxdy=\frac{\pi}{\alpha^{[0]}},\quad \int_{\mathbb T_{\frac{2\pi}{\alpha^{[0]}}}\times[-1,1]} \omega_1(x,y)\phi_2^*(x,y) dxdy=0,\\
  &\int_{\mathbb T_{\frac{2\pi}{\alpha^{[0]}}}\times[-1,1]} \omega_2(x,y)\phi_2^*(x,y) dxdy=\frac{\pi}{\alpha^{[0]}},\quad \int_{\mathbb T_{\frac{2\pi}{\alpha^{[0]}}}\times[-1,1]} \omega_2(x,y)\phi_1^*(x,y) dxdy=0.
\end{align*}

We introduce the spectral projection coefficients
\begin{align*}
  P_1\omega=&[\omega]_1=\frac{\alpha^{[0]}}{\pi}\left\langle \omega, \phi_{1}^* \right\rangle =\frac{\alpha^{[0]}}{\pi}\int_{\mathbb T_{\frac{2\pi}{\alpha^{[0]}}}\times[-1,1]} \omega(x,y)\phi_1^*(x,y) dxdy ,\\
  P_2\omega=&[\omega]_2=\frac{\alpha^{[0]}}{\pi}\left\langle \omega, \phi_{2}^* \right\rangle =\frac{\alpha^{[0]}}{\pi}\int_{\mathbb T_{\frac{2\pi}{\alpha^{[0]}}}\times[-1,1]} \omega(x,y)\phi_2^*(x,y) dxdy .
\end{align*}
Then define the projection operator
\begin{align*}
  P\omega=(P_1\omega)\omega_1+(P_2\omega)\omega_2,
\end{align*}
and $Q=I-P$.
\begin{lemma}\label{lem-project-1-2}
  It holds that
  \begin{align}\label{eq-project-1-2-0}
    P_1\omega_2=0,\quad P_2\omega_1=0,\quad P_1\pa_x\omega_1=0,\quad P_2\pa_x\omega_2=0,\quad P_2\pa_x\omega_1=-\alpha^{[0]},
  \end{align}
  and
    \begin{align}\label{eq-project-1-2-deri}
    P_1\Delta\omega_1= \alpha^{[0]}\pa_\nu c_i,\qquad P_2\Delta\omega_1= \alpha^{[0]}\pa_\nu c_r.
  \end{align}
  For any $f\in Ran(Q)$, it holds that
  \begin{align*}
    P_1f=P_2f=P_1\pa_xf=P_2\pa_xf=0.
  \end{align*}
\end{lemma}
\begin{proof}
  From \eqref{eq-choice-phi*-i}, and the fact that $\int_{\mathbb T_{\frac{2\pi}{\alpha^{[0]}}}}\sin\left(\alpha^{[0]} x\right)\cos\left(\alpha^{[0]} x\right)dx=0$, we have
  \begin{align*}
    P_1\omega_2=&\frac{\alpha^{[0]}}{\pi}\int_{\mathbb T_{\frac{2\pi}{\alpha^{[0]}}}\times[-1,1]}\omega_2(x,y)\phi_1^*(x,y) dxdy\\
      =&\frac{\alpha^{[0]}}{\pi}\int_{\mathbb T_{\frac{2\pi}{\alpha^{[0]}}}\times[-1,1]}  \left(\Delta_\alpha \phi_{\nu^{[0]},\alpha^{[0]},r}(y)\sin \left(\alpha^{[0]} x\right)+\Delta_\alpha \phi_{\nu^{[0]},\alpha^{[0]},i}(y)\cos \left(\alpha^{[0]} x\right)\right)\\
  &\qquad\qquad\qquad\qquad\qquad\qquad\cdot \left( \phi_{\nu^{[0]},\alpha^{[0]},r}^*(y)\cos \left(\alpha^{[0]} x\right)- \phi_{\nu^{[0]},\alpha^{[0]},i}^*(y)\sin \left(\alpha^{[0]} x\right)\right) dxdy\\
  =&0.
  \end{align*}
It is clear that $\pa_x\omega_1=-\alpha^{[0]}\omega_2$. Then based on the result, we have
  \begin{align*}
  P_1\pa_x\omega_1=-\alpha^{[0]}P_1\omega_2=0.
\end{align*}
The rest results of \eqref{eq-project-1-2-0} can be derived in the same way.

Next, we turn to \eqref{eq-project-1-2-deri}. We have
\begin{align*}
  P_1\Delta\omega_1=&\frac{\alpha^{[0]}}{\pi}\int_{\mathbb T_{\frac{2\pi}{\alpha^{[0]}}}\times[-1,1]} \Delta\omega_1(x,y)\phi_1^*(x,y) dxdy\\
  =&\frac{\alpha^{[0]}}{\pi}\int_{\mathbb T_{\frac{2\pi}{\alpha^{[0]}}}\times[-1,1]}  \left(\Delta_\alpha^2 \phi_{\nu^{[0]},\alpha^{[0]},r}(y)\cos \left(\alpha^{[0]} x\right)-\Delta_\alpha^2 \phi_{\nu^{[0]},\alpha^{[0]},i}(y)\sin\left(\alpha^{[0]} x\right)\right)\\
  &\qquad\qquad\qquad\qquad\qquad\qquad\cdot \left( \phi_{\nu^{[0]},\alpha^{[0]},r}^*(y)\cos \left(\alpha^{[0]} x\right)- \phi_{\nu^{[0]},\alpha^{[0]},i}^*(y)\sin \left(\alpha^{[0]} x\right)\right) dxdy\\
  =&\int_{-1}^1\left( \Delta_\alpha^2 \phi_{\nu^{[0]},\alpha^{[0]},r}(y) \phi_{\nu^{[0]},\alpha^{[0]},r}^*(y) +\Delta_\alpha^2 \phi_{\nu^{[0]},\alpha^{[0]},i}(y) \phi_{\nu^{[0]},\alpha^{[0]},i}^*(y) \right)dy.
\end{align*}
Then by Lemma \ref{lem-der-nu-c-phi}, we have 
\begin{align*}
  P_1\Delta\omega_1=\alpha^{[0]}\pa_\nu c_i.
\end{align*}
The result for $P_2\Delta\omega_1$ can be deduced in the same way.

For $f\in Ran(Q)$, we have
\begin{align*}
  P_1\pa_xf=&\frac{\alpha^{[0]}}{\pi}\int_{\mathbb T_{\frac{2\pi}{\alpha^{[0]}}}\times[-1,1]} \pa_xf(x,y)\phi_1^*(x,y) dxdy=-\frac{\alpha^{[0]}}{\pi}\int_{\mathbb T_{\frac{2\pi}{\alpha^{[0]}}}\times[-1,1]} f(x,y)\pa_x\phi_1^*(x,y) dxdy\\
  =&\frac{(\alpha^{[0]})^2}{\pi}\int_{\mathbb T_{\frac{2\pi}{\alpha^{[0]}}}\times[-1,1]} f(x,y)\phi_2^*(x,y) dxdy=0.
\end{align*}
For the same reason, we have
\begin{align*}
  P_2\pa_xf=0.
\end{align*}
\end{proof}
\subsection{Lyapunov--Schmidt reduction}
\subsubsection{Expansion of the system} 
 
For $s$ sufficiently small, we seek a solution of the following form:
\begin{equation}\label{eq-s-expan}
  \begin{aligned}    
   \omega_s=&s\omega^{[1]}+s^2\omega^{[2]},\quad \omega^{[2]}\in Ran(Q),\\
   c_s=&c_r^{[0]}+sc_r^{[1]}+s^2c_r^{[2]},\\
   \nu_s=&\nu^{[0]}+s\nu^{[1]}+s^2\nu^{[2]}.    
  \end{aligned}
\end{equation}
In the rest of the proof, we take
\begin{align*}
  \omega^{[1]}=\omega_1.
\end{align*}
The choice $\omega^{[1]}=\omega_1$ fixes the phase of the bifurcating branch; the phase-shifted branch is generated by translations in $x$.

Substituting \eqref{eq-s-expan} into \eqref{eq-reformulation-nonlin}, we obtain
\begin{equation}\label{eq-exp-full}
  \begin{aligned}    
    &\left(u_p-c_r^{[0]}-sc_r^{[1]}-s^2c_r^{[2]}\right)\pa_x \left(s\omega_1+s^2\omega^{[2]}\right)\\
    &-u_p''\pa_x \left(s\phi_1+s^2\phi^{[2]}\right)-\left(\nu^{[0]}+s\nu^{[1]}+s^2\nu^{[2]}\right)\Delta \left(s\omega_1+s^2\omega^{[2]}\right)\\
    =&\pa_y\left(s\phi_1+s^2\phi^{[2]}\right)\pa_x\left(s\omega_1+s^2\omega^{[2]}\right)-\pa_x\left(s\phi_1+s^2\phi^{[2]}\right)\pa_y\left(s\omega_1+s^2\omega^{[2]}\right).
  \end{aligned}
\end{equation}
Our aim is to show that, for each sufficiently small $s$, there exist $\omega^{[2]}\in Ran(Q)$, $c_r^{[1]}$, $c_r^{[2]}$, $\nu^{[1]}$, $\nu^{[2]}\in \mathbb R$, such that \eqref{eq-exp-full} holds.

Collecting powers of $s$, we decompose \eqref{eq-exp-full} into the following equations.

At order $s$, we obtain
\begin{align*}
  \left(u_p-c_r^{[0]}\right)\pa_x \left(s\omega_1\right)-u_p''\pa_x \left(s\phi_1\right)-\nu^{[0]}\Delta \left(s\omega_1\right)=0.
\end{align*}
By the definition of $\omega_1$, this equation is identically satisfied.

At order $s^2$, we obtain
\begin{align*}
  &-s^2c_r^{[1]}\pa_x\omega_1-s^2\nu^{[1]}\Delta\omega_1+s^2 \left( \left(u_p-c_r^{[0]}\right)\pa_x \omega^{[2]}-u_p''\pa_x \phi^{[2]}-\nu^{[0]}\Delta \omega^{[2]}\right)\\
  =&s^2 \left(\pa_y\phi_1\pa_x\omega_1-\pa_x\phi_1\pa_y\omega_1\right), 
\end{align*}
Recalling the definition of $\mathcal L_{c_r^{[0]},\nu^{[0]}}$, we can rewrite the above equation to
\begin{align*}
  -s^2c_r^{[1]}\pa_x\omega_1-s^2\nu^{[1]}\Delta\omega_1+s^2 \mathcal L_{c_r^{[0]},\nu^{[0]}}\omega^{[2]}=s^2 \left(\pa_y\phi_1\pa_x\omega_1-\pa_x\phi_1\pa_y\omega_1\right).
\end{align*} 

At order $s^3$, we obtain
\begin{align*}
  &-s^3c_r^{[1]}\pa_x\omega^{[2]}-s^3\nu^{[1]}\Delta\omega^{[2]}-s^3c_r^{[2]}\pa_x\omega_1-s^3\nu^{[2]}\Delta\omega_1\\
  =&s^3 \left(\pa_y\phi_1\pa_x\omega^{[2]}-\pa_x\phi_1\pa_y\omega^{[2]}\right)+s^3 \left(\pa_y\phi^{[2]}\pa_x\omega_1-\pa_x\phi^{[2]}\pa_y\omega_1\right).
\end{align*}

At order $s^4$, we obtain the genuinely nonlinear terms
\begin{align*}
  -s^4c_r^{[2]}\pa_x\omega^{[2]}-s^4\nu^{[2]}\Delta\omega^{[2]}=s^4 \left(\pa_y\phi^{[2]}\pa_x\omega^{[2]}-\pa_x\phi^{[2]}\pa_y\omega^{[2]}\right).
\end{align*}

\subsubsection{Projection in each subspace} 
We now project the expanded equation onto the eigenspace and its complement. Since the order-$s$ equation is identically satisfied, we start with the order-$s^2$ equation. 
\begin{lemma}\label{lem-nu1-cr1}
  It holds that
  \begin{align*}
    \nu^{[1]}=c_r^{[1]}=0.
  \end{align*}
\end{lemma}
\begin{proof}

Applying $P_1$ to the order $s^2$ equation and using Lemma \ref{lem-project-1-2}, we obtain
\begin{align}\label{eq-proj-1-s2}
  &-s^2\nu^{[1]}\alpha^{[0]}\pa_\nu c_i +s^2 P_1\mathcal L_{c_r^{[0]},\nu^{[0]}}\omega^{[2]}=s^2 P_1\left(\pa_y\phi_1\pa_x\omega_1-\pa_x\phi_1\pa_y\omega_1\right).
\end{align}

By the definition of $\phi_1^*$, we have
\begin{align*}
  \int_{\mathbb T_{\frac{2\pi}{\alpha^{[0]}}}\times[-1,1]} \mathcal L_{c_r^{[0]},\nu^{[0]}}\omega^{[2]}\cdot \phi_1^*(x,y) dxdy=\int_{\mathbb T_{\frac{2\pi}{\alpha^{[0]}}}\times[-1,1]}\phi^{[2]}\cdot \mathcal L_{c_r^{[0]},\nu^{[0]}}^*\left(\Delta \phi_1^*\right)dxdy =0.
\end{align*}

Moreover, $\pa_y\phi_1$ and $\pa_y\omega_1$ are odd in $y$, while $\pa_x\phi_1$, $\pa_x\omega_1$, and $\phi_1^*$ are even in $y$. Hence
\begin{align*}
 \int_{\mathbb T_{\frac{2\pi}{\alpha^{[0]}}}\times[-1,1]} \left(\pa_y\phi_1\pa_x\omega_1-\pa_x\phi_1\pa_y\omega_1\right)\cdot \phi_1^*(x,y) dxdy=0.
\end{align*}

It follows from \eqref{eq-proj-1-s2} that
\begin{align*}
  -s^2\nu^{[1]}\alpha^{[0]}\pa_\nu c_i=0.
\end{align*}
By the transversal crossing condition in Theorem \ref{thm1}, $\pa_\nu c_i\neq0$. Hence $\nu^{[1]}=0$.

Next, applying $P_2$ to the order-$s^2$ equation, we similarly obtain
\begin{align*}
  &s^2c_r^{[1]}\alpha^{[0]}-s^2\nu^{[1]}\alpha^{[0]}\pa_\nu c_r+s^2 P_2\mathcal L_{c_r^{[0]},\nu^{[0]}}\omega^{[2]}=s^2 P_2\left(\pa_y\phi_1\pa_x\omega_1-\pa_x\phi_1\pa_y\omega_1\right).
\end{align*}
As above,
\begin{align*}
  P_2\mathcal L_{c_r^{[0]},\nu^{[0]}}\omega^{[2]}=P_2\left(\pa_y\phi_1\pa_x\omega_1-\pa_x\phi_1\pa_y\omega_1\right)=0.
\end{align*}
Thus
\begin{align*}
  s^2c_r^{[1]}\alpha^{[0]}-s^2\nu^{[1]}\alpha^{[0]}\pa_\nu c_r=0.
\end{align*}
Since $\nu^{[1]}=0$, we conclude that $c_r^{[1]}=0$.
\end{proof}

From Lemma \ref{lem-nu1-cr1}, the $P_1$ and $P_2$ projection on \eqref{eq-exp-full} is equivalent on the order $s^3$ and order $s^4$ equations. Using Lemma \ref{lem-project-1-2} and Lemma \ref{lem-nu1-cr1}, and applying $P_1$ to \eqref{eq-exp-full}, we obtain
\begin{align*}
  \nu^{[2]}\alpha^{[0]}\pa_\nu c_i=&- P_1 \left( \pa_y\phi_1\pa_x\omega^{[2]}-\pa_x\phi_1\pa_y\omega^{[2]} + \pa_y\phi^{[2]}\pa_x\omega_1-\pa_x\phi^{[2]}\pa_y\omega_1  \right)\\
  &-s\nu^{[2]}P_1\Delta\omega^{[2]}-sP_1\left(\pa_y\phi^{[2]}\pa_x\omega^{[2]}-\pa_x\phi^{[2]}\pa_y\omega^{[2]}\right).
\end{align*}
We denote the right hand side by $\mathcal N_1(\nu^{[2]},c_r^{[2]},\omega^{[2]})$.

Applying $P_2$ on \eqref{eq-exp-full}, we have
\begin{align*}
  \nu^{[2]}\alpha^{[0]}\pa_\nu c_r-c_r^{[2]}\alpha^{[0]}=&- P_2 \left( \pa_y\phi_1\pa_x\omega^{[2]}-\pa_x\phi_1\pa_y\omega^{[2]} + \pa_y\phi^{[2]}\pa_x\omega_1-\pa_x\phi^{[2]}\pa_y\omega_1  \right)\\
  &-s\nu^{[2]}P_2\Delta\omega^{[2]}-sP_2\left(\pa_y\phi^{[2]}\pa_x\omega^{[2]}-\pa_x\phi^{[2]}\pa_y\omega^{[2]}\right).
\end{align*}
We denote the right hand side by $\mathcal N_2(\nu^{[2]},c_r^{[2]},\omega^{[2]})$.

Finally, we apply $Q$ to \eqref{eq-exp-full}. In the complementary equation, the order-$s^2$ terms must be retained. Since $\mathcal L_{c_r^{[0]},\nu^{[0]}}\omega^{[2]}\in Ran(Q)$, we obtain
\begin{align*}
  \mathcal L_{c_r^{[0]},\nu^{[0]}}\omega^{[2]}=&Q  \left(\pa_y\phi_1\pa_x\omega_1-\pa_x\phi_1\pa_y\omega_1\right) \\
  &+s\nu^{[2]} Q\Delta\omega_1+sQ \left(\pa_y\phi_1\pa_x\omega^{[2]}-\pa_x\phi_1\pa_y\omega^{[2]}+\pa_y\phi^{[2]}\pa_x\omega_1-\pa_x\phi^{[2]}\pa_y\omega_1\right)\\
  &+s^2 \left( c_r^{[2]} Q\pa_x\omega^{[2]}+\nu^{[2]} Q\Delta\omega^{[2]} \right)+s^2 Q \left( \pa_y\phi^{[2]}\pa_x\omega^{[2]}-\pa_x\phi^{[2]}\pa_y\omega^{[2]} \right).
\end{align*}
We define the right hand side by $\mathcal N_3(\nu^{[2]},c_r^{[2]},\omega^{[2]})$.

\subsection{Solution operator}
Based on the projections derived in the previous subsection, \eqref{eq-exp-full} reduces to the following system:
\begin{equation*}
  \left(
    \begin{array}{ccc}
      \alpha^{[0]}\pa_\nu c_i&0&0\\
      \alpha^{[0]}\pa_\nu c_r&-\alpha^{[0]}&0\\
      0&0&\mathcal L_{c_r^{[0]},\nu^{[0]}}
    \end{array}
  \right)\cdot \left(
    \begin{array}{l}      
      \nu^{[2]}\\
      c_r^{[2]}\\
      \omega^{[2]}
    \end{array}
  \right) =\left(
    \begin{array}{l}      
      \mathcal N_1\\
      \mathcal N_2\\
      \mathcal N_3
    \end{array}
  \right).
\end{equation*}
Our aim is to find $(\nu^{[2]},c_r^{[2]},\omega^{[2]})$ that solves the following nonlinear system
\begin{equation}\label{eq-system-nonlinear}
\left(
    \begin{array}{l}      
      \nu^{[2]}\\
      c_r^{[2]}\\
      \omega^{[2]}
    \end{array}
  \right) =\left(
    \begin{array}{ccc}
      \frac{1}{\alpha^{[0]}\pa_\nu c_i}&0&0\\
      \frac{\pa_\nu c_r}{\alpha^{[0]}\pa_\nu c_i}&\frac{-1}{\alpha^{[0]}}&0\\
      0&0&\mathcal L_{c_r^{[0]},\nu^{[0]}}^{-1}
    \end{array}
  \right)\cdot \left(
    \begin{array}{l}      
      \mathcal N_1\\
      \mathcal N_2\\
      \mathcal N_3
    \end{array}
  \right).
\end{equation}
By the transversal crossing condition, $\pa_\nu c_i\neq0$ at the critical point. Therefore, the finite-dimensional part of the linear system is invertible.

It remains to justify the bounded invertibility of $\mathcal L_{c_r^{[0]},\nu^{[0]}}$ on $Ran(Q)$. We first give the following lemma.
\begin{lemma}\label{lem-toy-ellip}
  For  $\hat f(\alpha,y)\in L^2([-1,1])$, $\alpha=k\alpha^{[0]}$ with $k\neq0$, there exists unique $\hat \phi(\alpha,y)$ solves
  \begin{equation}\label{eq-ellip-omega-al}
  \begin{aligned}    
     \left(\pa_y^2-\alpha^2\right)\hat \omega(\alpha,y)=\hat f(\alpha,y),\quad  \left(\pa_y^2-\alpha^2\right)\hat \phi(\alpha,y)=\hat \omega(\alpha,y),\\
  \hat \phi(\alpha,\pm1)=\pa_y\hat \phi(\alpha,\pm1)=0. 
  \end{aligned}
\end{equation}
Moreover, it holds that
\begin{align*}
  \left\|\hat \omega(\alpha,y)\right\|_{H^2_\alpha}+\left\|\hat \phi(\alpha,y)\right\|_{H^2_\alpha}\lesssim \left\|\hat f(\alpha,y)\right\|_{L^2},
\end{align*}
where
\begin{align*}
  \left\|\hat \omega(\alpha,y)\right\|_{H^2_\alpha}=(1+|\alpha|^2)\left\|\hat \omega(\alpha,y)\right\|_{L^2}+\left\|\pa_y^2\hat \omega(\alpha,y)\right\|_{L^2}.
\end{align*}
\end{lemma}
\begin{proof}
  To solve this elliptic equation, for $\alpha\neq0$, we introduce the following Green function with Dirichlet boundary condition
\begin{equation}
  G_{\alpha}(z,y)=-\frac{1}{\alpha \sinh(2\alpha)}\left\{
    \begin{array}{ll}
      \sinh(\alpha(1+z))\sinh(\alpha(1-y)),&z\le y,\\
      \sinh(\alpha(1+y))\sinh(\alpha(1-z)),&z\ge y,
    \end{array}
  \right.
\end{equation}
which satisfies
\begin{align*}
  \left(\pa_y^2-\alpha^2\right)G_{\alpha}(z,y)=\delta_z(y),\quad G_{\alpha}(z,\pm1)=0.
\end{align*}

We first define
\begin{align*}
  \hat \omega_{Diri}(\alpha,y)=\int^1_{-1}G_{\alpha}(z,y)f(z)dz,
\end{align*}
and then
\begin{align*}
  \hat \phi_{Diri}(\alpha,y)=\int^1_{-1}G_{\alpha}(z,y)\hat \omega_{Diri}(z)dz.
\end{align*}
Therefore,
\begin{align*}
  \pa_y\hat \phi_{Diri}(\alpha,\pm1)=\pm\int^1_{-1}\frac{\sinh(\alpha(1\pm z))}{\sinh(2\alpha)}\hat \omega_{Diri}(z)dz.
\end{align*}

Then we will introduce $\hat \phi_{a}(\alpha,y)$ and $\hat \phi_{b}(\alpha,y)$ to correct the boundary condition. Indeed,
\begin{align*}
  \hat \phi_{a}(\alpha,y)=\int^1_{-1}G_{\alpha}(z,y)\sinh(\alpha(1+z))dz,\\
  \hat \phi_{b}(\alpha,y)=\int^1_{-1}G_{\alpha}(z,y)\sinh(\alpha(1-z))dz
\end{align*}
which satisfies
\begin{align*}
  \left(\pa_y^2-\alpha^2\right)^2\hat \phi_{a}(\alpha,y)=\left(\pa_y^2-\alpha^2\right)^2\hat \phi_{b}(\alpha,y)=0.
\end{align*}

After matching the boundary condition, we get $\omega$ that satisfies \eqref{eq-ellip-omega-al}. It is easy to check that
\begin{align*}
  \left\|\hat \omega\right\|_{H^2_{\alpha}}\lesssim \left\|\hat f\right\|_{L^2}.
\end{align*}
\end{proof}

\begin{lemma}\label{lem-toy-ellip-0}
  For  $\hat f(0,y)\in L^2([-1,1])$, there exists unique $\hat \phi(0,y)$ solves
  \begin{equation}\label{eq-ellip-omega-al-0}
  \begin{aligned}    
     \pa_y^2\hat \omega(0,y)=\hat f(0,y),\quad  \pa_y^2\hat \phi(0,y)=\hat \omega(0,y),\\
  \hat\phi(0,-1)=\hat\phi'(0,\pm1)=\pa_y^2\hat\phi(0,1)-\pa_y^2\hat\phi(0,-1)=0. 
  \end{aligned}
\end{equation}
Moreover, it holds that
\begin{align*}
  \left\|\hat \omega(0,y)\right\|_{H^2}+\left\|\hat \phi(0,y)\right\|_{H^2}\lesssim \left\|\hat f(0,y)\right\|_{L^2}.
\end{align*}
\end{lemma}
\begin{proof}
 The proof is similar to Lemma \ref{lem-toy-ellip}.

We use the following Green function:
\begin{equation}
  G_{0}(z,y)=-\frac{1}{2}\left\{
    \begin{array}{ll}
      (1+z)(1-y),&z\le y,\\
      (1+y)(1-z),&z\ge y,
    \end{array}
  \right.
\end{equation}
to define $\hat \omega_{Diri}(0,y)$, $\hat \phi_{Diri}(0,y)$, $\hat \phi_{a}(0,y)$, $\hat \phi_{b}(0,y)$. Using these functions, we first construct $\hat \phi^\diamond(0,y)$ such that
\begin{align*}
  \hat \phi^\diamond(0,\pm1)=\pa_y\hat \phi^\diamond(0,\pm1)=0.
\end{align*}

 Moreover, we introduce $\hat \psi_{c}(0,y)=-y+\frac{y^3}{3}$, and $\hat \psi_{d}(0,y)=1$, and use $\hat \psi_{c}$ and $\hat \psi_{d}$ to correct the boundary condition of $\hat \phi^\diamond(0,y)$. 

 Finally, we get $\hat \phi(0,y)$ such that
 \begin{align*}
   \pa_y^2\hat \phi(0,1)-\pa_y^2\hat \phi(0,-1)=\pa_y\hat \phi(0,\pm1)=\hat \phi(0,-1)=0.
 \end{align*}

\end{proof}

\begin{lemma}\label{lem-inver-Lcnu}
The operator $\mathcal L_{c_r^{[0]},\nu^{[0]}}$ is invertible on $Ran(Q)$. More precisely, for $f\in Ran(Q)$, it holds that
  \begin{align*}
    \left\|\mathcal L_{c_r^{[0]},\nu^{[0]}}^{-1}f\right\|_{H^2 \left(\mathbb T_{\frac{2\pi}{\alpha^{[0]}}}\times[-1,1]\right)}\le C \left\|f\right\|_{L^2 \left(\mathbb T_{\frac{2\pi}{\alpha^{[0]}}}\times[-1,1]\right)}.
  \end{align*}
\end{lemma}
\begin{proof}
  We decompose $\mathcal L_{c_r^{[0]},\nu^{[0]}}$ into Fourier modes in $x$,
  \begin{align*}
     \mathcal L_{c_r^{[0]},\nu^{[0]}}=\sum_{\alpha=k\alpha^{[0]}}\mathcal L_{c_r^{[0]},\nu^{[0]},\alpha}.
   \end{align*} 

Let
\begin{align*}
  \mathcal L_{c_r^{[0]},\nu^{[0]},\alpha}\hat \omega(\alpha,y)=\hat f(\alpha,y),
\end{align*}
which is
\begin{align}\label{eq-omega-OS}
  -\nu^{[0]}\left(\pa_y^2-\alpha^2\right)\hat \omega(\alpha,y)+i\alpha\left(u_p-c_r^{[0]}\right)\hat \omega(\alpha,y)-i\alpha u_p''\hat \phi(\alpha,y)=\hat f(\alpha,y).
\end{align}

Since $f\in Ran(Q)$, the critical modes $\omega_1,\omega_2$ are removed. Hence, by the spectral properties established above, for each fixed Fourier mode $\alpha$ there exists $C(\alpha)$ such that
\begin{align}\label{eq-est-l2-l2}
  \left\|\hat \omega(\alpha,y)\right\|_{L^2}\le C(\alpha)\left\|\hat f(\alpha,y)\right\|_{L^2}.
\end{align}
For $|k|$ sufficiently large, the energy estimate gives a uniform bound. For the remaining finitely many modes, invertibility on $Ran(Q)$ gives finite constants. Taking the maximum over this finite set and the high-mode uniform bound, we obtain a constant $C_{\max}$ independent of $\alpha$.

Next, we rewrite \eqref{eq-omega-OS} to
\begin{align}\label{eq-omega-ellip}
  -\nu^{[0]}\left(\pa_y^2-\alpha^2\right)\hat \omega(\alpha,y)=\hat f(\alpha,y)-i\alpha\left(u_p-c_r^{[0]}\right)\hat \omega(\alpha,y)+i\alpha u_p''\hat \phi(\alpha,y).
\end{align}
Then by Lemma \ref{lem-toy-ellip} and Lemma \ref{lem-toy-ellip-0}, it holds that
\begin{align*}
  \left\|\hat \omega(\alpha,y)\right\|_{H^2_\alpha}\le C(\nu^{[0]}) \left(\left\|\hat \omega(\alpha,y)\right\|_{L^2}+\left\|\hat f(\alpha,y)\right\|_{L^2}\right).
\end{align*}
where $C(\nu^{[0]})$ depends only on $\nu^{[0]}$.

Then by \eqref{eq-est-l2-l2}, we have
\begin{align*}
  \left\|\hat \omega(\alpha,y)\right\|_{H^2_\alpha}\le C(\nu^{[0]}) \left(C_{max}+1\right)\left\|\hat f(\alpha,y)\right\|_{L^2}.
\end{align*}

Summing over all Fourier modes gives the desired $H^2$ estimate and completes the proof.
\end{proof}

\subsection{The fixed-point argument and completion of the proof}
We are now ready to solve for $(\nu^{[2]},c_r^{[2]},\omega^{[2]})$ and complete the proof of Theorem \ref{thm-bif}.

\begin{proof}[Proof of Theorem \ref{thm-bif}]

We define
\begin{align*}
  c_r^{[2],\left\{0\right\}}=0,\qquad \nu^{[2],\left\{0\right\}}=0,\qquad  \omega^{[2],\left\{0\right\}}=\mathcal L_{c_r^{[0]},\nu^{[0]}}^{-1}Q  \left(\pa_y\phi_1\pa_x\omega_1-\pa_x\phi_1\pa_y\omega_1\right).
\end{align*}

For $k\ge1$, we define
\begin{equation*}
\left(
    \begin{array}{l}      
      \nu^{[2],\left\{k\right\}}\\
      c_r^{[2],\left\{k\right\}}\\
      \omega^{[2],\left\{k\right\}}
    \end{array}
  \right) =\left(
    \begin{array}{ccc}
      \frac{1}{\alpha^{[0]}\pa_\nu c_i}&0&0\\
      \frac{\pa_\nu c_r}{\alpha^{[0]}\pa_\nu c_i}&\frac{-1}{\alpha^{[0]}}&0\\
      0&0&\mathcal L_{c_r^{[0]},\nu^{[0]}}^{-1}
    \end{array}
  \right)\cdot \left(
    \begin{array}{l}      
      \mathcal N_1\left(\nu^{[2],\{k-1\}},c_r^{[2],\{k-1\}},\omega^{[2],\{k-1\}}\right)\\
      \mathcal N_2\left(\nu^{[2],\{k-1\}},c_r^{[2],\{k-1\}},\omega^{[2],\{k-1\}}\right)\\
      \mathcal N_3\left(\nu^{[2],\{k-1\}},c_r^{[2],\{k-1\}},\omega^{[2],\{k-1\}}\right)
    \end{array}
  \right).
\end{equation*}

For $M_1,M_2>0$, define
\begin{align*}
  \mathcal S=\left\{
  \left(\nu^{[2]},c_r^{[2]},\omega^{[2]}\right):
  \nu^{[2]},c_r^{[2]}\in\mathbb R,\ 
  \omega^{[2]}\in Ran(Q),\ 
  \left|c_r^{[2]}\right|+\left|\nu^{[2]}\right|\le M_1,\ 
  \left\|\omega^{[2]}\right\|_{H^2}\le M_2
  \right\}.
\end{align*}
Here $\omega^{[2]}=\Delta\phi^{[2]}$, where $\phi^{[2]}$ satisfies the boundary conditions \eqref{bc-tw-1} and \eqref{bc-tw-2}.

For $\left(\nu^{[2],\left\{k-1\right\}},c_r^{[2],\left\{k-1\right\}},\omega^{[2],\left\{k-1\right\}}\right)\in \mathcal S$, from Lemma \ref{lem-inver-Lcnu}, we can see that
\begin{align*}
  \left\|\omega^{[2],\left\{k\right\}}\right\|_{H^2}\lesssim \left\|\omega^{[2],\left\{0\right\}}\right\|_{H^2}+s M_1+s M_2+s^2M_1M_2+s^2M_2^2,\\
  \left|c_r^{[2],\left\{k\right\}}\right|+\left|\nu^{[2],\left\{k\right\}}\right|\lesssim M_2+sM_1M_2+sM_2^2.
\end{align*}
We first choose $M_2$ sufficiently large so that it dominates $\left\|\omega^{[2],\{0\}}\right\|_{H^2}$. Then we choose $M_1$ sufficiently large depending on $M_2$. Finally, choosing $s>0$ sufficiently small, the above estimates imply that the map sends $\mathcal S$ into itself.

Next show the contraction of the iteration map in $\mathcal S$. It holds that
\begin{align*}
  &\omega^{[2],\left\{k\right\}}-\omega^{[2],\left\{k-1\right\}}\\
  =&\mathcal L_{c_r^{[0]},\nu^{[0]}}^{-1} \left(\mathcal N_3\left(\nu^{[2],\left\{k-1\right\}},c_r^{[2],\left\{k-1\right\}},\omega^{[2],\left\{k-1\right\}}\right)-\mathcal N_3\left(\nu^{[2],\left\{k-2\right\}},c_r^{[2],\left\{k-2\right\}},\omega^{[2],\left\{k-2\right\}}\right)\right)\\
  =&\mathcal L_{c_r^{[0]},\nu^{[0]}}^{-1} \bigg(s \left(\nu^{[2],\left\{k-1\right\}}-\nu^{[2],\left\{k-2\right\}}\right)Q\Delta\omega_1\\
  &\qquad+s Q\left(\pa_y\phi_1\pa_x \left(\omega^{[2],\left\{k-1\right\}}-\omega^{[2],\left\{k-2\right\}}\right)-\pa_x\phi_1\pa_y\left(\omega^{[2],\left\{k-1\right\}}-\omega^{[2],\left\{k-2\right\}}\right) \right)\\
  &\qquad+s Q\left(\pa_y\left(\phi^{[2],\left\{k-1\right\}}-\phi^{[2],\left\{k-2\right\}}\right)\pa_x\omega_1-\pa_x\left(\phi^{[2],\left\{k-1\right\}}-\phi^{[2],\left\{k-2\right\}}\right)\pa_y\omega_1\right)\\
  &\quad+s^2\left( \left(c_r^{[2],\left\{k-1\right\}}-c_r^{[2],\left\{k-2\right\}}\right) Q\pa_x\omega^{[2],\left\{k-1\right\}}+c_r^{[2],\left\{k-2\right\}}Q\pa_x\left(\omega^{[2],\left\{k-1\right\}}-\omega^{[2],\left\{k-2\right\}}\right)\right)\\
  &\quad+s^2\left( \left(\nu^{[2],\left\{k-1\right\}}-\nu^{[2],\left\{k-2\right\}}\right) Q\Delta\omega^{[2],\left\{k-1\right\}}+\nu^{[2],\left\{k-2\right\}}Q\Delta\left(\omega^{[2],\left\{k-1\right\}}-\omega^{[2],\left\{k-2\right\}}\right)\right)\\
  &\quad+s^2Q\left( \pa_y\left(\phi^{[2],\left\{k-1\right\}}-\phi^{[2],\left\{k-2\right\}}\right)\pa_x\omega^{[2],\left\{k-1\right\}}+\pa_y\phi^{[2],\left\{k-2\right\}}\pa_x\left(\omega^{[2],\left\{k-1\right\}}-\omega^{[2],\left\{k-2\right\}}\right)\right)\\
  &\quad-s^2Q\left( \pa_x\left(\phi^{[2],\left\{k-1\right\}}-\phi^{[2],\left\{k-2\right\}}\right)\pa_y\omega^{[2],\left\{k-1\right\}}+\pa_x\phi^{[2],\left\{k-2\right\}}\pa_y\left(\omega^{[2],\left\{k-1\right\}}-\omega^{[2],\left\{k-2\right\}}\right)\right)\bigg).
\end{align*}

Thus
\begin{align*}
  &\left\|\omega^{[2],\left\{k\right\}}-\omega^{[2],\left\{k-1\right\}}\right\|_{H^2}\\
  \lesssim& s \left(\left|\nu^{[2],\left\{k-1\right\}}-\nu^{[2],\left\{k-2\right\}}\right|+\left|c_r^{[2],\left\{k-1\right\}}-c_r^{[2],\left\{k-2\right\}}\right|+\left\|\omega^{[2],\left\{k-1\right\}}-\omega^{[2],\left\{k-2\right\}}\right\|_{H^2}\right).
\end{align*}
Similarly, we have
\begin{align*}
  \left|\nu^{[2],\left\{k\right\}}-\nu^{[2],\left\{k-1\right\}}\right| \lesssim \left\|\omega^{[2],\left\{k-1\right\}}-\omega^{[2],\left\{k-2\right\}}\right\|_{H^2}+s  \left|\nu^{[2],\left\{k-1\right\}}-\nu^{[2],\left\{k-2\right\}}\right|,
\end{align*}
and
\begin{align*}
  \left|c_r^{[2],\left\{k\right\}}-c_r^{[2],\left\{k-1\right\}}\right| \lesssim \left\|\omega^{[2],\left\{k-1\right\}}-\omega^{[2],\left\{k-2\right\}}\right\|_{H^2}+s \left(\left|\nu^{[2],\left\{k-1\right\}}-\nu^{[2],\left\{k-2\right\}}\right|+\left|c_r^{[2],\left\{k-1\right\}}-c_r^{[2],\left\{k-2\right\}}\right|\right). 
\end{align*}

For solution $\left(\nu^{[2]},c_r^{[2]},\omega^{[2]}\right)$, we define the following product norm:
\begin{align*}
  \left\|\left(\nu^{[2]},c_r^{[2]},\omega^{[2]}\right)\right\|_{PN}=s^{\frac{1}{2}}|\nu^{[2]}|+s^{\frac{1}{2}}|c_r^{[2]}|+\left\|\omega^{[2]}\right\|_{H^2},
\end{align*}
then we can see that
\begin{align*}
  &\left\|\left(\nu^{[2],\left\{k\right\}}-\nu^{[2],\left\{k-1\right\}},c_r^{[2],\left\{k\right\}}-c_r^{[2],\left\{k-1\right\}},\omega^{[2],\left\{k\right\}}-\omega^{[2],\left\{k-1\right\}}\right)\right\|_{PN}\\
  \lesssim&s^{\frac{1}{2}}\left\|\left(\nu^{[2],\left\{k-1\right\}}-\nu^{[2],\left\{k-2\right\}},c_r^{[2],\left\{k-1\right\}}-c_r^{[2],\left\{k-2\right\}},\omega^{[2],\left\{k-1\right\}}-\omega^{[2],\left\{k-2\right\}}\right)\right\|_{PN},
\end{align*}
by taking $s$ small enough, we have
\begin{align*}
  &\left\|\left(\nu^{[2],\left\{k\right\}}-\nu^{[2],\left\{k-1\right\}},c_r^{[2],\left\{k\right\}}-c_r^{[2],\left\{k-1\right\}},\omega^{[2],\left\{k\right\}}-\omega^{[2],\left\{k-1\right\}}\right)\right\|_{PN}\\
  \le&\frac{1}{2}\left\|\left(\nu^{[2],\left\{k-1\right\}}-\nu^{[2],\left\{k-2\right\}},c_r^{[2],\left\{k-1\right\}}-c_r^{[2],\left\{k-2\right\}},\omega^{[2],\left\{k-1\right\}}-\omega^{[2],\left\{k-2\right\}}\right)\right\|_{PN}.
\end{align*}
Then by fixed point theorem, we can see that there exists $\left(\nu^{[2]},c_r^{[2]},\omega^{[2]}\right)$ that solve the nonlinear system \eqref{eq-system-nonlinear}.

Then we complete the proof.

\end{proof}

\section*{Acknowledgements}
H. Li is partially supported by NSF of China under Grant 12501287.  Y. Wang is partially supported by NSF of China under Grant 12471200. Z. Zhang is partially supported by NSF of China under Grant 12288101.

\section*{Data Availability}
The manuscript has no associated data.
\section*{Ethics declarations}
{\bf Conflict of interest:} The authors declare that they have no conflict of interest.

\bibliographystyle{siam.bst} 
\bibliography{references.bib}

\end{document}